\documentclass[11pt,a4paper]{article}
\usepackage{a4wide} 

\usepackage{longtable} 
\usepackage{amsmath,amssymb,amsthm,                
            filecontents, cite}
\usepackage{hyperref}
\hypersetup{colorlinks=true,citecolor=blue,linkcolor=blue, urlcolor=blue}
\usepackage{enumerate}
\usepackage{calc}

\usepackage{tikz}
\usetikzlibrary{arrows} 
\usetikzlibrary{decorations.markings}
\usetikzlibrary{calc}

\usepackage{url}

\newcommand{\defendsymb}{$\qed$}
\newcommand{\defend}{\hfill\defendsymb}
\newcommand{\defenddisp}{\tag*{\defendsymb}}

 \newcommand{\statelargeell}{Consider any $r>1$. There is an $\ell_0$, depending on~$r$, such that the following holds for any $\ell\geq \ell_0$. }  
  \newcommand{\statelargen}{Consider any $r>1$. There is an $n_0$, depending on~$r$, such that the following holds. }   
\newcommand{\letlargen}{Let $n_0$ be an integer which is sufficiently large with respect to~$r$. }

\newcommand{\megastar}[1][k , \ell, m]{\ensuremath{{\mathcal M}_{#1}}}
\newcommand{\mf}[3]{{\mathcal{G}}_{#1,#2,#3}}
\newcommand\metaf{\mf{k}{\ell}{m}}

\renewcommand{\Pr}{\mathbb{P}}  
\newcommand{\E}{\mathbb{E}}

\newcommand{\Zzero}{\mathbb{Z}_{\geq 0}}  
\newcommand{\Zone}{\mathbb{Z}_{\geq 1}}

\newtheorem{theorem}{Theorem}
\newtheorem{lemma}[theorem]{Lemma}
\newtheorem{corollary}[theorem]{Corollary}
\newtheorem{proposition}[theorem]{Proposition}
\theoremstyle{definition}
\newtheorem{definition}[theorem]{Definition}
\newtheorem{observation}[theorem]{Observation}
\newtheorem{remark}[theorem]{Remark}

\newcommand\Z{\mathbb{Z}}
\newcommand\R{\mathbb{R}} 

\newcommand{\Tact}[2]{T_{\mathsf{a}}^{#1}(#2)}

\newcommand{\Tbirth}[1]{T_{\mathsf{b}}^{#1}}

\newcommand{\Ti}[1]{T^{t_0,j}_{\mathsf{c}}(#1)}  

 \newcommand{\Tdeath}[1]{T_{\mathsf{d}}^{#1}}
 
 \newcommand{\Tend}{T_{\mathsf{e}}}

\newcommand{\TTend}[1]{T_{\mathsf{end}}^{#1}}
\newcommand{\primetend}[1]{t_{\mathsf{end}}^{\prime}}

\newcommand{\Tinact}[2]{T_{\mathsf{in}}^{#1}(#2)} 
 
\newcommand{\Iabs}{h^{t_0,j}_{\mathsf{in}}}  
\newcommand{\Tabs}{T^{t_0,j}_{\mathsf{in}}} 
\newcommand{\Tabsnoh}[1]{T_{\mathsf{in}}^{#1}}

\newcommand{\TTmut}[1]{T_{\mathsf{m}}^{#1}}
\newcommand{\tmut}[1]{t_{\mathsf{m}}^{#1}} 
\newcommand{\Tmut}[3]{T_{\mathsf{m}}^{#2,#1}(#3)}

\newcommand{\tmax}{t_{\mathsf{max}}}

\newcommand{\TTnmut}[1]{T_{\mathsf{n}}^{#1}} 
\newcommand{\Tnmut}[3]{T_{\mathsf{n}}^{#2,#1}(#3)}
\newcommand{\tnmut}[3]{t_{\mathsf{n}}^{#2,#1}(#3)}

\newcommand\Tpabs{T_\mathsf{pa}} 
\newcommand\tpabs{t'_\mathsf{pa}}

\newcommand{\Tstart}{T_{\mathsf{s}}} 
\newcommand{\tspawn}{t_{\mathsf{sp}}} 

\newcommand{\Tvone}[1]{T_{v_1}^{#1}}

\newcommand{\tx}{t_{x_0}}

\newcommand{\W}[2]{W_{\mathsf{a}}^{#1}(#2)} 

\newcommand{\Win}[2]{W_{\mathsf{in}}^{#1}(#2)} 
\newcommand{\J}[2]{J^{#1}(#2)} 

\newcommand{\Wnoh}[1]{W_{\mathsf{a}}^{#1}}

\newcommand{\filt}{\mathcal{F}}
    
\newcommand{\head}{\mathcal{H}}
\newcommand{\transmatrix}{R}

\newcommand{\mc}[1]{M_{#1}}
\newcommand{\nc}[1]{N_{#1}}
\newcommand{\smc}[1]{M_{#1}^*}
\newcommand{\snc}[1]{N_{#1}^*}
\newcommand{\smbc}[1]{\overline{M}_{#1}^*}
\newcommand{\snbc}[1]{\overline{N}_{#1}^*}
\newcommand{\vmc}[2]{M_{(#1,#2)}}
\newcommand{\vnc}[2]{N_{(#1,#2)}}
\newcommand{\vsmc}[2]{M_{(#1,#2)}^*}
\newcommand{\vsnc}[2]{N_{(#1,#2)}^*}

\newcommand{\moranclocks}{\mathcal{C}}
\newcommand{\clocks}{\moranclocks}
\newcommand{\inclocks}{\mathcal{C}^*_{\text{mut}}}
\newcommand{\outclocks}{{\mathcal{C}}^*_{\text{nmut}}}

\newcommand{\iin}[3]{{i}_{\mathsf{m}}(#1,#2,#3)}
\newcommand{\iout}[3]{{i}_{\mathsf{n}}(#1,#2,#3)} 
\newcommand{\siin}[2]{{i}_{\mathsf{m}}(#1,#2)}
\newcommand{\siout}[2]{{i}_{\mathsf{n}}(#1,#2)}

\newcommand{\myref}[2]{(\textnormal{#1}\ref{#2})}

\newcommand{\sst}[3]{{\mathcal S}_{#1,#2,#3}}
\newcommand{\sstar}{\sst{k}{\ell}{m}}

\newcommand{\len}[1]{\mathrm{len}(#1)}

\newcommand{\extinct}{\zeta}
\newcommand{\extinctreg}{\extinct_{\mathrm{reg}}}
\newcommand{\extinctmeg}{\extinct_{\mathcal M}}
\newcommand{\rhoreg}{\rho_{\mathrm{reg}}}

\newcommand{\calC}{\mathcal{C}}
\newcommand{\calE}{\mathcal{E}}  
\newcommand{\calR}{\mathcal{R}}

\newcommand{\Y}{Y^{t_0,j}}

\newcommand{\floor}[1]{\lfloor #1 \rfloor}
\newcommand{\ceil}[1]{\lceil #1 \rceil}
\newcommand{\Floor}[1]{\left\lfloor #1 \right\rfloor}

\newcommand{\sP}[1]{\ensuremath{\mathcal{P}_{#1}}} 
\newcommand{\Pmutclock}{\mathcal{P}_1}
\newcommand{\Pnmutclock}{\mathcal{P}_2}
\newcommand{\Pcentremut}{\mathcal{P}_3}
\newcommand{\Pfastabsorb}{\mathcal{P}_4}
\newcommand{\Pvinitspawn}{\mathcal{P}_5}
\newcommand{\Pstrainfire}{\mathcal{P}_6}
\newcommand{\Pstrainlife}{\mathcal{P}_7}
\newcommand{\Pstraindanger}{\mathcal{P}_8}
\newcommand{\Ptotalmut}[1]{\ensuremath{\mathcal{P}_1(#1)}}
\newcommand{\Presmut}[1]{\ensuremath{\mathcal{P}_2(#1)}}
\newcommand{\Pcliquefull}[1]{\ensuremath{\mathcal{P}_3(#1)}}
\newcommand{\Pmostcliques}[1]{\ensuremath{\mathcal{P}_4(#1)}}

\let\phi=\varphi

\date{5 May 2016}
\title{Amplifiers for the Moran Process\thanks
{ A preliminary version of this paper will appear in Track A of ICALP 2016 (best paper prize).  The research leading to these results has received funding from the European Research Council under
  the European Union's Seventh Framework Programme (FP7/2007--2013) ERC grant agreement no.\ 334828. The paper  reflects only the authors' views and not the views of the ERC or the European Commission.
  The European Union is not liable for any use that may be made of the information contained therein.
  Department of Computer Science, University of Oxford, Wolfson Building, Parks Road, Oxford, OX1~3QD, UK.  }
}
\author{Andreas Galanis \quad Andreas G\"obel \quad Leslie Ann Goldberg \\
John Lapinskas \quad David Richerby}

\begin{document}

\maketitle

\begin{abstract}
The Moran process, as studied by Lieberman, Hauert and Nowak, 
is a  randomised algorithm modelling the spread of genetic mutations in populations. 
The  algorithm runs on  an underlying graph where individuals correspond to vertices.
Initially, one vertex (chosen uniformly at random) possesses a mutation, with
fitness $r>1$. All other individuals have fitness~$1$.
During each step of the algorithm, an individual is chosen with probability
proportional to its fitness, and its state (mutant or non-mutant) is passed on to an
out-neighbour which is chosen uniformly at random.
If the underlying graph is strongly connected then the algorithm will  eventually reach
\emph{fixation}, in which all individuals are mutants, or \emph{extinction}, in which 
no individuals are mutants.  
 An infinite family of directed graphs 
is said to be \emph{strongly amplifying} if, for every $r>1$, the extinction probability tends
to~$0$ as the number of vertices increases. A formal definition is provided in the paper.
Strong amplification is a rather surprising property --- it means that 
in such graphs, the fixation probability of a uniformly-placed initial mutant
tends to~$1$ even though the initial mutant only has a fixed
selective advantage of $r>1$ (independently of~$n$). The name ``strongly amplifying'' comes from the fact that this
selective advantage is ``amplified''.
Strong amplifiers have received quite a bit of attention, and
Lieberman et al.\ proposed two potentially strongly-amplifying families --- 
superstars and metafunnels.
Heuristic arguments have been published, arguing that there
are infinite families of superstars that are strongly amplifying.
The same has been claimed for metafunnels.
In this paper, we give the first rigorous proof that there is an infinite family of directed graphs
that is strongly amplifying. We call the graphs in the family ``megastars''.  When the algorithm is run on an $n$-vertex graph in this family, starting with a uniformly-chosen
mutant, the  extinction probability    is roughly
$n^{-1/2}$ (up to logarithmic factors).
We prove that all infinite families of superstars and metafunnels have
 larger extinction probabilities (as a function of~$n$). 
Finally, we prove that our analysis of megastars is fairly tight --- there is
no infinite family of megastars such that the Moran algorithm
gives a smaller extinction probability (up to logarithmic factors).
Also, we provide a counter-example which clarifies the literature concerning the isothermal theorem of Lieberman et al. 
 \end{abstract}
 
\section{Introduction}
\label{sec:intro}

This paper is about a randomised algorithm called the Moran process.
This algorithm was introduced in biology~\cite{Mor1958:Moran, LHN2005:EvoDyn}
to model the spread of genetic mutations in populations.
Similar algorithms have
been used to model the spread of epidemic diseases, the behaviour of
voters, the spread of ideas in social networks, strategic interaction
in evolutionary game theory, the emergence of monopolies, 
and cascading
failures in power grids and transport networks
\cite{ARLV2001:Influence, Ber2001:Monopolies, Gin2000:GT,
  KKT2003:Influence, Lig1999:IntSys}.

There has been past work about analysing the expected
convergence time of the algorithm~\cite{DGMRSS2014:approx, DGRS}.
In fact, the fast-convergence result of~\cite{DGMRSS2014:approx} implies
that when the algorithm is run on an undirected graph,
and the ``fitness'' of the initial mutation is some constant $r>1$,
there is an FPRAS for the   ``fixation probability'', which is the probability that
a randomly-introduced initial mutation spreads throughout the whole graph.

This paper answers an even more basic question, originally raised in~\cite{LHN2005:EvoDyn},
about the long-term behaviour of the algorithm when it is run on 
directed graphs. In particular, the question is
whether there even exists an infinite family of (directed) graphs 
such that, when the algorithm is run on an $n$-vertex graph in this family, the
fixation probability is $1-o(1)$, as a function of~$n$.
A heuristic argument that this is the case was given in~\cite{LHN2005:EvoDyn}, but
a counter-example to the argument (and to the hypothesized bound on the fixation probability) was
given in \cite{DGMRSS2013:Superstars}.
A further heuristic argument (with a revised bound) was given in \cite{JLH2015:Superstars}.
Here we give the first rigorous proof
that there is indeed a family of ``amplifiers'' with fixation probability $1-o(1)$.  
Before describing this, and the other results of this paper, we describe the model.

The Moran algorithm has a parameter~$r$ which is the fitness of
``mutants''. All non-mutants have fitness~$1$.  
The algorithm runs on a directed graph. In the initial state,
one vertex is chosen uniformly at random to become a mutant.
After this, the algorithm runs in discrete steps as follows.
At each step, a vertex
is selected at random, with probability proportional to its fitness.
Suppose that this is vertex~$v$. Next, an   out-neighbour $w$ of~$v$
is selected uniformly at random. Finally, the state of vertex~$v$ (mutant or non-mutant) 
is copied to vertex~$w$.

If the graph is finite and strongly connected then with probability~$1$, the
process will 
either reach the state where there are only mutants (known as
\emph{fixation}) or 
it will reach the state where there are only non-mutants (\emph{extinction}).  In this
paper, we are interested in the probability that fixation is reached,
as a function of the mutant fitness~$r$, given the topology of the
underlying graph.  If $r<1$ then the single initial mutant has lower
fitness than  the non-mutants that occupy every other vertex in the
initial configuration, so the mutation is overwhelmingly likely to go extinct.
If $r=1$,  
an easy symmetry argument shows that
the fixation
probability is~$\tfrac1n$ in any strongly connected graph on
$n$~vertices~\cite[Lemma~1]{DGMRSS2014:approx}.\footnote{The result is
  stated in~\cite{DGMRSS2014:approx} for undirected graphs but the
  proof goes through unaltered for strongly connected directed
  graphs.}  Because of this, we restrict attention to the case $r>1$.
Perhaps surprisingly, a single advantageous mutant can have a very
high probability of reaching fixation, despite being heavily
outnumbered in the initial configuration.

A directed graph is said to be \emph{regular}
if there is some positive integer~$d$ so that the in-degree and out-degree of
every vertex is~$d$.
In a strongly connected regular graph on $n$~vertices, the fixation
probability of a mutant with fitness $r>1$ when the Moran algorithm is run is given by
\begin{equation}\label{eq:fp-regular}
    \rhoreg(r,n) = \frac{1-\tfrac1r}{1-\tfrac1{r^n}},
\end{equation}
so the extinction probability of such a mutant is given by
\begin{equation}\label{eq:ep-regular}
\extinctreg(r,n) = \frac{\tfrac1r-\tfrac{1}{r^n}}{1-\tfrac1{r^n}}.\end{equation}
Thus, in the limit, as~$n$ tends to~$\infty$, the extinction probability tends to~$1/r$.
To see why~\eqref{eq:fp-regular} and~\eqref{eq:ep-regular} hold,  note that, for every configuration of mutants, the number of
edges from mutants to non-mutants is the same as the number of edges
from non-mutants to mutants.  Suppose that the sum of the individuals'
fitnesses is~$W$ and consider an edge $(u,v)$.  If $u$~is a mutant in
the current state, it is selected to reproduce with probability~$r/W$,
and, if this happens, the offspring is placed at~$v$ with
probability~$1/d$.  Similarly, if $u$~is not a mutant, reproduction
happens along $(u,v)$ with probability~$1/(dW)$.  So, in any state,
the number of mutants is $r$~times as likely to increase at the next
step of the process  as it is to decrease.
If we observe the number of mutants
every time it changes, the resulting stochastic process is a random
walk on the integers, that starts at~$1$, absorbs at $0$ and~$n$,
increases with probability~$\tfrac{r}{r+1}$ and decreases with
probability~$\frac1{r+1}$. 
It is well known that this walk absorbs 
at~$n$ with probability~\eqref{eq:fp-regular} and
at~$0$ with probability~\eqref{eq:ep-regular}.   In particular, 
the undirected $n$-vertex complete graph  is regular.  
Thus, by~\eqref{eq:ep-regular}, its extinction probability tends to~$1/r$. 

When the Moran process is run on 
non-regular graphs the  extinction probability
may be  quite a bit lower than~$1/r$.   
Consider the undirected $(n+1)$-vertex ``star'' graph,
which  consists of single centre vertex that is connected by edges to each of $n$~leaves. 
In the limit as
$n\to\infty$, the $n$-leaf star has  extinction probability $\tfrac1{r^2}$
\cite{LHN2005:EvoDyn, BR2008:FixProb}. 
Informally, the reason that the extinction probability is  so small
is that the initial mutant is likely to be placed in a leaf, and, at each step,
a mutation at a leaf is relatively unlikely to be overwritten.

Lieberman et al.~\cite{LHN2005:EvoDyn} refer to graphs which have smaller extinction probability than~\eqref{eq:ep-regular}
(and therefore have larger fixation probability than~\eqref{eq:fp-regular})
as \emph{amplifiers}.
The terminology comes from the fact that the selective 
advantage of the mutant is being ``amplified'' in such graphs.

The purpose of this paper is to explore the long-term behaviour of the Moran process by 
quantifying   how good amplifiers can be.
For this, it helps to have some more formal definitions.

\begin{definition}
Consider a function $\extinct(r,n)\colon \mathbb{R}_{>1} \times \Zone \rightarrow \mathbb{R}_{\geq 0}$.
An infinite family~$\Upsilon$ of directed graphs
is said to be \emph{up-to-$\extinct$ fixating}
if,  for every $r>1$, there is an $n_0$ (depending on~$r$)
so that, for  every  graph  $G\in \Upsilon$ with $n\geq n_0$ vertices, 
the following is true:
When the Moran process is run on~$G$,
starting from a uniformly-random initial mutant,
the extinction probability is at most $\extinct(r,n)$.
\defend{}
\end{definition}
 
 Equation~\eqref{eq:ep-regular} demonstrates that
 the infinite family of strongly-connected regular graphs
 is up-to-$\extinctreg$ fixating and 
 since $\extinctreg \leq 1/r$, this family 
  is also up-to-$1/r$ fixating.
Informally, an infinite family of graphs is said to be \emph{amplifying}
if it is up-to-$\extinct$ fixating for a
function $\extinct(r,n)$ which is
``smaller''  than $\extinctreg(r,n)$. Here is the formal definition.

\begin{definition} 
 An infinite family of directed graphs
is   \emph{amplifying} 
if it is up-to-$\extinct$ fixating 
for a function $\extinct(r,n)$ which, for every $r>1$, satisfies
$\lim_{n\rightarrow \infty} \extinct(r,n) < 1/r$.
\defend{}
\end{definition}

The infinite family of graphs containing all undirected stars
(which can be viewed as directed graphs with edges in both directions) is   
up-to-$\extinct(r,n)$ fixating for
a function $\extinct(r,n)$ satisfying
$\lim_{n\rightarrow \infty} \extinct(r,n) = 1/r^2$, so
this family of graphs is amplifying.

Lieberman et al.~\cite{LHN2005:EvoDyn} were interested in infinite families of
digraphs for which the extinction probability tends to~$0$, prompting the following definition.

\begin{definition}
An infinite family of directed graphs is
  \emph{strongly amplifying}
if it is up-to-$\extinct$ fixating 
for a function $\extinct(r,n)$ which, for every $r>1$, satisfies
$\lim_{n\rightarrow \infty} \extinct(r,n) =0$. \defend{}\end{definition}
  
Note that the infinite family of undirected stars is not strongly amplifying since the extinction probability 
of stars tends to~$1/r^2$ rather than to~$0$.

Prior to this paper,  
there was no (rigorous) proof that a strongly amplifying family of digraphs exists
(though there were heuristic arguments, as we explain later). Proving rigorously
that there is an infinite family of directed graphs that is strongly amplifying for the Moran algorithm is
one of our main contributions.
 
Lieberman et al.~\cite{LHN2005:EvoDyn} 
produced  
good intuition about strong amplification and
defined two infinite families of graphs --- superstars and metafunnels ---
from which it turns out that strongly amplifying families can be constructed.
It is extremely difficult to analyse the Moran process on these 
families, due mostly to the complexity of the graphs, and the
difficulty of dealing with 
issues of dependence and concentration.
Thus, all previous arguments have been heuristic.
For completeness, we discuss these heuristic arguments in Section~\ref{sec:JLH}. 

In this paper, we define a new family of digraphs called megastars.
The definition of megastars is heavily influenced by the superstars of
Lieberman et al.
Our main theorem is the following.

\begin{theorem}
\label{thm:mainone}
There exists an infinite family of  megastars that is strongly amplifying.
\end{theorem}

 Megastars
  are not easier to analyse than superstars or metafunnels.
The reason for our focus on  this class of graphs is that  it turns out to be provably 
better  amplifying than any of the previously-proposed families.
We will present several theorems along these lines.
Before doing so, we define the classes of graphs.

\subsection{Metafunnels, superstars and megastars} 
 
\label{sec:intro:graphs}

\subsubsection{Metafunnels}\label{def:meta}
 
We start by defining the metafunnels of~\cite{LHN2005:EvoDyn}.
 Let $k$, $\ell$ and $m$ be positive integers.
The \emph{$(k,\ell,m)$-metafunnel} is 
the directed graph $\metaf$ defined as follows. (See Figure~\ref{fig:mf}.)

The vertex set $V(\metaf)$  is
the union of $k+1$ disjoint sets $V_0,\ldots,V_k$.
The set $V_0$ contains the single vertex~$v^*$ which is called the \emph{centre vertex}.
For $i \in [k]$, $V_i$ is the union of
$\ell$ disjoint sets $V_{i,1},\ldots,V_{i,\ell}$, each of which has size $m^i$.
The edge set of $\metaf$ is
\[ (V_0 \times V_k) \cup (V_1 \times V_0) \cup \bigcup_{i\in[k-1]}\bigcup_{j\in[\ell]}  (V_{i+1,j} \times V_{i,j})\,.\]
 
Lieberman et al.\ refer to metafunnels with $\ell=1$ as ``funnels''. 
 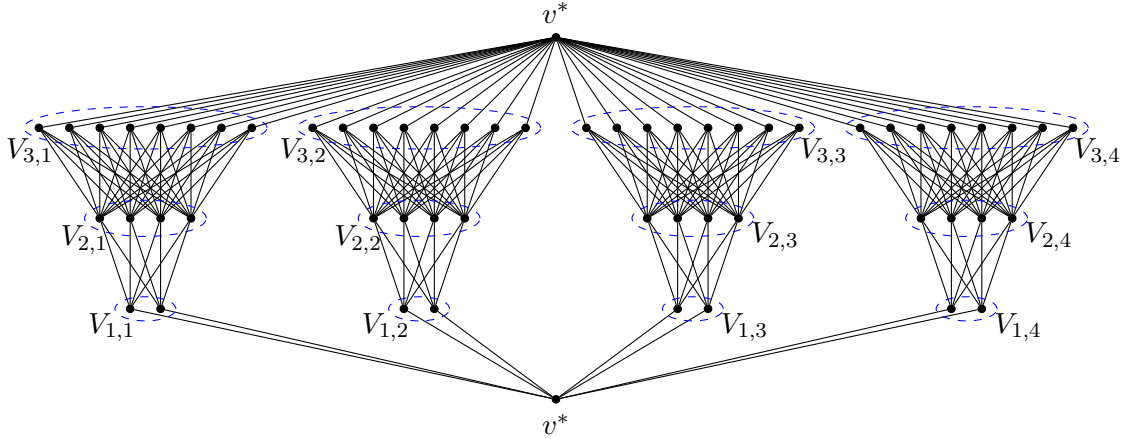
\begin{figure} 
\begin{center}
\begin{tikzpicture}[scale=.4, dedge/.style={ ->, >=stealth}]
\tikzstyle{vertex}=[fill=black, draw=black, circle, inner sep=1pt]

\node[vertex] (centt) at (5,1 )[label=90:{$v^*$}] {};
\node[vertex] (centb) at (5,-11)[label=-90:{$v^*$}] {};

%Middle Leaf

% Levels 
\foreach \x in {-3,-2,-1,0,1,2,3,4} \node[vertex] (\x A) at (\x,-2) {};
\foreach \x in {-1,0,1,2} \node[vertex] (\x B) at (\x,-5) {};
\foreach \x in {0,1} \node[vertex] (\x C) at (\x,-8) {};

%Connection of levels
\foreach \x in {-3,-2,-1,0,1,2,3,4} \draw (centt)--(\x A);
\foreach \x in {-3,-2,-1,0,1,2,3,4} \foreach \y in {-1,0,1,2} \draw (\x A)--(\y B);
\foreach \x in {-1,0,1,2} \foreach \y in {0,1} \draw (\x B)--(\y C);
\foreach \x in {0,1} \draw (\x C)--(centb);

%Labeling and marking
\draw[dashed,blue] (.5,-2) ellipse (4cm and .7cm);
\node at (-3.3,-2.8) {$V_{3,2}$};
\draw[dashed,blue] (.5,-5) ellipse (2cm and .6cm);
\node at (-1.5,-5.7) {$V_{2,2}$};
\draw[dashed,blue] (.5,-8) ellipse (1cm and .4cm);
\node at (-.6,-8.6) {$V_{1,2}$};

% Left Leaf

%Levels
\foreach \x in {-12,-11,-10,-9,-8,-7,-6,-5} \node[vertex] (\x A) at (\x,-2) {};
\foreach \x in {-10,-9,-8,-7} \node[vertex] (\x B) at (\x,-5) {};
\foreach \x in {-9,-8} \node[vertex] (\x C) at (\x,-8) {};

%Connection of levels
\foreach \x in {-12,-11,-10,-9,-8,-7,-6,-5} \draw (centt)--(\x A);
\foreach \x in {-12,-11,-10,-9,-8,-7,-6,-5} \foreach \y in {-10,-9,-8,-7} \draw (\x A)--(\y B);
\foreach \x in {-10,-9,-8,-7} \foreach \y in {-9,-8} \draw (\x B)--(\y C);
\foreach \x in {-9,-8} \draw (\x C)--(centb);

%Labeling and marking
\draw[dashed,blue] (-8.5,-2) ellipse (4cm and .7cm);
\node at (-12.3,-2.8) {$V_{3,1}$};
\draw[dashed,blue] (-8.5,-5) ellipse (2cm and .6cm);
\node at (-10.5,-5.7) {$V_{2,1}$};
\draw[dashed,blue] (-8.5,-8) ellipse (1cm and .4cm);
\node at (-9.6,-8.6) {$V_{1,1}$};

%Right Leaf

% Levels 
\foreach \x in {6,7,8,9,10,11,12,13} \node[vertex] (\x A) at (\x,-2) {};
\foreach \x in {8,9,10,11} \node[vertex] (\x B) at (\x,-5) {};
\foreach \x in {9,10} \node[vertex] (\x C) at (\x,-8) {};

%Connection of levels
\foreach \x in {6,7,8,9,10,11,12,13} \draw (centt)--(\x A);
\foreach \x in {6,7,8,9,10,11,12,13} \foreach \y in {8,9,10,11} \draw (\x A)--(\y B);
\foreach \x in {8,9,10,11} \foreach \y in {9,10} \draw (\x B)--(\y C);
\foreach \x in {9,10} \draw (\x C)--(centb);

\draw[dashed,blue] (9.5,-2) ellipse (4cm and .7cm);
\node at (13.8,-2.8) {$V_{3,3}$};
\draw[dashed,blue] (9.5,-5) ellipse (2cm and .6cm);
\node at (12.2,-5.6) {$V_{2,3}$};
\draw[dashed,blue] (9.5,-8) ellipse (1cm and .4cm);
\node at (11.2,-8.6) {$V_{1,3}$};

%Far right Leaf

% Levels 
\foreach \x in {15,16,17,18,19,20,21,22} \node[vertex] (\x A) at (\x,-2) {};
\foreach \x in {17,18,19,20} \node[vertex] (\x B) at (\x,-5) {};
\foreach \x in {18,19} \node[vertex] (\x C) at (\x,-8) {};

%Connection of levels
\foreach \x in {15,16,17,18,19,20,21,22} \draw (centt)--(\x A);
\foreach \x in {15,16,17,18,19,20,21,22} \foreach \y in {17,18,19,20} \draw (\x A)--(\y B);
\foreach \x in {17,18,19,20} \foreach \y in {18,19} \draw (\x B)--(\y C);
\foreach \x in {18,19} \draw (\x C)--(centb);

%Labeling and marking
\draw[dashed,blue] (18.5,-2) ellipse (4cm and .7cm);
\node at (22.8,-2.8) {$V_{3,4}$};
\draw[dashed,blue] (18.5,-5) ellipse (2cm and .6cm);
\node at (21.2,-5.6) {$V_{2,4}$};
\draw[dashed,blue] (18.5,-8) ellipse (1cm and .4cm);
\node at (20.2,-8.6) {$V_{1,4}$};
\end{tikzpicture}
\end{center}
\caption{The metafunnel $\mf{3}{4}{2}$.  All edges are directed downwards 
in the diagram and the centre vertex~$v^*$ is shown twice,  once at the top and once at the
bottom of the diagram.  There are $\ell=4$ copies of the basic unit, 
each of which consists of $k=3$ levels $V_{1,j}$, $V_{2,j}$ and~$V_{3,j}$, with $|V_{i,j}| = m^i = 2^i$.}
\label{fig:mf}
\end{figure}

 \subsubsection{Superstars}\label{def:super}
 
 We next define the superstars of~\cite{LHN2005:EvoDyn}. 
Let $k$, $\ell$ and $m$ be positive integers. The \emph{$(k,\ell,m)$-superstar} is the directed graph $\sstar$ defined as follows. (See Figure~\ref{fig:superstar}.) The vertex set $V(\sstar)$ of $\sstar$ is the disjoint union of $\ell$ size-$m$ sets $R_1, \dots, R_\ell$  
(called \emph{reservoirs})
together with $k\ell$ vertices $v_{1,1}, v_{1,2}, \dots, v_{\ell, k}$ and a single centre vertex~$v^*$.   The edge set of $\sstar$ is given by
\[E(\sstar) = \bigcup_{i=1}^{\ell} \Bigg((\{v^*\} \times R_i) \cup (R_i \times \{v_{i,1}\}) \cup \{(v_{i,j}, v_{i,j+1}) \mid j \in [k-1]\} \cup \{(v_{i,k}, v^*)\}\} \Bigg).\]
\begin{figure} 
\begin{center}
\begin{tikzpicture}[scale=.7,node distance = 1.5cm, dedge/.style={ ->, >=stealth'}]
\tikzstyle{vertex}=[fill=black, draw=black, circle, inner sep=1.5pt]

%Central Vertices
\node[vertex] (centt) at (3,1 )[label=90:{$v^*$}] {};
\node[vertex] (centb) at (3,-6.5)[label=-90:{$v^*$}] {};

%Middle Reservoir
\foreach \x in {1,2,3,4,5} \node[vertex] (\x) at (\x,-0.5) {};
\foreach \x in {1,2,3,4,5} \draw[dedge] (centt)--(\x) ;
\draw[dashed,blue] (3,-0.5) ellipse (2.7cm and .7cm);
\node at (.8,-1.3) {$R_2$};

%Middle path
\foreach \x in {1,2,3,4} \node[vertex] (\x M) at (3,-\x-1) [label=0:{$v_{2,\x}$}] {};
\foreach \x in {1,2,3,4,5} \draw[dedge] (\x)--(1M);
\draw[dedge] (1M)--(2M); 
\draw[dedge] (2M)--(3M);
\draw[dedge] (3M)--(4M);
\draw[dedge] (4M)--(centb);
%\draw[dashed] (3,-3.5) ellipse (.7cm and 2.1cm);

%Left Reservoir
\foreach \x in {-1,-2,-3,-4,-5} \node[vertex] (\x) at (\x,-0.5) {};
\foreach \x in {-1,-2,-3,-4,-5} \draw[dedge] (centt)--(\x) ;
\draw[dashed,blue] (-3,-0.5) ellipse (2.7cm and .7cm);
\node at (-5.2,-1.3) {$R_1$};

%Left path
\foreach \x in {1,2,3,4} \node[vertex] (\x L) at (-3,-\x-1)[label=180:{$v_{1,\x}$}] {};
\foreach \x in {-1,-2,-3,-4,-5} \draw[dedge] (\x)--(1L);
\draw[dedge] (1L)--(2L);
\draw[dedge] (2L)--(3L);
\draw[dedge] (3L)--(4L);
\draw[dedge] (4L)--(centb);
%\draw[dashed] (-3,-3.5) ellipse (.7cm and 2.1cm);

%Right Reservoir
\foreach \x in {7,8,9,10,11} \node[vertex] (\x) at (\x,-0.5) {};
\foreach \x in {7,8,9,10,11} \draw[dedge] (centt)--(\x) ;
\draw[dashed,blue] (9,-0.5) ellipse (2.7cm and .7cm);
\node at (6.8,-1.3) {$R_3$};

%Right Path
\foreach \x in {1,2,3,4} \node[vertex] (\x R) at (9,-\x-1)[label=0:{$v_{3,\x}$}] {};
\foreach \x in {7,8,9,10,11} \draw[dedge] (\x)--(1R);
\draw[dedge] (1R)--(2R);
\draw[dedge] (2R)--(3R);
\draw[dedge] (3R)--(4R);
\draw[dedge] (4R)--(centb);
%\draw[dashed] (9,-3.5) ellipse (.7cm and 2.1cm);
\end{tikzpicture}
\end{center}
\caption{The superstar $\sst{4}{3}{5}$, with $\ell=3$ reservoirs $R_1$, $R_2$ and~$R_3$, each of size $m=5$, connected by a path with $k=4$ vertices to~$v^*$.  The centre
vertex~$v^*$ is shown twice, at both the top and bottom of the diagram.}\label{fig:superstar}\end{figure}
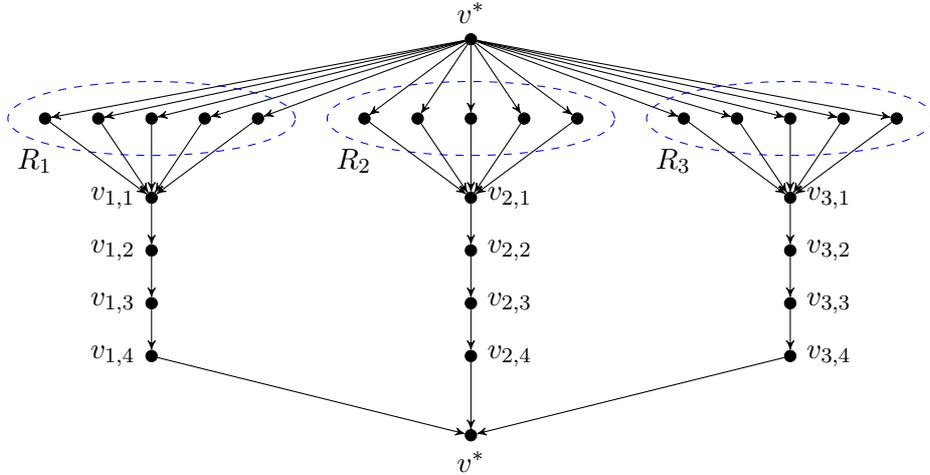

\subsubsection{Megastars}\label{def:mega}

Finally, we define the new class of megastars, which turn out to be provably-better
amplifiers than either metafunnels or superstars.
The intuition behind the design of this class of graphs is that the path $v_{i,1} v_{i,2} \dots v_{i,k}$ 
linking the $i$'th reservoir~$R_i$ of a superstar to the centre vertex~$v^*$ is good
for amplifying but that a clique is even better.

Let $k$, $\ell$ and $m$ be positive integers. 
The  $(k , \ell, m)$-megastar is 
the directed graph~$\megastar$ defined as follows. (See Figure~\ref{fig:ms}.)
The vertex set $V(\megastar)$ of $\megastar$ is the 
disjoint union of $\ell$ sets $R_1,\ldots, R_\ell$ of size $m$, called reservoirs,
$\ell$ sets $K_1, \ldots, K_\ell$ of size $k$, called cliques,
$\ell$ ``feeder vertices'' $a_1,\ldots,a_\ell$ and a single centre vertex $v^*$.
The edge set of $\megastar$    consists of the following edges:
\begin{itemize}
\item an edge from $v^*$ to every vertex in $R_1\cup \dots \cup R_{\ell}$,
\item for each $i\in[\ell]$, an edge from each vertex in~$R_i$ to $a_i$,
\item for each $i\in[\ell]$, an edge from $a_i$ to each vertex in $K_i$,
\item for each $i\in[\ell]$, 
edges in both directions between every pair of distinct vertices in~$K_i$,
\item an edge from every vertex in $K_1\cup\dots\cup K_{\ell}$ to~$v^*$.
 \end{itemize}

\begin{figure}
\begin{center}
\begin{tikzpicture}[scale=.7,node distance = 1.5cm, dedge/.style={ ->, >=stealth'}, nedge/.style={<->, 
>=stealth', <=stealth'}]
\tikzstyle{vertex}=[fill=black, draw=black, circle, inner sep=1.5pt]

%Central Vertices
\node[vertex] (centt) at (5,1.5 )[label=90:{$v^*$}] {};
\node[vertex] (centb) at (5,-5)[label=-90:{$v^*$}] {};

%Left Reservoir
\foreach \x in {1,2,3,4} \node[vertex] (\x) at (\x,0) {};
\foreach \x in {1,2,3,4} \draw[dedge] (centt)--(\x);

\node[vertex] (L) at (2.5,-1.5) {};
\foreach \x in {1,2,3,4} \draw[dedge] (\x)--(L);

\node at (2,-1.5) {$a_1$};

\draw[dashed,blue] (2.5,0) ellipse (2cm and .5cm);
\node at (.8,-.6) {$R_1$};
%Left Clique

\node[vertex] (LC1) at (1.5,-2.5) {};
\node[vertex] (LC2) at (3.5,-2.5) {};
\node[vertex] (LC3) at (2.5,-3) {};

\draw (LC1)--(LC2)--(LC3)--(LC1);
\foreach \x in {1,2,3} \draw[dedge] (L)--(LC\x);
\foreach \x in {2,3} \draw[dedge] (LC\x)--(centb);
\draw[dedge] (LC1) ..controls (2.5,-3.8) and (3.3,-4.5) ..(centb);
\draw[dashed,red] (2.5,-2.6) ellipse (1.4cm and .7cm);
\node at (1.3,-3.35) {$K_1$};

%Right Reservoir
\foreach \x in {6,7,8,9} \node[vertex] (\x) at (\x,0) {};
\foreach \x in {6,7,8,9} \draw[dedge] (centt)--(\x);

\node[vertex] (R) at (7.5,-1.5) {};
\foreach \x in {6,7,8,9} \draw[dedge] (\x)--(R);

\node at (8.2,-1.5) {$a_2$};

\draw[dashed,blue] (7.5,0) ellipse (2cm and .5cm);
\node at (9.2,-.6) {$R_2$};
%Right Clique

\node[vertex] (RC1) at (6.5,-2.5) {};
\node[vertex] (RC2) at (8.5,-2.5) {};
\node[vertex] (RC3) at (7.5,-3) {};

\draw (RC1)--(RC2)--(RC3)--(RC1);
\foreach \x in {1,2,3} \draw[dedge] (R)--(RC\x);
\foreach \x in {1,3} \draw[dedge] (RC\x)--(centb);
\draw[dedge] (RC2) ..controls (7.5,-3.8) and (6.7,-4.5) ..(centb);
\draw[dashed,red] (7.5,-2.6) ellipse (1.4cm and .7cm);
\node at (8.7,-3.35) {$K_2$};
\end{tikzpicture}
\end{center}
\caption{The megastar $\megastar[{3,2,4}]$, with $\ell=2$ reservoirs $R_1$ and~$R_2$, each of size $m=4$. Each reservoir~$R_i$ is attached, via the feeder vertex~$a_i$ to a clique of size $k=3$. The centre
vertex $v^*$ is shown twice, once at the top and once at the bottom of the diagram. The edges within the cliques $K_1$ and~$K_2$ are bidirectional.
}\label{fig:ms}
\end{figure}
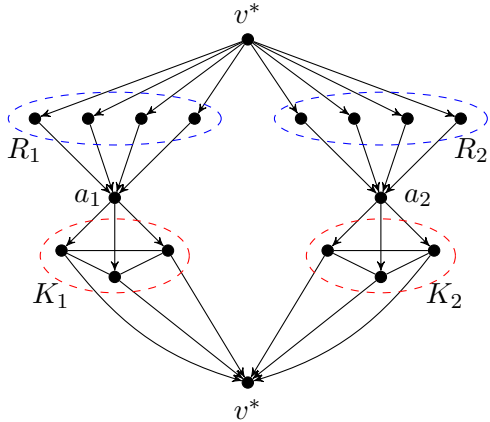

\subsection{Our results}\label{sec:ourresults}\label{sec:results}

Our main result is that there is an infinite family  of megastars that is strongly amplifying,
so we start by defining this family.
Although megastars are parameterised by three parameters, $k$, $\ell$ and $m$, 
the megastars in the family that we consider have a single parameter~$\ell$, so
we define $k$ and $m$ to be functions of~$\ell$.

\begin{definition} \label{def:newmegas}
Let $m(\ell) = \ell$ and $k(\ell) = \ceil{(\log \ell)^{23}}$. 
Let
\begin{equation*}
\Upsilon_{\mathcal{M}} = \{ \megastar[k(\ell),\ell,m(\ell)] \mid \ell \in \Z, \ell \ge 2\}.\defenddisp{}
\end{equation*}
\end{definition}

Our main result can then be stated as follows.

 \begin{theorem}
\label{thm:maintwo}
Let $\extinctmeg(r,n) = (\log n)^{23} n^{-1/2}$.
The family $\Upsilon_{\mathcal{M}}$ is  up-to-$\extinctmeg$ fixating.
\end{theorem} 

 \begin{corollary}
\label{cor:maintwo}
The family $\Upsilon_{\mathcal{M}}$ is  strongly amplifying.
\end{corollary}

The proof of Theorem~\ref{thm:maintwo} requires a complicated analysis, accounting for
dependencies and concentration.
The theorem, as stated here, follows directly from Theorem~\ref{thm:megastar} which is proved in 
Section~\ref{sec:megastar} (see Page~\pageref{thm:megastar}).
  
The reason that we studied megastars rather than the previously-introduced superstars and
metafunnels is that megastars turn out to be provably better amplifying than any
of the previously-proposed families.
To demonstrate this, we prove the following theorem about superstars.

\begin{theorem}
\label{thm:mainsuperstar}  
Let  $\extinct(r,n)$ be any function such that, for any $r>1$,
$$\lim_{n \rightarrow \infty}  \extinct(r,n) {(n \log n)}^{1/3} = 0.$$
Then there is no infinite family of superstars that is up-to-$\extinct$ fixating.
\end{theorem}
  
The function $\extinctmeg(r,n)$ from Theorem~\ref{thm:maintwo} certainly
satisfies  
$\lim_{n \rightarrow \infty}  \extinctmeg(r,n) {(n \log n)}^{1/3} = 0$,
so Theorem~\ref{thm:mainsuperstar} 
shows that there is no infinite family of superstars that is up-to-$\extinctmeg$ fixating.  
More mundanely, it shows, for example, that
if
$\extinct(r,n) = n^{-1/3} {(\log n)}^{-1}$,  then no infinite family of superstars is up-to-$\extinct$ fixating.
 Theorem~\ref{thm:mainsuperstar} is a direct consequence of Theorem~\ref{thm:star-lower}
which is proved in 
Section~\ref{sec:superstar} (see Page~\pageref{thm:star-lower}).
It turns out that analysing superstars is a little bit easier than analysing 
megastars or metafunnels, so this is the first proof that we present.
 
Taken together, Theorems~\ref{thm:maintwo} and~\ref{thm:mainsuperstar} show that
superstars are worse amplifiers than megastars. We next show that metafunnels are
\emph{substantially} worse. We start with the following simple-to-state theorem.

\begin{theorem}\label{thm:mainfirstmetafunnel}
Fix any $\delta>0$ and let $\extinct(r,n) = n^{-\delta}$.
Then there is no infinite family of metafunnels that is up-to-$\extinct$ fixating.
\end{theorem}

In fact, Theorem~\ref{thm:mainfirstmetafunnel} can be  strengthened by an exponential amount.

\begin{theorem}\label{thm:mainmetafunnel} 
Fix any $\epsilon < 1/2$ 
and let $\extinct(r,n) = n^{-1/{(\log n)}^{\epsilon}}$.
Then there is no infinite family of metafunnels that is up-to-$\extinct$ fixating.
\end{theorem}

Theorems~\ref{thm:mainfirstmetafunnel} and~\ref{thm:mainmetafunnel} 
are a direct consequence of Theorem~\ref{thm:funnel-upper} 
which is proved in 
Section~\ref{sec:metafunnel} (see Page~\pageref{thm:funnel-upper}).
In fact, Theorem~\ref{thm:funnel-upper} 
provides even tighter bounds, though these  are more difficult to state. 

 The theorems that we have already described (Theorem~\ref{thm:maintwo}, Theorem~\ref{thm:mainsuperstar}
 and Theorem~\ref{thm:mainmetafunnel}) are the main contributions of the paper.
 Together, they show that there is a family of megastars that is strongly amplifying, and
 that there are no families of superstars or metafunnels that amplify as well.
 For completeness, we present a theorem
showing that the analysis of Theorem~\ref{thm:maintwo} is fairly tight,
in the sense that there
are no infinite families of megastars
that amplify substantially better than~$\Upsilon_{\mathcal M}$ --- in particular,
our bound on extinction probability can only be improved by factors of $\log(n)$. It cannot
be improved  more substantially.

\begin{theorem}
\label{thm:mainothermegastar} 
Let $\extinct(r,n) = n^{-1/2} /(52 r^2)$. 
There is no infinite family of  megastars 
that is up-to-$\extinct$ fixating. 
\end{theorem}

Theorem~\ref{thm:mainothermegastar} follows from  Theorem~\ref{lem:megastar-lower-final}, which is 
straightforward, and is proved in  
Section~\ref{sec:megaUB} (see~Page~\pageref{lem:megastar-lower-final}).
We conclude the paper with a digression which perhaps clarifies the literature.
It is stated, and seems to be commonly believed, that an evolutionary
graph (a weighted version of the Moran process --- see Section~\ref{sec:isothermal} for details) is ``isothermal'' if and only if
the fixation probability of a mutant placed uniformly at random  is $\rhoreg(r,n)$.
This belief seems to have come from an informal statement 
of the ``isothermal theorem'' in the main body of \cite{LHN2005:EvoDyn} (the formal statement in the supplementary
material of~\cite{LHN2005:EvoDyn} is correct,  however)
and it has spread, for example, as Theorem~1 of \cite{SRJ2012:Review}.
In the final section of our paper, we clear this up 
by proving the following proposition, which says that there is a counter-example. 

\newcommand{\stateexamplethm}{There is an evolutionary graph that is not isothermal, but has fixation probability $\rhoreg(r,n)$.}
\begin{proposition}
\label{prop:counterexample}
\stateexamplethm
\end{proposition}

The definitions needed to prove Proposition~\ref{prop:counterexample} are deferred to 
Section~\ref{sec:isothermal} (see Page~\pageref{sec:isothermal}).

\subsection{Proof techniques}

As we have seen, it is easy to study the Moran process on a $d$-regular graph
by considering the transition matrix of the corresponding Markov chain (which looks like a one-dimensional random walk).
Highly symmetric graphs such as undirected stars can also be handled in a straightforward matter, by directly
analysing the transition matrix.
Superstars, metafunnels and megastars are more complicated,  
and the number of mutant-configurations is exponential,
so instead we resort to dividing the process into phases, as is typical in the study
of randomised algorithms and stochastic processes.

An essential and common trick in the area of stochastic processes
(for example, in work on the voter model)
is  moving to continuous time.
Instead of directly studying the discrete-time Moran process, one could consider the following
natural continuous-time model which was studied in~\cite{DGRS}: Given a set of mutants at time~$t$,
each vertex waits an amount of time before reproducing. For each vertex, the period
of time is chosen according to the exponential distribution with parameter equal to the
vertex's fitness, independently of the other vertices. If the first vertex to reproduce is~$v$ at time 
$t + \tau$ then, as in the standard, discrete-time version of the process, 
one of its out-neighbours~$w$ is chosen uniformly at random, the individual at~$w$ is replaced by a copy of the 
one at~$v$, and the time at which $w$~will next reproduce is exponentially distributed with parameter given by its new fitness. The discrete-time process is recovered by taking the sequence of configurations each time a vertex reproduces.
Thus, the fixation probability of the discrete-time process is
\emph{exactly the same} as the fixation probability of the continuous-time process.
So moving to the continuous-time model causes no harm.
As~\cite{DGRS} explains, analysis can be easier in the continuous-time model
because certain natural stochastic domination techniques apply in the continuous-time setting
but not in the discrete-time setting.

It turns out that moving to the model of \cite{DGRS} does not suffice for our purposes.
A major problem in our proofs is dealing with dependencies.
In order to make this feasible, we instead study a continuous-time model 
(see ``the clock process'' in Section~\ref{sec:clock})
in which every
edge of the underlying graph~$G$ is equipped with two Poisson processes, one of which is called a 
\emph{mutant clock} and the other of which is called a \emph{non-mutant clock}.
The clock process is a stochastic process in which  all of these clocks
run independently.
The continuous-time Moran process (Definition~\ref{def:contMoran})
can be recovered as a function of the times at which these clocks trigger.

Having all of these clocks available still does not give us the flexibility that
we need.
We say that a vertex~$u$ ``spawns a mutant'' in the Moran process
if, at some point in time, $u$ is a mutant, and it is selected for reproduction.
 We wish to be able to
discuss events such as the event that the vertex $u$ does not spawn a mutant until it has 
already been a mutant 
for some  particular amount of time.
In order to express such events in a clean way, 
making all conditioning explicit, we 
define additional stochastic processes called ``star-clocks'' (see Section~\ref{sec:starclock}).
All of the star-clocks run independently in the star-clock process.

In Section~\ref{sec:coupled} we provide a coupling of the star-clock process with the Moran process.
The coupling is valid in the sense that the two projections are correct --- the projection 
onto the Moran process runs according to the correct distribution and so does the projection onto the
star-clock process.
The point of the coupling is that the different star-clocks can be viewed 
as having their own ``local'' times. In particular,
there is a star-clock 
 $\smc{(u,v)}$ which controls reproductions from vertex~$u$ onto vertex~$v$ 
 \emph{during the time that $u$ is a mutant}. 
 The coupling enables us to focus on relevant parts of the stochastic process,
 making all conditioning explicit.
 
 The processes that we have described so far are all that we need to
 derive our upper bound on the fixation probability of superstars (Section \ref{sec:superstar}). This is the easiest
 of our main results.
  
 Analysing the Moran process on metafunnels is more difficult. 
 By design, the initial mutant~$x_0$ is likely to be placed
 in the ``top of a funnel'' (in the set~$V_k)$.
   In the analysis, it is useful to be
 able to create independence by considering a ``strain'' of mutants 
 which contains all of the descendants of a particular mutant spawned by~$x_0$.
Like the Moran process itself, a
strain can be viewed as a stochastic process depending
on the triggering of the clocks.
In order to facilitate the proof, we define a general notion of ``mutant process''
(Section~\ref{sec:mutantprocess}) --- so the Moran process is one example of a mutant process,
and a strain is another. The analysis of the Moran process on metafunnels involves both of
these and also a third mutant process which is essentially the bottom level of a strain
(called its head). Strains and heads-of-strains share some common properties, and they are analysed
together as ``colonies'' in Section~\ref{sec:colony}.
The analysis of the metafunnel is the technically most difficult of our results.

Fortunately, the analysis of the megastar in Section~\ref{sec:megastar}
does not require three different types of mutant processes --- it only requires one.
The process that is considered is not the Moran process itself. Instead, 
it is a modification of the Moran process called the megastar process.
The megastar process is similar to the Moran process except that the feeder vertices are forced
to be non-mutants,  except when their corresponding cliques are completely full or completely empty.
It is easy to show (see the proof of Theorem~\ref{thm:megastar}) 
that the fixation probability of the  Moran process is at least as high as the fixation probability of the  megastar process.
However, the megastar process is somewhat easier to analyse because the cliques evolve
somewhat independently.  The proof of the key lemma (Lemma~\ref{lem:megastar-fixates}) is
fairly long but it is not conceptually difficult. The point is to prove that, with high probability, the cliques fill up and
cause fixation.

\subsection{Comparison with previous work}
\label{sec:intro:previous-work}
\label{sec:JLH}

The Moran process is similar to a discrete version of
 directed percolation known as the contact process. There is a vast literature
 (e.g.,   \cite{Lig1999:IntSys,Durrett:NACS,Shah,DS}) on the contact process and other related infection processes
 such as the voter model and susceptible-infected-susceptible (SIS) epidemic models. 
  Often, the questions that are studied in these models are different from
 the question that we study here. For example, in voter systems~\cite{DS}
 the two states (mutant/non-mutant) are often symmetric (similar to our $r=1$ case)
 and the models are often studied on infinite graphs where the question
 is whether the process absorbs or not (both kinds of absorption, fixation 
 and extinction, are therefore called ``fixation'' in some of this work).
 The particular details
 of the Moran   process  
 are very important for us because the details of the algorithm determine the long-term behaviour.
 For example, unlike the Moran process, 
 in the contact process~\cite{Grimmett}, the rate at which a node becomes a non-mutant is
 typically taken to be~$1$, whereas the rate at which a node becomes a mutant is 
 proportional to the number of mutant neighbours. 
In the discrete-time versions of many commonly-studied models,
 a node is chosen randomly at each step for replacement, rather than (as in the Moran process)
 for reproduction.
  In any case, the important point for us is that the details of the algorithm are important --- results do not carry over
  from one algorithm to the other. 
  Therefore,
  we concentrate in this section on 
 previous work  about calculating the fixation probability
 of the Moran process itself.

Lieberman et al.~\cite{LHN2005:EvoDyn} studied the fixation probability of the
Moran process and introduced superstars and metafunnels.
Intuitively, a superstar is a good amplifier because
(as long as $m$ is sufficiently large) the initial mutation is likely to be placed in 
a reservoir and (as long as $\ell$ is sufficiently large) this is unlikely to
be killed quickly by the centre vertex.
Moreover, the paths of a superstar are good for amplifying the selective advantage of mutants
because,  after the infection spreads from a reservoir vertex to the beginning of
a path, it is likely to ``pick up momentum'' as it travels down the path,
arriving at the centre vertex as a chain of $\Theta(k)$ mutants
(which, taken together,  are more likely to cause the centre to spread the infection
than a single mutant arriving at the centre would be).
As we have seen (Theorems~\ref{thm:maintwo} and \ref{thm:mainsuperstar})
megastars are provably better for amplification than superstars.
The reason for this is that a clique is substantially better than a path
at doing this ``amplification''. Nevertheless, the amplifying properties of superstars
strongly influenced our decision to study megastars.

Lieberman et al.~\cite[Equation (2)]{LHN2005:EvoDyn} claimed\footnote{The reader who consults \cite{LHN2005:EvoDyn} might
wonder why ``$k$'' as written in Equation (2) of \cite{LHN2005:EvoDyn} has become $k+2$ here.
The reason is just that 
we use a slightly   different
parameterisation from that of \cite{LHN2005:EvoDyn}.
To allow appropriate comparison, we describe all previous work using the parameterisation that
we give in Section~\ref{sec:intro:graphs}.}
that for sufficiently large~$n$,
the fixation probability of a superstar with parameter~$k$
tends to  $1-r^{-(k+2)}$,
 and that ``similar results hold for the funnel and metafunnel''.
They provided a heuristic sketch-proof for the superstar, but not for the funnel or metafunnel.
Hauert~\cite[Equation (5)]{Hau2008:EvoDyn} claims specifically that
the fixation probability of funnels tends to~$1-r^{-(k+1)}$.
As far as we know, no heuristic arguments have been given for funnels or metafunnels.

In any event, D\'iaz, Goldberg, Mertzios, Richerby, Serna and
Spirakis~\cite{DGMRSS2013:Superstars} showed that the $1-r^{-(k+2)}$
claim for superstars is incorrect for the case $k=3$. In particular, for this case
they showed that the fixation probability is at most
$1 - \frac{r+1}{2r^5 + r+1}$, which is less than the originally
claimed value of $1 - r^{-5}$ for all $r\geq 1.42$.

Subsequently, Jamieson-Lane and Hauert~\cite[Equation (5)]{JLH2015:Superstars} made a
more detailed   but still heuristic\footnote{
A full discussion of the argument of Jamieson-Lane and Hauert 
(and of the obstacles to making it a rigorous proof)
are discussed in Section~\ref{sec:extra}. This section is not necessary for the rest of the paper.} 
analysis of the fixation probability
of superstars.  They claim
that for   superstars with parameter $k$ 
and with 
$\ell=m$, the fixation probability $\rho_k$ has the following bounds for fixed $r>1$,  \begin{equation}
\label{eq:JLH}
    1 -\frac{1}{r^4(k-1)(1-\tfrac1r)^2} -o(1) \leq \rho_k
        \leq 1 - \frac{1}{r^4(k - 1)}+o(1),
\end{equation} 
 where the $o(1)$ terms tend to~$0$ as $\ell\rightarrow \infty$.
They claim that their bounds are a good approximation as   long as
$k \ll \ell = m \sim \sqrt{n}$.
It is not clear exactly what ``$\ll$'' means in this context. 
Certainly there are parameter regimes
where $k = o(\ell)$ and $\ell=m\sim \sqrt{n}$
but nevertheless the
extinction probability is much larger than  the proposed upper bound  
$1/(r^4(k-1){(1-1/r)}^2)$ from~\eqref{eq:JLH}.
For example,  suppose that $\ell=m=k^{3/2}$.
In this case (see Lemma~\ref{lem:star-lower-die-instant}),
the extinction probability is at least 
$$\frac{k}{2r(m+k)}  = \frac{1}{ 2r(k^{1/2}+1)},$$
which  is larger than   
$1/(r^4(k-1){(1-1/r)}^2)$ for all sufficiently large~$k$.
Nevertheless, the bounds proposed by Jamieson-Lane and Hauert   \eqref{eq:JLH}
seem to be close to the truth when $k$ is very small compared to $\ell$ and $m$.
 
 Our Corollary~\ref{cor:superstark} 
identifies a wide class of parameters 
for which the extinction probability is provably
at least $1/(1470 r^4 k)$.  
This is weaker than the suggested bound of  Jamieson-Lane and Hauert
by a factor of $1470$.  
This constant factor is explained by the
fact that our rigorous proof needs to show concentration of all random variables.
We use lots of Chernoff bounds and other
bounds on probabilities. In  writing  the proof, we optimised readability rather than
optimising our constants, so our constants can presumably be improved.    

There is   recent work on other related aspects of the Moran process.
For example, \cite{Mertz1, Mertz2} give fixation probability bounds
 on connected \emph{undirected} graphs.  
   \cite{ACN2015:amplifiers} studies 
 amplification with respect to adversarial or ``temperature-based'' placement of the initial mutation,
 in which the ``temperature'' of a vertex is proportional to the sum of all incoming edge weights.
 Also, \cite{Mertz2} considers the extent to which the number of ``good starts'' for fixation can be bounded.

 \subsection{Outline of the paper} 
 
\begin{itemize} 
\item Section~\ref{sec:defprelim}, starting on Page~\pageref{sec:defprelim},
defines some notation and states some well-known probabilistic bounds
(Chernoff bounds and analysis of gambler's ruin) which will be used in the proof.
\item Section~\ref{sec:processes}, starting on Page~\pageref{sec:processes},
defines several stochastic processes which we use to study the Moran process.
This section is important. It is impossible to read any of the proofs without understanding
these processes.
\item Section~\ref{sec:superstar}, starting on Page~\pageref{sec:superstar},
gives an upper bound on the fixation probability of superstars.
The main result of the section is Theorem~\ref{thm:star-lower},
which immediately implies Theorem~\ref{thm:mainsuperstar}.
This is the technically easiest of our main proofs, so we present it first. 

\item Section~\ref{sec:metafunnel}, starting on Page~\pageref{sec:metafunnel},
gives a stronger upper bound on the fixation probability of metafunnels
(and hence of funnels). The main result of the section is Theorem~\ref{thm:funnel-upper},
which immediately implies 
Theorems~\ref{thm:mainfirstmetafunnel} and~\ref{thm:mainmetafunnel}. The proof of 
Theorem~\ref{thm:funnel-upper} has high-level similarity to the proof of Theorem~\ref{thm:star-lower}, but
it is much more difficult. Dependencies cause complications, and we must analyse several
mutant processes to deal with these.

\item Section~\ref{sec:megastar}, starting on Page~\pageref{sec:megastar},
establishes the existence of an infinite family of megastars which is strongly amplifying.
The main theorem is Theorem~\ref{thm:megastar}, which immediately implies
Theorem~\ref{thm:maintwo} and hence Theorem~\ref{thm:mainone}. In order to deal with dependencies, we study a
mutant process called a 
``megastar process''.
We show in the proof of Theorem~\ref{thm:megastar}
that this process is dominated by the Moran process. Thus, the main work of the section
is to prove the key lemma, Lemma~\ref{lem:megastar-fixates}, which analyses the megastar process.

\item Section~\ref{sec:megaUB}, starting on Page~\pageref{sec:megaUB}, 
gives an upper bound showing that the analysis in Section~\ref{sec:megastar} is fairly tight.
The main theorem, Theorem~\ref{lem:megastar-lower-final}, is straightforward and it immediately implies
Theorem~\ref{thm:mainothermegastar}.

\item Section~\ref{sec:isothermal}, starting on Page~\pageref{sec:isothermal},
gives a simple example of an evolutionary graph that is not isothermal but has fixation probability $\rhoreg(r,n)$ (Proposition~\ref{prop:counterexample}),
clearing up a misconception in the literature. 

\item Section~\ref{sec:extra}, starting on Page~\pageref{sec:extra},
discusses earlier heuristic analysis of superstars. 
\end{itemize}

\section{Definitions and preliminaries}
\label{sec:defprelim} 

\subsection{Notation}
 
We use $N^-(v)$ to refer to the set of in-neighbours of a vertex~$v$
and $N^+(v)$ to refer to the set of out-neighbours of~$v$.
We use $d^-(v) = |N^-(v)|$ and $d^+(v) = |N^+(v)|$.

We refer to the Lebesgue measure of a (measurable) subset $S \subseteq \R$ as the \emph{measure} of that set, and denote it by $\len{S}$.

We use base-$e$ for logarithms unless the base is given explicitly.

We write $\Zzero = \{0, 1, 2, \dots\}$, $\Zone = \{1, 2, \dots\}$, and $[n] = \{1, 2, \dots, n\}$.
 
If $b < a$, we consider the interval $[a,b]$ to be well-defined but empty. Likewise if $b \le a$, we consider the intervals $(a,b)$, $(a,b]$ and $[a,b)$ to be well-defined but empty.
We define empty sums, products, unions etc.\ to be the identities of the corresponding operations. For example, $\prod_{i=1}^0 i = 1$ and $\bigcup_{i=1}^0 A_i = \emptyset$.

  Throughout the paper, we use lower case $t$'s to denote fixed times and upper case $T$'s to denote stopping times.

\subsection{Chernoff bounds}
 
We often use the following simple bound 
which applies to any real number
$x\in[0,1]$.
\begin{equation}
\label{eq:ebounds}
x/2 \leq 1-e^{-x} \leq x.
\end{equation}
We will require the following well-known Chernoff bounds. The first appears as Theorem~5.4 of~\cite{MU}.
  
\begin{lemma}\label{lem:pchernoff} 
Let $Y$ be a Poisson random variable with parameter~$\rho\geq 0$.
 If $y > \rho$ and $z < \rho$, then
\[\Pr(Y \ge y) \le \frac{e^{-\rho}(e\rho)^y}{y^y}\qquad\text{and}\qquad
\Pr(Y \le z) \le \frac{e^{-\rho}(e\rho)^z}{z^z}\,.\]
\end{lemma}

\begin{corollary}\label{cor:pchernoff}\label{cor:pchernoff-2}
Let $Y$ be a Poisson random variable with parameter~$\rho\geq 0$.
Then $\Pr(Y \ge 2\rho) \le e^{-\rho/3}$ and $\Pr(Y \le 2\rho/3) \le e^{-\rho/16}$.
\end{corollary}
\begin{proof}
Lemma~\ref{lem:pchernoff} applied with $y=2\rho$ and $z=2\rho/3$ implies that
\begin{gather*}
\Pr(Y \ge 2\rho) \le \frac{e^{-\rho}(e\rho)^{2\rho}}{(2\rho)^{2\rho}}  = e^{-\rho}\left(\frac{e^2}{4}\right)^\rho = e^{(1 - \log 4)\rho} \le e^{-\rho/3}\,, \\
\Pr(Y \le 2\rho/3) \le \frac{e^{-\rho}(e\rho)^z}{z^z} = e^{-\rho} \left(\frac{3e}{2}\right)^{2\rho/3} = e^{-(1 - 2/3 - 2\log(3/2)/3)\rho} \le e^{-\rho/16}\,.
\qedhere
\end{gather*}
\end{proof}

\begin{corollary}\label{cor:pchernoff-3}
Let $Y$ be a Poisson random variable with parameter $\rho > 0$. If $y \geq 8\rho$, then
\[\Pr(Y \ge y) \le e^{-y}.\]
\end{corollary}
\begin{proof}
Note that $y > e^2 \rho$. Thus by Lemma~\ref{lem:pchernoff}, we have
\[\Pr(Y \ge y) \le \frac{e^{-\rho}(e\rho)^y}{y^y} \le \left(\frac{e\rho}{y}\right)^y \le e^{-y}.\qedhere\]
\end{proof}

\begin{corollary}\label{cor:expsum}
Let $s$ be a positive integer and let $Y$ be the sum of $s$ i.i.d.\@ exponential random variables, each with parameter $\lambda$.
Then, 
for any $j\geq 3s/(2\lambda)$,
$\Pr(Y < j) \geq 1-e^{-\lambda j/16}$.
\end{corollary}
\begin{proof}
First, note that 
$\Pr(Y<j) = \Pr(Y \leq j)$ since $\Pr(Y=j)=0$.
Then   $\Pr(Y\leq j)$    
is equal to the probability that   a Poisson process
with parameter~$\lambda$  triggers at least~$s$
times in  the interval $[0,j]$. This is the 
  same as the probability
 that a Poisson random variable
with parameter $\lambda j$ is at least~$s$.
Since $s \leq 2 \lambda j/3$, we can now use Corollary~\ref{cor:pchernoff-2}.
\end{proof}  
  
The following is Corollary~2.4 of~\cite{JLR}.

\begin{lemma}\label{lem:bchernoff}
Suppose that $Y$ follows the binomial distribution with $n$ Bernoulli trials, each with success probability~$p \in 
(0,1)$ and that $c>1$.  Then, for all $y \ge cnp$, $\Pr(Y \ge y) \le e^{-\varphi(c)y}$, where 
$\varphi(c)=\log{c}-1+1/c$. Note that $\varphi(2)>1/6$ and $\varphi(7)>1$.
\end{lemma}

We define the \emph{geometric distribution} as follows. Given a biased coin which comes up heads with probability $p>0$, imagine tossing it until it comes up heads. Then the total number of tosses which came up tails follows the geometric distribution with parameter $p$.

\begin{lemma}\label{lem:geo-chernoff}
Let $Y_1, \dots, Y_t$ be a sequence of i.i.d.\ geometric variables 
with parameter $p \ge 13/14$. Then
\[\Pr\big(Y_1 + \dots + Y_t \ge 14t(1-p)\big) \le e^{-14t(1-p)}\,.\]
\end{lemma}
\begin{proof}
Consider a series of independent coin tosses, each with probability $p$ of coming up heads. Then the probability that $Y_1 + \dots + Y_t \ge 14t(1-p)$ is exactly the probability that at least $\ceil{14t(1-p)}$ of the first $t+\ceil{14t(1-p)}-1$ coin tosses come up tails. By Lemma~\ref{lem:bchernoff}, the probability that at least $\ceil{14t(1-p)}$ of the first $2t$ coin tosses come up tails is at most $e^{-14t(1-p)}$, and $2t \ge t + \ceil{14t(1-p)}-1$, so the result follows.
\end{proof}

\subsection{Gambler's ruin}\label{sec:gambler}

The following analysis of the classical gambler's ruin problem is well-known.
See, for example, \cite[Chapter~XIV]{Fel1968:Probability}.

\begin{lemma}[Classical gambler's ruin]\label{lem:gambler}
Consider a random walk on $\mathbb{Z}_{\geq 0}$ 
that absorbs at~$0$ and~$a$ (for some positive integer~$a$), starts at $z\in\{0,\dots,a\}$
and from each state in $\{1,\ldots,a-1\}$
has probability $p\neq 1/2$ of  increasing (by~$1$) and
probability $q=1-p$ of  decreasing (by~$1$).
\begin{enumerate}[(i)]
\item \label{ruin:one} The probability of   reaching state $a$ is 
\[\frac{(q/p)^{z}-1}{(q/p)^a-1}\,.\]
\item \label{ruin:2} The expected number of transitions  until absorption is
\[\frac{z}{q-p} - \left(\frac{a}{q-p} \right)\left( \frac{1-(q/p)^{z}}{1-(q/p)^a}\right)\,.\]
 \end{enumerate}
\end{lemma}

\begin{corollary}[Gambler's ruin inequalities]\label{cor:gambler}
Consider a random walk on $\mathbb{Z}_{\geq 0}$ 
that absorbs at~$0$ and~$a$ (for some positive integer~$a$), starts at $z\in[0,a]$
and from each state in $\{1,\ldots,a-1\}$
has probability $p\neq 1/2$ of  increasing (by~$1$) and
probability $q=1-p$ of  decreasing (by~$1$).
\begin{enumerate}[(i)]
\item \label{ruin:three} If $p > q$ then the probability of 
reaching state~$a$ is  at least $1 - (q/p)^{z}$.
\item \label{ruin:four} If $q > p$ then the expected number of transitions 
until absorption  is at most $z/(q-p)$.
\item \label{ruin:five}
If $p > q$ then the expected number of transitions 
until absorption  is at most $(a-z)/(p-q)$.
\end{enumerate}

\end{corollary}
\begin{proof}
Items~(\ref{ruin:three}) and~(\ref{ruin:four}) are immediate.
To see (\ref{ruin:five}) for $p>q$, rewrite the expected number
of transitions as
\[\left(\frac{a}{p-q}\right)\left(  \frac{1-(q/p)^{z}}{1-(q/p)^a}\right) - \frac{z}{p-q} \leq \frac{a-z}{p-q}\,.\qedhere\]
\end{proof}

 We also consider a variant of the gambler's ruin in which the probability of upwards transitions depends on the current state. 
 \begin{lemma}\label{lem:backtoback}
 Let $a$, $b$, $c$ and $d$ be integers satisfying
  $a<b < c-1$ and $c+1<d$. 
  Consider $p_1 \in (1/2,1)$. 
  Consider the discrete Markov chain on states
 $\{a,\ldots,d\}$ with the following transition matrix.
 \begin{align*}
p_{a,a} &= 1,\\
p_{i,i+1} &= 
	\begin{cases}
		p_1   & \textnormal{ if }a+1\le i\le c,\\
		1/3   & \textnormal{ if }c < i \le d-1,\\
	\end{cases}\\
p_{i,i-1} &= 1 - p_{i,i+1} \textnormal{ for all }i\in \{a+1,\ldots,d-1\},\\
p_{d,d} &= 1.
\end{align*}
For integers $x$, $y$ and $z$ in $\{a,\ldots,d\}$
and subset $S$ of $\{a,\ldots,d\}$,
 let $p_{x\rightarrow y; z}$ 
denote the probability that, starting from state~$x$, the chain visits state~$y$
without passing through~$z$
and let $H_{x;S}$ be the   number of transitions that the chain takes to hit a state in~$S$,
starting from state~$x$.
\begin{enumerate}[(i)]
\item \label{backtobackone} $p_{c\rightarrow d; b} \geq 
1 - {\left(\frac{1-p_1}{p_1}\right)}^{c-b} 2^{d-c}$.
\item \label{backtobacktwo} 
$\E[H_{c;\{a,d\}}] \leq  
2^{d-c+1} \left(\frac{3p_1-1}{2 p_1-1}\right)$.
\end{enumerate}
 \end{lemma}
 \begin{proof}
 
 We first prove (\ref{backtobackone}).
 It is immediate that
 \begin{align*}
 p_{c\rightarrow d; b} &= 
(1-p_1) p_{c-1\rightarrow d;b} + p_1 p_{c+1\rightarrow d;b} \\
&=
  (1-p_1) p_{c-1\rightarrow c;b} p_{c\rightarrow d;b} + 
 p_1 p_{c+1\rightarrow d;c} + 
 p_1
 (1-p_{c+1\rightarrow d;c})
 p_{c\rightarrow d;b}.\end{align*}
 Rearranging,
  \begin{equation}\label{eq:willkeep}
	p_{c\rightarrow d;b} = \frac
 {p_1 p_{c+1\rightarrow d;c}}
 {(1-p_1)  (1-p_{c-1\rightarrow c;b} )  + p_1  p_{c+1\rightarrow d;c}}.
\end{equation}
 Now   
 from   Lemma~\ref{lem:gambler}\eqref{ruin:one},
\begin{equation}\label{eq:willkeeptwo}
 p_{c+1\rightarrow d;c} = 
\frac{ 2-1}{{2}^{d-c}-1} \geq 2^{-(d-c)}\,.
\end{equation}
Also, from  Corollary~\ref{cor:gambler}\eqref{ruin:three},
  \[  p_{c-1\rightarrow c;b}  \geq 1- {\left(\tfrac{1-p_1}{p_1}\right)}^{c-b-1},
  \]
  so 
  \[1- p_{c-1\rightarrow c;b}  \leq {\left(\tfrac{1-p_1}{p_1}\right)}^{c-b-1}.\]
Plugging these bounds into~\eqref{eq:willkeep}, we get
 \[p_{c\rightarrow d;b} \geq \frac
 { 1}
 { {\left(\frac{1-p_1}{p_1}\right)}^{c-b} 2^{d-c}    + 1    }
 = 1 -
 \frac
 {{\left(\frac{1-p_1}{p_1}\right)}^{c-b} 2^{d-c}  }
 { {\left(\frac{1-p_1}{p_1}\right)}^{c-b} 2^{d-c}    + 1    }
 \geq 
 1 - {\left(\frac{1-p_1}{p_1}\right)}^{c-b} 2^{d-c}\,.
 \]    
     
We now prove (\ref{backtobacktwo}).
Clearly,
\[\E[H_{c;\{a,d\}}] = 1 + p_1 \E[H_{c+1;\{a,d\}}] + (1-p_1) \E[H_{c-1;\{a,d\}}].\]
But 
 \[\E[H_{c+1;\{a,d\}}] = \E[H_{c+1;\{c,d\}}] + p_{c+1\rightarrow c;d} \E[H_{c;\{a,d\}}]\]
and
\[
 \E[H_{c-1;\{a,d\}}] = \E[H_{c-1;\{a,c\}}] + p_{c-1 \rightarrow c;a} \E[H_{c;\{a,d\}}].
\]
So solving, we get
\[ \E[H_{c;\{a,d\}}]
= \frac{1 + p_1 \E[H_{c+1;\{c,d\}}] + (1-p_1) \E[H_{c-1;\{a,c\}}] }
{1 - p_1 p_{c+1\rightarrow c;d} - (1-p_1)p_{c-1 \rightarrow c;a}}.\]
Now $p_{c-1\rightarrow c;a} \leq 1$ and by Equation~\eqref{eq:willkeeptwo},
$p_{c+1\rightarrow c;d} \leq 1-2^{-(d-c)}$.
By   Corollary~\ref{cor:gambler}\eqref{ruin:four},
$\E[H_{c+1;\{c,d\}}] \leq 1/(\tfrac23-\tfrac13) = 3$. 
Finally, by   Corollary~\ref{cor:gambler}\eqref{ruin:five},
\[\E[H_{c-1;\{a,c\}}] \leq \frac{1}{p_1 - (1-p_1)} = \frac{1}{2p_1-1}.\]
Plugging all of these in, 
\[ \E[H_{c;\{a,d\}}]
\leq \frac{1 + 3 p_1   + \frac{1-p_1}{2 p_1 -1}  }
{1 - p_1  (1- 2^{-(d-c)}) - (1-p_1)} = 2^{d-c+1} \left(\frac{3p_1-1}{2 p_1-1}\right).\qedhere\]
\end{proof}

\section{Stochastic processes}
\label{sec:processes}
     
We will be concerned with the discrete-time  
 Moran process~\cite{Mor1958:Moran}, as adapted by Lieberman, Hauert and
Nowak~\cite{LHN2005:EvoDyn} and described in Section~\ref{sec:intro}.
This is a discrete model of evolution on an underlying directed graph~$G$
where the reproduction rate of mutants is a parameter~$r>0$ called the ``fitness''.

In this paper, we consider the situation $r>1$, which corresponds to the
situation in which a mutation is advantageous.
 The   fitness $r$ is a parameter of all of our processes.
Our results apply to any fixed $r>1$.
Since the value of $r$ is fixed, we simplify the presentation by not including
it in the explicit notation and terminology. Thus, from now on, we say ``Moran process''
to signify ``Moran process with fitness~$r$''.

Following \cite{DGRS} we will simplify our proofs by studying a
continuous-time version of the Moran process.  The continuous-time version 
is also parameterised by~$G$ and~$r$ and it 
has the same fixation probability as the discrete-time version, so our results will
carry over immediately to the discrete process.

In order to deal with conditioning in the proofs we will in fact define several general stochastic
processes, all of which depend on $G$ and~$r$ --- one of these will be equivalent to the continuous-time 
Moran process and others will be useful for dominations.

All of the processes that we study evolve over time.
For any process~$P$, we use $\filt(P)$ to denote the filtration of~$P$
so $\filt_t(P)$ captures the history of the process~$P$ up to and including time~$t$.

\subsection{The clock process}
\label{sec:clock} 
For each 
edge $e = (u, v) $ of~$G$ we define two Poisson processes --- a Poisson process
$\mc{e}$  with  parameter $r/d^+(u)$
and a Poisson process
$\nc{e}$   with parameter $1/d^+(u)$. 
We refer to these processes as \emph{clocks}, and when an event occurs in one of
 them, we say that the relevant clock \emph{triggers}.
We refer to $\mc{e}$ as a \emph{mutant clock with 
source~$u$ and target~$v$} and $\nc{e}$ as a 
\emph{non-mutant clock with source~$u$ and target~$v$}.

We use $\moranclocks(G)$ to denote the set
of all clocks so
$\moranclocks(G) = \bigcup_{e\in E(G)} \{\mc{e},\nc{e}\}$.
We use $P(G)$ to denote the Cartesian product of
all process in $\moranclocks(G)$. $P(G)$ is the stochastic process in which all clocks
in $\moranclocks(G)$ evolve simultaneously and independently, starting 
at time~$0$.  

With probability~$1$, the clocks trigger a countably infinite number of times
and these can be indexed by an increasing sequence $\tau_1,\tau_2,\ldots$.
Also, no clocks trigger simultaneously 
and the clocks trigger for an infinitely long period --- that is, for every clock and every $t$,
the clock triggers at some $\tau_i > t$.
For convenience, we take $\tau_0=0$. We will use the random variables
$\tau_0,\tau_1,\ldots$ 
(which depend on the process $P(G)$)
 in our arguments.

\subsection{Mutant processes}
\label{sec:mutantprocess}

A \emph{mutant process} $\mu $   has an underlying graph $G(\mu)$ and initial state $\mu_0$.
At every time $t$, the state $ \mu_t$ is a subset of $V(G(\mu))$,
which we sometimes refer to as the ``set of mutants'' at time~$t$.
Every mutant process satisfies the following two constraints.
\begin{enumerate}[(1)]
\item For all $t\geq 0$, $\mu_t$ is determined by $\filt_t(P(G(\mu)))$. \label{constone}
\item For all $t,t'\geq 0$, if there is  a non-negative integer~$i$ so that \label{consttwo}
$t$ and $t'$ are both in the range $[\tau_i,\tau_{i+1})$ then $\mu_t = \mu_{t'}$.
\end{enumerate}

We define some terminology associated with the mutant process~$\mu$.
\begin{itemize}
\item If the clock $\mc{(u,v)}$ triggers at time $t$ 
and $u\in \mu_t$
we say that
\emph{$u$ spawns a mutant onto $v$ in $\mu$ at time $t$}
and that   $\mu$ \emph{spawns a mutant} onto $v$ at time $\tau_i$.  
\item If the clock $\nc{(u,v)}$ triggers at time $t$ 
and $u\notin \mu_t$
we say that
\emph{$u$ spawns a non-mutant onto $v$ in $\mu$ at time~$t$}. We say that $\mu$ \emph{spawns a non-mutant} onto $v$ at time $\tau_i$.
\item If   $v\in \mu_{\tau_i}$
and $v\not\in \mu_{\tau_{i-1}}$ we say that 
 \emph{$v$ becomes a mutant in~$\mu$ at time $\tau_{i}$}. 
\item If  $v\in \mu_{\tau_{i-1}}$ 
and $v\not\in \mu_{\tau_i}$
we say that 
\emph{$v$ becomes a non-mutant in~$\mu$ at time $\tau_{i}$} or that \emph{$v$ dies 
in~$\mu$ at time $\tau_{i}$}.
\end{itemize}
When the mutant process is absolutely clear from the context, we sometimes drop the phrase
``in $\mu$''.
Note that $v$~does not necessarily become a mutant at time~$\tau_i$ when some $u$~spawns a mutant onto~$v$ at time~$\tau_i$ since $v$~may already be a mutant at that time.

For convenience, we include the filtration  $\filt_t(P(G))$
 in the filtration 
$\filt_t(\mu)$ of the  mutant process
so the sequence of trigger-times $\tau_0,\tau_1,\ldots$ up to time~$t$
can be determined from $\filt_t(\mu)$.

\begin{remark}Sometimes we  will consider a mutant process $\mu$ 
in which the initial state $\mu_0$ is a randomly chosen subset of $V(G(\mu))$.
When we do this, we assume that the choice of the initial state~$\mu_0$
is independent of the triggering of the   clocks in~$\moranclocks(G(\mu))$.
\end{remark}

We will define several mutant processes in the course of our proofs, but
the most fundamental is the Moran process itself, which is a particular mutant process.

\begin{definition}[\textbf{the Moran process}] \label{def:contMoran}
The (continuous-time) Moran process
on graph~$G$ 
with initial mutant $x_0\in V(G)$  
is a mutant process $X$ with $G(X)=G$ and 
$X_0=\{x_0\}$  
defined as follows. 
Recall that, for every positive integer~$i$,
a clock $C\in \clocks(G)$ triggers at $\tau_i$. 
For $t\in (\tau_{i-1},\tau_i)$, we set
$X_t = X_{\tau_{i-1}}$. Then
we define $X_{\tau_i}$ as follows.
\begin{enumerate}[(i)]
 \item If  $C = \mc{(u,v)}$ for some $(u,v)\in E(G)$ \label{spawneventone}
 and $u \in X_{\tau_{i-1}}$ then
 $X_{\tau_i} = X_{\tau_{i-1}} \cup \{v\}$.  
 \item If $C = \nc{(u,v)}$ for some $(u,v) \in E(G)$ \label{spawneventtwo}
 and $u\notin X_{\tau_{i-1}}$ then
   $X_{\tau_{i}} = X_{\tau_{i-1}} \setminus \{v\}$. 
 \item Otherwise, $X_{\tau_i} = X_{\tau_{i-1}}$. \label{otherevent}
\end{enumerate} 
Considering the positive integers~$i$ in order, this completes the definition of
the Moran process~$X_t$.\defend{}
\end{definition}

\begin{remark}
It is clear from Definition~\ref{def:contMoran} that the Moran process~$X_t$
is a mutant process.
In Definition~\ref{def:contMoran}, say that $\tau_i$ is a ``relevant trigger time''
if \eqref{spawneventone} or \eqref{spawneventtwo} occurs rather than \eqref{otherevent}.
 The discrete-time  
 Moran process~\cite{Mor1958:Moran}, as adapted by Lieberman, Hauert and
Nowak~\cite{LHN2005:EvoDyn} 
is the Markov chain 
$X_{\tau_{0}},X_{\tau_{i_1}},X_{\tau_{i_2}},\ldots$,
where  
$\tau_{i_1},\tau_{i_2},\ldots$ is the increasing sequence of
 relevant trigger times.
 Note that the fixation probability 
of the discrete-time Moran process is the same as the fixation probability of the continuous-time process
 $X_t$, so we will study the process $X_t$ in this paper.
\end{remark}

\begin{definition} We say that a mutant process is \emph{extinct by time~$t$}
if, for all $t'\geq t$, $\mu_{t'} = \emptyset$.
We say that it \emph{fixates by time~$t$} if, for all $t'\geq t$, $\mu_{t'} = V(G(\mu))$.
We say that it \emph{absorbs by time~$t$} if it is extinct by time~$t$ or it fixates by time~$t$.
The \emph{fixation probability} is the probability that, for some $t$, it fixates by time~$t$.
The \emph{extinction probability} is the probability that, for some $t$, it is extinct by time~$t$.\defend{} \end{definition}

\begin{remark} The Moran process $X_t$ is extinct by time~$t$ if $X_t = \emptyset$ and
fixates by time~$t$ if $X_t = V(G(X))$.
If $G$ is strongly connected then the fixation probability and the extinction probability sum to~$1$. 
\end{remark}

\begin{definition}\label{def:localtime} 
For any mutant process~$\mu$, 
any vertex $u \in V(G(\mu))$,
and any $t\geq 0$,
we define $\iin{\mu}{u}{t}$
to be the measure of the  set $\{t' \leq t \mid u\in \mu_{t'}\}$.
Similarly, we define $\iout{\mu}{u}{t}$
  to be the measure of the set $\{t' \leq t \mid u\notin \mu_{t'}\}$.
\defend{}
\end{definition}

The subscript ``$\mathsf m$'' stands for ``mutant'' 
since $\iin{\mu}{u}{t}$
is the amount of time that $u$ is a mutant in~$\mu$, up until time~$t$.
Similarly,
the subscript ``$\mathsf n$'' stands for
non-mutant.
The random variables 
$\iin{\mu}{u}{t}$ and $\iout{\mu}{u}{t}$ are  determined by $\filt_t(\mu)$.
Also,  
$\iin{\mu}{u}{t} + \iout{\mu}{u}{t} = t$.

\subsection{The  star-clock process}
\label{sec:starclock}
 
Consider a mutant process~$\mu$.
We wish to be able to
discuss events such as the event that a vertex $u$ does not spawn a mutant until it has been a mutant 
for time $t$.
In order to express such events in a clean way, 
making all conditioning explicit, we 
define additional stochastic processes.

For each 
edge $e = (u, v) $ of~$G$ we define 
four further Poisson processes --- Poisson processes
$\smc{e}$  
and $\smbc{e}$ each
with  parameter $r/d^+(u)$ and  Poisson processes
$\snc{e}$ and $\snbc{e}$  each  with parameter $1/d^+(u)$. 
We refer to these processes as \emph{star-clocks}.
We identify sources and targets of star-clocks in the same way that we did for   clocks.
For example, the star-clock $\smc{(u,v)}$ has source~$u$ and target~$v$.

We use $\inclocks(G)$ to denote the set
$\inclocks(G) = \bigcup_{e\in E(G)} \{\smc{e},\snbc{e}\}$.
We use $\outclocks(G)$ to denote the set
$\outclocks(G) = \bigcup_{e\in E(G)} \{\snc{e},\smbc{e}\}$.

The star-clock process  $P^*(G)$ is the stochastic
process where all star-clocks in $\inclocks(G) \cup \outclocks(G)$ evolve simultaneously and independently,
starting at time~$0$.

\subsection{A coupled process}
\label{sec:coupled}

Given a mutant process~$\mu$
let $G=G(\mu)$.
We will now define a stochastic process $\Psi(\mu)$
which is a coupling of $\mu$
(which includes the clock process $P(G)$)
with the newly-defined star-clock process $P^*(G)$. 
Intuitively, the idea
of the coupling
is that each clock $\mc{(u,v)}$ in $P(G)$ will 
evolve following $\smc{(u,v)}$ when $u$ is a mutant
and following $\smbc{(u,v)}$ when $u$ is a non-mutant.
Similarly, $\nc{(u,v)}$ will evolve following
$\snc{(u,v)}$ when $u$ is a non-mutant 
and $\snbc{(u,v)}$ when $u$ is a mutant.
In the coupling, 
we pause the star-clocks in $\inclocks(G) \cup \outclocks(G)$ while they are not being used to drive clocks in 
$\moranclocks(G)$, so that, e.g., the ``local time'' of a clock $\smc{(u,v)}$ at global time $t$ is $\iin{\mu}{u}{t}$.  
 
We will be able to deduce both $\filt_t(\mu)$
and $\filt_t(P^*(G))$ from the 
filtration $\filt_{T}(\Psi(\mu))$ 
of the coupled process at an appropriate stopping time~$T$ --- the details are given below. 
The fact that the coupling is valid (which we will show  below)
will ensure that both of  the marginal processes, 
$\mu$ and $P^*(G)$,  evolve according
to their correct distributions.
 
To construct the coupling 
we start with a 
copy
of the star-clock process~$P^*(G)$ and with the initial 
state $\mu_0$ of the mutant process~$\mu$.
We define 
$\tau_0 = 0$ 
(so  we have  implicitly  defined  $\filt_{\tau_0}(\mu)$).

Suppose that, for some non-negative integer~$j$, we have defined
$\filt_{\tau_j}(\mu)$.  Given this and 
the evolution of the star-clock process $P^*(G)$,
we will show how to   
define $\tau_{j+1}$
and $\filt_{\tau_{j+1}}(P(G))$ which determine
$\filt_{\tau_{j+1}}(\mu)$.
To do this,  
 let $t_j$ be  the minimum $t>0$ such that one of the following occurs.
\begin{itemize}
\item  For some $u\in \mu_{\tau_j}$, a star-clock
in $\inclocks(G)$ with source~$u$ triggers at time $\iin{\mu}{u}{\tau_j}+t$, or
\item for some $u\notin \mu_{\tau_j}$, a star-clock
in $\outclocks(G)$ with source~$u$  triggers at time $\iout{\mu}{u}{\tau_j}+t$.
\end{itemize}
We define $\tau_{j+1} = \tau_j+t_j$.
No clocks in $\moranclocks(G)$ trigger
in the interval $(\tau_j,\tau_{j+1})$.
We now determine which clock from $\moranclocks(G)$ triggers at time~$\tau_{j+1}$
by reconsidering each case.
\begin{itemize}
\item  If  $u\in \mu_{\tau_j}$ and
$\smc{(u,v)}$ triggers at time $\iin{\mu}{u}{\tau_j} + t_j$
then $\mc{(u,v)}$ triggers at time $\tau_{j+1}$. 
\item  If $u\in \mu_{\tau_j}$ and
$\snbc{(u,v)}$ triggers at time $\iin{\mu}{u}{\tau_j}+t_j$
then $\nc{(u,v)}$ triggers at time $\tau_{j+1}$.
\item  If $u\notin \mu_{\tau_j}$, and
$\snc{(u,v)}$ triggers at time $\iout{\mu}{u}{\tau_j}+t_j$
then $\nc{(u,v)}$ triggers at time $\tau_{j+1}$.
\item  If $u\notin \mu_{\tau_j}$ and
$\smbc{(u,v)}$ triggers at time $\iout{\mu}{u}{\tau_j}+t_j$
then $\mc{(u,v)}$ triggers at time $\tau_{j+1}$.
\end{itemize}
  This
fully defines 
$\filt_{\tau_{j+1}}(P(G))$ and hence
$\filt_{\tau_{j+1}}(\mu)$.
So we have fully defined the coupling
and therefore the process $\Psi(\mu)$.

Before showing that the coupling is valid, it will be helpful to state   exactly what information is
contained in $\filt_{t}(\Psi(\mu))$.
Certainly this includes $\filt_{t}(\mu)$ which itself includes $\filt_t(P(G))$.
Also, $\filt_t(P(G))$ defines a non-negative integer~$j$ so
that  $t \in [\tau_j,\tau_{j+1})$.
We will use~$j$ to state the information that $\filt_{t}(\Psi(\mu))$ contains about 
the evolution of~$P^*(G)$.
\begin{itemize}
\item For each star-clock $C\in \inclocks(G)$ with source~$u\in \mu_{\tau_j}$, 
$\filt_t(\Psi(\mu))$
includes a list of the times in $[0,\iin{\mu}{u}{\tau_j}+t-\tau_j]$ when  $C$ triggers.
\item For each star-clock $C\in \inclocks(G)$ with source $u\notin \mu_{\tau_j}$,
$\filt_t(\Psi(\mu))$
includes a list of the times in $[0,\iin{\mu}{u}{\tau_j}]$ when  $C$ triggers.
\item For each star-clock $C\in \outclocks(G)$ with source~$u\notin \mu_{\tau_j}$, 
$\filt_t(\Psi(\mu))$
includes a list of the times in $[0,\iout{\mu}{u}{\tau_j}+t-\tau_j]$ when  $C$ triggers.
\item For each star-clock $C\in \outclocks(G)$ with source $u\in \mu_{\tau_j}$,
$\filt_t(\Psi(\mu))$
includes a list of the times in $[0,\iout{\mu}{u}{\tau_j}]$ when  $C$ triggers. 
\end{itemize}

To show that the coupling is valid we must show that both of the marginal processes,
$\mu$ and $P^*(G)$, evolve according to their correct distributions.
The fact that $P^*(G)$ does so is by construction.
To show that $\mu$ does so, it suffices to prove that for all $j\in\Zzero$ and all possible values $f_j$ of $\filt_{\tau_j}(\mu)$, the distribution of $\filt_{\tau_{j+1}}(\mu)$ conditioned on $\filt_{\tau_j}(\mu) = f_j$ is correct. Note that the only information contained in $\filt_{\tau_{j+1}}(\mu)$ but not $\filt_{\tau_j}(\mu)$ is the value of $\tau_{j+1}$ and the identity of the clock in $\moranclocks(G)$ that triggers at time $\tau_{j+1}$.

Let $f_j'$ be an arbitrary possible value of $\filt_{\tau_j}(\Psi(\mu))$ consistent with the event $\filt_{\tau_j}(\mu) = f_j$, in the sense that the intersection of the events $\filt_{\tau_j}(\Psi(\mu)) = f_j'$ and $\filt_{\tau_j}(\mu)=f_j$ is non-empty. Recall from the definition of $\Psi(\mu)$ that, conditioned on $\filt_{\tau_j}(\Psi(\mu)) = f_j$, $\filt_{\tau_{j+1}}(\mu)$ depends only on particular star-clocks in particular intervals, as follows. 
\begin{itemize}
\item For each $u\in \mu_{\tau_j}$,
it depends on   the evolution of each star-clock  in $\inclocks(G)$ with source~$u$
only during the interval $(\iin{\mu}{u}{\tau_j},\infty)$.  
It does not depend on the evolution of star-clocks in $\outclocks(G)$ with source~$u$.
	
\item For each $u\notin\mu_{\tau_j}$,
it depends on the evolution 
of each star-clock  in $\outclocks(G)$ with source~$u$
only during the interval $(\iout{\mu}{u}{\tau_j},\infty)$.
It does not depend on the evolution of star-clocks in $\inclocks(G)$ with source~$u$.

\end{itemize}

For each star-clock, these intervals are disjoint from the intervals exposed in $f_j'$, and the start of each interval is determined by $f_j'$. Moreover, in the interval $(\tau_j, \tau_{j+1}]$, each clock in $\moranclocks(G)$ is triggered by a unique clock in $\inclocks(G) \cup \outclocks(G)$ with the same rate. Thus all clocks in $\moranclocks(G)$ trigger with the correct rates in this period and they are independent of each other (since all of the star-clocks in $P^*(G)$ evolve independently). We conclude that $\filt_{\tau_{j+1}}(P(G))$, and hence $\filt_{\tau_{j+1}}(\mu)$, has the appropriate distribution. The coupling is therefore valid.

By construction, we   have the following observation.
\begin{observation} \label{lem:clock-coupling}\label{obs:one}
Let $\mu$ be a mutant process and consider 
$\Psi(\mu)$. 
Let $(u,v)$ be an edge of~$G(\mu)$.
Given  $t> 0$, let
  $j$ be the maximum integer such that $\tau_j < t$.
Then  the following are true.
\begin{itemize}
\item $\mc{(u,v)}$ triggers at time $t$ if and only if either $u \in \mu_{\tau_j}$ and 
$\smc{(u,v)}$   triggers at time $\iin{\mu}{u}{t}$ or $u \notin \mu_{\tau_j}$ and $\smbc{(u,v)}$ triggers at time $\iout{\mu}{u}{t}$.
\item $\nc{(u,v)}$ triggers at time $t$ if and only if either $u \notin \mu_{\tau_j}$ and $\snc{(u,v)}$ triggers at time $\iout{\mu}{u}{t}$ or $u \in \mu_{\tau_j}$ and $\snbc{(u,v)}$ triggers at time $\iin{\mu}{u}{t}$.
\end{itemize}
\end{observation}

\section{An upper bound on the fixation probability of superstars}
\label{sec:superstar}
 
Recall the definition of a
$(k,\ell,m)$-superstar 
from Section~\ref{def:super}.
We use
$n = \ell(k + m)+1$ to denote the number of vertices
of a  
$(k,\ell,m)$-superstar.

Given any $i\in [\ell]$,
we say that $v_{i,1} v_{i,2} \dots v_{i,k}$ is the \emph{path associated with} 
the reservoir~$R_i$.
We will often consider the case that the initial mutant $x_0$ is in 
a reservoir.
When it is possible, we simplify the notation 
by dropping the index~$i$ of the reservoir.
Thus, we write  $R$ for the reservoir containing $x_0$ 
and we write $v_1 \dots v_k$ for the path associated with $R$. 
So if $R=R_i$ then for each $j\in[k]$, we  write $v_j$ as a synonym for $v_{i,j}$.
The main result of this section is the   following upper bound on the fixation probability of the superstar.

\newcommand{\statemainsuperthm}{Let $r > 1$. Then there exists $c_r>0$ such that the following holds for all positive 
integers $k$, $\ell$ and $m$. 
Choose $x_0$ uniformly at random from $V(\sstar)$.
Let $X$ be the Moran process (with fitness $r$) with 
$G(X)=\sstar$ and $X_0=\{x_0\}$.
Then the probability that $X$  goes extinct  is at least $1/(c_r(n\log n)^{1/3})$.}

\begin{theorem}\label{thm:star-lower}
\statemainsuperthm 
\end{theorem}

\subsection{Proof Sketch}

In this Section, we give an informal sketch of the proof of Theorem~\ref{thm:star-lower}.
The presentation of the proof itself does not depend upon the sketch
so the reader may prefer to skip directly to the proof.
In  all of our proof sketches,  we use the word ``likely'' to mean ``sufficiently likely''. 
We leave the details of ``how likely'' to the actual proofs.

If $m$ is small relative to $k$
(in particular, if $m< k {(n \log n)}^{1/3}$) 
then 
the initial mutant~$x_0$ is  likely
to be placed in a path, rather than in a reservoir.
If this happens, then it is  likely to go extinct. 
This easy case is dealt with in Lemma~\ref{lem:star-lower-die-instant}
and corresponds to Case~2 in the proof of Theorem~\ref{thm:star-lower}.
(Case~1 is the trivial case where $n<n_0$.)

Another easy case arises if $\ell$ is sufficiently small relative to~$n$
(in particular, if $\ell = O( {(n \log n)}^{1/3})$).
This case is dealt with in Lemma~\ref{lem:star-lower-die-before-path}
and corresponds to Case~3 in the proof of Theorem~\ref{thm:star-lower}.
In this case,  even when $x_0$ is placed in a reservoir~$R$,
it is still likely that $x_0$ dies before $v_{2}$ ever becomes a mutant.
This is because it takes roughly $\Theta(m)$ time for the
mutation to spread from~$v_{1}$ to~$v_{2}$
since a mutant at~$v_{1}$ has only probability~$\Theta(1/m)$ of spawning a mutant
before it dies.
On the other hand, since $\ell$ is small, $x_0$ is sufficiently likely to die in $\Theta(m)$ time.
For details, see the proof of Lemma~\ref{lem:star-lower-die-before-path}.

The remaining case, Case~4 in the proof of  Theorem~\ref{thm:star-lower},
is deemed the ``difficult regime'' and is dealt with in Section~\ref{sec:supdifficult}.
In this case, it is 
easy to show that $\ell = \Omega(\kappa \log n)$
and $m = \Omega(\kappa)$ where $\kappa = \max \{3k, 70 r^4 \log n\}$.

It is likely that the initial mutant $x_0$ is placed in a reservoir~$R$, and the key
lemma, showing that it is sufficiently likely to go extinct, is Lemma~\ref{lem:star-lower-hard}.

 At a very high level, the argument proceeds as follows.
 Suppose that $v^*$ does not spawn a mutant before $x_0$ dies.
 Then it is very easy to see
 that, after $x_0$ dies, 
 the path of reservoir~$R$ is likely to go extinct quickly.
 
 Thus, the   crux of the argument is to show that $x_0$ is   likely to die   before $v^*$ spawns a mutant. Each 
time $v^*$ becomes a mutant it has an $O(1/\ell)$ chance of spawning a mutant   before dying, so 
roughly our goal is to show  that $x_0$ is  sufficiently  likely to die   before $v^*$ becomes a mutant 
$\Omega(\ell)$ times. 

Very roughly, our high-level approach is to
partition time into  intervals of length $\kappa=O(m)$. In each 
block of $O(m/\kappa)$
such intervals,
 $v_2$ is likely to become a mutant $O(1)$ times. 
 Each time this happens, 
 it is likely that $R$'s path will again fill with non-mutants  within $O(\kappa)$ time, 
 so it is likely that $v_k$ is   a mutant for at most $O(\kappa)$ time 
 during the block and it is likely that $v^*$ becomes a mutant at most $O(\kappa)$ times during the
 block. 
 Combining $O(\ell/\kappa)$  blocks, it is likely 
 that $v^*$ becomes a mutant at most $O(\ell)$ times by time~$ \ell m/\kappa$. 
 Since $\vnc{v^*}{x_0}$ has rate $1/(\ell m)$, 
 it is also likely that $x_0$ dies by time~$\ell m/\kappa$.

In more detail, the proof of Lemma~\ref{lem:star-lower-hard} shows that 
$x_0$ dies before $v^*$ spawns a mutant as long as
certain events called \sP1--\sP5 occur. These events are defined in the statement of Lemma~\ref{lem:star-masterlist}.
They  formalise the high-level approach that we have just  described.
It is important that most of  these events are defined in terms of clock-triggers
so that we can 
get good upper bounds on the probability that they fail
and thus prove (in  Lemma~\ref{lem:star-masterlist}) that they are likely to occur simultaneously.

The proof of Lemma~\ref{lem:star-lower-hard} tracks a quantity $\sigma(t)$ 
which is the number of times that $v_k$ 
(the end of the path of the reservoir containing~$x_0$)
spawns a mutant onto the centre vertex~$v^*$
by time~$t$.
The proof uses \sP1--\sP5  to show that $\sigma(t)$ stays $O(\ell)$
up to a fixed time $\tx = O(\ell m/\kappa)$.
As we noted, 
the analysis divides the period up to time~$\tx$ into intervals of size $\kappa$.
 Event~\sP5 ensures that
during most such intervals,
non-mutant clocks with target~$v_1$
and mutant clocks with targets~$v_1$ and~$v_2$ behave appropriately
so that,  if~$x_0$ is the only mutant in~$R$ during the interval, then $v_2$ does not become
a mutant during the interval.  
The fact that $x_0$ is indeed the only mutant in~$R$  
  follows from event~\sP1 which ensures that 
$v^*$ does not   spawn a mutant while $\sigma(t)$ is small.
Then since $v_2$ does not become a mutant during the interval,  event~\sP3 ensures
that the clocks along the path trigger in such a way  
that (unless $v_1$ or $v^*$ spawn a mutant) the only mutants remaining at the end of the interval 
are in $\{x_0,v_1\}$. This ensures that $\sigma(t)$ stays small through another interval.
 Event~\sP5 only ensures the above during ``most such intervals'' but  event~\sP4
ensures that the mutant clock with source~$v_k$  does not trigger
too often, so the remaining intervals are not too problematic.
Thus, events~\sP1, \sP3, \sP4 and \sP5,   taken together, ensure that $\sigma(\tx)$ is  $O(\ell)$.

Given that $\sigma(\tx)$ is  $O(\ell)$, it is easy to show that the initial mutant goes extinct 
during  the next two intervals (beyond time~$\tx$).  Event~\sP1 ensures that $v^*$ doesn't spawn any  
mutants.  Event~\sP2 ensures that the initial mutant $x_0$ 
has already died by time $\tx$.
Finally,  event~\sP3 ensures that any remaining mutants die in the path during the next two intervals.

The difficult part of the  proof is defining  events~\sP1--\sP5 in such a way
that we  can show (in Lemma~\ref{lem:star-masterlist}) that they are likely to occur simultaneously.
It turns out  (Lemma~\ref{lem:star-lower-path-clear}, Corollary~\ref{cor:star-lower-full-path-not-bad}
and Lemma~\ref{lem:star-lower-path-not-fill}) that  events \sP3--\sP5
are so unlikely to fail that we  bound this probability with a simple union bound, avoiding any
complicating conditioning. (Of course, for this it was necessary to express  these events in terms of clocks
rather than in terms of the underlying Moran process.)
In order to simplify the presentation, we deal with  \sP1 and \sP2 together, in 
Lemma~\ref{lem:star-lower-apply-centre-mut}.   Roughly, they correspond to the event that, as long as 
$\sigma(t) = O(\ell)$ then $v^*$ does not spawn a mutant at time~$t$ and, for $t=\tx$, 
$x_0$ dies by time~$t$.
 This event is
implied by the conjunction of three 
further events.
\begin{itemize}
\item
$\mathcal{E}_1$ is the event that
no star-clock 
$\vsmc{v^*}{v}$  (for any~$v$)
triggers in $[0,1/r]$.
\item
$\mathcal{E}_2$  is the event that the 
star-clock $\vsnc{v^*}{x_0}$ triggers in $[0,\tx-1]$. 
\item  
$\mathcal{E}_3$ corresponds informally to the event that
$v^*$ is a mutant for a period of time shorter than~$1/r$ during the first
 $O(\ell)$ times that it becomes a mutant (though the formal 
definition is expressed in terms of clocks, and is a little more complicated). Note the intention, though,
which is to ensure that $v^*$ is a mutant for a period of time shorter than~$1/r$, which makes
$\mathcal{E}_1$ relevant.
\end{itemize}
Lemma~\ref{lem:star-lower-centre-mut} shows that $\mathcal{E}_3$ is very  likely to hold.
In the proof of
Lemma~\ref{lem:star-lower-apply-centre-mut}, 
it is observed that $\mathcal{E}_1$ and $\mathcal{E}_2$ are independent
(by the definition of the star-clocks) and that $\Pr(\mathcal{E}_1) = 1/e$.
The proof demonstrates that $\mathcal{E}_2$ is sufficiently likely, giving the desired bound.

\subsection{Glossary}

\begin{longtable}{p{\linewidth-\widthof{Definition 00, Page 00}}@{}l@{ }l@{}}
clears before spawning \dotfill & Definition~\ref{def:clear-before},& Page~\pageref{def:clear-before} \\
clears within $I_i$ \dotfill & Definition~\ref{def:star-lower-path-clear},& Page~\pageref{def:star-lower-path-clear} \\
$c_r$  \dotfill & Theorem~\ref{thm:star-lower},& Page~\pageref{thm:star-lower} \\
$I_i=(i \kappa, (i+1)\kappa]$ \dotfill & Definition~\ref{def:tenth},& Page~\pageref{def:tenth} \\
$\iin{\mu}{u}{t}$ \dotfill & Definition~\ref{def:localtime},& Page~\pageref{def:localtime}\\
$\iout{\mu}{u}{t}$ \dotfill & Definition~\ref{def:localtime},& Page~\pageref{def:localtime}\\
$K=70$ \dotfill & Definition~\ref{def:star-lower-cr},& Page~\pageref{def:star-lower-cr} \\
$\kappa=\max\{3k,Kr^4 \log n\}$ \dotfill & Definition~\ref{def:star-lower-cr},& Page~\pageref{def:star-lower-cr} \\ 
$(\sP1),(\sP2),(\sP3),(\sP4),(\sP5)$ \dotfill & Lemma~\ref{lem:star-masterlist},& 
Page~\pageref{lem:star-masterlist} \\
protected \dotfill & Definition~\ref{def:protected},& Page~\pageref{def:protected} \\
$\Psi(X)$ \dotfill & Section~\ref{sec:coupled},& Page~\pageref{sec:coupled} \\
$\sigma(t)$ \dotfill & Definition~\ref{def:tenth}, & Page~\pageref{def:tenth} \\
$\sstar$ \dotfill & Section~\ref{def:super}, & Page~\pageref{def:super}\\
$\tmax=2\ell m$ \dotfill & Definition~\ref{def:tenth},& Page~\pageref{def:tenth} \\
$\TTmut{h}$  \dotfill & Definition~\ref{def:star-lower-tmut},& Page~\pageref{def:star-lower-tmut} \\
$\TTnmut{h}$ \dotfill & Definition~\ref{def:star-lower-tmut},& Page~\pageref{def:star-lower-tmut} \\
$\tx= \ell m /(K r^4 \kappa)$ \dotfill & Definition~\ref{def:tenth},& Page~\pageref{def:tenth} \\
$Y_h=\TTnmut{h}-\TTmut{h}$ \dotfill & Definition~\ref{def:star-lower-tmut},& Page~\pageref{def:star-lower-tmut}
\end{longtable}

\subsection{The easy regimes}
\label{sec:supeasy}

\begin{lemma}\label{lem:star-lower-die-instant}
Choose $x_0$ uniformly at random from $V(\sstar)$.
Let $X$ be the Moran process with $G(X)=\sstar$ and $X_0=\{x_0\}$.
The extinction probability of~$X$ is at least $k/(2r(m+k))$.
\end{lemma}
\begin{proof}
We have
\[\Pr(x_0 \notin R_1 \cup \dots \cup R_\ell) = 1 - \frac{\ell m}{\ell(m+k)+1} \ge 1 - \frac{m}{m+k} = 
\frac{k}{m+k}.\]
Moreover, if $x_0 \notin R_1 \cup \dots \cup R_\ell$, then $x_0$ has an in-neighbour of 
out-degree~$1$
so,  with probability at least $1/(1+r) \ge 1/(2r)$,
$x_0$ dies 
before spawning a mutant. The result therefore follows.
\end{proof}

\begin{lemma}\label{lem:star-lower-die-before-path}
Suppose $m \ge 12r$ and $x_0 \in R_1 \cup \dots \cup R_\ell$. 
Let $X$ be the Moran process with $G(X)=\sstar$ and $X_0=\{x_0\}$.
The extinction probability of~$X$  is at least $1/(26 r^2\ell)$.
\end{lemma}
\begin{proof}
Let $R$ be the reservoir containing $x_0$, and let $v_1 \dots v_k$ be the path associated with $R$. 
Let $\xi=\lfloor m/(2r) \rfloor$, $t^* = m/(4r^2)$ and   $J = [0,t^*]$.
For all $t\geq0$,
let $\calE^1$, $\calE^2$ and $\calE^3_t$  be events defined as follows. 
\begin{description}
\item $\mathcal{E}^1$:  $\vnc{v^*}{x_0}$   triggers in   $J$.
\item $\mathcal{E}^2$:  $\vmc{x_0}{v_1}$ triggers at most $\xi$ times in~$J$.
\item $\mathcal{E}^3_t$:   
$\displaystyle \min \{ t' > t \mid \mbox{for some $v \ne x_0$, $\vnc{v}{v_1}$  triggers at $t'$} \}
 < \min \{t' > t \mid \mbox{$\vmc{v_1}{v_2}$  triggers at $t'$}\}$.
\end{description}
Finally, let 
 $\Tvone{i}$ be the $i$'th time at which the clock $\vmc{x_0}{v_1}$ triggers and define
$\mathcal{E}^3 =  \bigcap_{i=1}^{ \xi} \mathcal{E}^3_{\Tvone{i}}$.

Suppose that events~$\calE^1$, $\calE^2$ and $\calE^3$ occur. We will show that
$X$ goes extinct. 
Let $\xi'$ be the number of times 
  that $v_1$ becomes a mutant   in~$J$.
By $\mathcal{E}^2$, $\xi' \leq \xi$.
By $\mathcal{E}^3$, for each of  the first $\xi'$ times that $v_1$ becomes
a mutant, it dies before spawning a mutant.
Thus, for all $t\in J$, $X_t \subseteq \{x_0,v_1\}$.
Also, by $\mathcal{E}^1$, $x_0$ dies in $J$.
As soon as $x_0$ dies, and $v_1$ dies for the $(\xi')$'th time,
$X$ is extinct.

We   bound $\Pr(\mathcal{E}^1 \cap \mathcal{E}^2 \cap \mathcal{E}^3)$ below. Since $\vnc{v^*}{x_0}$ has rate 
$1/(\ell m)$, we have
\begin{equation}\label{eqn:star-lower-die-before-path-1}
\Pr(\mathcal{E}^1) = 1-e^{-t^*/(\ell m)} = 1 - e^{-m/(4r^2\ell m)} \ge 1 - \left(1 - \frac{1}{8r^2\ell}\right) = 
\frac{1}{8r^2\ell}.
\end{equation}
Here the inequality follows by~\eqref{eq:ebounds}. Moreover, since $\vmc{x_0}{v_1}$ has 
rate $r$, by Corollary~\ref{cor:pchernoff} we have
\begin{equation}\label{eqn:star-lower-die-before-path-2}
\Pr(\mathcal{E}^2) \ge 1 - e^{-m/(12r)} \ge 1 - e^{-1}.
\end{equation}

For any $t\in J$, let $f$   be a possible value of 
$\filt_{t }(X)$. Let~$\Phi$ be the random variable containing the list of times in~$J$
at which  $\vnc{v^*}{x_0}$ and $\vmc{x_0}{v_1}$ trigger. 
Let $\varphi$ be a possible value of $\Phi$ which is consistent with the events $\filt_{t}( X) = f$ 
and $\mathcal{E}^1 \cap \mathcal{E}^2$.
Note that $\varphi$ determines $\mathcal{E}^1 \cap \mathcal{E}^2$. By 
memorylessness and independence of clocks in $\moranclocks(\sstar)$, we have
\[\Pr\left(\mathcal{E}^3_{t} \mid \filt_{t}(X) = f, \Phi = \varphi \right) = 
\frac{m-1}{m-1+r} = 1 - \frac{r}{m-1+r} \ge 1 - \frac{r}{m}.\]
Thus for all $i\in [\xi]$,
$\Pr\big(\mathcal{E}^3_{\Tvone{i}} \mid \mathcal{E}^1 \cap \mathcal{E}^2\big) \ge 1 - r/m$. 
It follows by a union bound that
\begin{equation}\label{eqn:star-lower-die-before-path-3}
\Pr\left(\mathcal{E}^3 \mid \mathcal{E}^1 \cap \mathcal{E}^2\right) \ge 1 -  \xi  \left(\frac{r}{m}\right) 
\geq \frac{1}{2}.
\end{equation}

Since $\mathcal{E}^1$ and $\mathcal{E}^2$ depend entirely on distinct clocks in $\moranclocks(\sstar)$ in fixed 
intervals, the two events are independent. Thus by  
\eqref{eqn:star-lower-die-before-path-1}--\eqref{eqn:star-lower-die-before-path-3}, we have
\[\Pr\left(\mathcal{E}^1 \cap \mathcal{E}^2 \cap \mathcal{E}^3\right) = \Pr\left(\mathcal{E}^1\right)\, 
\Pr\left(\mathcal{E}^2\right)\, \Pr\left(\mathcal{E}^3 \mid \mathcal{E}^1 \cap \mathcal{E}^2\right) \geq
\left(\frac{1}{8r^2\ell}\right)\left(1-\frac{1}{e}\right) \frac{1}{2}\geq
 \frac{1}{26r^2\ell},\]
and the result follows.
\end{proof}

\subsection{The difficult regime}
\label{sec:supdifficult}

\begin{definition}\label{def:star-lower-cr}
Let   $K=70$ 
and
$\kappa = \max \{3k, K r^4 \log n\}$.
\defend{}
\end{definition}

\newcommand{\statestarlowerlem}{
\statelargen
Suppose that 
$\ell \ge  K r^4 \kappa \log n$, 
 $m \ge 6r^2\kappa $ 
and $n \ge n_0$.
Fix $x_0 \in  R_1 \cup \dots \cup R_\ell$.
Let $X$ be the Moran process with $G(X)=\sstar$ and $X_0=\{x_0\}$.
Then the extinction probability of~$X$ 
is at least $1/(7K r^4 \kappa )$.
}

\begin{lemma}\label{lem:star-lower-hard}
\statestarlowerlem
\end{lemma}

The following corollary, which applies to the regime in which $\kappa=3k$, is immediate.
 
\begin{corollary}\label{cor:superstark}
\statelargen
Suppose that 
$k\geq  (K/3) r^4 \log n$,
$\ell \ge  3 K r^4 k \log n$, $m \ge 18 r^2 k $ and $n \ge n_0$.
Fix $x_0 \in  R_1 \cup \dots \cup R_\ell$.
Let $X$ be the Moran process with $G(X)=\sstar$ and $X_0=\{x_0\}$.
Then the extinction probability of~$X$ 
is at least $1/(21K r^4  k )$.

\end{corollary}

The crux of our proof of Lemma~\ref{lem:star-lower-hard}  is Lemma~\ref{lem:star-masterlist}. In order to state this 
lemma, we require the following additional definitions. 

\begin{definition}\label{def:tenth}
Let  $\tx = \ell m/(K r^4 \kappa )$,  and  $\tmax = 2\ell m$. For all $i \in \Zzero$, let $I_i = (i\kappa , 
(i+1)\kappa ]$.
For all $t \in [0, \tmax]$, let
\begin{equation*}
\sigma(t) = \big|\{t' \le t \mid v_k \textnormal{ spawns a mutant onto $v^*$ 
  at time }t'\}\big|\,.\defenddisp{}
\end{equation*}
\end{definition}

The reason that we give $\tx$ its name is
that we will be most concerned with the case in which $x_0$ dies  in the interval $[0,\tx]$.

\begin{definition}\label{def:clear-before}
Let $I \subseteq [0, \infty)$ be an interval, let $R \in \{R_1, \dots, R_\ell\}$, suppose $x_0 \in R$, 
and let $v_1 \dots v_k$ be the path associated with $R$. We say that \emph{$v_1$ clears before spawning a mutant 
within $I$} if at least one of the following statements holds:
\begin{enumerate}[(i)]
\item $\vmc{v_1}{v_2}$ does not trigger in $I$, or
\item for some $v \ne x_0$, $\vnc{v}{v_1}$ triggers in $I$ before $\vmc{v_1}{v_2}$ first triggers in $I$.\defend
\end{enumerate}
\end{definition}

\begin{definition}\label{def:protected}
Let $i \in \Zzero$, let $R \in \{R_1, \dots, R_\ell\}$, suppose $x_0 \in R$, and let $v_1 \dots v_k$ 
be the path associated with $R$. We say that \emph{$v_2$ is protected in $I_i$} if both of the following properties 
hold.
\begin{enumerate}[(i)]
\item $v_1$ clears before spawning a mutant within $I_i$.
\item For all $t \in I_i$ such that $\vmc{x_0}{v_1}$ triggers at time $t$, $v_1$ clears before spawning a mutant 
within $(t,(i+1)\kappa ]$.\defend{}
\end{enumerate}
\end{definition}

In particular, suppose that $v_2$ is protected in $I_i$ and that $x_0$ is the only mutant in $R$ for the duration 
of $I_i$. Then as we will see in the proof of Lemma~\ref{lem:star-lower-hard}, $v_2$ does not become a mutant in $I_i$.

\begin{definition}\label{def:star-lower-path-clear}
Let $i \in \Zzero$, let $R \in \{R_1, \dots, R_\ell\}$, suppose $x_0 \in R$, and let $v_1 \dots v_k$ 
be the path associated with $R$. We say that \emph{$v_1 \dots v_k$ clears within $I_i$} if there exist $v_0 \in 
R \setminus \{x_0\}$ and times 
$ i\kappa < 
t_1 <  \dots < t_{k+1} \leq (i+1)\kappa$ 
satisfying 
both
of the following properties.
\begin{enumerate}[(i)]
\item  
For all $j \in [k]$, 
$\vnc{v_{j-1}}{v_j}$ triggers at time $t_j$,
and $\vnc{v_k}{v^*}$ triggers at time $t_{k+1}$.
\item $\vmc{x_0}{v_1}$ does not trigger in the interval $[t_1, t_2]$.\defend
\end{enumerate}
\end{definition}

The purpose of this definition is the following. Suppose
that $X_{i\kappa } \subseteq \{x_0, v_1, \dots, v_k, v^*\}$, that neither $v_1$ nor $v^*$ spawns a mutant 
within $I_i$, and that $v_1 \dots v_k$ clears within $I_i$. Then, as we will see in the proof of 
Lemma~\ref{lem:star-lower-hard}, we will have $X_{(i+1)\kappa } \subseteq \{x_0,v_1\}$.

Our main task will be to prove the following lemma.

\newcommand{\statestarlowermaster}{
\statelargen  Suppose 
$\ell \ge K r^4 \kappa \log n$, 
 $m \ge 6r^2\kappa $ 
and $n \ge n_0$. Let $R \in \{R_1, \dots, R_\ell\}$, 
and let $v_1 \dots v_k$ be the path associated with $R$.
Fix $x_0 \in R$.
Let $X$ be the Moran process with $G(X)=\sstar$ and $X_0=\{x_0\}$.  
Then, with probability at least $1/(7K r^4 \kappa )$, the following events
occur simultaneously.
\begin{description}
\item \sP1:  $\forall t\leq \tmax$,
$\sigma(t) \geq \lfloor \ell/(2r) \rfloor + 1$ or 
$v^*$ does not spawn a mutant at time~$t$. 
\item \sP2: $\sigma(\tx) \geq \lfloor \ell/(2r) \rfloor + 1$ or $x_0 \notin X_{\tx}$. 
\item \sP3:  For all integers $i$ with $0 \le i \le \tx$, $v_1 \dots v_k$ clears within $I_i$.
\item \sP4:  For all integers $i$ with $0 \le i \le \tx$, the clock $\vmc{v_k}{v^*}$ triggers at most $\lfloor 
2r\kappa \rfloor$ times 
within~$I_i$.
\item \sP5:  For all but at most $8r^2\tx/m$ integers $i$ with $0 \le i \le \tx/\kappa $, $v_2$ is protected in $I_i$.
\end{description}
}

\begin{lemma}\label{lem:star-masterlist}
\statestarlowermaster
\end{lemma}

Note that the definition of  \sP5 considers $i$ up to $\tx/\kappa$,
because it corresponds to at most $\tx/\kappa$ intervals of length~$\kappa$.
The definitions of  \sP3 and  \sP4 consider larger values of~$i$.
In fact, it is only necessary to take $i$ up to $\tx/\kappa+2$ in  \sP3 and  \sP4 but
we state the lemma as we did to avoid clutter. 
As a first step towards proving Lemma~\ref{lem:star-masterlist},  
  we prove Lemmas~\ref{lem:star-lower-centre-mut} and~\ref{lem:star-lower-apply-centre-mut} which 
give a lower bound on the probability that  \sP1 and  \sP2 hold.

\begin{definition}\label{def:star-lower-tmut}
Let $R \in \{R_1, \dots, R_\ell\}$, let $v_1 \dots v_k$ be the path associated with $R$, and suppose 
$x_0 \in R$. 
Let $X$ be the Moran process with $G(X)=\sstar$ and $X_0=\{x_0\}$.
We give two mutual recurrences to 
define stopping times $\TTnmut{h}$ for all $h\in\Zzero$ and $\TTmut{h}$ for all $h \in \Zone$. 
The subscript ``$\mathsf n$'' stands for ``non-mutant'' and the subscript ``$\mathsf m$'' stands for ``mutant''.
\begin{align*}\TTnmut{h} &= \begin{cases}
 0, & \mbox{if $h=0$,}\\
\min\{t > \TTmut{h} \mid
\mbox{$t=\tmax$ or some clock $\vnc{v}{v^*}$ with $v \ne v_k$ triggers at~$t$} 
\}, & \mbox{otherwise.}
\end{cases}\\
\TTmut{h} &= 
\min\{t > \TTnmut{h-1} \mid
\mbox{$t=\tmax$ or $v_k$ spawns a mutant onto~$v^*$ at~$t$} 
\}.\end{align*}
Finally, for all $h\in\Zone$,  let $Y_h = \TTnmut{h} - \TTmut{h}$.\defend{}
\end{definition}

\begin{lemma}\label{lem:star-lower-centre-mut}
\statelargen Suppose $\ell \ge K r^4\kappa \log n$
and $n \ge n_0$. Let $R \in \{R_1, \dots, 
R_\ell\}$ and let $v_1 \dots v_k$ be the path associated with $R$. Let $x_0 \in R$.
Let $X$ be the Moran process with $G(X)=\sstar$ and $X_0=\{x_0\}$. Then
\[\Pr\left(\sum_{i=1}^{\lfloor\ell/(2 r) \rfloor} Y_i < \frac1r \right) \ge 1-\frac{1}{n^2}.\]
\end{lemma}
\begin{proof} \letlargen
We claim that $Y_1, \dots, Y_{\lfloor \ell/(2r) \rfloor}$ are stochastically dominated above by 
$\lfloor\ell/(2r)\rfloor$ independent exponential variables, each with parameter $\ell-1$. To see this claim, fix $i 
\in \Zone$,
 $t \ge 0$, and $y_1, \dots, y_{i-1}, \tmut{i} > 0$. Let $f_i$ be a possible value of $\filt_{\tmut{i}}( X)$. 
Suppose that the events $Y_1 = y_1, \dots, Y_{i-1} = y_{i-1}$, $\TTmut{i} = \tmut{i}$ and $\filt_{\tmut{i}}( X) 
= f_i$ are consistent, and note that in this case $\filt_{\tmut{i}}( X) = f_i$  determines the other 
events. 

  If $t\geq 0$ satisfies
  $\tmut{i} + t \ge \tmax$, it follows that $\tmut{i}+t \ge \TTnmut{i}$ and hence 
  if $\filt_{\tmut{i}}(X) = f_i$ then $Y_i = \TTnmut{i} - 
\tmut{i} \le t$. Hence, for  all such~$t$,
\begin{equation}\label{eqn:star-lower-centre-mut-1}
\Pr\big(Y_i \le t\mid \filt_{\tmut{i}}( X)=f_i\big) = 1 \ge 1 - e^{-(\ell-1)t}.
\end{equation}
Suppose instead that $\tmut{i}+t < \tmax$.  
If $\filt_{\tmut{i}}(X)=f_i$ then $Y_i \le t$ if and only if some clock $\vnc{v}{v^*}$ with $v \ne 
v_k$ triggers in the interval $(\tmut{i}, \tmut{i}+t]$. These clocks have total rate $\ell-1$, and so by 
memorylessness we have
\begin{equation}\label{eqn:star-lower-centre-mut-2}
\Pr\big(Y_i \le t\mid \filt_{\tmut{i}}( X) = f_i\big) = 1 - e^{-(\ell-1)t}.
\end{equation}

Since \eqref{eqn:star-lower-centre-mut-1} and \eqref{eqn:star-lower-centre-mut-2} apply to every value of $f_i$ 
consistent with $Y_1 = y_1, \dots, Y_{i-1} = y_{i-1}$ and $\TTmut{i} = \tmut{i}$, it follows that 
\[\Pr\big(Y_i \le t\mid Y_1 = y_1, \dots, Y_{i-1} = y_{i-1}\big) \ge 1 - 
e^{-(\ell-1)t}.\]
Thus $\sum_{i=1}^{\lfloor\ell/(2r)\rfloor} Y_i$ is stochastically dominated above by a sum $S$ of $\lfloor \ell/(2r) 
\rfloor$ i.i.d.\ exponential variables with parameter $\ell-1$. 
Corollary~\ref{cor:expsum} applies since 
\[\frac{1}{r} \geq \frac{3 \lfloor \ell/(2r) \rfloor}{2(\ell-1)},\] so we have
\[
\Pr\left(\sum_{i=1}^{\lfloor \ell/(2 r) \rfloor} Y_i < \frac{1}{r}\right) 
\geq 
\Pr\left( S < \frac{1}{r}\right) \geq 1 - e^{- (\ell-1) /(16 r)}\geq 1-\frac{1}{n^2},\]
 as required.  
\end{proof}

We are now in a position to prove that \sP1 and \sP2 occur with reasonable probability.

\begin{lemma}\label{lem:star-lower-apply-centre-mut}
\statelargen  Suppose $\ell \ge K r^4\kappa \log n$, $m \ge 2$ and $n \ge n_0$. Let $R \in \{R_1, \dots, 
R_\ell\}$ and let $v_1 \dots v_k$ be the path associated with $R$. Let $x_0 \in R$.
Let $X$ be the Moran process with $G(X)=\sstar$ and $X_0=\{x_0\}$. 
 Then
$\Pr(\sP1 \cap \sP2)   \ge 
1/( 6 K r^4 \kappa )$.\end{lemma}
\begin{proof}
\letlargen 
Consider the process $\Psi(X)$.  Define the following three events.
\begin{description}
\item $\calE_1$:  
no star-clock $\vsmc{v^*}{v}$  (for any~$v$)
triggers in $[0,1/r]$.
\item  $\mathcal{E}_2$: the
star-clock $\vsnc{v^*}{x_0}$ triggers in $[0,\tx-1]$. 
\item  $\mathcal{E}_3$:    $\sum_{i=1}^{\lfloor\ell/(2r)\rfloor} Y_i < 1/r$. 
\end{description}
We will prove that $\Pr(\mathcal{E}_1 \cap \mathcal{E}_2 \cap \mathcal{E}_3) \ge 1/(6 K r^4 \kappa )$, and that
if
$\mathcal{E}_1 \cap \mathcal{E}_2 \cap \mathcal{E}_3$  occurs then so does
\sP1 and \sP2.

We  first bound $\Pr(\mathcal{E}_1 \cap \mathcal{E}_2 \cap \mathcal{E}_3)$ below. The sum of the parameters of the 
star-clocks in $\{\vsmc{v^*}{v}  \}$ is $r$, so $\Pr(\mathcal{E}_1) = e^{-1}$. We have $\tx = \ell m/(K r^4\kappa ) 
\ge m\log n \ge 2\log n_0$ by hypothesis so, by choice of~$n_0$, we  may assume $\tx\ge 25$.  The parameter of the star-clock 
$\vsnc{v^*}{x_0}$ is $1/(\ell m)$, so using \eqref{eq:ebounds} we have
\[\Pr(\mathcal{E}_2) = 1 - e^{-(\tx-1)/(\ell m)} \ge 1-e^{-24\tx/(25\ell m)} \ge \frac{12\tx}{25\ell m} = 
\frac{12}{25 K r^4 \kappa }.\]

Note that $\mathcal{E}_1$ and $\mathcal{E}_2$ 
are independent of each other by the definition of the
star-clock process $P^*(\sstar)$
and the fact that the intervals in the definitions of $\mathcal{E}_1$ and $\mathcal{E}_2$ 
are fixed: $\tx = \ell m/(K r^4 \kappa )$ does not depend on the
evolution of $\Psi(X)$.
So we have $\Pr(\mathcal{E}_1 \cap \mathcal{E}_2) \ge 
12/(25 e K r^4 \kappa)$. Finally, by Lemma~\ref{lem:star-lower-centre-mut} together with the fact that $\kappa  \le n$, 
it follows that
\begin{equation}\label{eqn:star-lower-apply-centre-mut}
\Pr(\mathcal{E}_1 \cap \mathcal{E}_2 \cap \mathcal{E}_3) \ge \Pr(\mathcal{E}_1 \cap \mathcal{E}_2) - 
\Pr(\overline{\mathcal{E}_3}) \ge \frac{12}{25 e K r^4\kappa } - \frac{1}{n^2} \ge \frac{1}{6 K r^4 \kappa }.
\end{equation}

We  next show that $\mathcal{E}_1$ and $\mathcal{E}_3$ together imply that 
\sP1 occurs.  
If $v^*$ does not spawn a mutant before time $\tmax$ then \sP1 occurs, so suppose
instead that
$v^*$ spawns a mutant   for the first time at some time $\tspawn \le \tmax$. 
( The ``$\mathsf s$'' subscript in~$\tspawn$ stands for ``spawn''.)
We must show that $\sigma(\tspawn) \geq \lfloor \ell/(2r) \rfloor + 1$.
This will ensure that \sP1 occurs since $\sigma(t)$ is monotonically increasing.

Since $v^*$ spawns no mutants   before time $\tspawn$, we have $X_t \subseteq \{x_0, v_1, \dots, 
v_k, v^*\}$ for all $t < \tspawn$, and so (recalling Definition~\ref{def:localtime})
\begin{equation}\label{eqn:star-lower-apply-centre-mut-2}
\iin{X}{v^*}{ \tspawn} \leq \sum_{\substack{i \geq 1, \\ \TTmut{i} < \tspawn}} Y_i.
\end{equation}

Since $\mathcal{E}_3$ occurs, we have 
\begin{equation}\label{eqn:star-lower-apply-centre-mut-3}
  \sum_{\substack{1 \leq i \leq \lfloor \ell/(2r)\rfloor,  \\ \TTmut{i} <  \tspawn}} Y_i < \frac1r.
\end{equation}
However, since $\mathcal{E}_1$ occurs and $v^*$ spawns a mutant   at $\tspawn$, we have 
$\iin{X}{v^*}{\tspawn} \ge 1/r$. Hence by \eqref{eqn:star-lower-apply-centre-mut-2} and 
\eqref{eqn:star-lower-apply-centre-mut-3}, 
\[  \sum_{\substack{i \geq 1, \\ \TTmut{i} < \tspawn}} Y_i
>  \sum_{\substack{1 \leq i \leq \lfloor \ell/(2r)\rfloor,  \\ \TTmut{i} <  \tspawn}} Y_i.\]
Hence $\TTmut{\lfloor \ell/(2r)\rfloor +1} < \tspawn$. 
Thus, $\sigma( \tspawn) \geq \lfloor \ell/(2r)\rfloor +1$, and \sP1 occurs.

Finally, we  show that \sP1,  $\mathcal{E}_2$ and $\mathcal{E}_3$ together imply that 
\sP2 occurs.  Suppose 
$\sigma(\tx) \leq \lfloor \ell/(2r) \rfloor$.  We have $\tx \le \ell m \le \tmax$, so 
 by \sP1, $v^*$ spawns no mutants   in $[0,\tx]$. Hence as in~\eqref{eqn:star-lower-apply-centre-mut-2}, we have
\[\iin{X}{v^*}{\tx}\ \  \le 
\sum_{\substack{i \ge 1, \\ \TTmut{i} < \tx}} Y_i =
\sum_{\substack{i \ge 1, \\ \TTmut{i} \le \tx}} Y_i.\]
Since $\sigma(\tx) \leq \lfloor \ell/(2r) \rfloor$, it follows by $\mathcal{E}_3$ that $\iin{X}{v^*}{\tx} < 1/r$ and 
hence $\iout{X}{v^*}{\tx} \geq \tx-1/r > \tx -1$. Since $\mathcal{E}_2$ occurs, it follows that $v^*$ spawns a 
non-mutant onto $x_0$   at some time $t \le \tx$. Since $v^*$ spawns no mutants    in $[0,\tx]$, $x_0$ cannot 
become a mutant in $(t,\tx]$, so $x_0 \notin X_{\tx}$ and  \sP2 occurs.

Thus $\mathcal{E}_1 \cap \mathcal{E}_2 \cap \mathcal{E}_3$ implies that 
\sP1 and \sP2 occur, and so the result follows from~\eqref{eqn:star-lower-apply-centre-mut}.
\end{proof}
 
Lower bounds for the probabilities that properties \sP3--\sP5   hold follow from Chernoff bounds without too much 
difficulty.

\begin{lemma}\label{lem:star-lower-path-clear}
\statelargen Suppose $\ell \ge K r^4 \kappa \log n$, $m \ge 2$ and $n \ge n_0$. Let $R \in \{R_1, \dots, 
R_\ell\}$ and let $v_1 \dots v_k$ be the path associated with $R$. Let $x_0 \in R$.
Let $X$ be the Moran process with $G(X)=\sstar$ and $X_0=\{x_0\}$. Then
$\Pr(\sP3)  \ge 1-1/n$.
\end{lemma}
\begin{proof}
\letlargen Fix $i \in \Zzero$ --- we will bound the probability that $v_1 \dots v_k$ clears 
 within $I_i$ (as in Definition~\ref{def:star-lower-path-clear}).
 Let $v_0 \in R \setminus \{x_0\}$ be arbitrary.  
For all $h\in\Zzero$, let
\begin{align*}
T_{1,h} &= \min
\{ t \geq i \kappa + h \mid \mbox{$t=i \kappa + h + 1/2$ or $\vnc{v_0}{v_1}$ triggers at $t$}\}\\
    T_{2,h} &= 
    \min\{t > T_{1,h} \mid \mbox{
$t= i\kappa+h+1$ or  ($\vmc{x_0}{v_1}$ does not trigger in $[T_{1,h},t)$ and
$\vnc{v_1}{v_2}$ triggers at $t$)}\}.
\end{align*}
Let $\calE_h$ be the event that $T_{1,h} < i\kappa+h+1/2$ and $T_{2,h} < i\kappa+h+1$.

The probability that the clock $\vnc{v_0}{v_1}$ triggers in $[i\kappa +h,i\kappa +h+1/2)$ 
is $1-e^{-1/2}$.
For any $t_1 \in  [i\kappa +h,i\kappa +h+1/2)$ 
the probability that there is a $t_2 \in (t_1,i\kappa+h+1)$
such that 
 $\vnc{v_1}{v_2}$ triggers 
at $t_2$, and $\vmc{x_0}{v_1}$ does not trigger in $[t_1,t_2]$ is
$(1-e^{-(r+1)(i\kappa+h+1-t_1)})/(r+1)$. To see this, note that the
$1-e^{-(r+1)(i\kappa+h+1-t_1)}$ factor corresponds to the probability that either $\vnc{v_1}{v_2}$ or 
$\vmc{x_0}{v_1}$  
triggers in the relevant interval (together, they correspond to a Poisson process with rate $r+1$). 
The $1/(r+1)$ factor corresponds to the probability that it is actually $\vnc{v_1}{v_2}$ rather than $\vmc{x_0}{v_1}$
that triggers first. Since the relevant interval has length at least~$1/2$,
the product of these two factors is at least
$(1-e^{-(r+1)/2})/(r+1)$.
So
\[\Pr(\calE_h) \geq \big(1-e^{-1/2}\big) \big(1-e^{-(r+1)/2}\big)/\big(r+1\big) \geq (1-e^{-1/2})(1-e^{-1})/(r+1) \geq 
\frac{1}{5(r+1)}\geq \frac{1}{10r}\,.\]

Moreover, the events $\{\calE_h \mid h \in \Zzero\}$ are mutually independent, as they depend only on the behaviour of 
clocks in $\moranclocks(\sstar)$ in fixed disjoint intervals. Thus
\begin{align}\label{eqn:star-lower-path-clear-1}
\Pr\big(\calE_h\textnormal{ holds for some }0 \le h \le (K r^4 \log n)/3\big) 
	&\ge 1 - \left(1-\frac{1}{10r}\right)^{\lfloor (K r^4 \log n)/3\rfloor +1} \nonumber \\
	&\ge 1 - \left(1-\frac{1}{10r}\right)^{(K r^4 \log n)/3} \nonumber \\
	&\ge 1 - e^{-(K r^3 \log n)/30} \nonumber\\
	&\ge 1 - e^{-2 \log n} = 1 - \frac{1}{n^2}\,.
\end{align}
 
Now, for all integers $j$ with $3 \le j \le k+1$, define
\[T_j = 
\begin{cases}
\min \{t > i\kappa +\kappa /2    \mid \vnc{v_2}{v_3}  \textnormal{ triggers at }t\}
	& \textnormal{ if } j = 3,\\
\min \{t > T_{j-1} \mid \vnc{v_{j-1}}{v_j}  \textnormal{ triggers at }t\}
    & \textnormal{ if } 4 \le j \le k,\\
\min \{t > T_{k} \mid \vnc{v_k}{v^*} \textnormal{ triggers at }t\} 
	& \textnormal{ if } j = k+1.\\
\end{cases}\]
Note that if $T_{k+1} < (i+1)\kappa $, then $T_{3}, \dots, T_{k+1} \in I_i$.
By memorylessness and independence of distinct clocks in $\moranclocks(\sstar)$, the random variables $T_{3} - (i\kappa 
+\kappa /2), T_{4} - T_{3}, \dots, T_{k+1} - T_{k}$ are $k-1$ i.i.d.\ exponential variables with rate~$1$, and 
$T_{k+1}-(i\kappa +\kappa /2)$ is their sum. 
Corollary~\ref{cor:expsum} applies because $\kappa/2\geq 3(k-1)/2$, so
\begin{align}\label{eqn:star-lower-path-clear-2}\Pr\big(T_{k+1} < (i+1)\kappa \big) &= \Pr\big(T_{k+1} - (i\kappa 
+\kappa /2) < \kappa /2\big) \nonumber\\
&\geq 1 - e^{- \kappa/32}\ge 1-e^{-(K r^4\log n)/32} \ge 1 - \frac{1}{n^2}.
\end{align}

Suppose $\calE_{h}$ holds for some $0 \le h \le (K r^4\log n)/3$. Note that $T_{2,h} \le i\kappa + \kappa/2$. Suppose  
further that $T_{k+1} < (i+1)\kappa $. 
Then setting $t_1 = T_{1,h}$, $t_2 = T_{2,h}$ and $t_j = T_{j}$
for all $j\in\{3, \dots, k+1\}$,
we see that $t_1, \dots, t_{k+1}$ satisfy the requirements of Definition~\ref{def:star-lower-path-clear} and so $v_1 
\dots v_k$ clears within $I_i$. It therefore follows by \eqref{eqn:star-lower-path-clear-1}, 
\eqref{eqn:star-lower-path-clear-2}, and a union bound that
\[\Pr(v_1 \dots v_k \textnormal{ clears within }I_i) \ge 1 - \frac{2}{n^2}.\]
Hence by a union bound over all integers $i$ with $0 \le i \le \tx$ and by $\tx\leq n/4$ it follows that
\[\Pr(\sP3)  \ge 1 - 2\tx\cdot \frac{2}{n^2} \ge 1-\frac{n}{2} \cdot 
\frac{2}{n^2} \ge 1-\frac{1}{n}\,.\qedhere\]
\end{proof}

\begin{lemma}\label{lem:star-lower-clocks-trigger-rarely}
\statelargen Suppose 
 $n \ge n_0$. Fix $x_0 \in R_1 \cup \dots \cup 
R_\ell$.
Let $X$ be the Moran process with $G(X)=\sstar$ and $X_0 = \{x_0\}$.  With probability at least $1 - 1/n^2$, for all 
integers $i$ with $0 \le i \le \tx$, 
every  clock in $\moranclocks(\sstar)$ triggers at most $\lfloor 2r\kappa \rfloor$ times in $I_i$.
\end{lemma}
\begin{proof}
\letlargen Fix a given clock $C \in \moranclocks(\sstar)$, fix $i \in \Zzero$ with $i \le \tx$, and write $a \le r$ for 
the rate of $\moranclocks(\sstar)$. The number of times that $C$ triggers in $I_i$ follows the Poisson distribution with 
parameter $a\kappa $. Since the number of triggers is an integer and $2r\kappa  \ge 2a\kappa $, by 
Corollary~\ref{cor:pchernoff} we have
\begin{align*}
\Pr(C\textnormal{ triggers at most }\lfloor 2r\kappa \rfloor\textnormal{ times in }I_i) &= \Pr(C\textnormal{ triggers at 
most }2r\kappa \textnormal{ times in }I_i)\\
&\ge \Pr(C\textnormal{ triggers at most }2a\kappa \textnormal{ times in }I_i)\\
&\ge 1 - e^{-a\kappa /3} \ge 1 - e^{-(K r^4\log n)/3}.
\end{align*}
There are at most $2n^2$ clocks in $\moranclocks(\sstar)$ and at most $\tx+1\le 2n^2$ choices of $i$. Thus by a union 
bound, with probability at least $1 - 4n^4e^{-K r^4\log n/3} \ge 1 - 1/n^2$, no single clock in $\moranclocks(\sstar)$ 
triggers more than $\lfloor 2r\kappa \rfloor$ times in any interval $I_i$ with $0 \le i \le \tx$.
\end{proof}

The following corollary follows immediately from Lemma~\ref{lem:star-lower-clocks-trigger-rarely}.
Of course, the probability bound in the corollary can be strengthened to $1-1/n^2$, but we state what we 
will later use.

\begin{corollary}\label{cor:star-lower-full-path-not-bad}
\statelargen  Suppose 
  $n \ge n_0$. Fix $x_0 \in R_1 \cup \dots 
\cup R_\ell$.
Let $X$ be the Moran process with $G(X)=\sstar$ and $X_0=\{x_0\}$.  Then
 $
\Pr(\sP4)  \ge 1 -  1/n$.\hfill \qedsymbol\end{corollary}

The following lemma gives 
a lower bound on the probability that \sP5 occurs.
In this lemma,
 we require that 
$m \ge 6r^2\kappa$,  
rather than $m \ge 2$, which we have so far been assuming.

\begin{lemma}\label{lem:star-lower-path-not-fill}
\statelargen Suppose $\ell \ge K r^4 \kappa \log n$, 
 $m \ge 6r^2\kappa $ 
 and $n \ge n_0$. Let $R \in \{R_1, 
\dots, R_\ell\}$ and let $v_1 \dots v_k$ be the path associated with $R$. Fix $x_0 \in R$.
Let $X$ be the Moran process with $G(X)=\sstar$ and $X_0=\{x_0\}$. Then
$\Pr(\sP5) \geq 1-1/n$. 
\end{lemma}
\begin{proof}
\letlargen Fix $i \in \Zzero$. For all $t \in I_i$, define the following events $\calE_t^1$ and $\calE_i^2$.
\begin{description}
\item $\calE_t^1$:  
	$\min \{ t' > t \mid
\mbox{for some $v \ne x_0$, $\vnc{v}{v_1}$ triggers at $t'$} \} < 
	\min \{t' > t \mid \mbox{$\vmc{v_1}{v_2}$   triggers at $t'$}\}$.
\item $\mathcal{E}_i^2$:
	$\vmc{x_0}{v_1}$  triggers at most $\lfloor 2r\kappa \rfloor$   times in $I_i$.
\end{description}
Thus $\mathcal{E}^1_t$ occurs if and only if $v_1$ clears before spawning a mutant within $(t,\infty)$. Let $T_{i,0} = 
i\kappa $, and let $T_{i,h}$ be the $h$'th time in $I_i$ at which the clock $\vmc{x_0}{v_1}$ triggers, or 
$(i+1)\kappa $ if no such time exists. Note that if $\bigcap_{h=0}^{\lfloor 2r\kappa \rfloor} \mathcal{E}^1_{T_{i,h}} 
\cap \mathcal{E}^2_i$ occurs, then $v_2$ is protected in $I_i$.

Now consider any $t \in I_i$ and let $f_t$ be a possible value of $\filt_t( X)$. By memorylessness, we have
\[\Pr\big(\mathcal{E}_t^1 \mid \filt_t( X) = f_t\big) = \frac{m-1}{m-1+r} = 1 - \frac{r}{m-1+r} \ge 1 - \frac{r}{m}.\]
In particular, since the event $T_{i,h} = t$ is determined by $\filt_t( X)$, it follows by a union bound that
\[\Pr\left(\bigcap_{h=0}^{\lfloor 2r\kappa \rfloor}\mathcal{E}_{T_{i,h}}^1\right) \ge 1 - \frac{(\lfloor 2r\kappa\rfloor 
+1)r}{m} \ge 1 - \frac{3r^2\kappa }{m}.\]
By Lemma~\ref{lem:star-lower-clocks-trigger-rarely} we have $\Pr(\mathcal{E}_i^2) \ge 1-1/n^2$, so 
it 
follows by a union bound that
\begin{equation}\label{eqn:star-lower-path-reset}
\Pr(v_2 \textnormal{ is protected in } I_i) \ge \Pr\left(\bigcap_{h=0}^{\lfloor 2r\kappa\rfloor 
}\mathcal{E}_{T_{i,h}}^1 
\cap \mathcal{E}_i^2 \right) \ge 1 - \frac{3r^2\kappa }{m} - \frac{1}{n^2}.
\end{equation}

Since $I_0, I_1, \dots$ are disjoint intervals, the events that $v_2$ is or is not protected in  these
intervals 
are independent by memorylessness. Thus the number of intervals $I_i$ with $0 \le i \le \tx /\kappa $ in which $v_2$ is 
not protected is stochastically dominated above by a binomial distribution consisting of
$\lfloor \tx /\kappa\rfloor+1$  Bernoulli trials, each with success probability $3r^2\kappa /m+1/n^2$. This 
distribution has expectation 
\[\left(\left\lfloor \frac{\tx}{\kappa}\right\rfloor+1\right) \left( \frac{3r^2\kappa}{m}+\frac{1}{n^2} \right)\le 
\frac{4r^2\tx}{m} = \frac{4r^2\ell}{K r^4 \kappa},\]
so by Lemma~\ref{lem:bchernoff} we have
\[\Pr(\overline{\sP5})\le  e^{-(1/6)8r^2 \tx/m} 
=
 e^{-(4/3)r^2 \ell/(K r^4 \kappa )} \le 
e^{-(4/3)r^2 \log n} \le \frac{1}{n}.\]
Here the penultimate inequality follows since $\ell \ge K r^4 \kappa \log n$ by hypothesis. The result therefore 
follows.
\end{proof}

Now that we have proved lower bounds on the probability that each of 
\sP1--\sP5 occur, Lemma~\ref{lem:star-masterlist} follows easily.

{\renewcommand{\thetheorem}{\ref{lem:star-masterlist}}
\begin{lemma}
\statestarlowermaster
\end{lemma}
\addtocounter{theorem}{-1}
}

\begin{proof} 
$\Pr(\sP1 \cap \cdots \cap \sP5) \geq
\Pr(\sP1 \cap \sP2) -    \Pr(\overline{\sP3}) - \Pr(\overline{ \sP4}) - \Pr(\overline{ \sP5})$.
\letlargen Then we
bound each term on the right-hand side   by applying (in order) Lemma~\ref{lem:star-lower-apply-centre-mut}, 
Lemma~\ref{lem:star-lower-path-clear}, Corollary~\ref{cor:star-lower-full-path-not-bad} and 
Lemma~\ref{lem:star-lower-path-not-fill} to obtain
\[\Pr( \sP1 \cap \dots \cap  \sP5) \ge 
\frac{1}{6 K r^4 \kappa } - \frac{3}{n} \ge \frac{1}{7 K r^4 \kappa},\]
as required. The final inequality follows since, by hypothesis, $\kappa  \le \ell/(K r^4 \log n) \le n/\log n$.
\end{proof}

We are now at last in a position to prove Lemma~\ref{lem:star-lower-hard}, which we will then use to prove 
Theorem~\ref{thm:star-lower}.  

{\renewcommand{\thetheorem}{\ref{lem:star-lower-hard}}
\begin{lemma}
\statestarlowerlem 
\end{lemma}
\addtocounter{theorem}{-1}
}

\begin{proof}
\letlargen 
Let $R$ be a reservoir in $\{R_1,\ldots,R_\ell\}$
and let $v_1\dots v_k$ be the path associated with~$R$. Suppose $x_0\in R$.
By Lemma~\ref{lem:star-masterlist}, it suffices to 
assume that 
\sP1 -- \sP5 occur and to prove that $X$ goes extinct.

Recall the definition of~$\sigma(t)$ from Definition~\ref{def:tenth}.
Note that $\sigma(0)=0$ and
$\sigma(t)$ is monotonically increasing in~$t$. We will first bound $\sigma(\tx)$ from above (assuming that \sP1--\sP5 
occur).
Consider an interval $I_i$ with $0 \le i \le \tx$ (technically, we need only consider 
$0\leq i \leq \tx/\kappa$, but the extra generality does no harm
and  we will later need to consider slightly larger~$i$). 
Note that $I_i \subseteq [0,\tmax]$ since $(i+1)\kappa  \le 2\tx\kappa = \tmax/(K r^4)$. Suppose that $\sigma(i\kappa ) 
\le \lfloor\ell/(2r) \rfloor - \lfloor 2r\kappa \rfloor$.  We will derive an upper bound on $\sigma((i+1)\kappa )$ by 
splitting into cases.

\textbf{Case 1: $\boldsymbol{i > 0}$ and $\boldsymbol{v_2}$ is protected in $\boldsymbol{I_{i-1}}$ and 
$\boldsymbol{I_i}$.}
First note that 
since \sP1 occurs and $\sigma(i\kappa ) \le \lfloor\ell/(2r)\rfloor$, $v^*$ does not spawn a mutant   over the 
course of $[0,i\kappa]$ and so 
\[X_t \subseteq \{x_0, v_1, \dots, v_k, v^*\}\textnormal{ for all } t \in [0,i\kappa].\]

Now, suppose for a contradiction that $v_2$ becomes a mutant at some time 
$\widehat{t}_2 \in I_{i-1}$. 
Then $v_1$ must have become a mutant beforehand. Let 
$\widehat{t}_1$ be the latest time in $[0,\widehat{t}_2]$ at which this occurs, and note that $\vmc{x_0}{v_1}$ must 
have triggered at time $\widehat{t}_1$. Since $v_2$ is protected in $I_{i-1}$, if it were the case that $\widehat{t}_1  
\in I_{i-1}$, then $v_1$ would clear before spawning within $(\widehat{t}_1,i\kappa]$ and so $v_1$ would die 
in $(\widehat{t}_1, \widehat{t}_2)$. This is impossible since $v_1$ spawns a mutant 
  at time $\widehat{t}_2$, and $v_1$ does not become a mutant in $(\widehat{t}_1,\widehat{t}_2]$ by the definition of 
  $\widehat{t}_1$. We therefore have $\widehat{t}_1  \notin I_{i-1}$, so $\widehat{t}_1  \le (i-1)\kappa$. Since $v_2$ 
is protected in $I_{i-1}$, $v_1$ clears before spawning a mutant within $I_{i-1}$, so $v_1$ dies   in 
$((i-1)\kappa,\widehat{t}_2)$ --- again contradicting the fact that $v_1$ spawns a mutant  at time $\widehat{t}_2$. Thus 
we can conclude that $v_2$ does not become a mutant in $I_{i-1}$.

Since \sP3 occurs,  $v_1 \dots v_k$ clears within $I_{i-1}$. Let $v_0 \in R \setminus \{x_0\}$ and $t_1, \dots, 
t_{k+1} \in I_{i-1}$ be as in Definition~\ref{def:star-lower-path-clear}. Since $\vnc{v_0}{v_1}$ triggers at time 
$t_1$, it follows that $v_1 \notin X_{t_1}$. Since $\vmc{x_0}{v_1}$ does not trigger in $[t_1, t_2]$, $v_1$ does not 
become a mutant in $[t_1, t_2]$ and so $v_1 \notin X_{t_2}$. Since $\vnc{v_1}{v_2}$ triggers at time $t_2$, it follows 
that $v_2 \notin X_{t_2}$. We have already seen that $v_2$ does not become a mutant in $I_{i-1}$, so it follows that 
$v_2 \notin X_t$ for all $t \in [t_2, i\kappa]$. 

Now, $v_3 \notin X_{t_3}$ since $\vnc{v_2}{v_3}$ triggers at time $t_3 \in [t_2, i\kappa]$. Since $v_2$ is a non-mutant 
throughout $[t_2, i\kappa ]$, it follows that $v_3 \notin X_t$ for all $t \in [t_3, i\kappa ]$. Repeating the argument 
for $t_4, \dots, t_{k+1}$, we see that $v_2, \dots, v_k, v^* \notin X_{i\kappa }$ and hence $X_{i\kappa } 
\subseteq \{x_0, v_1\}$. 

Since $X_{i\kappa } \subseteq \{x_0, v_1\}$, no mutants can be spawned   in $I_i$ until $v_2$ next becomes a mutant.  
However, by the same argument as above, the fact that $v_2$ is protected in $I_i$ implies that $v_2$ does not become a 
mutant in $I_i$. Hence $X_t \subseteq \{x_0, v_1\}$ for all $t \in I_i$, and in particular $v_k$ does not spawn a 
mutant onto $v^*$   in $I_i$. Thus $\sigma((i+1)\kappa ) = \sigma(i\kappa )$  This gives the desired upper bound on 
$\sigma((i+1)\kappa )$.

\textbf{Case 2: Case 1 does not hold.}
Suppose for a contradiction that $\sigma((i+1)\kappa ) \ge \sigma(i\kappa ) + \lfloor 2r\kappa \rfloor + 1$.
Then $v_k$ spawns a mutant onto~$v^*$ at least $\lfloor 2 r \kappa \rfloor + 1$ times in~$I_i$, contradicting
\sP4.
Thus $\sigma((i+1)\kappa ) \le \sigma(i\kappa )+\lfloor 2r\kappa \rfloor$. 
Again, we have the desired upper bound on $\sigma((i+1)\kappa)$. 

Combining Cases 1 and 2, we have proved that whenever $0 \le i \le \tx$ and $\sigma(i\kappa ) \le 
\lfloor\ell/(2r)\rfloor - \lfloor 2r\kappa \rfloor$,
\begin{align}\label{eqn:star-lower-hard-3}
\begin{split}
\sigma((i+1)\kappa ) &=   \sigma(i\kappa ), \mbox{ if $i>0$ and $v_2$ is protected in $I_{i-1}$ and $I_i$, and}\\
\sigma((i+1)\kappa ) &\le \sigma(i\kappa ) + \lfloor 2r\kappa \rfloor,  \mbox{ otherwise.}
\end{split}
\end{align}
Since \sP5 occurs and
$\ell \ge K r^4 \kappa \log n$, 
the number of intervals $I_i$ such that $0 \le i \le \lfloor \tx/\kappa \rfloor$ and Case 1 does not hold is at most
\begin{align*}
1  + 2 \left|\left\{i \in \Zzero \,\,\middle|\,\, i \le \left\lfloor \frac{\tx }{\kappa } \right\rfloor,\, v_2 
\textnormal{ is not protected in }I_i\right\}\right| 
	&\le 1+ 2\cdot \frac{8r^2\tx }{m}\\
	&= 1+\frac{16r^2\ell}{K r^4 \kappa} \le \frac{17r^2\ell}{K r^4 \kappa}.
\end{align*}
Moreover, again using the fact that $\ell \ge K r^4 \kappa\log n$,
\[\lfloor 2r\kappa\rfloor \cdot \frac{17 r^2\ell}{K r^4 \kappa} \le \frac{34 r^3\ell}{K r^4} = \frac{34 \ell} {K 
r} \le \left\lfloor\frac{\ell}{2 r}\right\rfloor - \lfloor 2r\kappa \rfloor.\]
Since $\sigma(0) = 0$, it therefore follows by repeated application of~\eqref{eqn:star-lower-hard-3} that
\begin{align}\label{eqn:star-lower-hard-2}
\sigma(\tx ) \le \sigma\left(\left\lfloor \frac{\tx }{\kappa } +1 \right\rfloor \kappa \right) \le 
\left\lfloor\frac{\ell}{2r}\right\rfloor - \lfloor 2r\kappa \rfloor.
\end{align}

Now consider the  behaviour of the process in the interval $(\tx, \lfloor \tx /\kappa  + 2 \rfloor \kappa ]$. From 
\eqref{eqn:star-lower-hard-2}, we have that $\sigma(\lfloor \tx /\kappa  + 1\rfloor \kappa ) \le 
\lfloor\ell/(2r)\rfloor - \lfloor 2r\kappa\rfloor$, so by \eqref{eqn:star-lower-hard-3} it follows that 
$\sigma(\lfloor \tx /\kappa  + 2\rfloor \kappa ) \le \lfloor\ell/(2r)\rfloor$ and so,
since \sP1 occurs,  $v^*$ does not spawn a mutant in the interval $[0, \lfloor \tx /\kappa  + 2 \rfloor \kappa ]$. 

Since $\sP2$ occurs  and \eqref{eqn:star-lower-hard-2} holds, we have $x_0 \notin X_{\tx} $, so 
for all $t \in (\tx , \lfloor \tx /\kappa  + 2 \rfloor \kappa ]$, we have
$X_t \subseteq \{v_1, \dots, v_k, v^*\}$. 
Since $\sP3$ occurs, $v_1 \dots v_k$ clears within $I_{\lfloor\tx/\kappa+1\rfloor}$. 
Let   $v_0 \in R \setminus \{x_0\}$ and 
 the sequence of times
$t_1, \dots, t_{k+1} \in I_{\lfloor\tx/\kappa+1\rfloor}$  be as in Definition~\ref{def:star-lower-path-clear}. Then 
for all $i \in [k]$, $\vnc{v_{i-1}}{v_i}$ triggers at time $t_i$ and so $v_i \notin X_t$ for all $t \in 
[t_i, \lfloor \tx/\kappa+2 \rfloor]\kappa$. 
Likewise, $v^* \notin X_t$ for all $t \in [t_{k+1}, \lfloor 
\tx/\kappa+2 \rfloor]\kappa$.
In particular, $ 
X_{\lfloor \tx/\kappa+2 \rfloor \kappa} 
= \emptyset$, so $ X$ goes extinct and the result holds.
\end{proof}

\subsection{Proving the main theorem (Theorem~\ref{thm:star-lower})}

We now have everything we need to prove Theorem~\ref{thm:star-lower}, which follows relatively easily from 
Lemmas~\ref{lem:star-lower-die-instant}, \ref{lem:star-lower-die-before-path} and \ref{lem:star-lower-hard}.

{\renewcommand{\thetheorem}{\ref{thm:star-lower}}
\begin{theorem}
\statemainsuperthm
\end{theorem}
\addtocounter{theorem}{-1}
}

\begin{proof}

Fix $r>1$ as in the statement of the theorem.
Recall from Definitions \ref{def:star-lower-cr} and~\ref{def:tenth} that
$K = 70$
and $\kappa  = \max\{3k,K r^4 \log n\}$. 
Let $n_0$ be
the smallest integer 
 such that, for $n\geq n_0$, 
 Lemma~\ref{lem:star-lower-hard} and Lemma~\ref{lem:star-lower-apply-centre-mut}
  applies  and also
\begin{equation}
\label{eq:nlogn-bound}
(n \log n)^{1/3} \geq n^{1/3} \geq
\max\{18 r^2, 6 K r^6 \log n,Kr^4 (\log n)^2\}\,.\end{equation}

We split into cases depending on the values of $k$, $\ell$, $m$ and $n$.
We show that in each case, the statement of the theorem holds, provided $c_r \ge \max \{2rn_0, 156r^6K\}$.

\textbf{Case 1: $\boldsymbol{n < n_0}$.}
We show that with probability at least $1/2rn_0$, $x_0$ dies 
  before spawning a single mutant. Indeed, at the start of the process $x_0$ spawns a mutant with rate $r$, and every 
choice of $x_0 \in V(\sstar)$ has an in-neighbour so $x_0$ dies 
  with rate at least $1/n$. Thus $ X$ goes extinct with probability at least
\[\frac{\frac{1}{n}}{\frac{1}{n}+r} \ge \frac{1}{2rn} \ge \frac{1}{2rn_0},\]
so the result follows since $c_r \geq 2rn_0$.

\textbf{Case 2: $\boldsymbol{n \ge n_0}$, $\boldsymbol{m < k (n\log n)^{1/3}}$.} 
By Lemma~\ref{lem:star-lower-die-instant}, $ X$ goes extinct with probability at least
\[\frac{k}{2r(m+k)} \ge \frac{k}{2r(k (n\log n)^{1/3} + k)} \ge \frac{1}{4r(n\log n)^{1/3}},\]
where the final inequality holds
since $(n \log n)^{1/3} \geq 1$.
The result follows since $c_r \geq 4 r$.

\textbf{Case 3: $\boldsymbol{n \ge n_0}$, $\boldsymbol{m \ge k (n\log n)^{1/3}}$ and $\boldsymbol{\ell < 3K r^4 
(n\log n)^{1/3}}$.}
Note that
\begin{equation}\label{eqn:star-lower-1}
\Pr(x_0 \in R_1 \cup \dots \cup R_\ell) = \frac{\ell m}{\ell(m+k)+1} \ge \frac{m}{m+k+1} \ge 
\frac{1}{2},
\end{equation}
where the final inequality is valid since
$m\geq k (n \log n)^{1/3} \geq 2k$.
We  will therefore condition on $x_0 \in R_1 \cup \dots \cup R_\ell$. Moreover, we have $m \ge k (n\log 
n)^{1/3} \ge 12r$. Thus by Lemma~\ref{lem:star-lower-die-before-path} and~\eqref{eqn:star-lower-1}, $X$ goes extinct 
with probability at least
\[\frac{1}{2} \cdot \frac{1}{26r^2\ell} \ge \frac{1}{156r^2 K r^4 (n\log n)^{1/3}},\]
so the statement holds since $c_r \geq 156 K r^6$.

\textbf{Case 4: $\boldsymbol{n \ge n_0}$, $\boldsymbol{m \ge k (n\log n)^{1/3}}$ and $\boldsymbol{\ell \ge 3K r^4 
(n\log n)^{1/3}}$.}
Note that 
\[m \ge k (n\log n)^{1/3} \ge 6r^2\max \{3 k, K r^4 \log n\} = 6r^2\kappa.\]
We  will also show that $\ell \ge K r^4 \kappa \log n$, in order to apply Lemma~\ref{lem:star-lower-hard}. Since 
$\ell \ge 3K r^4 (n\log n)^{1/3}$ and $n = \ell(m+k)+1 \ge \ell m$, we have $m \le n/\ell \le (n^2/\log 
n)^{1/3}$.
It is also immediate from \eqref{eq:nlogn-bound} and the hypothesis on~$\ell$ that
\begin{equation}\label{eqn:star-lower-2}
\ell \ge K^2 r^8   (\log n)^2.
\end{equation}
Therefore,
\begin{equation}\label{eqn:star-lower-3}
3k \le \frac{3m}{(n\log n)^{1/3}} \le \frac{3n^{1/3}}{(\log n)^{2/3}} = \frac{3K r^4 (n\log n)^{1/3}}{K r^4 \log n} 
\le \frac{\ell}{K r^4 \log n}.
\end{equation}
It therefore follows from \eqref{eqn:star-lower-2}, \eqref{eqn:star-lower-3} and the definition of $\kappa $ 
(Definition~\ref{def:star-lower-cr}) that $\ell \ge K r^4 \kappa \log n$, and so we may apply 
Lemma~\ref{lem:star-lower-hard}.

As in Case 3, \eqref{eqn:star-lower-1} holds. Thus by~\eqref{eqn:star-lower-1} and Lemma~\ref{lem:star-lower-hard}, $ 
X$ 
goes extinct with probability at least $1/(14K r^4 \kappa )$. By \eqref{eqn:star-lower-3} we have $3k \le 3n^{1/3}$, 
and so
\begin{equation*}
\frac{1}{14K r^4 \kappa } = \frac{1}{14K r^4 \max \{3k, K r^4 \log n\}} \ge \frac{1}{14K r^4 \max \{3n^{1/3}, K r^4 
\log n\}}
= \frac{1}{42K r^4 n^{1/3}} ,
\end{equation*}
and the result follows since $c_r \geq 156Kr^6$.
\end{proof}

\section{An upper bound on the fixation probability of metafunnels}
\label{sec:metafunnel}

The $(k,\ell,m)$-metafunnel is defined in Section~\ref{def:meta}.
We use $n = 1 + \ell \sum_{i=1}^{k} m^i$
to denote the number of vertices.

The main result of this section is the following upper bound on the fixation probability of
the metafunnel.
 
\newcommand{\statemainthm}{Let $r > 1$. Then there  is a constant~$c_r>0$, depending on~$r$,
such that the following holds
for all $k, \ell, m\in\Zone$
such that the $(k,\ell,m)$-metafunnel $\metaf$ has
$n\geq 3$ vertices. 
Suppose that the initial state $X_0$ of the Moran process
with fitness~$r$
is chosen uniformly at random from all 
singleton subsets of $V(\metaf)$.
The probability that the Moran process goes extinct is at least 
$e^{-\sqrt{\log r \cdot \log n}}(\log n)^{-c_r}$.
}

\begin{theorem}\label{thm:funnel-upper}
\statemainthm 
\end{theorem}

\subsection{Proof sketch} 

If $k=1$ then $\metaf$ is a star and has extinction probability roughly $1/r^2$
so Theorem~\ref{thm:funnel-upper} follows easily.
So for most of the proof (and the rest of this sketch) we assume $k\geq 2$.
To prove the theorem, we divide  the parameter space into two regimes. 

In the first regime,
$m \le r^{\sqrt{\log_r n}} $. Since $m$ is small,
  $V_k$ is not too large compared to $V_0 \cup \dots \cup V_{k-1}$. 
Thus, it is fairly likely that    $x_0$ is born outside $V_k$, and dies  before it can spawn a single mutant. 
This straightforward analysis is contained in the short Section~\ref{sec:smallm}.

Most of the proof (Section~\ref{sec:largem}) focusses on the second regime, where 
$m \geq  r^{\sqrt{\log_r n}}$
which, since $n \geq \ell m^k$, implies $k \le \sqrt{\log_r n}$.
In this regime it is likely that  
a uniformly-chosen initial mutant
$x_0$ is   born in $V_k$
(Lemma~\ref{prop:funnel-vinit}) so we assume that this is the case in most of the proof 
(and the rest of this sketch). 
The key lemma is Lemma~\ref{lem:funnel-die-slow}
which shows that, in this case, it is (sufficiently) likely that $x_0$ dies  before $v^*$ spawns a mutant.

In more detail, Definition~\ref{def:Tme}
defines a stopping time $\Tpabs$ which is the first time $t$
that one of the following occurs.
  \begin{enumerate}[({A}1)]
\item $X_{t} = \emptyset$, or  
\item $|X_{t}|$ exceeds  a given  threshold $m^*$ which is a polynomial in $\log n$, or
\item By time $t$, $v^*$ has already become a mutant in~$X$ 
more than $b^*$ times, where $b^*$ is about half as large as its number $\ell m$ of in-neighbours, or
 \item $t$ exceeds some  threshold $\tmax$ which is (very) exponentially large in~$n$.
\end{enumerate}  
The subscript ``$\mathsf {pa}$'' is for ``pseudo-absorption time''
because (A1)   implies that the Moran process absorbs by going extinct and 
(A2) is a prerequisite for absorbing by fixating.  
The proof of Lemma~\ref{lem:funnel-die-slow} 
shows that, with sufficiently high probability, (A2)--(A4)
 do not hold, and so the Moran process~$X$ must go extinct by~$\Tpabs$.

Conditioning makes it difficult to prove that (A2)--(A4) fail. To alleviate this, we divide the mutants into groups called ``strains'' which are easier to analyse. In particular, a strain contains all of the descendants of a particular mutant spawned by~$x_0$.
Formally, for each positive integer~$i$, the $i$'th strain $S^i$
is defined as a mutant process in Definition~\ref{definition:strains}.
Informally, $S^i$ is ``born'' at the $i$'th time at which $x_0$ spawns a mutant in~$X$.
It ``dies'' when all of the descendants of this spawn have died.
It is ``dangerous'' if one (or more) of these descendants spawns a mutant onto~$v^*$
before~$\Tpabs$.

Lemma~\ref{lem:funnel-masterlist} defines eight events \sP1--\sP8.
These are defined in such a way that we can show (in the proof of Lemma~\ref{lem:funnel-die-slow})
that if \sP1--\sP8 simultaneously occur, then (A2)--(A4) do not hold.
The definitions are engineered in such a way that we can also  
show that it is fairly likely that they do hold simultaneously --- this takes up most of the proof.
Informally, the events are defined as follows.
\begin{description}
\item \sP1: No star-clock $\smc{(v^*,v)}$ triggers in $[0,1]$.  
\item \sP2: 
For some threshold $\tx < n$, the star-clock $\snc{(v^*,x_0)}$ triggers in $[0,\tx-2]$. 
\item \sP3: $v^*$ is a mutant for at most one unit of time up to time $\Tpabs$.
\item \sP4: The Moran process absorbs (either fixates or goes extinct) by time $\tmax/2$. 
\item \sP5: Break   $[0,\tx]$ into intervals of length ${(\log n)}^2$.
During each interval, $x_0$ spawns at most  $2r(\log n)^2$ mutants in   $X$.  
\item \sP6:  Define~$s$ to be around $3r \tx$.
Each of the strains $S^1, \dots, S^s$ spawns at most $  \log n  $ mutants before $\Tpabs$.  
\item \sP7: Each of the strains $S^1,\ldots,S^s$ dies within ${(\log n)}^2$ steps. 
\item \sP8: At most $b^*/\log n$ of $S^1, \dots, S^s$ are dangerous.  
\end{description}
The rough sketch of Lemma~\ref{lem:funnel-die-slow} is as follows.
\sP1 and \sP3 guarantee that $v^*$ does not spawn a mutant in~$X$ until after $\Tpabs$.
This together with \sP2 and \sP3 guarantees that the only mutants
in the process before time $\Tpabs$ are part of strains that are born before $\tx$.
By \sP5, there are at most $s$ such strains.
By \sP6,  each of these strains only has about $\log n$ mutants.
Together with \sP7, this implies that (A2) does not hold at $t=\Tpabs$.
\sP8 and \sP6 imply that (A3) does not hold at $t=\Tpabs$. Finally, \sP4 implies that (A4) does not hold at $t=\Tpabs$.

The bulk of the proof involves showing (Lemma~\ref{lem:funnel-masterlist}) that \sP1--\sP8 are sufficiently likely to simultaneously occur.
Of these, \sP3--\sP7 are all so likely to occur that the probability that they do not occur
can be subtracted off using a union bound (so conditioning on the other   \sP{i}'s is not an issue).
The majority of the failure probability comes from the probability that \sP2 does not occur.
 This is handled in the  straightforward Lemma~\ref{prop:funnel-clocks-behave}
 which gives a lower bound on the probability that \sP1 and \sP2 both occur.
 The remaining event, \sP8, is sufficiently unlikely to occur
 that careful conditioning is required. This is (eventually)
 handled in Lemma~\ref{lem:funnel-few-dangerous-strains}, which shows that it is fairly likely to
 occur, conditioned on the fact that both \sP1 and \sP2 occur.

 In order to get a good estimate on the probability that a strain is dangerous (in \sP8), we need
 to consider the number of mutants spawned from the ``layer'' of the strain closest to the centre vertex~$v^*$.
 In order to do this, we define  a new mutant process called the ``head'' of a strain.
Strains and heads-of-strains share some common properties, and they are analysed
together as ``colonies'' in Section~\ref{sec:colony}. 
Informally (see Definition~\ref{defn:colony}) 
a ``colony'' is a mutant process~$Z$ whose mutants are in $V_1 \cup \cdots \cup V_{k-1}$
(and not in $V_0$ or $V_k$). Once a colony becomes empty, it stays empty.
Since a colony is a mutant process but not necessarily a Moran process, vertices
may enter and/or leave whenever a clock triggers but we say that the colony is \emph{hit} when
a vertex leaves a colony specifically because a non-mutant is spawned onto it in the underlying Moran process.
We define the ``spawning chain'' $Y^Z$ of a colony and show that it
increases whenever the colony spawns a mutant and that 
it only decreases when the colony is hit. By analysing the jump chain of a spawning chain  
we are able to obtain the desired bounds on the probability that \sP6, \sP7 and \sP8 fail to occur.

\subsection{Glossary}

\begin{longtable}{p{\linewidth-\widthof{Definition 00, Page 00}}@{}l@{ }l@{}}
$\mathcal{A} = \{\mc{(v^*,u)}   \mid u \in V_{k}\} \cup \{\nc{(v^*,x_0)}\}$ \dotfill & Definition~\ref{def:Upsilon},&
Page~\pageref{def:Upsilon} \\
$\mathcal{A}^* = \{\smc{(v^*,u)} \mid u \in V_k\} \cup \{\smbc{(v^*,u)} \mid u \in V_k\} \cup 
\{\snc{(v^*,x_0)},\snbc{(v^*,x_0)}\}$ \dotfill & Definition~\ref{def:Upsilon},&
Page~\pageref{def:Upsilon} \\
\myref{A}{A:extinct}, \myref{A}{A:fill}, \myref{A}{A:spread}, \myref{A}{A:timeout} \dotfill & Definition~\ref{def:Tme},& Page~\pageref{def:Tme} \\
$b^* = \lfloor \ell m /2\rfloor$ \dotfill & Definition~\ref{defn:funnel-constants},& 
Page~\pageref{defn:funnel-constants} \\
$C_r = \lceil 2\log_r 20 \rceil$ \dotfill & Definition~\ref{defn:funnel-constants},& 
Page~\pageref{defn:funnel-constants} \\
$c_r$  \dotfill & Theorem~\ref{thm:funnel-upper},& Page~\pageref{thm:funnel-upper} \\
dangerous strain \dotfill & Definition~\ref{definition:strains},& Page~\pageref{definition:strains} \\
$\metaf$ \dotfill & Section~\ref{def:meta}, & Page~\pageref{def:meta}\\
hit \dotfill & Definition~\ref{defn:colony},& Page~\pageref{defn:colony} \\
$\head^i_t$ (head of a strain) \dotfill & Definition~\ref{defn:head}, & Page~\pageref{defn:head} \\
$\iin{\mu}{u}{t}$ \dotfill & Definition~\ref{def:localtime},& Page~\pageref{def:localtime}\\
$\iout{\mu}{u}{t}$ \dotfill & Definition~\ref{def:localtime},& Page~\pageref{def:localtime}\\
$I_j = [(j-1)(\log n)^2, j(\log n)^2)$ \dotfill & Definition~\ref{defn:funnel-constants},&
Page~\pageref{defn:funnel-constants} \\
$m^* = \lceil 5 r (\log n)^3 \rceil$ \dotfill & Definition~\ref{defn:funnel-constants},&
Page~\pageref{defn:funnel-constants} \\
$m' = m - m^*$ \dotfill & Definition~\ref{defn:funnel-constants},&
Page~\pageref{defn:funnel-constants} \\
$\Pmutclock,\dots,\Pstraindanger$ \dotfill & Lemma~\ref{lem:funnel-masterlist},&
Page~\pageref{lem:funnel-masterlist} \\
$\Phi$ \dotfill & Definition~\ref{def:Upsilon},& Page~\pageref{def:Upsilon} \\
$\Psi(X)$ \dotfill & Section~\ref{sec:coupled},& Page~\pageref{sec:coupled} \\
$\Psi(X,Z)$ \dotfill & Definition~\ref{def:theprocessPsiXZ},& Page~\pageref{def:theprocessPsiXZ} \\
$\transmatrix_{a,b}$ \dotfill &  Definition~\ref{defn:spawn-chain},& Page~\pageref{defn:spawn-chain} \\
$s = \lceil 3r\tx \rceil$ \dotfill & Definition~\ref{defn:funnel-constants},&
Page~\pageref{defn:funnel-constants} \\
$S^i_t$ (strain) \dotfill & Definition~\ref{definition:strains},& Page~\pageref{definition:strains} \\
$\Tbirth{i}$ \dotfill & Definition~\ref{definition:strains},& Page~\pageref{definition:strains} \\
$\Tdeath{i}$ \dotfill & Definition~\ref{definition:strains},& Page~\pageref{definition:strains} \\
$\Tend(Z)$ \dotfill & Definition~\ref{defn:colony},& Page~\pageref{defn:colony} \\
$\tmax = n^{3n}$ \dotfill & Definition~\ref{defn:funnel-constants},& Page~\pageref{defn:funnel-constants} \\
$\Tpabs$ \dotfill & Definition~\ref{def:Tme},& Page~\pageref{def:Tme} \\
$\Tstart(Z)$ \dotfill & Definition~\ref{defn:colony},& Page~\pageref{defn:colony} \\
$\tx = \ell m^k/(r^k(\log n)^{C_r+5})$ \dotfill & Definition~\ref{defn:funnel-constants},&
Page~\pageref{defn:funnel-constants} \\
$\tau_i$ \dotfill & Section~\ref{sec:clock}, & Page~\pageref{sec:clock}\\
$Y^Z$ (spawning chain) \dotfill & Definition~\ref{defn:spawn-chain},& Page~\pageref{defn:spawn-chain} \\
$\widehat{Y}^Z$ (jump chain of $Y^Z$)\dotfill & Definition~\ref{defn:spawnjumpchain},& 
Page~\pageref{defn:spawnjumpchain} \\
$Z$ (colony) \dotfill & Definition~\ref{defn:colony},& Page~\pageref{defn:colony} \\
(Z\ref{Z:one}), (Z\ref{Z:two}) \dotfill & \dotfill& Page~\pageref{Z:one}\\
\end{longtable}

\subsection{The small $\boldsymbol{m}$ case}\label{sec:smallm}
 
We first show that if $m$ is small, then $x_0$ is likely to be born outside $V_k$ and die before spawning a mutant. This is relatively easy.

\begin{lemma}\label{lem:funnel-die-instant}
Suppose $k\geq 2$. 
Choose $x_0$ uniformly at random from $V(\metaf)$.
Let $X$ be the Moran process  with $G(X) = \metaf$ and $X_0=\{x_0\}$.
The extinction probability of~$X$ is  at least $1/(2(m+r))$. 
\end{lemma}
\begin{proof}
 
Note that $\Pr(x_0 \in V_k) = \ell m^k/n$. 
 First suppose $m=1$ and $k \ge 2$. We have
\[\Pr(x_0 \in V_k) = \frac{\ell}{\ell k + 1} \le \frac{\ell}{2\ell+1} < \frac{1}{2}.\]
Moreover, if $x_0 \notin V_k$, then $x_0$ has an in-neighbour of out-degree~$1$.
In this case, with probability at least $1/(1+r)$, $x_0$ dies
   before spawning a mutant. It follows that $X$ goes extinct with probability at least $1/(2(1+r))$, as required.

 Now suppose $m,k \ge 2$. We have
\begin{equation*}
    \Pr(x_0\notin V_k)
        = \frac{1 + \ell \sum_{i=1}^{k-1} m^i}{1 + \ell \sum_{i=1}^{k} m^i}
        \geq \frac{\sum_{i=1}^{k-1} m^i}{\sum_{i=1}^{k} m^i}
        = \frac{m^k-m}{m(m^k-1)} \geq \frac1{2m}.
\end{equation*}
If $x \in V_i$ for some $i \in\{0,\ldots,k-1\}$, then $x$ has at least $m^{i+1}$ in-neighbours with out-degree $m^i$, so 
with probability at least $m/(m + r)$,
$x_0$ dies   
before spawning a mutant. Hence $X$ has extinction probability at least
\[\frac{1}{2m}\cdot \frac{m}{m+r} = \frac{1}{2(m+r)}\,.\qedhere\]
\end{proof}

\subsection{The large $\boldsymbol{m}$ case}\label{sec:largem}

We now consider the case where $m$ is large. 
 
\begin{lemma}\label{prop:funnel-vinit}
Suppose $m \ge 2$. 
Choose $x_0$ uniformly at random from $V(\metaf)$.
Then $\Pr(x_0 \in V_k) \ge 1/2$.
\end{lemma}
\begin{proof}
We have
\begin{equation*}
\Pr(x_0 \in V_k) 
	= \frac{\ell m^k}{\ell \sum_{i=0}^k m^i - (\ell-1)} 
	\ge \frac{m^k}{\sum_{i=0}^k m^i} 
	= \frac{m^k (m-1)}{m^{k+1}-1} 
	\ge 1 - \frac{1}{m} 
	\ge \frac{1}{2}\,.\qedhere
\end{equation*}
\end{proof}

For the remainder of  Section~\ref{sec:metafunnel}, we  will
fix an arbitrary vertex $x_0 \in V_k$ and
let $X$ be the Moran process with $G(X)=\metaf$ and
$X_0 = \{x_0\}$.  We  first define some constants. Then, we define a  ``pseudo-absorption time'' 
$\Tpabs$ which, by \myref{A}{A:extinct}  and \myref{A}{A:fill}
  in the definition below, is at most the absorption time of the Moran process~$X$.

\begin{definition}[\textbf{Constants}]
\label{def:constants}
\label{defn:funnel-constants}
We will use the following definitions for the rest of  Section~\ref{sec:metafunnel}. 
\begin{itemize} 
\item $b^* = \lfloor \ell m /2\rfloor$,
\item $C_r = \lceil 2\log_r 20 \rceil$,
\item for each $j \in \Zone$, $I_j = [(j-1)(\log n)^2, j(\log n)^2)$. 
\item $m^* = \lceil 5 r (\log n)^3 \rceil$, 
\item $m' = m - m^*$,  
\item $\tmax = n^{3n}$,
\item $\tx = \ell m^k/(r^k(\log n)^{C_r+5})$, and
\item $s = \lceil 3r\tx \rceil$.\defend{}
\end{itemize}
\end{definition}

\begin{definition}[\textbf{The stopping time $\boldsymbol{\Tpabs}$}]
\label{def:Tme}
We define the stopping time $\Tpabs$ to be the 
first time $t$ that  one of the following occurs:
\begin{enumerate}[({A}1)]
\item $X_{t} = \emptyset$, or \label{A:extinct}
\item $|X_{t}| \geq m^*$, or \label{A:fill}
\item $v^*$ becomes a mutant in~$X$ at time $t$ for the $(b^*+1)$'st time, or \label{A:spread}
\item $t \ge \tmax$.\label{A:timeout}\defend{}
\end{enumerate}
\end{definition}

The definition of $\Tpabs$ is motivated as follows. Certainly  \myref{A}{A:extinct} must hold when the process~$X$ goes extinct, and 
\myref{A}{A:fill} must hold before $X$~fixates. If  \myref{A}{A:spread} holds, we expect $v^*$ to spawn a mutant in~$X$, which makes the process significantly harder to analyse (so we will stop the analysis before this). Actually, this is also why we stop at $m^*$ mutants in \myref{A}{A:fill} --- if the process contains too many mutants then it becomes harder to analyse. Finally,  \myref{A}{A:timeout} ensures that $\Tpabs < \infty$.

We will prove that, with sufficiently high probability, \myref{A}{A:fill}--\myref{A}{A:timeout}
do not hold, and so the Moran process~$X$ must go extinct by~$\Tpabs$. To do this, we will group the descendants of each mutant spawned by $x_0$ in~$X$ together and analyse each group 
as a separate mutant process.

\begin{definition}[\textbf{Strains}]
\label{definition:strains}
Consider the Moran process $X$ with  
$G(X) = \metaf$ and
$X_0 = \{x_0\}$ for
some  $x_0\in V_k$.
For each positive integer~$i$, we define a mutant process
$S^i$ 
(called the \emph{$i$'th strain}) with
$G(S^i)=\metaf$.
Let $\Tbirth{i}$ be the $i$'th time at which $x_0$ spawns a mutant in $X$, or~$\infty$ if $x_0$~spawns fewer than $i$~mutants.
The subscript ``$\mathsf b$'' stands for the ``birth'' of the strain.
Clearly, $\Tbirth{i}$ is a function of the evolution of the process~$X$. We let
$S^i_t = \emptyset$ for all $t < \Tbirth{i}$. If $\Tbirth{i} < \infty$, then we have $\Tbirth{i} = \tau_j$ for some $j$. Let $u_i$ be the vertex onto which the mutant is spawned in~$X$, and let $S^i_{\tau_j} = \{u_i\}$. The process~$S^i$ now evolves discretely as follows. Suppose we are given $S^i_{\tau_a}$ for some $a \ge j$. We define $S^i_t = S^i_{\tau_a}$ for all $t \in (\tau_a,\tau_{a+1})$. We then define $S^i_{\tau_{a+1}}$ by dividing into cases.

\medskip\noindent
\textbf{Case 1: Some vertex $\boldsymbol{u \in S^i_{\tau_a}}$ spawns a mutant onto some vertex $\boldsymbol{v}$ in $\boldsymbol{X}$ at time $\boldsymbol{\tau_{a+1}}$.}
If  $v \notin V_0 \cup V_k$, then we set $S^i_{\tau_{a+1}} = S^i_{\tau_a} \cup \{v\}$.
Otherwise, we set $S^i_{\tau_{a+1}} = S^i_{\tau_a}$.

\medskip\noindent
\textbf{Case 2: Some vertex $\boldsymbol{v \in S^i_{\tau_a}}$ dies in $\boldsymbol{X}$  
at time $\boldsymbol{\tau_{a+1}}$.} 
We set $S^i_{\tau_{a+1}} = S^i_{\tau_a} \setminus \{v\}$.

\medskip\noindent
\textbf{Case 3: Neither Case 1 nor Case 2 holds.} 
We set $S^i_{\tau_{a+1}} = S^i_{\tau_a}$.

If $\Tbirth{i}=\infty$ then we define $\Tdeath{i}=\infty$.
Otherwise, 
we define $\Tdeath{i} = \sup\{t \mid S^i_t \ne \emptyset\}$. 
The subscript ``$\mathsf d$'' stands for the ``death'' of the strain.
Note that the definition maintains the invariant that $S^i_t \subseteq X_t$.
Finally, we define the notion of a \emph{dangerous} strain. The strain $S^i$ is said to be \emph{dangerous} if it spawns a mutant onto $v^*$  
 during the interval $[0,\Tpabs]$.\defend{}
\end{definition}

Note that we allow $S^i_t$ and $S^{i'}_t$ to intersect for $i \ne i'$. Intuitively, $S^i_t$ is the set of all living descendants 
at time~$t$
(within $V_1 \cup \dots \cup V_{k-1}$) of the $i$'th mutant spawned by $x_0$ in $X$. 

We now set out a list of events $\Pmutclock, \dots, \Pstraindanger$  which, as we will see in the proof of Lemma~\ref{lem:funnel-die-slow}, together imply extinction. We state these events and claim they hold with reasonable probability in Lemma~\ref{lem:funnel-masterlist}. 

\newcommand{\statemainlemma}{There exists $n_0>0$, depending on $r$, such that the following holds. Suppose $n \ge n_0$, $m \ge r^{\sqrt{\log_r n}}$ and $2\le k \le \sqrt{\log_r n}$. 
Suppose $x_0 \in V_k$.
Let $X$ be the Moran process  
with $G(X) = \metaf$ and 
$X_0 = \{x_0\}$. 
With probability at least $r^{-k}/(\log n)^{C_r + 7}$, 
all of the following events occur in $\Psi(X)$.
\begin{description}
\item [$\Pmutclock$:] no star-clock $\smc{(v^*,v)}$ triggers in $[0,1]$.
\item[$\Pnmutclock$:] the star-clock $\snc{(v^*,x_0)}$ triggers in $[0,\tx-2]$.
\item[$\Pcentremut$:] $\iin{X}{v^*}{\Tpabs}\leq 1$.
\item[$\Pfastabsorb$:] $X_{\tmax/2} \in \{\emptyset, V(\metaf)\}$.
\item[$\Pvinitspawn$:] for all $j \le \lceil \tx/(\log n)^2 \rceil$, $x_0$ spawns at most $2r(\log n)^2$ mutants in $I_j$ in $X$.
\item[$\Pstrainfire$:] each of $S^1, \dots, S^s$ spawns at most $  \log n  $ mutants  in $(0,\Tpabs]$.
\item[$\Pstrainlife$:] for all $i \in [s]$, $\min\{\Tdeath{i},\Tpabs\} \le \Tbirth{i} + (\log n)^2$.
\item[$\Pstraindanger$:] at most $b^*/\log n$ of $S^1, \dots, S^s$ are dangerous.
\end{description}}
\begin{lemma}\label{lem:funnel-masterlist}
\statemainlemma
\end{lemma}

The majority of the failure probability in Lemma~\ref{lem:funnel-masterlist} comes from~$\Pnmutclock$.
In addition, $\Pmutclock$ and $\Pstraindanger$ may fail  with reasonably high probability, so we  will need to be careful with conditioning for these events. The remaining events each occur with high enough probability that we can apply a union bound. 

We first show that $\Pmutclock\cap\Pnmutclock$ occurs with reasonable probability.

\begin{lemma}\label{prop:funnel-clocks-behave}
There exists $n_0>0$, depending on $r$, such that the following holds. Suppose $n \ge n_0$, $m \ge r^{\sqrt{\log_r n}}$ and $2\le k \le \sqrt{\log_r n}$. 
Suppose $x_0 \in V_k$.
Let $X$ be the Moran process  
with $G(X) = \metaf$ and 
$X_0 = \{x_0\}$.  Then, in the process $\Psi(X)$,
 \[\Pr(\Pmutclock\cap\Pnmutclock) \ge \frac{1}{r^k (\log n)^{C_r+6}}.\]
\end{lemma}
\begin{proof}
Let $n_0$ be a large integer relative to $r$. Note that $\Pmutclock$ and~$\Pnmutclock$ depend  only on the star-clock process $P^*(G)$, and so they are independent by the definition of $P^*(G)$.  The
sum of the parameters of the star-clocks in $\{\smc{(v^*,u)} \mid u \in V_k\}$
is~$r$, so the
 definition of Poisson processes ensures that
$\Pr(\Pmutclock) = e^{-r}$. 

The assumptions in the statement of the lemma guarantee that $\tx\geq 4$.
The parameter of the star-clock $\snc{(v^*,x_0)}$
is $1/(\ell m^k)$, so
\[\Pr(\Pnmutclock) = 
1 - e^{-(\tx-2)/(\ell m^k)} \geq
1-e^{-\tx/(2\ell m^k)} 
.\]
Using \eqref{eq:ebounds}, we get
\[\Pr(\Pnmutclock) \geq  
1-e^{-\tx/(2\ell m^k)} \ge \frac{\tx}{4\ell m^k} = \frac{1}{4r^k(\log n)^{C_r+5}} \ge \frac{e^r}{r^k(\log n)^{C_r+6}}.
\]

Since $\Pmutclock$ and~$\Pnmutclock$ are independent, the result follows.
\end{proof}

\begin{lemma}\label{lem:funnel-many-infections-needed}
There exists $n_0>0$, depending on $r$, such that the following holds. Suppose 
$n \ge n_0$, $m \ge r^{\sqrt{\log_r n}}$ and $k\geq 2$.
Suppose $x_0 \in V_k$.
Let $X$ be the Moran process  
with $G(X) = \metaf$ and 
$X_0 = \{x_0\}$.  
Then $\Pr(\Pcentremut) \ge 1-1/n$.
\end{lemma}
\begin{proof} 
    
For  every positive integer~$i$,  
we define $\TTmut{i}$ and $\TTnmut{i}$ as follows.
If~$v^*$ becomes a mutant at least $i$ times in $[0,\Tpabs]$ then
$\TTmut{i}$ is the $i$'th time that it does so. Otherwise, $\TTmut{i}=\Tpabs$.
If~$v^*$ becomes a non-mutant at least $i$ times in $[0,\Tpabs]$ then    
$\TTnmut{i}$ is the $i$'th time that it does so. Otherwise, $\TTnmut{i}=\Tpabs$.

 By
item
\myref{A}{A:spread} in
the definition of $\Tpabs$ (Definition~\ref{def:Tme}),
$v^*$ may become a mutant at most $b^*$ times in the interval $[0,\Tpabs)$, so 
$$\iin{X}{v^*}{\Tpabs}  =  \sum_{i=1}^{b^*} (\TTnmut{i} - \TTmut{i}).$$
 
Now consider any $i\in [b^*]$ and any $t_0 \geq 0$.
Consider any possible value $f$ for $\filt_{t_0}(X)$
that is consistent with $\TTmut{i} = t_0$.
We will show 
\begin{equation}
\label{eq:xxgen}
\forall y\geq 0,\quad \Pr(\TTnmut{i} - t_0 \leq y \mid \filt_{t_0}(X)=f) \geq 1 - \exp(-\ell m' y).
\end{equation}

The event
$\TTmut{i}=t_0$ is determined by~$f$.
Since $\TTmut{i} \leq \Tpabs$, we have $t_0\leq \Tpabs$.
The value~$f$ also determines whether or not $t_0=\Tpabs$.
If so, then \eqref{eq:xxgen} is trivial since
$\TTnmut{i} = \Tpabs=t_0$, so $\Pr(\TTnmut{i} - t_0 \leq y\mid \filt_{t_0}(X)=f)=1$.

From now on, we assume that $f$ implies that $\Tpabs>t_0$.
Since $\TTmut{i}=t_0$,
this implies that $v^*\in X_{t_0}$.
Let $\mathcal{B}$ be the set of all non-mutant clocks with target~$v^*$
and let $\Xi$ encapsulate the behaviour of every clock in $\moranclocks(\metaf)\setminus \mathcal{B}$ over the
interval $(t_0,\tmax]$. Consider any value $\xi$ of $\Xi$ which is consistent with $\filt_{t_0}(X)=f$.

We now define a time $\tpabs$ which depends only on the values~$f$ and~$\xi$.
To do so, consider the situation in which $\filt_{t_0}(X)=f$, $\Xi=\xi$ and no clock in $\mathcal{B}$
triggers in $(t_0,\tmax]$, so that the evolution of~$X$ in this interval is entirely determined by~$f$ and~$\xi$.
Let $\tpabs$ be the time at which $\Tpabs$ would occur in this situation.

It is easy to see that  $\TTnmut{i} \leq \tpabs$. 
If a non-mutant is spawned onto $v^*$ in~$X$ at some time $t'\in(t_0,\tpabs]$
then $\TTnmut{i}\leq t' \leq \tpabs$. Otherwise, the evolution of $X$ in $(t_0,\tpabs]$
is  exactly the same as it would be if no clocks in $\mathcal{B}$ triggered,
so $\Tpabs=\tpabs=\TTnmut{i}$.

  We will now prove \eqref{eq:xxgen}.
 First, if $y\geq \tpabs - t_0$,
 then since $\TTnmut{i} \leq \tpabs$, we have
 \begin{equation}
  \label{eq:Dec3first}
\Pr(\TTnmut{i} - t_0 \leq y \mid \filt_{t_0}(X)=f ,\Xi=\xi)  
  \geq
 \Pr( \tpabs - t_0 \leq y \mid \filt_{t_0}(X)=f,\Xi=\xi )    = 1. 
 \end{equation}  
 So suppose $y < \tpabs - t_0$.
 Let $t_1 < \cdots < t_z$ be the times in $(t_0,t_0+y]$ at which clocks in 
 $\moranclocks(\metaf)\setminus \mathcal{B}$ trigger and let $t_{z+1} = t_0+y$.
 Thus $t_0 < \cdots < t_z \leq t_{z+1} < \tpabs$.
 For all $h\in \{0,\ldots,z\}$, let $\chi(h)$ be the value that $X_{t_h}$ would take in the situation where no clock in 
$\mathcal{B}$ triggers in $(t_0,t_h]$, $\filt_{t_0}(X)=f$ and $\Xi=\xi$.
Thus $\tpabs$, $z$, $t_0,\ldots,t_{z+1}$ and $\chi(0),\ldots,\chi(z)$ are all uniquely determined by~$f$ and~$\xi$.

 For each $h \in [z+1]$, let $\mathcal{E}_h$ be the event that a non-mutant is spawned onto $v^*$ in the interval $(t_{h-1}, t_h)$. Note that with probability 1, no non-mutant is spawned onto $v^*$ at any time $t_h$. Thus
\begin{align}\label{eqn:ker}
\Pr(\TTnmut{i}-t_0 \le y \mid \filt_{t_0}(X)=f, \Xi=\xi) &= 1 - \Pr(\overline{\mathcal{E}_1} \cap \dots \cap \overline{\mathcal{E}_{z+1}} \mid \filt_{t_0}(X)=f, \Xi=\xi)\\\nonumber
&= 1 - \prod_{h=1}^{z+1} \Pr(\overline{\mathcal{E}_h} \mid \filt_{t_0}(X)=f, \Xi=\xi, \overline{\mathcal{E}_1} \cap \dots \cap \overline{\mathcal{E}_{h-1}}).
\end{align}

Now fix $h \in [z+1]$, and consider any possible value $f_{h-1}$ of $\filt_{t_{h-1}}(X)$ which implies that 
$\filt_{t_0}(X)=f$ and $\overline{\mathcal{E}_1} \cap \dots \cap \overline{\mathcal{E}_{h-1}}$ 
and is consistent with $\Xi = \xi$. 
Consider the evolution of~$X$ given 
 $\filt_{t_{h-1}}(X) = f_{h-1}$ and $\Xi=\xi$. Since $\overline{\mathcal{E}_1} \cap \dots \cap \overline{\mathcal{E}_{h-1}}$ occurs, no non-mutant  is spawned onto $v^*$ in the interval $(t_0, t_{h-1}]$ and so $X_{t_{h-1}} = \chi(h-1)$. Moreover, $X$ remains constant in $[t_{h-1}, t_{h})$ unless a non-mutant is spawned onto $v^*$. Thus, 
given 
the condition that $\filt_{t_{h-1}}(X)=f_{h-1}$ and $\Xi=\xi$,   $\mathcal{E}_{h}$ occurs if and only if a non-mutant clock  whose source is in $V_1 \setminus \chi(h-1)$   triggers in the interval $(t_{h-1}, t_{h})$. 
Since $t_{h-1} < \tpabs$, property (A2) in the definition of $\Tpabs$ (Definition~\ref{def:Tme})
ensures that $|\chi(h-1)| < m^*$
so $|V_1 \setminus  \chi(h-1)| \geq \ell m - m^* \geq \ell m'$.
Thus,
$$\Pr( \overline{\mathcal{E}_{h}} \mid \filt_{t_{h-1}}(X) = f_{h-1},\Xi=\xi) \leq 
\exp(-\ell m'(t_{h}-t_{h-1})).$$ 
 Combining this with \eqref{eqn:ker} (by multiplying over all $h\in[z+1]$), we get
 \begin{equation} \label{eq:Dec3second}
 \Pr(\TTnmut{i} - t_0 \leq y \mid \filt_{t_0}(X)=f,  \Xi=\xi)
\geq 1-  \exp \left(
- \ell m' y 
\right).\end{equation}
Equation~\eqref{eq:xxgen} follows from \eqref{eq:Dec3first} and \eqref{eq:Dec3second}.
Equation~\eqref{eq:xxgen} shows that 
$\sum_{i=1}^{b^*} (\TTnmut{i} - \TTmut{i})$ is dominated from above by a sum    $S$ of $b^*$ i.i.d.\ exponential random variables with rate $ \ell m'$.  
It follows that
$$\Pr(\sP3) = \Pr (\iin{X}{v^*}{\Tpabs} \leq 1) =  
\Pr \left( \sum_{i=1}^{b^*} (\TTnmut{i} - \TTmut{i}) \leq 1 \right)   \geq \Pr(S \le 1).$$

The hypothesis of the lemma guarantees that $m \geq 4 m^*$ so 
$b^* = \lfloor \ell m /2 \rfloor \leq 2 \ell (m - m^*)/3 = 2 \ell m'/3$.
Therefore, by Corollary~\ref{cor:expsum},
$\Pr(S <1) \geq 1-e^{-\ell m'/16} 
\geq 1-1/n$,
where the last inequality  holds since $\log n \leq \ell (m - m^*)/16$ by the hypothesis of the lemma.
\end{proof}
 
\begin{lemma}\label{lem:funnel-eventually-absorb}
There exists $n_0>0$, depending on $r$, such that the following holds. Suppose 
$n \ge n_0$ and $k\geq 2$. 
Suppose $x_0 \in V_k$.
Let $X$ be the Moran process  
with $G(X) = \metaf$ and 
$X_0 = \{x_0\}$.  
Then $\Pr(\Pfastabsorb) \ge 1 - 1/n$.
\end{lemma}
\begin{proof}
Let $n_0$ be a large constant. For all $i \in \Zone$, let 
\[J_i^{-}=(i-1)n^2,\quad J_i^{+}=i n^2,\]
and $J_i$ be the interval $(J_i^{-}, J_i^{+}]$. Let $v_1, \dots, v_n$ be a fixed ordering of $V(\metaf)$.  Define an event $\mathcal{E}_i$ as follows. If $X_{J_i^-} = \emptyset$, then $\mathcal{E}_i$ holds. Otherwise, let $j(i) = \min \{j \mid v_j \in X_{J_i^-}\}$. Let $e(i,1), \dots, e(i,{n-1})$ be the sequence of edges returned by a breadth-first search of $\metaf$ starting from $v_{j(i)}$. Then $\mathcal{E}_i$ holds if and only if clocks in $\moranclocks(\metaf)$ trigger at least $n-1$ times in $J_i$, and the first $n-1$ such trigger events correspond to $\mc{e(i,1)}, \dots, \mc{e(i,{n-1})}$, in that order. Note that if $\mathcal{E}_i$ holds for some $i$ then the Moran process reaches absorption no later than $J_i^+$. 

Now let $f_i$ 
be any possible value
for the filtration $\filt_{J_i^-}(\Psi(X))$.
The event 
$\filt_{J_i^-}(\Psi(X))=f_i$
contains all information about $\mathcal{E}_1, \dots, \mathcal{E}_{i-1}$ and $j(i)$.
We will show that 
\begin{equation}
\label{eq:twelth}
\Pr\big({\mathcal{E}_i} \,\big|\, \filt_{J_i^-}(\Psi(X))=f_i\big) \ge \frac{1}{(2n^{2})^n}.
\end{equation}

First, if $\filt_{J_i^-}(\Psi(X))=f_i$
implies that $X_{J_i^-}=\emptyset$, then the probability in~\eqref{eq:twelth} is~$1$.
Otherwise, $f_i$ implies that $X_{J_i^-}$ is non-empty and
the first $n-1$ triggers of clocks in $\moranclocks(\metaf)$ in the interval $(J_i^-,\infty)$ are as in $\mathcal{E}_i$ with probability at least 
\[\left(\frac{(r/n)}{(1+r)n}\right)^{n-1}\ge \frac{1}{(2n^{2})^{n-1}}.\]
By Corollary~\ref{cor:pchernoff-2}, the probability that clocks in $\moranclocks(\metaf)$ trigger at least $n$ times in $J_i$ is at least $1-e^{-n^2/16}$. Thus by a union bound, we have 
established~\eqref{eq:twelth}.

  Let $i' = \min \{i \mid \mathcal{E}_i \textnormal{ holds}\}$. Then $i'-1$ is dominated below by a geometric distribution with parameter $1/(2n^{2})^n$, so 
\[\Pr(\overline{\Pfastabsorb}) \le \Pr(i' > n{(2n^2)}^n) \le \left(1 - \frac{1}{(2n^{2})^n} \right)^{n(2n^{2})^n} \le e^{-n} \le \frac{1}{n}.\qedhere\]
\end{proof}

\begin{lemma}\label{lem:funnel-few-strains-born}
There exists $n_0>0$, depending on $r$, such that the following holds. Suppose 
$n \ge n_0$ and $k\geq 2$. 
Suppose $x_0 \in V_k$.
Let $X$ be the Moran process  
with $G(X) = \metaf$ and 
$X_0 = \{x_0\}$. 
Then $\Pr(\Pvinitspawn) \ge 1-1/n$.
\end{lemma}
\begin{proof}
Let $n_0$ be large relative to $r$. For each $j \in \Zone$, 
 the number of times in $I_j$ that $x_0$ spawns a mutant in $X$ 
is dominated from above
by the number of times in $I_j$ that mutant clocks  with source $x_0$ trigger, which follows a Poisson distribution with parameter $r\len{I_j} = r(\log n)^2$.
By Corollary~\ref{cor:pchernoff}, we have
\[\Pr(x_0\textnormal{ spawns  at least }2r(\log n)^2\textnormal{ mutants in }I_j\textnormal{ in }X) \le e^{-r(\log n)^2 /3} \le e^{-(\log n)^2 /3}.\]
We have $\lceil{\tx/(\log n)^2}\rceil \le \ell m^k \le n$, so the result follows by a union bound over the $I_j$'s with $j \le \lceil \tx/(\log n)^2 \rceil$.
\end{proof}

\subsubsection{Colonies and spawning chains}
\label{sec:colony}

To deal with $\Pstrainfire$, $\Pstrainlife$ and~$\Pstraindanger$, we will need to analyse mutant spawns and deaths in 
strains and in the ``bottom layers'' of strains. We  will use  similar ideas for both cases, so to avoid redundancy we introduce the following  
general definitions.

\begin{definition}[\textbf{Colonies}]\label{defn:colony}
Fix $\metaf$ and fix $x_0 \in V_k$.
Consider the Moran process $X$ with $G(X)=\metaf$ and $X_0 = \{x_0\}$.
A \emph{colony} is a mutant process 
$Z$ with $G(Z)=\metaf$
satisfying the following conditions.
\begin{itemize}
\item For all~$t\geq 0$,
$Z_t \subseteq X_t \cap (V_1 \cup \dots \cup V_{k-1})$.
  
\item If for some $t<t'$, $Z_t$ is non-empty and $Z_{t'}=\emptyset$ then
for all $t''\geq t'$, $Z_{t''} = \emptyset$.
\end{itemize} 
 
We define the 
start and end times of any colony~$Z$ as follows. 
The subscript ``$\mathsf s$'' stands for ``start'' and the subscript ``$\mathsf e$''
stands for ``end''.
\begin{align*}
\Tstart(Z) &=  \min \{t \ge 0 \mid 
\mbox{$Z_t \ne \emptyset$ or $t=\Tpabs$}\}.
\\
\Tend(Z) &= \min \{t \ge \Tstart(Z) \mid 
\mbox{$Z_t = \emptyset$ or $t=\Tpabs$}
\}.
\end{align*} 

Further, we say that a colony~$Z$ in a Moran process~$X$ is \emph{hit}
at time~$t$ if $t=\tau_j$ for some~$j\ge 1$ and there is a vertex $v\in
Z_{\tau_{j-1}}$ such that some vertex $u$ spawns a non-mutant onto~$v$ in $X$ at time~$\tau_j$.\defend{}
\end{definition}

Since a colony~$Z$ is a mutant process but not necessarily a Moran
process, vertices may enter and/or leave~$Z$ at any time~$\tau_j$, as
long as the conditions of Definition~\ref{defn:colony} are respected.
A colony being hit at a particular time means that a vertex left the
colony at that time specifically because a non-mutant was spawned onto
it in the underlying Moran process.

Note that a strain (Definition~\ref{definition:strains}) is a colony.
Also,
$\Tstart(S^i) = \min\{\Tbirth{i},\Tpabs\}$ and $\Tend(S^i)= \min\{\Tdeath{i},\Tpabs\}$.

\begin{definition}[\textbf{Spawning chains}]\label{defn:spawn-chain}
\label{spawningchainpageref}
Fix $\metaf$ with $m > m^*$ and fix $x_0 \in V_k$. Consider the Moran process $X$ with $G(X) = \metaf$ and $X_0 = \{x_0\}$. The \emph{spawning chain} $Y^Z$  of the colony~$Z$ 
is a continuous-time stochastic process with states in $\Z$ which
evolves   as follows.  
First, for all $t \in [0,\Tstart(Z)]$, we define $Y^Z_t=1$. 
We next define $Y^Z_t$ for all $t \in (\Tstart(Z), \Tend(Z)]$. 
If $\Tstart(Z) \ge \Tend(Z)$ there is nothing to define. 
Suppose instead that $\Tstart(Z) < \Tend(Z)$, so that $\Tstart(Z) = \tau_i$ for some $i$. 

Now for any $j\geq i$ with $\tau_j < \Tend(Z)$, suppose that we are given $Y^Z_{\tau_j}$.
If $\tau_{j+1} > \Tend(Z)$, then 
for all $t\in (\tau_j,\Tend(Z)]$,
we set $Y^Z_t = Y^Z_{\tau_j}$.
Otherwise,  for all $t\in (\tau_j,\tau_{j+1})$, we set $Y^Z_t=Y^Z_{\tau_j}$ 
and we define $Y^Z_{\tau_{j+1}}$ according to the following cases.
 
\medskip\noindent
\textbf{Case 1: $\boldsymbol{Z}$ spawns a mutant at time $\boldsymbol{\tau_{j+1}}$.} 
We set $Y^Z_{\tau_{j+1}} = Y^Z_{\tau_j} + 1$.

\medskip\noindent
\textbf{Case 2: $\boldsymbol{Z}$ is hit at time $\boldsymbol{\tau_{j+1}}$.} 
With probability
\begin{equation}\label{eq:pcolony}
\frac
{m'|Z_{\tau_j}|}
{\sum_{v \in Z_{\tau_j}} \sum_{u \in N^-(v) \setminus X_{\tau_j}} (1/d^+(u))}
\end{equation}
(independently of all other events)
we set
$Y^Z_{\tau_{j+1}} = Y^Z_{\tau_j} - 1$; with the remaining probability, we set $Y^Z_{\tau_{j+1}} = Y^Z_{\tau_j}$.
We will show below that the probability in~\eqref{eq:pcolony} is well-defined.

\medskip\noindent
\textbf{Case 3: Neither Case 1 nor Case 2 holds.}
We set $Y^Z_{\tau_{j+1}} = Y^Z_{\tau_j}$.

We have now defined $Y^Z_t$ for all $t \le \Tend(Z)$. Finally,   for  $t > \Tend(Z)$ the spawning chain
$Y^Z_t$ evolves independently of $\Psi(X)$
as a continuous-time Markov chain on 
$\Z$ with 
start state $Y^Z_{\Tend(Z)}$ and
the following transition rate matrix. 
\[
 \transmatrix_{a,b}  = 
	\begin{cases}
		r 	& \textnormal{if }b=a+1,\\
		m'  & \textnormal{if }b=a-1,\\
		0 					& \textnormal{otherwise.}
	\end{cases}
\]
\defend{}
\end{definition}

\begin{definition}[\bf The process $\boldsymbol{\Psi(X,Z)}$ and its filtration]\label{def:theprocessPsiXZ}
Fix $\metaf$ with $m > m^*$ and fix $x_0 \in V_k$.
Consider the Moran process $X$ with $G(X)=\metaf$ and $X_0 = \{x_0\}$.
Consider a colony~$Z$.
Let $\Psi(X,Z)$   be the stochastic process consisting of
$\Psi(X)$ together with the spawning chain $Y^Z$
whose evolution is coupled with that of $\Psi(X)$ in the manner described above.
The filtration  $\filt_t(\Psi(X,Z))$  of $\Psi(X,Z)$ consists of
$\filt_t(\Psi(X))$ together with all information about the transitions of $Y^Z$ up to time~$t$.\defend{}
\end{definition}

\begin{definition}[\bf The jump chain $\boldsymbol{\widehat{Y}^Z}$]\label{defn:spawnjumpchain}
Fix $\metaf$ with $m > m^*$ and fix $x_0 \in V_k$.
Consider the Moran process $X$ with $G(X)=\metaf$ and $X_0 = \{x_0\}$.
Consider a colony~$Z$. Then the \emph{jump chain}
$\widehat{Y}^Z$ of $Y^Z$ is defined straightforwardly.
$\widehat{Y}^Z(0) = Y^Z_0=1$, and 
if $Y^Z$ makes its $i$'th transition at time $ t$ then
$\widehat{Y}^Z(i) = Y^Z_{ t}$.\defend{}
\end{definition}
 
We now show that the probability in~\eqref{eq:pcolony} is well-defined. First note that since $m > m^*$, the numerator of \eqref{eq:pcolony} is positive. Moreover, since $Z$~is hit at time~$\tau_{j+1}$, there must be a vertex $v\in Z_{\tau_j}$ and a vertex $u \in N^-(v) \setminus X_{\tau_j}$ that spawned onto~$v$ at time~$\tau_{j+1}$.  So, certainly, the denominator of~\eqref{eq:pcolony} is nonzero.  Now, 
let $v \in Z_{\tau_j}$ be arbitrary, and 
let $i$ be the integer in $[k-1]$ such that $v\in V_i$. Then since 
$\tau_j < \Tend(Z) \le \Tpabs$, by \myref{A}{A:fill} we have $|X_{\tau_j}|\le m^*$. It follows that
\begin{align*}
\sum_{u \in N^-(v) \setminus X_{\tau_j}} \frac{1}{d^+(u)}
	  = \frac{|N^-(v) \setminus X_{\tau_j}|}{m^i}
	  \ge \frac{m^{i+1} - |X_{\tau_j}|}{m^i} 
	  \ge \frac{m^{i+1} - m^*}{m^i} 
	  \ge m'.
\end{align*}
It follows that
\[\sum_{v \in Z_{\tau_j}} \sum_{u \in N^-(v) \setminus X_{\tau_j}} 
\frac{1}{d^+(u)} \ge m'|Z_{\tau_j}|,\]
so \eqref{eq:pcolony} is, as claimed, a probability.

Note that either $\Tend(Z) = \Tpabs$ or $Z_t = \emptyset$ for all $t \in [\Tend(Z), \Tpabs]$. In either case, we see that spawning chains as defined above satisfy two important properties. 

\begin{enumerate} [(Z1)]
\item  $\{t \le \Tend(Z) \mid Y^Z \textnormal{ increases at }t\} = \{t \le \Tpabs \mid Z\textnormal{ spawns a mutant at }t\}$.\label{Z:one}
\item  $\{t \le \Tend(Z) \mid Y^Z \textnormal{ decreases at }t\}
    \subseteq \{t \le \Tpabs \mid Z\textnormal{ is hit at }t\}$.\label{Z:two}
\end{enumerate}

Intuitively, we expect that for all $t \in (\Tstart(Z), \Tend(Z)]$, 
the spawning chain
$Y^Z_t$ should behave similarly to a continuous-time Markov chain on $\Z$ which increments with rate $r |Z_t|$ and decrements with rate $m' |Z_t|$. For technical convenience we will not prove this. Instead, we will prove that the jump chain 
$\widehat{Y}^Z$
 evolves as a random walk on $\Z$ with 
 appropriate probabilities.

\begin{definition}\label{def:Upsilon} Let
\begin{align*}
\mathcal{A} &= \{\mc{(v^*,u)}   \mid u \in V_{k}\} \cup \{\nc{(v^*,x_0)}\},\textnormal{ and}\\
\mathcal{A}^* &= \{\smc{(v^*,u)} \mid u \in V_k\} \cup \{\smbc{(v^*,u)} \mid u \in V_k\} \cup \{\snc{(v^*,x_0)},\snbc{(v^*,x_0)}\}.
\end{align*}
Let $\Phi$ 
contain, for each star-clock $C \in \mathcal{A}^*$, a list of the times at which $C$ triggers in $[0,\tmax]$. 

\defend{}
\end{definition}

The star-clocks in $\mathcal{A}^*$ are   part of the star-clock process~$P^*(\metaf)$, so of course the
times in $\Phi$ 
are ``local'' and don't  necessarily correspond to the times that clocks in~$\mathcal{A}$ trigger.
However, these are related by Observation~\ref{obs:one}.

In order to prove Lemma~\ref{lem:funnel-masterlist}, we will need to show that $\Pstraindanger$~is reasonably likely to occur, conditioned on $\Pmutclock\cap\Pnmutclock$. The following lemma (among others) will be used for this purpose, so we prove it conditioned on 
$\Phi$.

\begin{lemma}\label{lem:funnel-firing-jump-chain} 
Let $k\geq 2$ and $m \ge 5m^*$. Fix $\metaf$ and fix $x_0 \in V_k$.
Consider the Moran process $X$   with $G(X)=\metaf$ and $X_0 = \{x_0\}$.
Let $Z$ be a colony. 
Let $t_0$ be a non-negative real number, and let 
$f_0$ be a possible value of the filtration $\filt_{t_0}(\Psi(X,Z))$.
Let $\varphi$ be a possible value of $\Phi$.
If the three events
$\Tstart(Z)=t_0$, $\filt_{t_0}(\Psi(X,Z))=f_0$ and $\Phi=\varphi$ are
consistent, 
then conditioned on these three events,  
the jump chain
$\widehat{Y}^Z$ evolves 
as a random walk on $\Z$ with initial state $1$ and the following transition matrix.
\[
\widehat{Y}^Z(a,b) = 
	\begin{cases}
		r/(r+m') 		& \textnormal{if }b=a+1, \\
		m'/(r+m')  	& \textnormal{if }b=a-1,\\
		0 				& \textnormal{otherwise.}
	\end{cases}
\]
\end{lemma} 
\begin{proof}

The definition of the  jump chain $\widehat{Y}^Z$ implies
 that $\widehat{Y}^Z(0) = 1$. 

 Let $T_0 =  t_0$. 
 For $i \in \Zone$, let $T_i$ be the 
 random variable that is the time of $Y^Z$'s $i$'th transition. 
 
 Now consider 
 an $i\in\Zzero$ and
 a non-negative real number~$t_i$.  (If $i=0$ then $t_0$ is already defined in the statement of the lemma. 
 Otherwise, consider any $t_i\geq t_0$.)
  Suppose that $f_i$ is a possible value for the filtration $\filt_{t_i}(\Psi(X,Z))$  and that the events
  $\Tstart(Z)=t_0$,
 $\filt_{t_0}(\Psi(X,Z))=f_0$,
 $\Phi=\varphi$,
  $\filt_{t_i}(\Psi(X,Z))=f_i$, 
and $T_i=t_i$
 are consistent.   
 Note that
 all of these events are determined by $\filt_{t_i}(\Psi(X,Z))=f_i$ and
 $\Phi=\varphi$, which also determine $\widehat{Y}^Z(0), \dots, \widehat{Y}^Z(i)$.
    We therefore wish to show that
 \begin{equation}\label{eqn:funnel-firing-jump-chain-0}
\begin{aligned}
\Pr\big(\widehat{Y}^Z(i+1) = \widehat{Y}^Z(i) + 1 \,\big|\,
\filt_{t_i}(\Psi(X,Z)) = f_i,\Phi=\varphi \big) &= \frac{r}{r+m'},\mbox{ and}\\
\Pr\big(\widehat{Y}^Z(i+1) = \widehat{Y}^Z(i) - 1 \,\big|\, \filt_{t_i}(\Psi(X,Z))=f_i, \Phi=\varphi\big) &= \frac{m'}{r+m'}.
\end{aligned}
\end{equation}

Let $t_{i+1} > t_i$ be arbitrary. Let $\Xi$ 
contain, for each clock $C \in \moranclocks(\metaf)$, a list of the times at which $C$ triggers in $(t_i, t_{i+1})$.  Suppose that $\xi$ is a possible value of $\Xi$ such that the event $\Xi=\xi$ is consistent with the events $\filt_{t_i}(\Psi(X,Z)) = f_i$, $T_{i+1} = t_{i+1}$, and
$\Phi = \varphi$.
Let $\mathcal{F}$ be the intersection of these four events.
Namely, $\mathcal{F}$ is the intersection of 
\begin{itemize}
\item $\filt_{t_i}(\Psi(X,Z)) = f_i$,
\item $\Xi=\xi$,
\item $T_{i+1}=t_{i+1}$, and
\item $\Phi = \varphi$.
\end{itemize} 
Let $\mathcal{E}_1$ be the event that $Y^Z_{t_{i+1}} = Y^Z_{t_i}+1$ and
let $\mathcal{E}_2$ be the event that $Y^Z_{t_{i+1}}=Y^Z_{t_i}-1$.
By integrating over all choices of~$t_{i+1}$ and~$\xi$,
Equation~\eqref{eqn:funnel-firing-jump-chain-0} will follow from
 \begin{equation}\label{eq:bettername}
\begin{aligned}
\Pr\left(  \mathcal{E}_1 \mid \mathcal{F} \right)&= \frac{r}{r+m'},\mbox{ and}\\
\Pr\left( \mathcal{E}_2  \mid  \mathcal{F}\right) &= \frac{m'}{r+m'}.
\end{aligned}
\end{equation}
Since $\mathcal{F}$ implies $T_i=t_i$ and $T_{i+1}=t_{i+1}$, it implies
  $\mathcal{E}_1 \cup \mathcal{E}_2$.
Also, $\mathcal{E}_1 \cap \mathcal{E}_2$ is empty.
$\mathcal{F}$ determines   the evolution of $\Psi(X,Z)$ throughout $[0,t_{i+1})$. In particular, 
it determines whether 
the event $\Tend(Z) < t_{i+1}$ occurs. We split into cases accordingly.

\medskip\noindent\textbf{Case 1:  $\mathcal{F}$ implies that $\boldsymbol{t_{i+1} > \Tend(Z)}$.} 
In this case, conditioned on 
$\mathcal{F}$, the behaviour of $Y_t^Z$ at $t_{i+1}$ is governed entirely by the transition rate matrix $R$ and is therefore independent of $\Phi$ and $\Psi(X)$. 
The definition of the spawning chain $Y^Z$ gives Equation~\eqref{eq:bettername}.

\medskip\noindent\textbf{Case 2:  $\mathcal{F}$ implies that $\boldsymbol{t_{i+1} \le \Tend(Z)}.$} 
Let $T^- = \max\{\tau_j \mid \tau_j < t_{i+1}\}$, and let $t^-$ be the unique value of $T^-$ consistent with $\filt$. Let $\chi_{t^-}$ be the unique value of $X_{t^-}$ consistent with $\filt$, and let $\zeta_{t^-}$ be the unique value of $Z_{t^-}$ consistent with $\filt$. Define
\begin{align*}
S &= \sum_{v \in \zeta_{t^-}} \sum_{u \in N^-(v) \setminus \chi_{t^-}} (1/d^+(u)),\\
\mathcal{B}_1 &= \{\vmc{u}{v} \mid u \in \zeta_{t^-}\},\\
\mathcal{B}_2 &= \{\vnc{u}{v} \mid u \in V(\metaf) \setminus \chi_{t^-} \textnormal{ and }v \in \zeta_{t^-}\}.
\end{align*}
Consider the following events.
\begin{itemize}
\item $\widehat{\mathcal{E}}_1$ is the event that a clock in $\mathcal{B}_1$ triggers at $t_{i+1}$, and
\item $\widehat{\mathcal{E}}_2$ is the event that
\begin{itemize}
\item a clock in $\mathcal{B}_2$ triggers at $t_{i+1}$, and
\item an (independent) coin toss (part of the spawning chain), with probability $m'|\zeta_{t^-}|/S$ of coming up heads, comes up heads. 
\end{itemize}
\end{itemize}
Note that, conditioned on~$\mathcal{F}$, event~$\widehat{\mathcal{E}}_1$ coincides with~$\mathcal{E}_1$
and $\widehat{\mathcal{E}}_2$ coincides with~$\mathcal{E}_2$.
It is easy to see, using the definition of a Poisson process, that
 \begin{align*}
 \Pr\left( \widehat{\mathcal{E}}_1 \mid 
 \widehat{\mathcal{E}}_1 \cup \widehat{\mathcal{E}}_2 \right)  &= \frac{r|\zeta_{t^-}|}{r|\zeta_{t^-}| + S(m'|\zeta_{t^-}|/S)} = \frac{r}{r+m'},\mbox{ and}\\
  \Pr\left( \widehat{\mathcal{E}}_2 \mid 
 \widehat{\mathcal{E}}_1 \cup \widehat{\mathcal{E}}_2 \right)   &= \frac{m'}{r+m'}.
\end{align*}
In order to establish~\eqref{eq:bettername}, we would like to  show that conditioning on~$\mathcal{F}$ is equivalent
to conditioning on~$ \widehat{\mathcal{E}}_1 \cup \widehat{\mathcal{E}}_2$.
This is straightforward, apart from the event
$\Phi=\phi$, which is part of~$\mathcal{F}$.
Unfortunately, we need the result of the lemma to be conditioned on $\Phi=\phi$ and
not merely on the rest of~$\mathcal{F}$, so the rest of this proof is 
merely technical, and is to deal with this.

To proceed, we consider the four events making up~$\mathcal{F}$.
\begin{itemize}
\item
First, consider the event $\filt_{t_i}(\Psi(X,Z)) = f_i$.  
Let $\hat{f}_i$ be the induced value of $\filt_{t_i}(\Psi(X))$.
The value~$f_i$ consists of $\hat{f}_i$, together
with   the extra
 information about the transitions of~$Y^Z$ up to time~$t_i$ 
 (giving the outcomes of the independent coin tosses that are part of the spawning chain $Y^Z$).
 The value~$\hat{f}_i$ contains, 
 for each clock  
 $C \in \moranclocks(\metaf)$, a list of the times at which $C$ triggers 
 in $[0,t_i]$.
 It also contains information about the times that 
 the star-clocks trigger, according to the coupling in Section~\ref{sec:coupled}.
Using Observation~\ref{obs:one}, we could translate $\hat{f}_i$
into a  (unique) equivalent event  
which is 
a list of times at which certain star-clocks trigger.
We will  therefore 
write $\filt^*_{t_i}=f^*_i$ to denote the event
$\filt_{t_i}(\Psi(X,Z)) = f_i$, expressed entirely in terms of star-clock triggers
and outcomes of spawning-chain coin tosses.
\item Similarly, given 
$\filt^*_{t_i}=f^*_i$, we can 
uniquely 
express
$\Xi=\xi$ as an event 
which is 
a list of times at which certain star-clocks trigger. 
We will denote this event as $\Xi^*=\xi^*$.
The definitions of $t^-$, $\chi_{t^-}$ and $\zeta_{t^-}$ 
can be deduced from $f^*_i$ and $\xi^*$.
\item 
Now let  
$\mathcal{B}^*_1 = \{\vsmc{u}{v} \mid u \in \zeta_{t^-}\}$ and
$\mathcal{B}^*_2 = \{\vsnc{u}{v} \mid u \in V(\metaf) \setminus \chi_{t^-} \textnormal{ and }v \in \zeta_{t^-}\}$.
Note that for every star-clock with source~$u$ in $\mathcal{B}^*_1 \cup \mathcal{B}^*_2$
the quantities $\iin{u}{X}{t_{i+1}}$
and $\iout{u}{X}{t_{i+1}}$
can be deduced  from $f^*_i$ and $\xi^*$.
Let $\mathcal{E}^*_1$ be the event that 
a star-clock with source~$u$ in $\mathcal{B}^*_1$ triggers at $\iin{u}{X}{t_{i+1}}$.
Note that by Observation~\ref{obs:one}, $\filt^*_{t_i}=f^*_i$ and
$\Xi^*=\xi^*$ implies that
$\mathcal{E}^*_1$ is equivalent to $\widehat{\mathcal{E}}_1$.
Similarly, let $\mathcal{E}^*_2$ be the event that
a star-clock with source $u$ in $\mathcal{B}^*_2$ triggers at $\iout{u}{X}{t_{i+1}}$ and the
appropriate coin toss comes up heads, so that 
$\filt^*_{t_i}=f^*_i$ and
$\Xi^*=\xi^*$ implies that
$\mathcal{E}^*_2$ is  equivalent to $\widehat{\mathcal{E}}_2$.
 \end{itemize}
 We now have 
 \begin{align*}
 \Pr\left(  {\mathcal{E}}^*_1 \mid 
 \filt^*_{t_i}=f^*_i, \Xi^*=\xi^*, 
  \ {\mathcal{E}}^*_1 \cup  {\mathcal{E}}^*_2 \right)  &=   \frac{r}{r+m'},\mbox{ and}\\
  \Pr\left(  {\mathcal{E}}^*_2 \mid 
  \filt^*_{t_i}=f^*_i, \Xi^*=\xi^*,   
  {\mathcal{E}}^*_1 \cup  {\mathcal{E}}^*_2 \right)   &= \frac{m'}{r+m'}.
\end{align*} 
Finally, 
since $\mathcal{B}^*_1 \cup \mathcal{B}^*_2$ is disjoint from $\mathcal{A}^*$,
we conclude that,
conditioned on $\filt^*_{t_i}=f^*_i$, $\Xi^*=\xi^*$, and
   ${\mathcal{E}}^*_1 \cup  {\mathcal{E}}^*_2 $, the event
  $\mathcal{E}^*_1$ is independent of $\Phi$ (by independence of star-clocks in the 
  star-clock process).
  Thus, we obtain 
\begin{align*}
 \Pr\left(  {\mathcal{E}}^*_1 \mid 
 \filt^*_{t_i}=f^*_i, \Xi^*=\xi^*, 
  \ {\mathcal{E}}^*_1 \cup  {\mathcal{E}}^*_2 ,\Phi=\varphi \right)  &=  \frac{r}{r+m'},\mbox{ and}\\
  \Pr\left(  {\mathcal{E}}^*_2 \mid 
  \filt^*_{t_i}=f^*_i, \Xi^*=\xi^*,   
  {\mathcal{E}}^*_1 \cup  {\mathcal{E}}^*_2 ,\Phi=\varphi\right)   &= \frac{m'}{r+m'},
\end{align*} 
which implies~\eqref{eq:bettername} by translating the events back to their original formulation. 
\end{proof}

We next prove that, with high probability, $Y^Z$ transitions many times shortly after $\Tstart(Z)$. 

\begin{lemma}\label{lem:funnel-transition-fast}
There exists $n_0>0$ such that the following holds for all $n \ge n_0$, $m \ge 5m^*$ and $k\geq 2$.
Fix $\metaf$ with $n$ vertices and fix $x_0 \in V_k$.
Consider the Moran process $X$  with $G(X)=\metaf$ and $X_0 = \{x_0\}$.
Let $Z$ be a colony. 
Let $t_0$ be a non-negative real number, and let 
$f_0$ be a possible value of the filtration $\filt_{t_0}(\Psi(X,Z))$.
If the  events
$\Tstart(Z)=t_0$ and $\filt_{t_0}(\Psi(X,Z))=f_0$  are
consistent  
then, conditioned on these events,  
 with probability at least $1-e^{-(\log n)^2 /16}$,
 $Y^Z$ increases at least $2\log n$ times in the interval $[t_0, t_0+(\log n)^2]$ .
\end{lemma}
\begin{proof}
Let $P$ be a Poisson process with rate~$r$.
Conditioned on $\Tstart(Z)=t_0$ and $\filt_{t_0}(\Psi(X,Z))=f_0$,
we will couple the evolution of $\Psi(X,Z)$  from time~$t_0$
with that of~$P$.
The coupling will have the property that every time $P$ triggers,
$Y^Z$ increases.

Given the coupling,
we can conclude that the probability that 
$Y^Z$ increases at least $2\log n$ times in the interval $[t_0, t_0+(\log n)^2]$ 
is at least the probability that 
$P$ triggers at least $2\log n$ times in an interval of length~$(\log n)^2$.
This is the probability that a Poisson random variable~$W$ with parameter
$\rho = r (\log n)^2$ has value at least
$2 \log n$. We have $2\log n < 2 \rho/3$, so by Corollary~\ref{cor:pchernoff-2},
$\Pr(W\leq 2\log n)\leq \exp(-\rho/16) \leq \exp(- (\log n)^2/16)$.

It remains to give the details of the coupling.
Roughly, the coupling will be constructed using the sequence
$\tau_1,\tau_2,\ldots$.
However, one technical detail arises,
since $\Tend(Z)$ might not occur at one of the instants
$\tau_1,\tau_2,\ldots$.
So, for the purposes of the proof, let
$\hat{\tau}_1,\hat{\tau}_2,\ldots$ be the increasing sequence containing $\Tend(Z)$ and all $\tau_1, \tau_2, \dots$ from $\Psi(X)$ (and nothing else).  This random sequence is a function of the
evolution of $\Psi(X,Z)$.

Now, conditioned on $\Tstart(Z)=t_0$ and $\filt_{t_0}(\Psi(X,Z))=f_0$  
there is a non-negative integer~$j$ such that $\hat{\tau}_j=t_0$.
We will define the coupling from each $\hat{\tau}_i$ for $i\geq j$.

So consider $i\geq j$ and suppose that for some time $t_i$
and some filtration value $f_i$,
we have $\hat{\tau}_i= t_i$ and $\filt_{t_i}(\Psi(X,Z)) = f_i$.
To continue the coupling   in the open interval 
from~$t_i$ there are two cases.
It is easy to determine which case applies, using~$f_i$.

If $\Tend(Z)\leq t_i$, then the evolution of $Y^Z$ after time $t_i$ is  
a continuous-time Markov chain with transition matrix~$\transmatrix$, evolving independently of  $\Psi(X)$.
The rate of an upwards transition in~$\transmatrix$ is~$r$
so use the triggering of~$P$ to dictate these upwards transitions.

If $\Tend(Z) > t_i$, then $Z_{t_i}$ is non-empty,
so choose some $u\in Z_{t_i}$.
Now use the triggering of~$P$ to dictate the triggering of the
mutant clocks with source~$u$ (which together have rate~$r$). 
Thus, every time $P$ triggers, a mutant clock in $\moranclocks(\metaf)$ with source~$u$
is chosen uniformly at random to trigger.
\end{proof}

\subsubsection{Continuing the proof of Lemma~\ref{lem:funnel-masterlist}}

We are now in a position to bound the probabilities with which events $\Pstrainfire$--$\Pstraindanger$ occur, using Lemmas~\ref{lem:funnel-firing-jump-chain} and~\ref{lem:funnel-transition-fast}. 
We start by showing that, up until $\Tend(S^i)$,
the state of the spawning chain
 $Y^{S^i}$ is at least as large as the size
 of the strain~$S^i$ (which is a colony).

\begin{lemma}\label{prop:funnel-strain-dominated}
Let $k\geq 2$ and $m \ge 5m^*$. Fix $x_0 \in V_k$.  Consider the Moran process $X$  with $G(X)=\metaf$ and $X_0 = \{x_0\}$.
Let $i\in\Zone$. For all 
$t \le \Tpabs$, $|S^i_t| \le 
\max\{Y^{S^i}_t,0\}$.
 \end{lemma}
\begin{proof}
For $t \le \Tstart(S^i)$ we have $|S^i_t| \le 1 = Y^{S^i}_t\!$, and for 
$\Tend(S^i) \le t \le \Tpabs$ we have $|S^i_t| = 0 \leq 
\max\{Y^{S^i}_t,0\}$. 
Since $S^i$ and $Y^{S^i}$ remain constant except at 
$\tau_1, \tau_2, \dots$, it suffices to prove that 
$|S^i_{\tau_j}| \le Y^{S^i}_{\tau_j}$ for 
all $\tau_j \in [\Tstart(S^i),\Tend(S^i))$.
We  will do so by induction on $j$. 

If $\Tstart(S^i) = \Tend(S^i) = \Tpabs$ then there is nothing to prove, so suppose 
$\Tstart(S^i) = \tau_x$ for some $x$. 
We have already seen that the claim holds for $j=x$. 
Suppose that the claim holds for some $j \ge x$, and that 
$\tau_{j+1} < \Tend(S^i)$. 
We will now prove the claim for $j+1$ by dividing into cases.

\medskip\noindent
\textbf{Case 1: $\boldsymbol{|S^i_{\tau_{j+1}}| > |S^i_{\tau_j}|}$.}
This case may arise only if $S^i$ spawns a mutant at time $\tau_{j+1}$. 
Thus by (Z\ref{Z:one}), we have $Y^{S^i}_{\tau_{j+1}} = Y^{S^i}_{\tau_j}+1 
\ge |S^i_{\tau_j}|+1 = |S^i_{\tau_{j+1}}|$.

\medskip\noindent
\textbf{Case 2: $\boldsymbol{Y^{S^i}_{\tau_{j+1}} < Y^{S^i}_{\tau_j}}$.}
This case may arise only if the spawning chain $Y^{S^i}$ is decremented at time $\tau_{j+1}$, 
which by (Z\ref{Z:two}) may happen only if $S^i$ is hit at time $\tau_{j+1}$. Thus we have 
$Y^{S^i}_{\tau_{j+1}} = Y^{S^i}_{\tau_j} - 1 \ge |S^i_{\tau_j}| - 1 = |S^i_{\tau_{j+1}}|$.

\medskip\noindent
\textbf{Case 3: $\boldsymbol{Y^{S^i}_{\tau_{j+1}} \ge Y^{S^i}_{\tau_j}}$ and $\boldsymbol{|S^i_{\tau_{j+1}}| \le |S^i_{\tau_j}|}$.}
In this case, we have $Y^{S^i}_{\tau_{j+1}} \ge Y^{S^i}_{\tau_j} \ge |S^i_{\tau_j}| \ge |S^i_{\tau_{j+1}}|$.

Thus the claim, and therefore the result, holds in all cases.
\end{proof}
\begin{corollary}\label{cor:funnel-tend-dom}
Let $k\geq 2$ and $m \ge 5m^*$. Fix $x_0 \in V_k$.  Consider the Moran process $X$  with $G(X)=\metaf$ and $X_0 = \{x_0\}$.
Let $i$ be a positive integer. Suppose that 
$t\geq \Tstart(S^i)$ and
$Y^{S^i}_t = 0$. Then $t\geq \Tend(S^i)$.
\end{corollary}
\begin{proof}
If $t \ge \Tpabs$, then $t \ge \Tend(S^i)$ by the definition of $\Tend(S^i)$. Otherwise $|S^i_t| = 0$ by  Lemma~\ref{prop:funnel-strain-dominated}, and so again $t \ge \Tend(S^i)$ by definition. 
\end{proof}

It now follows by (Z\ref{Z:one}) and Corollary~\ref{cor:funnel-tend-dom} that $\Pstrainfire$ occurs if,
for each $i\in[s]$,
$Y^{S^i}$ increases at most $\log n$ times before reaching zero. We  will then be able to apply Lemma~\ref{lem:funnel-firing-jump-chain} to reduce this to a simple question about random walks. We state the lemma below somewhat more generally since we will use it again to deal with~$\Pstraindanger$.

\begin{lemma}\label{lem:funnel-gamblers-ruin} 
There exists $n_0>0$, depending on $r$, such that the following holds. 
Suppose 
$n \ge n_0$, $m \geq 5 m^*$ and $k\geq 2$. 
Fix $x_0 \in V_k$. 
Consider the Moran process $X$  with $G(X)=\metaf$ and $X_0 = \{x_0\}$.
Suppose that $i$ and $y$ are positive integers and that
$\varphi$ is a possible value for $\Phi$.
Then, conditioned on $\Phi = \varphi$, 
  the probability in $\Psi(X)$
that $S^i$ spawns at least $y$ mutants  in $(0,\Tpabs]$
is at most $(20r/m)^{y}$.
\end{lemma}

\begin{proof}
By (Z\ref{Z:one}), the number of mutants $S^i$ spawns 
in $(0,\Tpabs]$ 
is equal to the number of times $Y^{S^i}$ increases in $(\Tstart(S^i), \Tend(S^i)]$. By Corollary~\ref{cor:funnel-tend-dom}, $Y^{S^i}$ cannot reach zero until $\Tend(S^i)$, and so it suffices to prove that
\begin{equation}\label{eqn:funnel-gamblers-ruin}
\Pr\big(Y^{S^i}\textnormal{ increases at least }y\textnormal{ times before reaching }0 \,\big|\, \Phi = \varphi\big) \le
(20r/m)^{y}.
\end{equation}

Let $t_0$ be a non-negative real number and let $f_0$ be a possible  
value of $\filt_{t_0}(\Psi(X,S^i))$ consistent with $\Tstart(S^i)=t_0$ and $\Phi=\phi$. Recall 
from Lemma~\ref{lem:funnel-firing-jump-chain}
the transition probabilities  of $\widehat{Y}^{S^i}\!$, conditioned on 
$\Tstart(S^i)=t_0$, 
$\filt_{t_0}(\Psi(X,S^i))=f_0$
and $\Phi=\varphi$.

For all $j \in \Zone$, let $\mathcal{E}_j$ be the event that $Y^{S^i}$ takes exactly $j$ forward steps from $1$ before reaching $0$. Let $\mathcal{E}_{\ge y} = \bigcup_{j=y}^\infty \mathcal{E}_{j}$. 
Since  
$m\geq 5m^* > r$, 
 $Y^{S^i}$ reaches $0$ with probability $1$, so we have
\[\Pr(\mathcal{E}_{\ge y}) = \sum_{j=y}^\infty\Pr(\mathcal{E}_{j}).\]
For $Y^{S^i}$ to reach $0$ after exactly $j$ forward steps, $Y^{S^i}$ must decrease exactly $j+1$ times for a total of $2j+1$ steps. Thus
\[\Pr(\mathcal{E}_{\ge y}) 
	\ =\ \sum_{j=y}^\infty \binom{2j+1}{j}\left(\frac{r}{r+m'}\right)^j \left(\frac{m'}{r+m'} \right)^{j+1} 
	\le\ \sum_{j=y}^\infty 2^{2j+1} \left(\frac{r}{r+m'} \right)^j.\]
Since $4r/(r+m') \le 1/2$, it follows that
\[\Pr(\mathcal{E}_{\ge y})
    \le 2\sum_{j=y}^\infty \left(\frac{4r}{r+m'} \right)^j
    \le 4\left(\frac{4r}{m'}\right)^y
    \le 4\left(\frac{5r}{m}\right)^y
    \le \left(\frac{20r}{m}\right)^y.\]
Here the penultimate inequality follows since $m \ge 5m^*$, so $m' \ge 4m/5$.
Thus \eqref{eqn:funnel-gamblers-ruin} holds, and the result follows.
\end{proof}

\begin{corollary}\label{cor:funnel-dangerous-strains-not-bad}
There exists $n_0>0$, depending on $r$, such that the following holds. Suppose $n \ge n_0$, $m \geq 5m^*$ and $k\geq 2$.
Fix $x_0 \in V_k$. 
Let $X$ be  the Moran process   with $G(X)=\metaf$ and $X_0 = \{x_0\}$.
Then $\Pr(\Pstrainfire) \ge 1-1/n$.
\begin{proof} 
Recall from Definition~\ref{defn:funnel-constants} that $s = \lceil 3r\tx \rceil$. Fix $i \in [s]$. Then by Lemma~\ref{lem:funnel-gamblers-ruin}, the probability that $S^i$ spawns at most $\lfloor \log n
\rfloor $ mutants 
in $(0,\Tpabs]$ is at least
\[1 - \left(\frac{20r}{m}\right)^{\lfloor \log n\rfloor + 1} \ge 1 - \left(\frac{1}{\log n}\right)^{\log n} = 1 - e^{-\log n \log \log n} \ge 1 - \frac{1}{n^2}.\]
By a union bound, the probability that each of $S^1, \dots, S^s$ spawn at most $\log n$ mutants 
in $(0,\Tpabs]$ 
is at least $1-1/n$ as required.
\end{proof}
\end{corollary}

Now, $\Pstrainlife$ is implied by $\Pstrainfire$ and Lemma~\ref{lem:funnel-transition-fast}.

\begin{lemma}\label{lem:funnel-strains-die-fast} There exists $n_0>0$, depending on $r$, such that the following holds. Suppose $n \ge n_0$, $k\geq 2$ and $m \ge 5m^*$. 
Fix $x_0 \in V_k$. 
Let $X$ be  the Moran process    with $G(X)=\metaf$ and $X_0 = \{x_0\}$.
Then $\Pr(\Pstrainlife) \ge 1-2/n$.
\end{lemma}
\begin{proof}
By applying a union bound to Corollary~\ref{cor:funnel-dangerous-strains-not-bad} and Lemma~\ref{lem:funnel-transition-fast}, with probability at least $1 - 2/n$, $\Pstrainfire$~occurs and, for all $i \in [s]$, $Y^{S^i}$ increases at least $2\log n$ times in $[\Tstart(S^i), \Tstart(S^i) + (\log n)^2]$. By (Z\ref{Z:one}) and the occurrence of~$\Pstrainfire$, $Y^{S^i}$ can increase at most $\log n$ times in $[\Tstart(S^i), \Tend(S^i))$, so we must have $\Tend(S^i) \le \Tstart(S^i) + (\log n)^2$. Thus, $\Pstrainlife$~occurs by the definitions of $\Tstart(S^i)$ and $\Tend(S^i)$.
\end{proof}

It remains only to bound the probability of~$\Pstraindanger$. Note that while Lemma~\ref{lem:funnel-gamblers-ruin} bounds the number of mutants spawned by any vertex in a strain, to tightly bound the probability that the strain is dangerous we need to look at the number of mutants spawned from the ``layer'' of the strain closest to the centre vertex. 
\begin{definition}[\textbf{Head $\boldsymbol{\head^i_t}$ of a strain}]\label{defn:head}
Fix $\metaf$ and fix $x_0 \in V_k$.
Consider the Moran process $X$ with $G(X)=\metaf$ and $X_0 = \{x_0\}$.
For each positive integer~$i$, we define the \emph{head} $\head^i$ of the strain $S^i$ as follows. For $t\geq 0$, let
\[
\head^i_t = 
	\begin{cases}
		\emptyset                                                  & \textnormal{ if }S^i_t = \emptyset,\\
		S^i_t \cap V_{\min\{j \mid S^i_t \cap V_j \ne \emptyset\}} & \textnormal{ otherwise.}
	\end{cases}
\]
Note that, for all~$i$,
$\head^i$ is a colony with $\Tstart(\head^i) =\Tstart(S^i)$ and 
$\Tend(\head^i)=\Tend(S^i)$. \defend{}
\end{definition}

The following lemma relates 
the behaviour of the spawning chain $Y^{\head^i}$ to
 the question of whether or not $S^i$ is dangerous.

\begin{lemma}\label{lem:funnel-H-useful} 
Let $k\geq 2$ and $m \ge 5m^*$. Fix $x_0 \in V_k$. 
Let $X$ be  the Moran process   with $G(X)=\metaf$ and $X_0 = \{x_0\}$.
Let $i$ be a positive integer. 
Suppose that $S^i$ is dangerous, spawning a mutant onto $v^*$ 
for the first time at some time 
$\tspawn \le \Tpabs$. Then $\tspawn \le \Tend(S^i)$, and the following two statements hold in $\Psi(X,\head^i)$.
\begin{enumerate}[(i)]
\item\label{fun-H-use-1} $|\{t \le \tspawn \mid Y^{\head^i} \textnormal{ increases at }t\}| \ge k-1$.
\item\label{fun-H-use-2} $|\{t \le \tspawn \mid S^i \textnormal{ spawns a mutant at }t \}|\ge (k-1)+ |\{t \le \tspawn \mid Y^{\head^i} \textnormal{ decreases at }t\}| $.
\end{enumerate}
\end{lemma}
\begin{proof}

First note that since $\tspawn \le \Tpabs$ and $S^i$ is non-empty at $\tspawn$, it is immediate that $\Tstart(S^i) < \tspawn \le \Tend(S^i)$. We now define some notation. Clearly $\tspawn$ is in the sequence $\tau_1,\tau_2,\ldots$, so write $\tspawn=\tau_j$. For $t \in [\Tstart(S^i), \Tend(S^i))$, let $h^i(t) = \min \{y \mid S_t^i \cap V_y \ne \emptyset\}$, so that $\head^i_t = S_t^i \cap V_{h^i(t)}$. Let $a = |\{t \le \tspawn \mid \head^i\textnormal{ spawns a mutant at }t\}|$. 

We first bound $a$ below. Recall that $h^i(\Tstart(S^i)) = k-1$, and note that $h^i(\tau_{j-1}) = 1$ since $S^i$ spawns a mutant onto $v^*$ at time $\tau_j$. Moreover, every time $h^i$ decreases, it only decreases by 1 and $\head^i$ spawns a mutant. Also, $h^i$ increases (by at least 1) at time $\tau_x$ whenever $\head^i$ is hit at $\tau_x$ and $|\head^i_{\tau_{x-1}}| = 1$. Finally, note that $\head^i$ spawns a mutant at time $\tau_j$. Thus
\begin{align}
a &\ge (k-2) + |\{x < j \mid \head^i\textnormal{ is hit at }\tau_x\textnormal{ and }|\head^i_{\tau_{x-1}}| = 1\}|+1\nonumber\\
&= (k-1) + |\{x \le j \mid \head^i\textnormal{ is hit at }\tau_x\textnormal{ and }|\head^i_{\tau_{x-1}}| = 1\}|,\label{eqn:head-bounds-1}
\end{align}
where the equality follows since no two clocks trigger at the same time. By (Z\ref{Z:one}) this implies part (i) of the result. (In fact it is substantially stronger, and we will use this extra strength later in the proof.)

We now bound $a$ above. We say that a layer $V_y$ is \emph{empty} at time $t$ if $S^i_t \cap V_y = \emptyset$, and \emph{non-empty} otherwise. Since $\tspawn$ is the first time that $\head^i$ spawns a mutant onto $v^*$, every time $\head^i$ spawns a mutant in $[0,\tspawn)$, a layer must become non-empty in $S^i$. Since $\head^i$ spawns no mutants in $[0, \Tstart(S^i)]$, it follows that
\begin{align*}
a &= |\{t \in (\Tstart(S^i), \tspawn] \mid \head^i\textnormal{ spawns a mutant at }t\}|\\
&\le |\{t \in (\Tstart(S^i), \tspawn) \mid \textnormal{for some $y\in[k-1]$, } V_y \textnormal{ becomes non-empty at }t\}| + 1.
\end{align*}
Since $V_{k-1}$ becomes non-empty at $\Tstart(S^i)$, it follows that
\[
a \le |\{t < \tspawn \mid \textnormal{for some $y\in[k-1]$, } V_y \textnormal{ becomes non-empty at }t\}|.
\]
Every time a layer becomes non-empty in $[0,\tspawn]$, either it subsequently becomes empty again in $[0,\tspawn]$ or it contains at least one mutant at time $\tspawn$. Thus
\[
a \le |\{t \le \tspawn \mid \textnormal{for some $y\in[k-1]$, } V_y \textnormal{ becomes empty at }t\}| + |S^i_{\tspawn}|.
\]
If a layer $V_y$ becomes empty at time $\tau_x$, then $S^i$ is hit at time $\tau_x$ and $|S^i_{\tau_{x-1}} \cap V_y| = 1$. Thus
\[
a \le |\{t \le \tspawn \mid S^i \setminus \head^i \textnormal{ is hit at time }t\}| + |\{x \le j \mid \head^i\textnormal{ is hit at }\tau_x\textnormal{ and }|\head^i_{\tau_{x-1}}| = 1\}| + |S^i_{\tspawn}|.
\]
Every time a vertex becomes a mutant in $S^i$ during $[0,\tspawn]$, that mutant must either die in $[0,\tspawn]$ (at which point $S^i$ is hit) or still be alive at $\tspawn$. Thus
\begin{align*}
a &\le |\{t \le \tspawn \mid \textnormal{a vertex becomes a mutant in $S^i$ at }t\}| - |\{x \le j \mid \head^i\textnormal{ is hit at }\tau_x\textnormal{ and }|\head^i_{\tau_{x-1}}| > 1\}|.
\end{align*}
Since the only time a vertex becomes a mutant in $S^i$ without $S^i$ spawning a mutant is $\Tstart(S^i)$, and $S^i$ spawns a mutant onto $v^*$ at $\tspawn$, it follows that 
\begin{align*}
a \le |\{t \le \tspawn \mid S^i\textnormal{ spawns a mutant at }t\}| - |\{x \le j \mid \head^i\textnormal{ is hit at }\tau_x\textnormal{ and }|\head^i_{\tau_{x-1}}| > 1\}|.
\end{align*}

It now follows from \eqref{eqn:head-bounds-1} that
\[|\{t \le \tspawn \mid S^i \textnormal{ spawns a mutant at }t \}| \ge (k-1) + |\{t \le \tspawn \mid \head^i\textnormal{ is hit at time }t\}|.\]
Part (ii) of the result follows immediately from (Z\ref{Z:two}).
\end{proof}

We next prove that $Y^{\head^i}$ is unlikely to increase $k-1$ times before decreasing $C_r+2$ times. This, combined with Lemma~\ref{lem:funnel-H-useful}, will allow us to show that $S^i$ is dangerous with probability at most roughly $(r/m)^{k-1}$.

\begin{lemma}\label{lem:funnel-strain-bottom} There exists $n_0>0$, depending on $r$, such that the following holds. Suppose $n \ge n_0$,
$m \geq m^* (\log n)^3$ and $2\le k \le \sqrt{\log_r n}$. 
Fix $x_0 \in V_k$. 
Let $X$ be  the Moran process    with $G(X)=\metaf$ and $X_0 = \{x_0\}$.
Let $\varphi$ be a possible value of $\Phi$ and let $i\in\Zone$. Let $\mathcal{E}_i$ be the event in $\Psi(X,\head^i)$ that $Y^{\head^i}$ increases $k-1$ times before decreasing $C_r+2$ times. Then
\[\Pr(\mathcal{E}_i \mid \Phi=\varphi) \le (\log n)^{C_r+3}\left(\frac{r}{m}\right)^{k-1}.\]
\end{lemma}

\begin{proof}
Let $n_0$ be a large integer relative to $r$. 
Consider any positive integer~$t_0$ and  filtration value $f_0$ such that the events
$\Tstart(\head^i)=t_0$, $\filt_{t_0}(\Psi(X,\head^i))=f_0$ and $\Phi=\varphi$
are consistent.
Recall from Lemma~\ref{lem:funnel-firing-jump-chain} that, at each step and conditioned on these events, 
$\widehat{Y}^{\head^i}$ increases with probability $r/(r+m')$ and decreases with probability $m'/(r+m')$.

For  
$0 \le K \le C_r+1$, 
let $\mathcal{E}_{i,K}$ be the event that $Y^{\head^i}$ increases precisely $k-1$ times within its first $k-1+K$ transitions. Thus $\mathcal{E}_i = \bigcup_{i=0}^{C_r+1}\mathcal{E}_{i,K}$.

The number of backward steps among the first $k-1+K$ transitions of $Y^{\head^i}$ follows the binomial distribution 
consisting of
$k-1+K$ Bernoulli trials, each with success probability $m'/(r+m')$, and $\mathcal{E}_{i,K}$ holds if and only if this quantity is equal to $K$. Hence
\begin{align*}
\Pr(\mathcal{E}_{i,K} &\mid 
\Tstart(\head^i)=t_0, \filt_{t_0}(\Psi(X,\head^i))=f_0, \Phi=\varphi
 ) \\
 & = \binom{k-1+K}{K} \frac{(m')^{K}r^{k-1}}{(m'+r)^{k-1+K}} 
        \leq (k-1+K)^K \frac{m^{K}r^{k-1}}{(m')^{k-1+K}} \\
 & \le (\log n)^{C_r+1} \cdot \frac{m^{K}r^{k-1}}{(m')^{k-1+K}} = (\log n)^{C_r+1} \left(\frac{r}{m}\right)^{k-1} \left(\frac{m}{m'}\right)^{k-1+K},
\end{align*}
where the final inequality holds since $k \le \sqrt{\log_r n}$ and $K \le C_r+1$. Moreover,
\begin{align*}
\left(\frac{m}{m'}\right)^{k-1+K} 
  =   \left(1 + \frac{m^*}{m'} \right)^{k-1+K}
  \le \left(1 + \frac{2}{(\log n)^2} \right)^{k-1+K}
  \le e^{2(k+K)/(\log n)^2}
  \le e.
\end{align*}
It therefore follows that 
\[
\Pr\big(\mathcal{E}_{i,K} \,\big|\, \Tstart(\head^i)=t_0, \filt_{t_0}(\Psi(X,\head^i))=f_0, \Phi=\varphi\big)
\le (\log n)^{C_r+2} \left(\frac{r}{m} \right)^{k-1}.
\]
It now follows by a union bound over $K$ that
\begin{align*}
\Pr\big(\mathcal{E}_i \,\big|\,
\Tstart(\head^i)=t_0, \filt_{t_0}(\Psi(X,\head^i))=f_0, \Phi=\varphi\big)
	& \le (C_r+2) (\log n)^{C_r+2} \left(\frac{r}{m}\right)^{k-1}\\
	& \le (\log n)^{C_r+3} \left(\frac{r}{m}\right)^{k-1}.\qedhere
\end{align*}
\end{proof}

We are now finally in a position to deal with~$\Pstraindanger$. 
\begin{lemma}\label{lem:funnel-few-dangerous-strains}
There exists $n_0>0$, depending on $r$, such that the following holds. Suppose $n \ge n_0$, $m \ge r^{\sqrt{\log_r n}}$ and $2 \le k \le \sqrt{\log_r n}$. 
Fix $x_0 \in V_k$. 
Let $X$ be  the Moran process    with $G(X)=\metaf$ and $X_0 = \{x_0\}$.
Then, in the process $\Psi(X)$, 
$\Pr(\Pstraindanger \mid \Pmutclock\cap\Pnmutclock) \ge 1/2$.
\end{lemma}
\begin{proof}
Let $n_0$ be a large integer relative to $r$. 
Let $\varphi$ be a possible value for $\Phi$ so that the event
$\Phi=\varphi$ is consistent with $\Pmutclock\cap\Pnmutclock$.
Note that $\Pmutclock$ and~$\Pnmutclock$ are determined by $\Phi$.

For each $i \in [s]$, 
consider the process $\Psi(X,\head^i)$ on $\metaf$.
Let $\calE_i$ be the event that $S^i$ spawns at most $k+C_r$ mutants  in $(0,\Tpabs]$
and let $\calE_i'$ be the event that $Y^{\head^i}$ increases fewer than $k-1$ times before decreasing $C_r+2$ times. We first claim that $\calE_i \cap \calE_i'$ implies that $S^i$ is not dangerous. Indeed, suppose $\calE_i'$ holds and $S^i$ is dangerous, spawning a mutant onto $v^*$ for the first time at some time 
$\tspawn \le \Tpabs$.  By Lemma~\ref{lem:funnel-H-useful}(i), $Y^{\head^i}$ must increase at least $k-1$ times 
in $[0,\tspawn]$, and so since $\calE_i'$ holds $Y^{\head^i}$ must decrease at least $C_r+2$ times in $[0,\tspawn]$. By Lemma~\ref{lem:funnel-H-useful}(ii), it follows that $S^i$ must spawn at least $k+C_r+1$ mutants in $[0,\tspawn]$. Thus $\calE_i$ cannot hold, and so $\calE_i \cap \calE_i'$ implies that $S^i$ is not dangerous as claimed.

It therefore suffices to prove,  conditioned  on  $\Phi=\varphi$, that with probability at least $1/2$,
$\calE_i \cap \calE_i'$ holds for 
all but $b^*/\log n$ of the
 $i \in [s]$.

Fix $i \in [s]$. Then by Lemma~\ref{lem:funnel-gamblers-ruin}, 
\[\Pr(\overline{\calE_i} \mid \Phi=\varphi) \leq
 \left(\frac{20r}{m}\right)^{k+C_r+1}
 \le \left(\frac{20r}{m}\right)^{k+C_r} = \frac{20^{k+C_r}r^{C_r}}{m^{C_r}} \left(\frac{r}{m}\right)^k,\]
where
\[\frac{20^{k+C_r} r^{C_r}}{m^{C_r}} \le \frac{20^{\sqrt{\log_r n} + C_r}r^{C_r}}{r^{C_r\sqrt{\log_r n}}} 
= \frac{r^{(\sqrt{\log_r n}+C_r)(\log_r 20)+C_r}}{r^{C_r\sqrt{\log_r n}}} \le \frac{r^{2(\log_r 20)\sqrt{\log_r n}}}{r^{C_r\sqrt{\log_r n}}} \le 1.\]
(For the final inequality, we use the fact that $C_r = \lceil 2\log_r 20 \rceil$.) Moreover, by Lemma~\ref{lem:funnel-strain-bottom} we have
\[\Pr(\overline{\calE_i'} \mid \Phi=\varphi) \le \left(\frac{r}{m}\right)^{k-1}(\log n)^{C_r+3}.\]
From a union bound and the fact that $\Phi$ determines $\Pmutclock$ and~$\Pnmutclock$, it follows that
\[\Pr(S^i\textnormal{ is dangerous} \mid \Pmutclock\cap\Pnmutclock) \le \Pr\big(\overline{\calE_i}\cup \overline{\calE_i'} \,\big|\, \Pmutclock\cap\Pnmutclock\big) \le 2 \left(\frac{r}{m}\right)^{k-1}(\log n)^{C_r+3}.\]

We now simply apply Markov's inequality. By linearity of expectation, 
\[\E\big[|\{i \in [s] \mid S^i\textnormal{ is dangerous}\}| \,\big|\, \Pmutclock\cap\Pnmutclock\big] \le 2s \left(\frac{r}{m}\right)^{k-1}(\log n)^{C_r+3}.\]
Hence by Markov's inequality, with probability at least $1/2$, at most $4s(r/m)^{k-1}(\log n)^{C_r+3}$ strains are dangerous. Since $\tx = \ell m^k/(r^k(\log n)^{C_r+5})$, $s = \lceil 3r\tx\rceil 
\leq 4r\tx$, and $b^* = \lfloor \ell m/2\rfloor$,
we have 
\[4s\left(\frac{r}{m}\right)^{k-1}(\log n)^{C_r+3} \le \frac{16r^k (\log n)^{C_r+3}}{m^{k-1}}\tx = \frac{16\ell m}{(\log n)^2} \le \frac{b^*}{\log n},\]
so this implies that $\Pstraindanger$ occurs and the result follows.
\end{proof}

Lemma~\ref{lem:funnel-masterlist} now follows from everything we've done so far.  
 {\renewcommand{\thetheorem}{\ref{lem:funnel-masterlist}}
\begin{lemma}
\statemainlemma
\end{lemma}
\addtocounter{theorem}{-1}
} 
\begin{proof} 
Let $\mathcal{P} = \mathcal{P}_1 \cap \dots \cap \mathcal{P}_8$.
\begin{align*}
\Pr(\mathcal{P}) 
	& \ge \Pr(\Pmutclock \cap \Pnmutclock \cap \Pstraindanger) - \Pr(\overline{\Pcentremut}) - \Pr(\overline{\Pfastabsorb}) - \Pr(\overline{\Pvinitspawn}) - \Pr(\overline{\Pstrainfire}) - \Pr(\overline{\Pstrainlife})\\
	& = \Pr(\Pmutclock \cap \Pnmutclock)\,\Pr(\Pstraindanger \mid \Pmutclock \cap \Pnmutclock) - \Pr(\overline{\Pcentremut}) - \Pr(\overline{\Pfastabsorb}) - \Pr(\overline{\Pvinitspawn}) - \Pr(\overline{\Pstrainfire}) - \Pr(\overline{\Pstrainlife}).
\end{align*}
We bound each term on the right-hand side of the above by applying (in order) 
 Lemmas~\ref{prop:funnel-clocks-behave}, 
 \ref{lem:funnel-few-dangerous-strains},
 \ref{lem:funnel-many-infections-needed}, \ref{lem:funnel-eventually-absorb} and \ref{lem:funnel-few-strains-born}, Corollary~\ref{cor:funnel-dangerous-strains-not-bad}, and Lemma~\ref{lem:funnel-strains-die-fast}. This yields
\[
\Pr(\mathcal{P}) 
	\ge \frac{1}{r^k(\log n)^{C_r+6}}\cdot \frac{1}{2} - \frac{6}{n} \ge \frac{1}{r^k(\log n)^{C_r+7}},
\]
as required. The final inequality follows since $r^k \le r^{\sqrt{\log_r n}} \le \sqrt{n}$.
\end{proof}

\subsubsection{Applying Lemma~\ref{lem:funnel-masterlist}}

We now prove that $\Pmutclock\cap\dots\cap \Pstraindanger$ implies extinction for 
the Moran process~$X$, which together with Lemma~\ref{lem:funnel-masterlist} implies our lower bound on extinction probability.

\begin{lemma}\label{lem:funnel-die-slow}
There exists $n_0>0$, depending on $r$, such that the following holds. Suppose $n \ge n_0$, $m \ge r^{\sqrt{\log_r n}}$ and $2 \le k \le \sqrt{\log_r n}$. 
Fix $x_0 \in V_k$. 
Let $X$ be  the Moran process    with $G(X)=\metaf$ and $X_0 = \{x_0\}$.
 Then
\[\Pr(
\textnormal{$X$ goes extinct} 
) \ge \frac{1}{r^{k}(\log n)^{C_r+7}}.\]
\end{lemma}
\begin{proof}
Let $n_0$ be a large integer relative to $r$. By Lemma~\ref{lem:funnel-masterlist}, it suffices to prove that $\Pmutclock\cap\dots\cap \Pstraindanger$ implies extinction. We first show that we may restrict our attention to $S^1, \dots, S^s$. Note that by Observation~\ref{lem:clock-coupling}, $\Pmutclock\cap\Pcentremut$ implies that $v^*$ does not spawn a mutant in $X$ in $[0,\Tpabs]$. Thus the definition of the strains ensures that for all $t\leq \Tpabs$, 
\begin{equation}\label{Q:one}
X_t \setminus \{x_0, v^*\} \subseteq \bigcup_{i=1}^\infty S^i_t.
\end{equation}
Let $s' = |\{i \mid \Tbirth{i} \le \min\{\tx,\Tpabs\}\}|$.  Again by Observation~\ref{lem:clock-coupling}, $\Pnmutclock\cap\Pcentremut$ implies that,
in the interval $[0,\tx]$, either $x_0$ dies in~$X$ or $\Tpabs$ occurs or both.
Thus no strains are born in $(\tx,\Tpabs]$, so by~\eqref{Q:one}, it follows that 
\begin{equation}\label{Q:three}
X_t \setminus \{x_0, v^*\} \subseteq \bigcup_{i=1}^{s'} S^i_t  \textnormal{ for all } t \le \Tpabs.
\end{equation}
Moreover, by~$\Pvinitspawn$ and the definition of each interval $I_j$
in Definition~\ref{def:constants}, $x_0$ spawns at most $2r(\log n)^2 \cdot \lceil \tx/(\log n)^2 \rceil \le 3r\tx \le s$ mutants in $[0,\tx]$ in $X$ and so $s' \le s$.

We now show that $|X_{\Tpabs}| < m^*$, and so (A\ref{A:fill}) does not hold with $t=\Tpabs$. By~\eqref{Q:three}, each mutant in $X_{\Tpabs} \setminus \{x_0,v^*\}$ belongs to $S^i_t$ for some $i \in [s']$. By~$\Pstrainfire$, each such $S^i$ contains at most $\log n+1$ mutants in total. By $\Pstrainlife$ and the definition of $s'$, each such $S^i$ was born in the interval $[\Tpabs-(\log n)^2, \Tpabs] \cap [0,\tx]$. 
This interval spans at most two $I_j$'s with $j \le \lceil \tx/(\log n)^2\rceil$, so by~$\Pvinitspawn$ there are at most $4r(\log n)^2$ such $S^i$'s. Hence
\begin{equation*}
|X_{\Tpabs}| \le 
(\log n+1) \cdot 4r(\log n)^2 + 2 < 
m^*,
\end{equation*}
and so (A\ref{A:fill}) does not hold with $t=\Tpabs$.

By $\Pstraindanger$, $\Pstrainfire$ and~\eqref{Q:three},
$v^*$ becomes a mutant in~$X$ at most $(b^*/\log n) \log n = b^*$ times in $[0,\Tpabs]$.
Thus (A\ref{A:spread}) does not hold with $t=\Tpabs$.

Finally, since $\Pfastabsorb$ occurs, $V_{\tmax/2} \in \{\emptyset, V(\metaf)\}$. In either case, $\Tpabs < \tmax$ by (A\ref{A:extinct}) and (A\ref{A:fill}) and so (A\ref{A:timeout}) does not hold 
with~$t=\Tpabs$.
Since none of (A\ref{A:fill})--(A\ref{A:timeout}) holds at $\Tpabs$, (A\ref{A:extinct}) must hold at $\Tpabs$ by definition, and so 
the Moran process~$X$
goes extinct as required. 
\end{proof}

\subsection{Proof of the main theorem (Theorem~\ref{thm:funnel-upper})}

Theorem~\ref{thm:funnel-upper} now follows easily from Lemmas~\ref{lem:funnel-die-instant} and~\ref{lem:funnel-die-slow}.
  {\renewcommand{\thetheorem}{\ref{thm:funnel-upper}}
  \begin{theorem}
  \statemainthm
\end{theorem}
\addtocounter{theorem}{-1}
}  

\begin{proof} 
Let $n_0$ be the maximum of the value~$n_0$ in Lemma~\ref{lem:funnel-die-slow}
and $e^{2(r+1)}$. Recall that $C_r = \lceil 2\log_r 20 \rceil $
and 
take $c_r \ge  C_r + 8$ large enough that the result holds whenever $n < n_0$
and whenever $k=1$. 
(Note that if $k=1$ then $\metaf$ is a star 
so as $n$ tends to infinity, the extinction probability tends to $1/r^{2}$
\cite{LHN2005:EvoDyn, BR2008:FixProb}.)

We  now consider the case where $n \ge n_0$ and $k\geq 2$. 
Consider the coupled process $\Psi(X)$, with $x_0$ taken uniformly at random from $V(\metaf)$.

If $m \le e^{\sqrt{\log r \cdot \log n}}$ then, by Lemma~\ref{lem:funnel-die-instant}, we have
\[\Pr(X \textnormal{ goes extinct} ) \ge \frac{1}{2(m+r)} \ge 
\frac{1}{2(r+1) m}
\geq
\frac{1}{2(r+1)}e^{-\sqrt{\log r \cdot \log n}}\geq \frac{1}{\log n}e^{-\sqrt{\log r \cdot \log n}},\]
and the result follows. Suppose instead $m \ge e^{\sqrt{\log r \cdot \log n}} = r^{\sqrt{\log_r n}}$. Then we have $n \ge m^k$, so $k \le \sqrt{\log_r n}$. By  Lemma~\ref{prop:funnel-vinit}, we have $\Pr(x_0 \in V_k) \ge 1/2$. It follows from Lemma~\ref{lem:funnel-die-slow} that
\[\Pr( 
X \textnormal{ goes extinct} ) \ge \frac{1}{2} r^{-k} (\log n)^{-(C_r+7)} \ge e^{-\sqrt{\log r \cdot \log n}}(\log n)^{-c_r},\]
and again the result follows.
\end{proof}

\section{A lower bound on the fixation probability of $\Upsilon_{\mathcal{M}}$}
\label{sec:megastar}  
 
The definition of the $(k,\ell,m)$-megastar~$\megastar$ is given in Section~\ref{def:mega}.
Note that each of the cliques $K_1,\ldots, K_\ell$ is the vertex set of a
 complete graph on $k$ vertices
(contrary to the fact that some authors use the notation 
$K_i$ to denote a complete graph on $i$ vertices).
An infinite family of megastars is identified in Definition~\ref{def:newmegas}.
Recall that 
\[\Upsilon_{\mathcal{M}} = \{ \megastar[k(\ell),\ell,m(\ell)] \mid \ell \in \Zone\},\]
where $m(\ell) = \ell$ and $k(\ell) = \ceil{(\log \ell)^{23}}$.

 For convenience, we drop the argument $\ell$ in the functions $m(\ell)$ and
$k(\ell)$ and simply write $m$ and $k$.
Also, we use $\megastar[\ell]$ to denote the megastar $ \megastar[k(\ell),\ell,m(\ell)] $. 
We use $n=1+\ell(m+1+k)$ to denote the number of vertices of $\megastar[\ell]$.
Note   that $\sqrt{n}/2 \le \ell, m \le \sqrt{n}$ when $\ell$ is sufficiently large. Our main theorem is the following.

\newcommand{\statemegastar}{
\statelargeell Consider the Moran process~$X$
 with $G(X) = \megastar[\ell]$
 where 
 the initial mutant $x_0$ is chosen uniformly at random from $V(\megastar[\ell])$.
The fixation probability of~$X$ is at least $1-(\log n)^{23}/n^{1/2}$.
}
 
\begin{theorem}\label{thm:megastar}
\statemegastar
\end{theorem}

\subsection{Glossary}

\begin{longtable}{p{\linewidth-\widthof{Definition 000, Page 00}}@{}l@{ }l@{}}
$\alpha_i= \floor{\max\{(2\log n)^{2}, m/(2\log n)^{2i}\}}$ \dotfill & Definition~\ref{def:iteration-scheme},& 
Page~\pageref{def:iteration-scheme}\\
$\beta_i= \floor{\max\{(2\log n)^{2}, \ell m/(2\log n)^{2i}\}}$ \dotfill & Definition~\ref{def:iteration-scheme},& 
Page~\pageref{def:iteration-scheme}\\
$c_r = \lceil 2r/(r-1)\rceil$ \dotfill & Definition~\ref{def:Z},& Page~\pageref{def:Z}\\
\myref{D}{D:end}, \myref{D}{D:beta}, \myref{D}{D:alpha}, \myref{D}{D:clique-empty} \dotfill & 
Definition~\ref{def:tdoom},& Page~\pageref{def:tdoom}\\  
$\Delta = \floor{c_r(\log n)^3}$ \dotfill &  Definition~\ref{def:Delta},& Page~\pageref{def:Delta}\\
$\gamma_i= \beta_i (\log n)^{20}/k$ \dotfill & Definition~\ref{def:iteration-sch},& 
Page~\pageref{def:iteration-sch}\\
good (filtration) \dotfill & Definition~\ref{def:iteration-scheme},& Page~\pageref{def:iteration-scheme}\\
$H_{x;S}$ \dotfill & Lemma~\ref{lem:backtoback},& Page~\pageref{lem:backtoback}\\
$\Iabs= \min \{h \in \mathbb{Z}_{\geq 0} \mid K_j\text{ is inactive at } \Ti{h}\}$ \dotfill & 
Definition~\ref{def:Iabs},& Page~\pageref{def:Iabs}\\
$I_i= (I_i^-,I_i^+]$ \dotfill & Definition~\ref{def:wijh},& Page~\pageref{def:wijh}\\ 
$I_i^-= n^8 + in(\log n)^3$ \dotfill & Definition~\ref{def:iteration-scheme},& 
Page~\pageref{def:iteration-scheme}\\ 
$I_i^+=n^8 + (i+1)n(\log n)^3$ \dotfill & Definition~\ref{def:iteration-scheme},& Page~\pageref{def:iteration-scheme}\\
$\iin{\mu}{u}{t}$ \dotfill & Definition~\ref{def:localtime},& Page~\pageref{def:localtime}\\
$\iout{\mu}{u}{t}$ \dotfill & Definition~\ref{def:localtime},& Page~\pageref{def:localtime}\\
$\J{i,j}{h}$ \dotfill & Definition~\ref{def:wijh},& Page~\pageref{def:wijh}\\
jump \dotfill & Definition~\ref{def:clique-chain},& Page~\pageref{def:clique-chain}\\
$k=k(\ell) = \ceil{(\log \ell)^{23}}$ \dotfill &  Definition~\ref{def:newmegas},& Page~\pageref{def:newmegas}\\
megastar process \dotfill & Definition~\ref{def:megastar-process},& Page~\pageref{def:megastar-process}\\$\megastar$\dotfill & Section~\ref{def:mega},& Page~\pageref{def:mega}\\
$\megastar[\ell]=\megastar[k(\ell),\ell,m(\ell)]$\dotfill & Section~\ref{sec:megastar},& Page~\pageref{sec:megastar}\\
$m=m(\ell) = \ell$ \dotfill & Definition~\ref{def:newmegas},& Page~\pageref{def:newmegas}\\
\Ptotalmut{i}, \Presmut{i}, \Pcliquefull{i}, \Pmostcliques{i} \dotfill &  
Definition~\ref{def:iteration-scheme},& Page~\pageref{def:iteration-scheme}\\
$p_{i,j}$ \dotfill & Definition~\ref{def:Z},& Page~\pageref{def:Z}\\
$p'_{i,j}$ \dotfill & Lemma~\ref{lem:transmatrix},& Page~\pageref{lem:transmatrix}\\
$p_{x\rightarrow y; z}$ \dotfill & Lemma~\ref{lem:backtoback},& Page~\pageref{lem:backtoback}\\
$\Psi(X')$ \dotfill & Section~\ref{sec:coupled},& Page~\pageref{sec:coupled}\\
$r' = (r+1)/2$ \dotfill & Definition~\ref{def:Z},& Page~\pageref{def:Z}\\
$\mathcal{T}^i = \mathcal{T}^{i,1} \cup \dots \cup \mathcal{T}^{i,\ell}$ \dotfill & Definition~\ref{def:auxdefs},& 
Page~\pageref{def:auxdefs}\\
$\mathcal{T}^{i,j}= \{t \in (t_i^-, t_i^- + \gamma_i] \mid \textnormal{for some } v \in R_{j},\, \vsnc{v^*}{v} 
\textnormal{ triggers at }t\}$ \dotfill & Definition~\ref{def:auxdefs},& Page~\pageref{def:auxdefs}\\
$\Ti{h}$ \dotfill & Definition~\ref{def:cliquetimes}, & Page~\pageref{def:cliquetimes}\\
$\TTend{i}$ \dotfill & Definition~\ref{def:tdoom},& Page~\pageref{def:tdoom}\\
$t^-_i= \siout{v^*}{I_i^-}$ \dotfill & Definition~\ref{def:auxdefs},& Page~\pageref{def:auxdefs}\\
$t^+_i= \siin{v^*}{I_i^-}$ \dotfill & Definition~\ref{def:auxdefs},& Page~\pageref{def:auxdefs}\\
$\Tabs=\Ti{\Iabs}$ \dotfill & Definition~\ref{def:Tabs},& Page~\pageref{def:Tabs}\\
$\Tmut{v}{i}{h}$ \dotfill & Definition~\ref{defn:tmut},& Page~\pageref{defn:tmut}\\
$\Tnmut{v}{i}{h}$ \dotfill & Definition~\ref{defn:tmut},& Page~\pageref{defn:tmut}\\
$\tau_i$ \dotfill & Section~\ref{sec:clock}, & Page~\pageref{sec:clock}\\
$\mathcal{U}_i= \{v \in R_1 \cup \dots \cup R_{\ell} \mid \vsnc{v^*}{v}\textnormal{ triggers in }(t^-_i, 
t^-_i+\gamma_i]\}$ \dotfill & Definition~\ref{def:auxdefs},& Page~\pageref{def:auxdefs}\\	
$\Upsilon_{\mathcal{M}}$ \dotfill & Definition~\ref{def:newmegas},& Page~\pageref{def:newmegas}\\
$\Wnoh{i,j}$ \dotfill & Definition~\ref{def:newwijh},& Page~\pageref{def:newwijh}\\
$\W{i,j}{h}$, $\Win{i,j}{h}$ \dotfill & Definition~\ref{def:wijh},& Page~\pageref{def:wijh}\\
$X'_t$ \dotfill & Definition~\ref{def:megastar-process},& Page~\pageref{def:megastar-process}\\
$\Y(h)=|X'_{\Ti{h}} \cap K_j|$ \dotfill & Definition~\ref{def:clique-chain},& Page~\pageref{def:clique-chain}\\
$Z$ \dotfill & Definition~\ref{def:Z},& Page~\pageref{def:Z}\\
\end{longtable}

\subsection{The megastar process}

In  working with the megastar, it will be   helpful to isolate the evolution of the Moran process inside a clique $K_j$ from the state of the process in the rest of the graph.  To this end we  define a new mutant process~$X'$, which has  $G(X')=\megastar[\ell]$.
It is
defined in the same way as the Moran process $X$, except that each feeder vertex is forced to be a non-mutant while its corresponding clique contains both mutants and non-mutants.

\begin{definition}\label{def:megastar-process}
The \emph{megastar process} on 
$\megastar[\ell]$ with \emph{initial mutant} $x_0 \in V(\megastar[\ell])$ 
is a mutant process $X'$ with $G(X') = \megastar[\ell]$ and 
$X_0' = \{x_0\}$ defined as follows. 
Recall that $\tau_0=0$ and, for every $i\in\Zone$, a clock $C\in \moranclocks(\megastar[\ell])$ triggers at $\tau_i$.
For $t\in (\tau_{i-1},\tau_i)$, we set $X'_t = X'_{\tau_{i-1}}$. Then we define $X'_{\tau_i}$ as follows.
\begin{enumerate}[(i)]
\item If $C = \vmc{u}{v}$ for some $(u,v) \in E(\megastar[\ell])$ such that $u \in X'_{\tau_{i-1}}$ and $u,v \notin \{a_1, \dots, a_\ell\}$, then $X_{\tau_{i}}' = X'_{\tau_{i-1}} \cup \{v\}$.
\item If $C = \vmc{a_j}{v}$ for some $j \in [\ell]$ and $v \in K_j$ such that $a_j \in X'_{\tau_{i-1}}$, and $K_j \cap X'_{\tau_{i-1}} = \emptyset$, then $X'_{\tau_{i}} = (X'_{\tau_{i-1}} \cup \{v\}) \setminus \{a_j\}$.
\item If $C = \vmc{u}{a_j}$ for some $j \in [\ell]$ and $u \in R_{j}$ such that $u \in X'_{\tau_{i-1}}$, and $K_j \cap X'_{\tau_{i-1}} \in \{\emptyset, K_j\}$, then $X'_{\tau_{i}} = X'_{\tau_{i-1}} \cup \{a_j\}$.
\item If $C = \vnc{u}{v}$ for some $(u,v) \in E(\megastar[\ell])$ such that $u \notin X'_{\tau_{i-1}}$, then $X'_{\tau_{i}} = X'_{\tau_{i-1}} \setminus \{v\}$.
\item Otherwise, $X'_{\tau_{i}} = X'_{\tau_{i-1}}$.
\end{enumerate}
For $j \in [\ell]$, we say $K_j$ is \emph{active at time $t$} if $$\emptyset \subset X'_t \cap K_j \subset K_j,$$ and \emph{inactive at time $t$} otherwise.\defend{}
\end{definition}

$X'$ is the only mutant process that will be considered in Section~\ref{sec:megastar}.
Note that while a clique $K_j$ is active, 
$a_j$ is a non-mutant and the evolution of mutants in $K_j$
depends only on clocks with sources in $K_j \cup \{a_j\}$ and not on the state of the rest of $\megastar[\ell]$.

\subsection{Proof of main theorem (Theorem~\ref{thm:megastar}) assuming key lemma (Lemma~\ref{lem:megastar-fixates})}\label{sec:megastar-partial}

The key ingredient in the proof of our lower bound (Theorem~\ref{thm:megastar}) is the following lemma.
\newcommand{\statemegastarfixates}{
\statelargeell Suppose that  $x_0\in R_1\cup \dots \cup R_{\ell}$. Then there exists 
a $t\geq 0$ such that $\Pr\big(X_{t}'=V(\megastar[\ell])\big)\geq 1-42(\log n)^{2}/n^{1/2}$.
}
\begin{lemma}\label{lem:megastar-fixates}
\statemegastarfixates
\end{lemma}

The precise value of~$t$ is not important, but in fact, our proof will work
with $t=2n^{8}$. Most of Section~\ref{sec:megastar} will be devoted to the proof of Lemma~\ref{lem:megastar-fixates}.
Before giving the proof, we show how to use Lemma~\ref{lem:megastar-fixates}
to prove Theorem~\ref{thm:megastar}, which we restate here for convenience.

{\renewcommand{\thetheorem}{\ref{thm:megastar}}
\begin{theorem}
\statemegastar
\end{theorem}
\addtocounter{theorem}{-1}
}
\begin{proof}
Let $X'$ be the megastar process on $\megastar[\ell]$ with $X_0'=X_0$. Recall that both $X$ and $X'$ are defined in terms of the same clock process $\moranclocks(\megastar[\ell])$. It is therefore immediate that for all $t \ge 0$, $X'_t \subseteq X_t$.  Thus, if $X'$ fixates at time $t$, then $X_t = X'_t = V(\megastar[\ell])$ and so $X$ must also fixate at or before time $t$. Thus, $\Pr(X\textnormal{ fixates}) \ge \Pr(X'\textnormal{ fixates})$.

Let $\calR$ be the event that the initial mutant $x_0$ is in a reservoir. Clearly,
\begin{equation*}
\Pr(\calR)= \frac{\ell m}{\ell(k +m+1)+1}> \frac{m}{k  + m + 2}> 
1 - \frac{k+2}{m}\geq 1-\frac{(\log \ell)^{23}+3}{m}\,.
\end{equation*}
By Lemma~\ref{lem:megastar-fixates}, we have 
\begin{equation*}
\Pr(X'\textnormal{ fixates}\mid \calR)\geq 1-\frac{42(\log n)^{2}}{\sqrt{n}}\,.
\end{equation*}
Therefore,
\begin{align*}
\Pr(X'\textnormal{ fixates})
    &\geq \Pr(\calR)\,\Pr(X'\textnormal{ fixates}\mid \calR) \\
    &\geq 1 - \frac{(\log\ell)^{23} + 3}{m} - \frac{42(\log n)^{2}}{\sqrt{n}} \\
    &\geq 1 - \frac{2^{-23}(\log n)^{23} + 3}{\sqrt{n}/2} - \frac{42(\log n)^{2}}{\sqrt{n}}\,,
\end{align*}
and the result follows.
\end{proof}

The rest of  Section~\ref{sec:megastar} is
devoted to  the proof of Lemma~\ref{lem:megastar-fixates}. 

\subsection{Sketch of the proof of the key lemma (Lemma~\ref{lem:megastar-fixates})}

In this Section, we give an informal sketch of the proof of Lemma~\ref{lem:megastar-fixates}. The presentation of the proof itself does not depend upon the sketch so the reader may prefer to skip directly to the proof. Throughout, we assume that $n$ is ``large'' relative to $r$, leaving the details of how large to the actual proof.

At a very high level, the argument proceeds as follows. We set out some preliminary results concerning cliques in Section~\ref{sec:cliques}. With our choice of parameters, $x_0$ is very likely to spawn inside a reservoir, say $R_1$. Let $\Delta = \Theta((\log n)^3)$ (see the Glossary for the precise definition), and note that $\Delta$ is much smaller than $k$. In Section~\ref{sec:first-clique}, we prove that $K_1$ is likely to fill with mutants before $x_0$ dies, and likely to contain at most $\Delta$ non-mutants at time $n$. In Section~\ref{sec:other-cliques}, we prove that $K_1, \dots, K_\ell$ are all likely to contain at most $\Delta$ non-mutants at time $n^8$. Finally, in Section~\ref{sec:megastar-fixates} we prove that the process is likely to fixate by time $2n^8$.

We now discuss each part of the argument in more detail. We say that a clique is \emph{active} if it contains both mutants and non-mutants. The key idea of Section~\ref{sec:cliques} is that since we are working with the megastar process rather than the Moran process, the behaviour of an active clique is governed by a simple random walk on $\{0, \dots, k\}$ (see Lemma~\ref{lem:transmatrix}). This walk is forward-biased for almost its entire length, so we dominate it below by two back-to-back gambler's ruins (see Definition~\ref{def:Z} and Lemma~\ref{lem:clique-dom}). This, together with the fact that any clique containing both mutants and non-mutants changes state with rate at least $r$, allows us to prove several key properties of cliques in the megastar process which we state here in simplified form (see Corollary~\ref{cor:clique-absorb-fast-cts} and Lemma~\ref{lem:iteration-cliques}):
\begin{enumerate}[(C1)]
\item If a clique contains at least one mutant, then with at least constant probability it fills with mutants within time $k\log n$.
\item If a clique contains at most $\Delta$ non-mutants, then with very high probability it fills with mutants within time $(\log n)^7$.
\item Let $I$ be an interval with $(\log n)^7 \le \len{I} \le e^{(\log n)^2}$. Then if a clique contains at most $\Delta$ non-mutants at the start of $I$, with very high probability it contains at most $\Delta$ non-mutants at the end of $I$ and contains at most $2\Delta$ non-mutants at any time in $I$.
\end{enumerate}
Finally, we use (C2) and (C3) together with a careful domination to prove upper bounds on the likelihood of non-mutants being spawned onto $v^*$ from an active clique (see Lemma~\ref{lem:clique-nm-spawns}). 

We now discuss Section~\ref{sec:first-clique}. Heuristically, the argument is quite simple. Consider the interval $J = [0,\sqrt{n}(\log n)^3]$. With probability at least $1-O((\log n)^3/\sqrt{n})$, $\vnc{v^*}{x_0}$ does not trigger in $J$ and so $v^*$ remains a mutant throughout. Conditioned on this event, by Chernoff bounds $x_0$ is very likely to spawn $\Omega(\sqrt{n}(\log n)^3)$ mutants onto $a_1$ in $J$. Each time a mutant is spawned onto $a_1$, either $K_1$ already contains a mutant or there is an $\Omega(1/m) = \Omega(1/\sqrt{n})$ chance of $a_1$ spawning a mutant into $K_1$ before dying. Whenever $K_1$ contains a mutant, by (C1) there is an $\Omega(1)$ chance that $K_1$ will fill with mutants, so in expectation $K_1$ will fill with mutants $\Omega((\log n)^3)$ times over the course of $J$. Finally, when $K_1$ has filled with mutants, by (C3) it is likely to contain at most $\Delta$ non-mutants at time $n$ (see Lemma~\ref{lem:fill-first-clique}). Unfortunately, these events are not independent --- for example, a mutant may be spawned onto $a_j$ while $K_j$ is already active from a previous spawn --- so concentration is not guaranteed. To make the argument rigorous, we therefore divide $J$ into sub-intervals and apply domination.

Section~\ref{sec:other-cliques} is now relatively easy. (C3) tells us that $K_1$ is very likely to remain almost full of mutants for a superpolynomial length of time. While we could fill each subsequent clique with a similar argument to that used in Section~\ref{sec:first-clique}, we have enough wiggle room that we can instead use a substantially simpler argument to prove that $K_1, \dots, K_\ell$ each contain at most $\Delta$ non-mutants by time $n^8$ (see Lemma~\ref{lem:fill-all-cliques}). A side effect of this is that our bound on $t$ in the statement of Lemma~\ref{lem:megastar-fixates} is very loose.

The meat of the proof is in Section~\ref{sec:megastar-fixates}. Suppose that $K_1, \dots, K_\ell$ each contain at most $\Delta$ non-mutants. Since $\Delta$ is much smaller than $k$, it is tempting to simply dominate the number of mutants in reservoirs below by a random walk on $\{0, \dots, \ell m\}$. We could argue that by (C3), for superpolynomial time most of $v^*$'s in-neighbours will be mutants, and so $v^*$ will spawn far more mutants than non-mutants in this interval. While this is true, it will only take us so far --- even if each clique only contained one non-mutant, we should still expect $v^*$ to be a non-mutant for an $\Omega(1/k)$ proportion of the time, leaving us with $\Omega(m/k)$ non-mutants in each reservoir. However, all is not lost. Intuitively, when there are many mutants in a reservoir, the corresponding feeder vertex is more likely to be a mutant and so frequently its clique will contain no non-mutants at all. Developing this idea yields Lemma~\ref{lem:main-iteration}, the main result of the section.

For all $i\in\Zzero$, let
\begin{align*}
I_i^-    &= n^8 + in(\log n)^3,\\
I_i^+    &= n^8 + (i+1)n(\log n)^3,\\
I_i      &= (I_i^-, I_i^+],\\
\alpha_i &= \floor{\max\{(2\log n)^{2}, m/(2\log n)^{2i}\}},\\
\beta_i  &= \floor{\max\{(2\log n)^{2}, \ell m/(2\log n)^{2i}\}}.
\end{align*}
We say that the filtration at time $I_i^-$ is \emph{good} if the following events occur.
\begin{itemize}
\item \Ptotalmut{i}: $|(R_1 \cup \dots \cup R_{\ell}) \setminus X'_{I_i^-}| \le \beta_i$.
\item \Presmut{i}: For all $j \in [\ell]$, $|R_{j} \setminus X'_{I_i^-}| \le \alpha_i$.
\item \Pcliquefull{i}: For all $j \in [\ell]$, $|K_j \setminus X'_{I_i^-}| \le \Delta$.
\item \Pmostcliques{i}: For all but at most $\beta_i$ choices of $j\in[\ell]$, $R_{j} \cup \{a_j\} \cup K_j \subseteq X'_{I_i^-}$.
\end{itemize} 
Lemma~\ref{lem:main-iteration} implies that if the filtration at $I_i^-$ is good, then with very high probability so is the filtration at $I_{i+1}^-$. Thus the number of non-mutants in each reservoir drops by a factor of at least $(2\log n)^2$ to a minimum of $\floor{(2\log n)^2}$, as does the total number of non-mutants across all reservoirs. Moreover, if $i$ is sufficiently large, then Lemma~\ref{lem:main-iteration} also implies that the process fixates by time $I_{i+1}^-$ with probability at least $1/2$. 

Note that the filtration at $I_0^-$ is good; indeed, $\alpha_0 = m$ and $\beta_0 = \ell m$, so \Ptotalmut{0}, \Presmut{0} and \Pmostcliques{0} trivially occur, and \Pcliquefull{0} is very likely to occur by Lemma~\ref{lem:fill-all-cliques}. It is therefore relatively easy to prove Lemma~\ref{lem:megastar-fixates} using Lemmas~\ref{lem:fill-all-cliques} and~\ref{lem:main-iteration} (see Section~\ref{sec:iteration}).

The linchpin of the proof of Lemma~\ref{lem:main-iteration} is a stopping time $\TTend{i}$ defined to be the first time $t \ge I_i^-$ such that one of the following holds.
\begin{enumerate}[\normalfont(D1)]
\item $t = I_i^+$.
\item $v^*$ spawns $\beta_{i+1}$ non-mutants in the interval $(I_i^-,t]$.
\item For some $j \in [\ell]$, $v^*$ spawns $\alpha_{i+1}$ non-mutants onto vertices in $R_{j}$ in the interval $(I_i^-,t]$.
\item For some $j \in [\ell]$, $|K_j \setminus X'_t| > 2\Delta$.
\end{enumerate}
Note that the definition of $\TTend{i}$ guarantees, without any need for conditioning, that throughout $(I_i^-, \TTend{i}]$ our cliques remain almost full of mutants and not too many non-mutants are spawned onto reservoirs. We will therefore work in $(I_i^-, \TTend{i}]$ for most of the proof to facilitate dominations, with the eventual goal of proving that $(I_i^-, \TTend{i}] = I_i$.

In Section~\ref{sec:clique-active}, we prove an upper bound on the number of times cliques are likely to become active over the interval $(I_i^-, \TTend{i}]$ (see Lemma~\ref{lem:clique-active-bound}). In Section~\ref{sec:vcent}, we apply this together with Lemma~\ref{lem:clique-nm-spawns} to prove an upper bound on the length of time for which $v^*$ is likely to be a mutant over the interval $(I_i^-, \TTend{i}]$ (see Lemma~\ref{lem:centre-nm-short-time}). Unfortunately, due to the use of $\TTend{i}$, these proofs require a fairly technical series of dominations. More details can be found in the relevant sections.

In Section~\ref{sec:megastar-final}, we put all of this together to prove Lemma~\ref{lem:main-iteration} and hence Lemma~\ref{lem:megastar-fixates}. The key observation is that Lemma~\ref{lem:centre-nm-short-time} combines with Chernoff bounds on star-clocks to give strong upper bounds on the number of non-mutants spawned by $v^*$ in $(I_i^-, \TTend{i}]$. These bounds, together with (C3), imply that none of (D2)--(D4) are likely to hold at $\TTend{i}$ --- in which case $(I_i^-, \TTend{i}] = I_i$. Additional Chernoff bounds on star-clocks then imply that $v^*$ is likely to spawn a mutant onto every reservoir vertex over the first half of $I_i$, which implies that \Ptotalmut{i+1} and \Presmut{i+1} are likely to occur. Then \Pcliquefull{i+1} is likely to occur by (C3), and \Pmostcliques{i+1} is likely to occur by (C2) combined with a relatively simple argument. This implies that the filtration at $I_{i+1}^-$ is likely to be good, as required by Lemma~\ref{lem:main-iteration}. This part of the argument is mostly contained in Lemmas~\ref{lem:star-clock-bound} and~\ref{lem:branches-clear}.

It remains only to prove that if $i$ is sufficiently large, $X$ fixates with probability at least $1/2$. In this case, we use similar arguments to the above to show that with probability at least $5/6$, $v^*$ spawns no non-mutants at all over the course of $I_i$. In this case, similar arguments to those used to deal with \Pmostcliques{i+1} work to show that $X$ is likely to fixate as required.

\subsection{The behaviour of mutants within cliques}\label{sec:cliques}

 In order to describe the behaviour of mutants in cliques in the megastar process, we require the following definitions.

\begin{definition}\label{def:cliquetimes}
Given a clique $j\in [\ell]$ and
a time 
$t_0\geq 0$, define the following stopping times for $h\in\Zzero$.
\[\Ti{h} = \begin{cases}
t_0 & \mbox{if $h=0$,}\\
\min\{t > \Ti{h-1} \mid 
X'_t \cap K_j \neq 
X'_{\Ti{h-1}} \cap K_j
\} & \text{if $h>0$ and $K_j$ is}\\
   & \text{active at $\Ti{h-1}$,}\\
\Ti{h-1} & \mbox{otherwise.} 
\end{cases}\]
\defend{}
\end{definition}

The subscript ``$\mathsf c$'' in $\Ti{h}$ stands for ``change'' because if $K_j$ is active at $t_0$, then $\Ti{1}, \Ti{2}, \dots$ are the times at which the number of mutants in $K_j$ changes after $t_0$ until $K_j$ next becomes inactive. We also use the following definition.

\begin{definition}\label{def:Iabs}
Let $\Iabs = \min \{h \in \mathbb{Z}_{\geq 0} \mid K_j\text{ is inactive at } \Ti{h}\}$.
\defend{}
\end{definition}

The subscript ``$\mathsf{in}$'' in $\Iabs$ stands for ``inactive''.
Note that $\Iabs$ is finite with probability~$1$.
Thus, with probability~$1$, the following is well-defined.

\begin{definition}\label{def:Tabs}
Let $\Tabs = \Ti{\Iabs}$.\defend{}
\end{definition}

\begin{definition} \label{def:clique-chain}
For every $h\in\Zzero$,  let $\Y(h) = |X'_{\Ti{h}} \cap K_j|$.
For $h\in\{1,\ldots,\Iabs\}$, we say that $\Y$ \emph{jumps} at time~$\Ti{h}$.
\defend{}
\end{definition}
  
  We first show that $\Y$ evolves as a Markov chain.

 \begin{lemma}\label{lem:transmatrix}  
 Suppose $\ell \ge 3$, and consider any $j\in [\ell]$ and $t_0\geq 0$.
 Let $f$ be any possible value of $\filt_{t_0}(X')$. 
 Conditioned on the event $\filt_{t_0}(X') = f$,
 $\Y$ evolves as a discrete-time Markov chain on  states $\{0,\ldots,k\}$ 
starting from state $\Y(0)$
with
 the following transition matrix. 
\begin{align*}
p'_{0,0} &= 1,\\
p'_{i,i+1} &= \frac{r(k-i)}{(r+1)(k-i)+ 1} \textnormal{ for all }i\in [k-1],\\
p_{i,i-1}' &= 1 - p'_{i,i+1} \textnormal{ for all }i\in [k-1],\\
p_{k,k}' &= 1,\\
p'_{i,j} &=0, \textnormal{ otherwise.}
\end{align*}
\end{lemma}
\begin{proof}  

Consider any $t' \geq t_0$, 
any integer $h \ge 0$, any $y_0, \dots, y_h \in \{0, \dots, k\}$  
and 
any possible value $f'$ of $\filt_{t'}(X')$ 
that implies that
$\filt_{t_0}(X')=f$,
$\Y(0) = y_0, \dots, \Y(h) = y_h$,
and 
$\Ti{h} = t'$.
We will show that, conditioned on $\filt_{t'}(X')=f'$,
the distribution of $\Y(h+1)$  
 is as claimed in the lemma statement.

First, suppose that $y_h\in \{0,k\}$.
In this case, Definition~\ref{def:cliquetimes} guarantees that $\Ti{h+1}=\Ti{h}$
so $\Y(h+1)=\Y(h)$,
  which is consistent with $p'_{0,0}=p'_{k,k}=1$.

Next, suppose that $y_h \in [k-1]$.  
Then the clique $K_j$ is active throughout the interval $[t_0,\Ti{h+1})$.
By the definition of the megastar process~$X'$, 
the feeder vertex $a_j$ is a non-mutant throughout this interval. 

The out-degree of~$a_j$ and every vertex in $K_j$ is~$k$.
There are $y_h(k-y_h)$ mutant clocks whose sources are in 
$K_j\cap X'_{t'}$ and targets are in 
$K_j \setminus X'_{t'}$.
Similarly, there are $(k-y_h+1) y_h$ non-mutant clocks whose
sources are in 
$(K_j \cup \{a_j\}) \setminus X'_{t'}$
and targets are in $K_j \cap X'_{t'}$. 
Thus after $t'$, the number of mutants in $K_j$ increases with rate 
$u := r y_h(k-y_h)/k$ and decreases with rate 
$d:= (k-y_h+1)y_h/k$. 
 It follows that
\begin{align*}
\Pr(\Y(h+1)=y_h+1 \mid \filt_{t'}(X') = f') &= 
\frac{u}
{u+d}
  = p'_{y_h, y_h+1},\\
\Pr(\Y(h+1)=y_h-1 \mid \filt_{t'}(X') = f') &= p'_{y_h, y_h-1},
\end{align*}
as required. \end{proof}  
   
 We will dominate the Markov chain $\Y(h)$ 
 in terms of a simpler Markov chain~$Z$, which is defined as follows.

\begin{definition}\label{def:Z}
Let $c_r = \lceil 2r/(r-1)\rceil$ and let $r' = (r+1)/2$. 
Let $Z$ be the discrete-time Markov chain on states $\{0, \dots, k\}$ 
with the following transition matrix.
\begin{align*}
p_{0,0} &= 1,\\
p_{i,i+1} &= 
	\begin{cases}
		r'/(r'+1)   & \textnormal{ if }1\le i\le k-c_r,\\
		1/3         & \textnormal{ if }k-c_r < i \le k-1,\\
	\end{cases}\\
p_{i,i-1} &= 1 - p_{i,i+1} \textnormal{ for all }i\in [k-1],\\
p_{k,k} &= 1,\\
p_{i,j} &= 0, \textnormal{ otherwise.}\defenddisp{}
\end{align*}
\end{definition}
 
Note that the value~$k$ used in Definition~\ref{def:Z} is
the same as~$k=k(\ell)$ in our parameterisation of the megastar (Definition~\ref{def:newmegas}). We  now show that $\Y$ is  dominated below by 
$Z$, starting from state $\Y(0)$.

\begin{lemma}\label{lem:clique-dom}  
Suppose $\ell \ge 3$, and consider any $j \in [\ell]$ and $t_0 \geq 0$.
Let $f$ be any possible value of $\filt_{t_0}(X')$. 
Conditioned on the event $\filt_{t_0}(X') = f$, 
$\Y$ is dominated below by 
the Markov chain~$Z$, starting from state $\Y(0)$. \end{lemma}
\begin{proof}
Given Lemma~\ref{lem:transmatrix}, the two things to show are
\begin{enumerate}[(i)]
\item for all $i\in [k-1]$, $p'_{i,i+1} \geq p_{i,i+1}$, and
\item for $0<i<i+1<k$,
$p'_{i+1,i+2} \geq p_{i,i+1}$. 
\end{enumerate}
Since $p'_{i+1,i+2} \leq p'_{i,i+1}$, it suffices to prove the second of these,
together with $p'_{k-1,k} \geq p_{k-1,k}$.

We start with the latter. We have
\[p'_{k-1,k} = \frac{r}{(r+1)+ 1} \geq   \frac{1}{3},\]
which gives the desired bound since, from $c_r\geq 2$,
we have $k-c_r < k-1$ and so $p_{k-1,k}=1/3$.

So now we must consider $i$ satisfying
$0<i<k-1$ and we must show
$p'_{i+1,i+2} \geq p_{i,i+1}$.

\medskip\noindent
\textbf{Case 1: $\boldsymbol{i \le k-c_r}$.} 
We have
\[p'_{i+1,i+2} - p_{i,i+1} = \frac{r(k-(i+1))}{(r+1)(k-(i+1))+ 1} - \frac{r'}{r'+1}.\]
Since $k > i+1$,
the common denominator of these fractions is positive; the numerator of the difference is easily calculated to be $(k-i)(r-r')  -r  
 \geq c_r(r-1)/2-r \geq 0$.

\medskip\noindent
\textbf{Case 2: $\boldsymbol{i > k-c_r}$.} Since $i$ and~$k$ are integers, $i+1< k$ implies that $1 \leq k-(i+1)$.
Thus,
\[p'_{i+1,i+2} = \frac{r(k-(i+1))}{(r+1)(k-(i+1))+ 1} \geq   \frac{r(k-(i+1))}{(r+2)(k-(i+1))} = \frac{r}{r+2} \geq \frac{1}{3}.\]
As $i > k-c_r$, $p_{i,i+1} = 1/3$, so this gives the desired bound.
\end{proof}

We now use our observations about the gambler's ruin problem in Section~\ref{sec:gambler} to derive some simple bounds on the
behaviour of the Markov chain~$Z$.

\begin{definition}\label{def:Delta}
Let $\Delta = \floor{c_r(\log n)^3}$.\defend{}
\end{definition}
 
Note that $k \ge 2\Delta$ as long as $\ell$ is sufficiently large, and that 
having $\ell$ sufficiently large also implies a lower bound on $n$.
 
\begin{lemma}\label{lem:Z-k-c-r}  
\statelargeell Suppose that $Z$ is started in state $k-c_r$. Then the following statements hold.
\begin{enumerate}[(i)]
\item \label{it:Z-k-c-r-1} The probability of reaching state $k$ without passing through state $k-\Delta$ is at least $1-e^{-{(\log n)}^3}$.
\item  \label{it:Z-k-c-r-2}The expected number of transitions 
that it takes to reach
state $k$ or state $k-2\Delta$ is at most $\log n$.
\end{enumerate}
\end{lemma}
\begin{proof} 
Recall the notation $H_{x;S}$ and $p_{x\rightarrow y;z}$ from Lemma~\ref{lem:backtoback}. We first prove Item~(\ref{it:Z-k-c-r-1}).
Apply Lemma~\ref{lem:backtoback} with
$p_1 = r'/(r'+1) =(r+1)/(r+3)>1/2$,
$a = k-2\Delta$,  
$b = k-\Delta$,
$c = k-c_r$ and
$d = k$. By Item~(\ref{backtobackone}) of Lemma~\ref{lem:backtoback},
\[p_{c\rightarrow d; b} \geq 
1 - {\left(\frac{1-p_1}{p_1}\right)}^{c-b} 2^{d-c}.\]
Note that $(1-p_1)/p_1=2/(r+1)=1-(r-1)/(r+1)$ and $\Delta\geq c_r(\log n)^3-1$, so, for all $n$ sufficiently large with respect to~$r$, it holds that
\begin{align*} 
{\left(\frac{1-p_1}{p_1}\right)}^{\Delta-c_r} 2^{c_r}&\leq {\left(1-\frac{r-1}{r+1}\right)}^{c_r(\log n)^3} (r+1)^{c_r+1}\\
&\leq \exp\left(-c_r \frac{r-1}{r+1}(\log n)^3+(c_r+1)\log (r+1)\right)\leq  e^{ -{(\log n)}^3},
\end{align*}
where in the 
second inequality we used~\eqref{eq:ebounds} and
in the last inequality we used the fact that $c_r \frac{r-1}{r+1}>1$. Item~(\ref{it:Z-k-c-r-1}) thus follows.

We next prove Item~(\ref{it:Z-k-c-r-2}). Apply Lemma~\ref{lem:backtoback} with 
the same parameters as before.  
By Lemma~\ref{lem:backtoback}(\ref{backtobacktwo}), 
$\E[H_{c,\{a,d\}}] \leq  
2^{d-c+1} \left(\frac{3p_1-1}{2 p_1-1}\right) $ 
so it suffices to show 
\[ 2^{c_r+1} \left(\frac{3p_1-1}{2 p_1-1}\right) \leq \log n\,,\]
which again holds as long as $n$ is sufficiently large with respect to~$r$.
\end{proof}

Next, we use the domination of $\Y$ by $Z$ (Lemma~\ref{lem:clique-dom}) and
our observations about $Z$ (Lemma~\ref{lem:Z-k-c-r})
to derive some conclusions about mutants in cliques. 

\begin{lemma}\label{lem:clique-fill}
\statelargeell Let $j \in [\ell]$, $t_0 \ge 0$, and $y_0 \in \{0, \dots, k\}$. Let $f$ be a possible value of $\filt_{t_0}(X')$ which implies that $\Y(0) = y_0$. Then the following statements hold.
\begin{enumerate}[(i)]
\item \label{it:clique-fill-2} If $y_0 \ge k-1$, then 
\[\Pr\Big(
\forall t\in [t_0, \Tabs], |K_j \setminus X'_t| \le \Delta
\,\Big|\, \filt_{t_0}(X') = f \Big)\geq 1-e^{-{(\log n)}^3}.\]
 \item \label{it:clique-fill-3} If $y_0 \ge k - \Delta$, then 
\[\Pr\Big( 
\forall t \in [t_0, \Tabs], |K_j \setminus X'_t| \le 2\Delta
\,\Big|\, \filt_{t_0}(X') = f\Big) \geq 1-2e^{-(\log n)^3}.\]
\end{enumerate}
\end{lemma}
\begin{proof}
Recall the notation $p_{x \rightarrow y;z}$ from Lemma~\ref{lem:backtoback}. The Markov chain $Z$ is the same as the chain in Lemma~\ref{lem:backtoback} 
with $p_1 = r'/(r'+1)$,
$a=0$,  $c=k-c_r$ and $d=k$. We start with simple lower bounds concerning~$Z$.
\begin{itemize}
\item From any state $i\in\{k-c_r, \dots, k\}$, $Z$~must reach either state~$k$ or state $k-c_r$ before reaching state $k-\Delta$.  By Lemma~\ref{lem:Z-k-c-r}(\ref{it:Z-k-c-r-1}), $p_{k-c_r\rightarrow k;k-\Delta} \ge 1 - e^{-(\log n)^3}$, and we have
\begin{equation}\label{eqn:p_2}
p_{i \rightarrow k; k-\Delta}\ge p_{k-c_r\rightarrow k;k-\Delta}
\ge 1 - e^{-(\log n)^3}.
\end{equation}

\item Consider $i \in\{k-\Delta,\ldots,k-c_r-1\}$. By Corollary~\ref{cor:gambler}(\ref{ruin:three}), \[p_{i\rightarrow k-c_r;k-2\Delta} \geq 1 - (r')^{-\Delta}\geq 1-e^{-\Delta(r'-1)/r'} \geq 1 - e^{-(\log n)^3},\]
where the last inequality holds for all $n$ sufficiently large with respect to~$r$, using the fact that $c_r\frac{r'-1}{r'}=c_r\frac{r-1}{r+1}>1$. It follows from  Lemma~\ref{lem:Z-k-c-r}(\ref{it:Z-k-c-r-1}) that 
\begin{equation}\label{eqn:p_3}
p_{i\rightarrow k;k - 2 \Delta}  \geq p_{i \rightarrow k-c_r;k-2\Delta} - p_{k-c_r \rightarrow k-2\Delta; k} \geq 1 - 2e^{-(\log n)^3}.
\end{equation}
\end{itemize}

The result now follows easily. By Lemma~\ref{lem:clique-dom}, $\Y$ is dominated below by $Z$ with initial state $y_0$ conditioned on $\filt_{t_0}(X')=f$. Thus Item (\ref{it:clique-fill-2}) follows from \eqref{eqn:p_2}
and Item (\ref{it:clique-fill-3}) follows from~\eqref{eqn:p_2}
and~\eqref{eqn:p_3}.
\end{proof}

Note that in the process of proving Lemma~\ref{lem:clique-fill}, we proved the following (see~\eqref{eqn:p_2} and~\eqref{eqn:p_3}).

\begin{corollary}\label{cor:stupid} \statelargeell Let
$y_0$ be an integer satisfying $k - \Delta \leq y_0  \leq k$.  The probability that $Z$, when started from~$y_0$, 	
reaches state $k$ without passing through 
state $k- 2 \Delta$,  is at least $1-2e^{-(\log n)^3}$.\hfill\qedsymbol
\end{corollary}

We now use Lemma~\ref{lem:Z-k-c-r} and Corollary~\ref{cor:stupid} to give lower bounds on the probability that $Z$ reaches $k$ in a relatively short time.

 \begin{lemma}\label{lem:Z-fast-hit-k}
 \statelargeell Consider the Markov chain~$Z$, starting from
 $y_0\in[k-1]$.
 \begin{enumerate}[(i)]
\item  If $k-\Delta \le y_0 \le k-1$, then 
$\Pr(Z_{\ceil{15c_r^2(\log n)^6}} = k)\ge 1 - 3e^{-(\log n)^3}$.

 \item  If $1 \le y_0 \le k-\Delta$, then 
 $\Pr(Z_{\lceil 5c_r^2k \rceil } = k)  \ge \tfrac{r-1}{ 5r}  $.
\end{enumerate} 
 \end{lemma}
 \begin{proof}
 
 Recall the notation $H_{x;S}$ and $p_{x\rightarrow y;z}$ from Lemma~\ref{lem:backtoback}.
The Markov chain~$Z$ is the same as the Markov chain in Lemma~\ref{lem:backtoback}
with $p_1 = r'/(r'+1)$, $a=0$,
$c = k-c_r$ and $d=k$.
Let $\eta = \ceil{5ec_r\Delta}\,\ceil{(\log n)^3}$ and $\eta' = 
\ceil{4c_r^2  k}$. 

We first establish Item~(i). Note that $\ceil{15c_r^2(\log n)^6} \ge \eta$, so it suffices to establish the stronger statement that for all $y_0$ satisfying $k-\Delta \le y_0 \le k-1$, we have $\Pr(Z_{\eta} = k)\ge 1 - 3e^{-(\log n)^3}$. So suppose $k-\Delta \leq y_0 \leq k-1$.
Note that 
\[\Pr (Z_{ \eta} = k) \geq \Pr(H_{y_0;\{k-2\Delta,k\}} \leq \eta) - p_{y_0 \rightarrow k-2\Delta;k}.\]
Corollary~\ref{cor:stupid} shows that
$ p_{y_0 \rightarrow k-2\Delta;k} \leq 2e^{-(\log n)^3}$.
So, to show (i), we will show
\begin{equation}\label{eq:star}
\Pr(H_{y_0;\{k-2\Delta,k\}} \leq \eta) \geq 1 - e^{-(\log n)^3}.\end{equation}
To establish~\eqref{eq:star}, we first show the following.
\begin{equation}\label{eq:starstar}
\mbox{For every integer~$y$ satisfying $k-2\Delta < y < k$, }
\E[H_{y;\{k-2\Delta,k\}}] \leq 5c_r \Delta.
\end{equation} 
To see~\eqref{eq:starstar}, there are three cases to consider.

\medskip\noindent
\textbf{Case 1: $\boldsymbol{y=k-c_r}$.} By  Lemma~\ref{lem:Z-k-c-r}\eqref{it:Z-k-c-r-2}, we have the (tighter) bound
\begin{equation}\label{eqn:H-y-1}
\E[H_{y;\{k-2\Delta,k\}}] \le \log n.
\end{equation}

\medskip\noindent
\textbf{Case 2: $\boldsymbol{k-2\Delta < y < k-c_r}$.} By  Corollary~\ref{cor:gambler}\eqref{ruin:five}, we have
\begin{equation*}
\E[H_{y;\{k-2\Delta,k-c_r\}}]\le \frac{2\Delta(r'+1)}{r'-1}=\frac{2\Delta(r+3)}{r-1}\leq 4c_r\Delta .
\end{equation*}
Combining this and \eqref{eqn:H-y-1}, we obtain
\begin{equation*}
\E[H_{y;\{k-2\Delta,k\}}] \le \E[H_{y;\{k-2\Delta,k-c_r\}}] + \E[H_{k-c_r;\{k-2\Delta,k\}}]  \le 5c_r\Delta .
\end{equation*}

\medskip\noindent
\textbf{Case 3: $\boldsymbol{k-c_r < y < k}$.} By Corollary~\ref{cor:gambler}\eqref{ruin:four} we have $\E[H_{y;\{k-c_r,k\}}] \le 3c_r$, and so 
\begin{equation*}
\E[H_{y;\{k-2\Delta,k\}}] \le \E[H_{y;\{k-c_r,k\}}] + \E[H_{k-c_r;\{k-2\Delta,k\}}] \le 3c_r + \log n \le 2\log n.
\end{equation*}

These three cases establish~\eqref{eq:starstar}.  Applying Markov's inequality gives the following.
\begin{equation}\label{eq:blah}
\mbox{For every integer~$y$ satisfying $k-2\Delta < y < k$,}\quad
\Pr(H_{y;\{k-2\Delta,k\}} \geq \ceil{5ec_r\Delta}) \leq  \frac{1}{e}.
\end{equation}
Now \eqref{eq:star} follows by subdividing the set $[\eta]$ (indexing $\eta$ transitions from the initial state $y_0$) into $\ceil{(\log n)^3}$ disjoint sets of contiguous indices, each of size $\ceil{5ec_r\Delta}$, then applying \eqref{eq:blah} to each subset.

We next establish item~(ii), so 
suppose $1 \le y_0 \le k-\Delta$.
 By Corollary~\ref{cor:gambler}\eqref{ruin:five}, for all $y_0\in\{1, \dots, k-c_r\}$, we have
\begin{equation*}
\E[H_{y_0;\{0,k-c_r\}}] \le \frac{k(r'+1)}{r'-1} = \frac{k(r+3)}{r-1} \le 2c_rk.
\end{equation*}
Applying Markov's inequality, we obtain
\[\Pr(H_{y_0;\{0,k-c_r\}} \leq \ceil{4c_r^2 k}) \geq 1-\frac{1}{2c_r}\geq 1-\frac{r-1}{4r}.\]
Now by  Corollary~\ref{cor:gambler}\eqref{ruin:three}, 
$p_{y_0 \rightarrow 0;k-c_r} \leq  \frac{1}{r'}$.  
So 
\[\Pr (H_{y_0;\{k-c_r\}}\leq\eta') \geq 
\Pr(H_{y_0;\{ 0,k-c_r\}} \leq \eta') - p_{y_0 \rightarrow  0;k-c_r}
\geq   1-\frac{r-1}{4r}-\frac{2}{r+1}.
\]
Now, using the fact that, starting from 
$Z_0=k-c_r$, we have $\Pr(Z_{\eta} = k)\ge 1 - 3e^{-(\log n)^3}$,
which we proved in the derivation of~(i),
we have, starting from $Z_0=y_0$,
\[\Pr(Z_{\eta + \eta'} =k) \geq 1- 3 e^{-(\log n)^3} -\frac{r-1}{4r}-\frac{2}{r+1} =(r-1)\left(\frac{1}{r+1}-\frac{1}{4r}\right)-3 e^{-(\log n)^3}\geq \frac{r-1}{5r}\,,\]
which completes the proof, since $\eta+\eta' \leq \ceil{5c_r^2k}$.
\end{proof}
 
The following corollary follows immediately from Lemma~\ref{lem:clique-dom} and Lemma~\ref{lem:Z-fast-hit-k}.

 \begin{corollary}\label{cor:clique-absorb-fast-disc}
\statelargeell Let $j \in [\ell]$, let $t_0 \geq 0$, and let $y_0 \in [k-1]$. Let $f$ be a possible value of $\filt_{t_0}(X')$ which implies that $\Y(0) = y_0$. Then the following statements hold.
\begin{enumerate}[(i)]
\item \label{it:clique-absorb-fast-disc-1} 
If $k-\Delta \le y_0 \le k-1$, then 
\[\Pr\Big(
\Y(\lceil 15c_r^2{(\log n)}^6\rceil ) = k \,\Big|\, \filt_{t_0}(X') = f\Big) \ge 1 - 3e^{-(\log n)^3}.\]
\item \label{it:clique-absorb-fast-disc-2} If $1 \le y_0 \le k-\Delta$, then 
\[\Pr\Big(
\Y(\lceil 5c_r^2 k \rceil ) = k \,\Big|\, \filt_{t_0}(X') = f\Big) \ge  \frac{r-1}{5r}.\]
\end{enumerate} \end{corollary}

To translate Corollary~\ref{cor:clique-absorb-fast-disc} into a bound on the time it takes $K_j$ to fill with mutants, we will require the following lemma.

\begin{lemma}\label{lem:clique-disc-cts-interface}
\statelargeell Let $j \in [\ell]$ and $t_0 > 0$. Let $f$ be a possible value of $\filt_{t_0}(X')$ which implies that $K_j$ is active at $t_0$. Let $t^* \ge 16(\log n)^3$. Then, conditioned on the event $\filt_{t_0}(X') = f$, with probability at least $1-e^{-(\log n)^3}$, 
$\Ti{\ceil{t^*/2}} < t_0 + t^*$. 
\end{lemma}
\begin{proof}
We will first show that $\Ti{1} - \Ti{0}, \dots, \Ti{\lceil t^*/2 \rceil} - \Ti{\lceil t^*/2 \rceil - 1}$ are dominated above by i.i.d.\ exponential variables with rate 1.

Fix $0 \le h \le \lceil t^*/2 \rceil - 1$, let $x\ge 0$ and let $t_0, \dots, t_h \ge 0$. Suppose $f_h$ is a possible value of $\filt_{t_h}(X')$ which implies that $\Ti{0} = t_0, \dots, \Ti{h} = t_h$ and $\filt_{t_0}(X') = f$. Fix $t \ge 0$. Note that $f_h$ determines the event $\Tabs \le t_h$, as well as the value of $\Y(h)$ --- write $y_h = \Y(h)$. If $f_h$ is such that $\Tabs \le t_h$, then we have $\Ti{h+1} - t_h = 0$ and so
\begin{equation}\label{eqn:Tip-1}
\Pr\big(\Ti{h+1} - t_h \le x \mid \filt_{t_h}(X') = f_h\big) = 1 \ge 1 - e^{-x}.
\end{equation}

Suppose instead that $f_h$ is such that $\Tabs > t_h$, so that  $1 \le y_h \le k-1$. Then $\Ti{h+1} - t_h$ is the amount of time it takes after $t_h$ for a vertex in $K_j \cup \{a_j\}$ to spawn either a mutant onto a non-mutant in $K_j$ or a non-mutant onto a mutant in $K_j$. Thus, conditioned on $\filt_{t_h}(X') = f_h$, 
(see the proof of Lemma~\ref{lem:transmatrix})
$\Ti{h+1} - t_h$ is an exponential variable with rate
\[\frac{ry_h(k-y_h)}{k} + \frac{y_h(k-y_h)}{k} + \frac{y_h}{k} \ge \frac{(r+1)y_h(k-y_h)}{k} \ge \frac{(r+1)(k-1)}{k}\ge \frac{2(k-1)}{k} \ge 1.\]
In particular,  we have shown that
\begin{equation}\label{eqn:Tip-2}
\Pr\big(\Ti{h+1} - t_h \le x \mid \filt_{t_h}(X') = f_h\big) \ge 1 - e^{-x}.
\end{equation}

By \eqref{eqn:Tip-1} and \eqref{eqn:Tip-2} we have 
\[\Pr\big(\Ti{h+1}-t_h \le x \mid \Ti{0} = t_0, \dots, \Ti{h} = t_h,
\filt_{t_0}(X') = f
\big) \ge 1 - e^{-x},\]
and so $\Ti{1} - \Ti{0}, \dots, \Ti{\lceil t^*/2 \rceil} - \Ti{\lceil t^*/2 \rceil - 1}$ are dominated above by i.i.d.\ exponential variables with rate 1 as claimed. 

It follows that
$\Ti{\lceil t^*/2 \rceil} - t_0$ is dominated above by a sum of $\lceil t^*/2 \rceil$ i.i.d.\ exponential variables with rate 1. We have $t^* \ge 3\lceil t^*/2 \rceil/2$, so by Corollary~\ref{cor:expsum} we have
\begin{equation*}
\Pr\big(\Ti{\ceil{t^*/2}} < t_0 + t^*\big) = \Pr\big(\Ti{\ceil{t^*/2}} - t_0 < t^*\big) \ge 1 - e^{-t^*/16} \ge 1 - e^{-(\log n)^3}.\qedhere
\end{equation*}
\end{proof}

We now use Lemma~\ref{lem:clique-fill},  Corollary~\ref{cor:clique-absorb-fast-disc}
and Lemma~\ref{lem:clique-disc-cts-interface}, to prove three results which contain all the properties of cliques 
that we will need to prove our key lemma, Lemma 
\ref{lem:megastar-fixates}.

\begin{corollary}\label{cor:clique-absorb-fast-cts}
\statelargeell Let $j \in [\ell]$,
$t_0\geq 0$,
and $y_0 \in [k]$. Let $f$ be a possible value of $\filt_{t_0}(X')$ which implies that $\Y(0)=y_0$. Then the following statements hold.
\begin{enumerate}[(i)]
\item \label{it:clique-absorb-fast-cts-1}If $k-\Delta \le y_0 \le k$, then 
\[\Pr\big(\Tabs \le t_0 + 30c_r^2(\log n)^6 \textrm{ and }K_j \subseteq X_{\Tabs}' \,\big|\, \filt_{t_0}(X') = f\big) \ge 1 - 4e^{-(\log n)^3}.\]
\item \label{it:clique-absorb-fast-cts-2}If $1 \le y_0 \le k-\Delta$, then 
\[\Pr\big(\Tabs \le t_0 + 10c_r^2k \textrm{ and }K_j \subseteq X_{\Tabs}' \,\big|\, \filt_{t_0}(X') = f\big) \ge \frac{r-1}{6r}.\]
\end{enumerate}
\end{corollary}
\begin{proof}
If $y_0=k$, then $\Tabs = t_0$ and the result follows immediately, so suppose instead that $y_0\in[k-1]$. 
This implies $t_0>0$.
Part~(\ref{it:clique-absorb-fast-cts-1}) now follows on applying a union bound to Corollary~\ref{cor:clique-absorb-fast-disc}(\ref{it:clique-absorb-fast-disc-1}) and Lemma~\ref{lem:clique-disc-cts-interface}, taking $t^* =  30c_r^2(\log n)^6$. Part~(\ref{it:clique-absorb-fast-cts-2}) likewise follows immediately by applying a union bound to Corollary~\ref{cor:clique-absorb-fast-disc}(\ref{it:clique-absorb-fast-disc-2}) and Lemma~\ref{lem:clique-disc-cts-interface}, taking $t^* = 10c_r^2k$.
\end{proof}

We combine Lemma~\ref{lem:clique-fill} and Corollary~\ref{cor:clique-absorb-fast-cts} to prove the following.

\begin{lemma}\label{lem:iteration-cliques}
\statelargeell Let $j \in [\ell]$. Let $I = (I^-, I^+]$ be a time interval with $30c_r^2(\log n)^{6} <\len{I} \le e^{(\log n)^2}$, and let $f$ be a possible value of $\filt_{I^-}(X')$ which implies that $|K_j \setminus X_{I^-}'| \le \Delta$. Then, conditioned on $\filt_{I^-}(X') = f$, with probability at least $1-e^{-\frac{1}{2}(\log n)^3}$ the following statements all hold.
\begin{enumerate}[(i)]
\item \label{it:iteration-cliques-1}For all $t \in I$, $|K_j \setminus X'_t| \le 2\Delta$.
\item \label{it:iteration-cliques-2}$|K_j \setminus X'_{I^+}| \le \Delta$.
\item \label{it:iteration-cliques-3}For all $t_0 \in [I^-, I^+ - 30c_r^2(\log n)^{6})$, there exists $t_1 \in [t_0, t_0 + 30c_r^2(\log n)^{6}]$ such that $K_j \subseteq X'_{t_1}$.
\end{enumerate}
\end{lemma}

\begin{proof}
We start with the following mutually recursive definitions.
Let $\Tinact{}{-1}= I^-$.
Then, for $h\in \Zzero$,
\begin{align*}
\Tact{}{h} &=  \min \{ t\geq \Tinact{}{h-1}\mid
\mbox{$K_j$ is active at $t$  or $t=I^+$} \},  \\
\Tinact{}{h} &= \min \{t\geq \Tact{}{h}\mid 
\mbox{$K_j$ is inactive at $t$  or $t=I^+$} \}.
\end{align*}
These definitions of $\Tinact{}{h}$ and $\Tact{}{h}$ are local to this proof.
The subscript ``$\mathsf a$'' stands for ``active'' and the subscript ``$\mathsf{in}$'' stands for ``inactive''.
The notation must not be confused with the global variable $\Tabs$.
Let $\xi = \ceil{4r\len{I}}$ and let $\kappa =30c_r^2(\log n)^{6}$.   

We define the following events.
\begin{itemize}
\item Let $\mathcal{E}_1$ be the event that  $\Tact{}{ \xi} = I^+$. 
\item For $h\in \Zzero$, let $\mathcal{E}_2(h)$ be the event that $\Tinact{}{h} \le \Tact{}{h} +   \kappa$.
\item The event $\filt_{I^-}(X') = f$
determines whether $K_j$ is active at time $I^-$. 
\begin{itemize}
\item If so, let  $\mathcal{E}_3(0)$ be the event that,   for all $t \in [\Tact{}{0}, \Tinact{}{0}]$, $|K_j \setminus X'_t| \le 2 \Delta$.
\item If not, let $\mathcal{E}_3(0)$ be the event that,   for all $t \in [\Tact{}{0}, \Tinact{}{0}]$, $|K_j \setminus X'_t| \le \Delta$.
\end{itemize}
\item For $h\in \Zone$, let $\mathcal{E}_3(h)$ be the event that,   for all $t \in [\Tact{}{h}, \Tinact{}{h}]$, $|K_j \setminus X'_t| \le \Delta$.  
\item For $h\in {\mathbb Z}_{\geq -1}$, 
  let $\mathcal{E}_4(h)$ be the event that 
for all $t\in [\Tinact{}{h},\Tact{}{h+1})$, 
$K_j \subseteq X'_t$.
\end{itemize}

Note that
for all $h\geq 0$, $\mathcal{E}_3(h)$ implies $\mathcal{E}_4(h)$.
This is easy to see as long as $K_j$ is inactive at $\Tinact{}{h}$. In this case,
$\mathcal{E}_3(h)$ implies that $K_j \subseteq X'_{\Tinact{}{h}}$ -- since the number
of non-mutants is at most~$2\Delta$, but it is either~$0$ or~$k$, it must be~$0$.
On the other hand, if $K_j$ is active at $\Tinact{}{h}$ then $\Tinact{}{h}=I^+$
so $\Tact{}{h+1}=I^+$ so the interval in $\mathcal{E}_4(h)$ is empty.

We next observe that $\filt_{I^-}(X')=f$ implies that
$\mathcal{E}_4(-1)$ occurs.
 From the statement of the lemma,
 $\filt_{I^-}(X')=f$ implies that  $|K_j \setminus X_{I^-}'| \le \Delta$.
 If $K_j$ is inactive at $I^-$ then this implies that 
 $K_j \subseteq X'_{I^-}$, which implies
 $\mathcal{E}_4(-1)$.
 If instead $K_j$ is active at $I^-$ then the interval in $\mathcal{E}_4(-1)$ is empty
 so $\mathcal{E}_4(-1)$ occurs vacuously. 
 
For any integer~$q$, let
$\mathcal{E}_2^q = \bigcap^{q}_{h=0} \mathcal{E}_2(h)$,
$\mathcal{E}_3^q = \bigcap^{q}_{h=0} \mathcal{E}_3(h)$, and
$\mathcal{E}_4 ^q= \bigcap^{q}_{h=-1} \mathcal{E}_4(h)$. Let $\mathcal{E}_2 = \mathcal{E}_2^{\xi}$,
$\mathcal{E}_3 = \mathcal{E}_3^{\xi}$ and
$\mathcal{E}_4 = \mathcal{E}_4^{\xi}$.

We first show that if $\filt_{I^-}(X')=f$   and
$\mathcal{E}_1$, $\mathcal{E}_2$ and $\mathcal{E}_3$   all
occur, then statements (i), (ii) and (iii) hold. 
As we have just observed,  
$\filt_{I^-}(X')=f$   and  $\mathcal{E}_3$ imply
that $\mathcal{E}_4$ also occurs.
Then $\mathcal{E}_3$, $\mathcal{E}_4$ and $\mathcal{E}_1$ imply  (i).
They also imply (ii) except in the case 
where $K_j$ is active at $I^-$ and remains active for  
all of~$I$. This case is ruled out by $\mathcal{E}_2(0)$ since
$\len{I} > \kappa$.
We now turn to statement~(iii).  Consider any  $t_0 \in [I^-, I^+-\kappa)$.   
Suppose first that $K_j$ is inactive at time~$t_0$.
Since (i) holds, $K_j\subseteq X'_{t_0}$, so it suffices to take $t_1=t_0$.
Suppose instead that $K_j$ is active at time~$t_0$. 
Then, for some $h\geq 0$, $t_0 \in [\Tact{}{h},\Tinact{}{h}]$.
By $\mathcal{E}_1$, $h\leq \xi$. By $\mathcal{E}_2(h)$, we may assume $\Tinact{}{h} \leq t_0 + \kappa$
so we can choose $t_1 = \Tinact{}{h}$.
Since $t_0+\kappa < I^+$,  $t_1 < I^+$
so  $t_1 \in [\Tinact{}{h},\Tact{}{h+1})$ and
$\mathcal{E}_4(h)$ guarantees that $K_j \subseteq X'_{t_1}$.

During the remainder of the proof, we will show that
\begin{equation}\label{eqn:iteration-clique-events}
\Pr(\mathcal{E}_1 \cap \mathcal{E}_2 \cap \mathcal{E}_3  \mid \filt_{I^-}(X')=f) \ge 1-e^{-\frac{1}{2}(\log n)^3}.
\end{equation}

It is clear that $\mathcal{E}_1$ occurs if clocks with source~$a_j$
trigger (in total) fewer than $\xi$ times in~$I$. These
clocks have  total rate $1+r \le 2r$, so by Corollary~\ref{cor:pchernoff}  it follows that
\begin{equation}\label{eqn:iteration-clique-E1}
\Pr(\mathcal{E}_1 \mid \filt_{I^-}(X') = f) \ge 1 -  e^{-(1+r)\len{I}/3} \ge 1 - e^{-(\log  n)^3}.
\end{equation}

Now consider any $h\in \{0,\ldots,\xi\}$  and
any $t_h\in [I^-,I^+]$.
If the events 
$\Tact{}{h}=t_h$,
$\filt_{I^-}(X')=f$, $\mathcal{E}_2^{h-1}$,
$\mathcal{E}_3^{h-1}$ 
and $\mathcal{E}_4^{h-1}$
are consistent
then let $f_h$ be any value of $\filt_{t_h}(X')$
 such that $\filt_{t_h}(X')= f_h$
  implies all of these events.
   Suppose that  $\filt_{t_h}(X')= f_h$
and consider how many non-mutants $K_j$ can have at time~$t_h$.
\begin{description}
\item[Case 1.] If $h=0$ and $K_j$ is active at~$I^-$ then $t_h=I^-$ and  it follows from 
the assumption in the statement of the lemma that $K_j$ has at most $\Delta$ non-mutants at time~$t_h$.
\item[Case 2.] Otherwise, if  $\Tinact{}{h-1} < I^+$, then
$[\Tinact{}{h-1},\Tact{}{h})$ is non-empty,  so it follows from $\mathcal{E}_4(h-1)$
that $K_j$ has at most one non-mutant at time~$t_h$.
\item[Case 3.] Otherwise,   $\Tinact{}{h-1} = t_h = I^+$  so it follows from $\mathcal{E}_3(h-1)$ 
that $K_j$ has at most $\Delta$ non-mutants at $\Tinact{}{h-1} = t_h$.
\end{description}
 In any of these three cases,
 Corollary~\ref{cor:clique-absorb-fast-cts}(\ref{it:clique-absorb-fast-cts-1})
implies that 
 \begin{equation}
 \label{eq:tuesA} 
\Pr(\mathcal{E}_2(h) \mid \filt_{t_h}(X') = f_h) \geq 1-4 e^{-{(\log n)}^3}.
\end{equation}
 In Case~1, Lemma~\ref{lem:clique-fill}(ii) 
implies that
\begin{equation}
\label{eq:tuesB}
\Pr(\mathcal{E}_3(h)  \mid \filt_{t_h}(X') = f_h) \geq 1- 2e^{-{(\log n)}^3}.
\end{equation} 
 In Case~2, Lemma~\ref{lem:clique-fill}(i)
implies that 
\begin{equation*}
\Pr(\mathcal{E}_3(h)  \mid \filt_{t_h}(X') = f_h) \geq 1- e^{-{(\log n)}^3}.
\end{equation*}
In Case~3, $\Tinact{}{h-1}= \Tact{}{h} = \Tinact{}{h}$ so $\mathcal{E}_3(h)$
is implied by $\mathcal{E}_3(h-1)$ hence
\begin{equation*}
\Pr(\mathcal{E}_3(h)  \mid \filt_{t_h}(X') = f_h)  =1.
\end{equation*} 
 
Equation~\eqref{eq:tuesB} gives the worst bound of the three cases.
Combining this with  \eqref{eq:tuesA} 
using a union bound, we have
$$\Pr(\overline{\mathcal {E}_2(h)} \cup \overline{\mathcal{E}_3(h)} \mid  \filt_{I^-}(X')=f, \mathcal{E}_2^{h-1},
\mathcal{E}_3^{h-1} ) \leq  6 e^{-{(\log n)}^3},$$
so 
$$\Pr(\mathcal{E}_2 \cap \mathcal{E}_3 \mid  \filt_{I^-}(X')=f) \geq 1 - (1+\xi)
6 e^{-{(\log n)}^3},$$
which, together with \eqref{eqn:iteration-clique-E1} gives  \eqref{eqn:iteration-clique-events}.
\end{proof}

The final result of Section~\ref{sec:cliques}
  shows that if an active clique is almost full of mutants, then with high probability it spawns no non-mutants at all onto $v^*$ before becoming inactive, and with even higher probability it doesn't spawn too many non-mutants onto $v^*$ before becoming inactive.

\begin{lemma}\label{lem:clique-nm-spawns}
\statelargeell Let $j \in [\ell]$ and $t_0\geq0$. Let $f$ be a possible value of $\filt_{t_0}(X')$ which implies that $k-\Delta \le |X'_{t_0} \cap K_j| \le k-1$. Let $S$ be the total number of non-mutants spawned onto $v^*$ by vertices in $K_j$ within $(t_0, \Tabs]$. Then the following statements hold.
\begin{enumerate}[(i)]
\item \label{it:clique-nm-spawns-1}$\Pr(S=0\mid \filt_{t_0}(X') = f)\geq 1-(\log n)^{10}/k$.
\item \label{it:clique-nm-spawns-2}$\Pr(S\leq (\log n)^{3}\mid \filt_{t_0}(X') = f)\geq1 - 7e^{-(\log n)^3}$.
\end{enumerate}
\end{lemma}
\begin{proof}
Let $\mathcal{E}_1$ be the event that for all $t \in [t_0, \Tabs]$, $|X'_t \cap K_j| \ge k-2\Delta$, and let $\mathcal{E}_2$ be the event that $\Tabs \le t_0 +  30c_r^2(\log n)^{6}$. 
Let $\mathcal{A}$ be the set of all clocks which have both source and target in 
$K_j \cup \{a_j\}$. Let $\Phi$ contain, for each clock $C \in \mathcal{A}$, a list of the times at which $C$ triggers in $(t_0, \Tabs]$. Note that by 
the definition of the megastar process
(Definition~\ref{def:megastar-process}), for all $t \in [t_0, \Tabs]$, $a_j \notin X'_t$. It follows that for all $t \in [t_0, \Tabs]$, $\Phi$ and $\filt_{t_0}(X')$ together uniquely determine $K_j \cap X'_t$. They therefore determine 
$\Ti{h}$ and $\Y(h)$ for all $h\geq 0$ and hence they determine whether
$\mathcal{E}_1$ and $\mathcal{E}_2$ occur. 

Let $\chi_0 = K_j \cap X_{t_0}'$. Consider any integer $y \ge 0$, any sets $\chi_1, \dots, \chi_y \subseteq K_j$ and any $t_1, \dots, t_y > t_0$. Suppose that $\varphi$ is a possible value of $\Phi$ such that $f$ and $\varphi$ together imply that $\mathcal{E}_1 \cap \mathcal{E}_2$ occurs, that $\Iabs = y$, and that for all $h\in\{0, \dots, y\}$, we have $\Ti{h} = t_h$ and $K_j \cap X_{t_h}' = \chi_h$. If $\filt_{t_0}(X') = f$ and $\Phi = \varphi$, then $\mathcal{E}_1 \cap \mathcal{E}_2$ occurs and so
\begin{align}\label{eqn:S_i}
|K_j \setminus \chi_h| &\le 2\Delta \textnormal{ for all }h\in \{0, \dots, y\},\\\label{eqn:t_x}
t_y &\le t_0 + 30c_r^2(\log n)^{6}.
\end{align}

For all $h \in \{1, \dots, y\}$, let 
\[S_h = |\{t \in (t_{h-1},t_h) \mid \vnc{u}{v^*} \textnormal{ triggers at time }t\textnormal{ for some }u \in K_j \setminus \chi_{h-1}\}|.\] 
We will show that, conditioning on $\Phi=\phi$ and $\filt_{t_0}(X')=f$,
$S_1, \dots, S_y$ are dominated above by independent Poisson variables $S_1', \dots, S_y'$, where $S_h'$ has parameter $\lambda_h = 2\Delta(t_h - t_{h-1})/k$. Indeed, consider any $h \in [y]$. Consider any integers $s_1, \dots, s_{h-1}$, and any possible value $f_{h-1}$ of $\filt_{t_{h-1}}(X')$ which is consistent with $\Phi = \varphi$ and which implies that $\filt_{t_0}(X') = f$ and $S_1 = s_1, \dots, S_{h-1} = s_{h-1}$.  

$S_h$ is independent of $\Phi$ by the definition of $\moranclocks(\megastar[\ell])$, since no clocks with target $v^*$ are contained in $\mathcal{A}$, $(t_{h-1},t_h)$ is a fixed time interval, and $\chi_{h-1}$ is a fixed set. Moreover, $S_h$ is independent of $\filt_{t_{h-1}}(X')$ by memorylessness. Thus, conditioned on $\filt_{t_{h-1}}(X') = f_{h-1}$ and $\Phi=\phi$, $S_h$ is simply a Poisson variable with parameter $(t_h-t_{h-1})(k-|\chi_{h-1}|)/k$, which is at most $\lambda_h$ by \eqref{eqn:S_i}. It therefore follows that for all $a \ge 0$,
\[\Pr(S_h \le a \mid \filt_{t_{h-1}}(X') = f_{h-1}, \Phi = \varphi) \ge \Pr(S_h' \le a),\]
and hence
\[\Pr(S_h \le a \mid \filt_{t_0}(X') = f, \Phi = \varphi, S_1 = s_1, \dots, S_{h-1} = s_{h-1}) \ge \Pr(S_h' \le a).\]
Thus conditioned on $\filt_{t_0}(X') = f$ and $\Phi = \varphi$, $S_1, \dots, S_y$ are dominated above by $S_1', \dots, S_y'$ as claimed.

Note that with probability $1$, no non-mutants are spawned onto $v^*$ at times $t_1, \dots, t_y$, and so $S = S_1 + \dots + S_y$. It follows from \eqref{eqn:t_x} that conditioned on $\filt_{t_0}(X') = f$ and $\Phi = \varphi$, $S$ is dominated above by a Poisson variable $S'$ with parameter 
\[\lambda_1 + \dots + \lambda_y = \frac{2\Delta}{k}\sum_{h=1}^y (t_h - t_{h-1}) = \frac{2\Delta}{k}(t_y-t_0) \le \frac{60c_r^3(\log n)^{9}}{k}.\]
By a union bound applied to Lemma~\ref{lem:clique-fill}(\ref{it:clique-fill-3}) and Corollary~\ref{cor:clique-absorb-fast-cts}(\ref{it:clique-absorb-fast-cts-1}), $\Pr(\mathcal{E}_1 \cap \mathcal{E}_2 \mid \filt_{t_0}(X')=f) \ge 1 - 6e^{-(\log n)^3}$. Thus for all $a \ge 0$ we have
\begin{align}\label{eqn:W-W'-2}
\Pr(S \le a \mid \filt_{t_0}(X') = f) &\ge \Pr(S \le a \mid \filt_{t_0}(X') = f, \mathcal{E}_1 \cap \mathcal{E}_2) - \Pr(\overline{\mathcal{E}_1 \cap \mathcal{E}_2} \mid \filt_{t_0}(X')=f)\nonumber\\
& \ge \Pr(S' \le a) - 6e^{-(\log n)^3}.
\end{align}

It is immediate from \eqref{eqn:W-W'-2} using~\eqref{eq:ebounds} that 
\begin{align*}
\Pr(S = 0 \mid \filt_{t_0}(X') = f) &\ge e^{-60c_r^3(\log n)^{9}/k} - 6e^{-(\log n)^3} \ge 1 - \frac{(\log n)^{10}}{k},
\end{align*}
and so part~(\ref{it:clique-nm-spawns-1}) of the result holds. Moreover, by \eqref{eqn:W-W'-2} combined with Corollary~\ref{cor:pchernoff-3}, we have
\[\Pr(S \le (\log n)^{3} \mid \filt_{t_0}(X') = f) \ge \Pr(S' \le (\log n)^{3}) - 6e^{-(\log n)^3} \ge 1 - 7e^{-(\log n)^3},\]
and so part~(\ref{it:clique-nm-spawns-2}) of the result holds.
\end{proof}

\subsection{Filling cliques}\label{sec:fill-cliques}

Recall from Definition~\ref{def:megastar-process} that $X'_0$ is the set containing a single initial mutant, and write $X_0' = \{x_0\}$. Because of the
megastar's symmetry, without loss of
generality we may assume that $x_0\in R_1\cup K_1 \cup \{a_1, v^*\}$. In Section~\ref{sec:fill-cliques}, we will further restrict our attention to the case where $x_0$ is in a reservoir, i.e., $x_0\in R_1$.

\subsubsection{The first clique fills with mutants}\label{sec:first-clique}

In Section~\ref{sec:first-clique}, we will show that if the initial mutant of $X'$ lies in the reservoir $R_1$, then with high probability $K_1$ is almost full of mutants at time $n$ (see Lemma~\ref{lem:fill-first-clique}). We first prove an ancillary lemma.

\begin{lemma}\label{lem:into-clique}
\statelargeell Suppose $X_0' \subseteq R_1$, and write $X_0' = \{x_0\}$. Let $\mathcal{E}_1$ be the event that $\vnc{v^*}{x_0}$ does not trigger in $[0, 17m(\log n)^2]$. Let $t_0 \in [0, 17m((\log n)^2 - 1)]$, and let $f$ be a possible value of $\filt_{t_0}(X')$ which is consistent with $\mathcal{E}_1$. Then
\[\Pr(\textnormal{there exists } t_1 \in [t_0, t_0+17m] \textnormal{ such that }K_1 \subseteq X_{t_1}' \mid \mathcal{E}_1, \filt_{t_0}(X')=f) \ge \frac{r-1}{
12r}.\]
\end{lemma}
\begin{proof} Let 
    \[\calC_{a_1} = \{\vnc{u}{a_1} \mid u \in R_1\} \cup \{\vmc{a_1}{v} \mid v \in K_1\}.\]
    For $h\in\{0, \dots, 8m-1\}$, let
    \begin{equation*}
        T_h = \min\{t > t_0+2h\mid \vmc{x_0}{a_1}
                                      \text{ triggers at } t \mbox{ or } t=t_0+2h+1\}.
    \end{equation*}
    Let $\calE_2(h)$ be the event that $\vmc{x_0}{a_1}$ triggers at time $T_h$.  We have
    \begin{equation}\label{eqn:into-clique-3}
    \Pr(\calE_2(h)) = 1 - e^{-r} >  \tfrac12.
    \end{equation}
    
    Let
    \begin{equation*}
        T'_h = \min
                  \{t > T_h\mid \text{a clock in }\calC_{a_1}
                                      \text{ triggers at } t \mbox{ or } t=T_h+1\}.
    \end{equation*}
    Let $\calE'_2(h)$ be the event that some mutant clock with source $a_1$ triggers at time $T'_h$.
    Note that $\calE'_2(h)$ is independent of $\calE_2(h)$. The probability that some clock in $\calC_{a_1}$ triggers in $(T_h, T_h+1]$ is $1-e^{-r-m}$, and the probability that the first clock in $\calC_{a_1}$ to trigger in $(T_h, \infty)$ is a mutant clock with source $a_1$ is $r/(r+m)$. Hence by \eqref{eqn:into-clique-3} and a union bound,
    \begin{equation}\label{eqn:into-clique-2}
        \Pr(\calE_2(h) \cap \calE'_2(h))
            \geq \frac{1}{2}\left(\frac{r}{r+m} - e^{-r-m}\right)
            \geq \frac{1}{4(r+m)}
            \geq \frac{1}{8m}.
    \end{equation}

	Note that for all $h$, the event $\calE_2(h)\cap \calE_2'(h)$ depends only on fixed clocks in
$\calC_{a_1}\cup \{\vmc{x_0}{a_1}\}$ 
over the fixed interval $(t_0+2h, t_0+2h+2]$. As such, the events $\calE_2(h) \cap \calE'_2(h)$ are independent from each other, from $\mathcal{E}_1$ and from $\filt_{t_0}(X')$. It follows from \eqref{eqn:into-clique-2} that
	\begin{align}\label{eqn:into-clique-1}
	\Pr\left(\bigcup_{h=0}^{8m-1}(\mathcal{E}_2(h) \cap \mathcal{E}'_2(h)) \,\Bigg|\, \mathcal{E}_1, \filt_{t_0}(X')=f \right) 
	&\ge 1 - \left(1 - \frac{1}{8m} \right)^{8m}
	\ge 1-e^{-1} \ge \frac{1} 
{2}.
	\end{align}

	Let $\mathcal{E}_2$ be the event that there exists $t \in (t_0, t_0+16m]$ such that $|K_1 \cap X_t'| \ge 1$. 	
	We now show that if $\mathcal{E}_1 \cap \mathcal{E}_2(h) \cap \mathcal{E}'_2(h)$ occurs for some $h \in \{0, \dots, 8m-1\}$, then $\calE_2$ occurs. Since $\calE_1$ occurs and $T_h \le 17m(\log n)^2$, we have $x_0 \in X_{T_h}'$. Since $\calE_2(h)$ occurs, by the definition of $X'$, either $a_1$ is a mutant at time $T_h$ or $K_1$ is active at time $T_h \le t_0+16m$ (and hence $\mathcal{E}_2$ occurs). If $a_1$ is a mutant at time $T_h$, then since $\calE'_2(h)$ occurs, $a_1$ spawns a mutant at time $T_h' \le t_0+16m$, and so $\calE_2$ occurs. Thus in all cases, if $\calE_1 \cap \calE_2(h) \cap \calE'_2(h)$ occurs then $\mathcal{E}_2$ occurs. Hence by \eqref{eqn:into-clique-1},
	\begin{equation}\label{eqn:into-clique-4}
	\Pr(\mathcal{E}_2 \mid \mathcal{E}_1, \filt_{t_0}(X')=f) \ge \Pr\left(\bigcup_{h=0}^{8m-1}(\mathcal{E}_2(h) \cap \mathcal{E}'_2(h)) \,\Bigg|\, \mathcal{E}_1, \filt_{t_0}(X')=f \right) \ge \frac{1} 
{2}.	
	\end{equation}

Let $$T = \min\{t \geq t_0 \mid \mbox{$t=t_0+16m$ or $|K_1 \cap X_t'| \ge 1$}\}.$$
Let $\mathcal{E}_3$ be the event that $\Tabsnoh{T,1} \le T+m$ and
$K_1$ is full of mutants at time $\Tabsnoh{T,1}$.
Let $\mathcal{A}$ be the set of all clocks which have both source and target in $K_1 \cup \{a_1\}$.
For every $t\geq 0$, let $\Phi_t$ be a random variable which contains, for each clock $C\in \mathcal{A}$,
a list of the times at which $C$ triggers in $(t, t+m]$.
Now consider any $t\geq t_0$. Let $f'$
be any possible value of $\filt_{t}(X')$ such that the following events are consistent ---
$\filt_{t_0}(X')=f$, $\filt_{t}(X')=f'$, $\mathcal{E}_2$, $T=t$, and $\mathcal{E}_1$.
Note that the first four of these events are determined by $\filt_{t}(X')=f'$.
Conditioned on $\filt_{t}(X')=f'$, $\mathcal{E}_1$ and $\Phi_t$
are independent.
Also, conditioned on $\filt_{t}(X')=f'$, $\mathcal{E}_3$ is determined by $\Phi_t$
(since the definition of the megastar process ensures that $a_1$ is a non-mutant throughout $[t, \Tabsnoh{t,1}]$),
so, conditioned on $\filt_{t}(X')=f'$, $\mathcal{E}_3$ is independent of~$\mathcal{E}_1$.
Now, applying
 	Corollary~\ref{cor:clique-absorb-fast-cts}, we have 
	\begin{equation*}
	\Pr(\mathcal{E}_3 \mid \mathcal{E}_1, \filt_{t}(X')=f') = \Pr(\mathcal{E}_3 \mid \filt_{t}(X')=f') \ge \frac{r-1}{6r}.
	\end{equation*}
Thus, 
\begin{equation}
\label{eq:inter}
 \Pr(\mathcal{E}_3 \mid \mathcal{E}_1, \filt_{t_0}(X')=f, \mathcal{E}_2)   \ge \frac{r-1}{6r}.
 \end{equation}
 
	It therefore follows from \eqref{eqn:into-clique-4} and \eqref{eq:inter} that
	\begin{align*}
	\Pr(\mathcal{E}_3 \mid \mathcal{E}_1, \filt_{t_0}(X')=f) 
	&\ge \Pr(\mathcal{E}_3 \mid \mathcal{E}_1 \cap \mathcal{E}_2, \filt_{t_0}(X')=f) \cdot \Pr(\mathcal{E}_2 \mid \mathcal{E}_1, \filt_{t_0}(X')=f)
	\ge \frac{r-1} 
{12r},
	\end{align*}
	and so the result follows.
\end{proof}

We are now able to prove Lemma~\ref{lem:fill-first-clique}.
\begin{lemma}
\label{lem:fill-first-clique}
	\statelargeell Suppose $X_0' \subseteq R_1$. Then with probability at least $1 - 
19	
	(\log n)^2/\ell$, $K_1$ contains at most $\Delta$ non-mutants at time $n$.
\end{lemma}
\begin{proof}
	Let $\mathcal{E}_1$ be the event that $\vnc{v^*}{x_0}$ does not trigger in $[0, 17m(\log n)^2]$. For all $i\in\Zzero$, let $\mathcal{E}_2(i)$ be the event that there exists 
$t \in [17im, 17(i+1)m]$ 
such that $K_1 \subseteq X_t'$. Let $\mathcal{E}_2 = \mathcal{E}_2(0) \cup \dots \cup \mathcal{E}_2(\floor{(\log n)^2}-1)$. By Lemma~\ref{lem:into-clique}, for all integers $i$ with $0 \le i \le \floor{(\log n)^2}-1$ and all possible values $f_i$ of $\filt_{17im}(X')$ consistent with $\mathcal{E}_1$, we have
	\[\Pr(\mathcal{E}_2(i) \mid \mathcal{E}_1, \filt_{17im}(X')=f_i) \ge 
\frac{r-1}{12r}.\]
	Since $f_i$ determines $\mathcal{E}_2(0), \dots, \mathcal{E}_2(i-1)$, it follows that
	\begin{equation}\label{eqn:fill-first-clique-1}
	\Pr(\mathcal{E}_2 \mid \mathcal{E}_1) \ge 1 - \left(1 - \frac{r-1}{
12r} \right)^{\floor{(\log n)^2}}.
	\end{equation}
	Moreover,
by \eqref{eq:ebounds},
	\[\Pr(\mathcal{E}_1) = e^{-17(\log n)^2/\ell} \ge 1 - \frac{17(\log n)^2}{\ell}.\]
	It therefore follows by~\eqref{eqn:fill-first-clique-1} that
	\begin{equation}\label{eqn:fill-first-clique-2}
	\Pr(\mathcal{E}_2) \ge \Pr(\mathcal{E}_2 \mid \mathcal{E}_1)\,\Pr(\mathcal{E}_1) \ge 1 - \frac{18(\log n)^2}{\ell}.
	\end{equation}
	
Now let $\mathcal{E}_3$ be the event that $|K_1 \setminus X'_n| \le \Delta$. 
Let $$T = \min\{t \geq 0 \mid \mbox{$t =17m(\log n)^2$ or $K_1 \subseteq X'_t$} \}.$$
Consider any $t \in [0, 17m(\log n)^2]$ and any possible value $f'$ of $\filt_t(X')$ which
is consistent with $\mathcal{E}_2$ and $T=t$.
Then by Lemma~\ref{lem:iteration-cliques}(\ref{it:iteration-cliques-2}) applied to the interval $(t,n]$, with probability at least $1-e^{-\frac{1}{2}(\log n)^3}$ conditioned on $\filt_t(X')=f'$, $\mathcal{E}_3$ occurs. Thus, using \eqref{eqn:fill-first-clique-2}, we obtain
	\[\Pr(\mathcal{E}_3) \ge \Pr(\mathcal{E}_3 \mid \mathcal{E}_2)\, \Pr(\mathcal{E}_2) \ge 1 - 
\frac{19(\log n)^2}{\ell},\]
	as required.
\end{proof}

\subsubsection{The other cliques become almost full}\label{sec:other-cliques}

In Section~\ref{sec:other-cliques}, we will show that if the initial mutant of $X'$ lies in a reservoir, without loss of generality $R_1$, then with high probability $K_1, \dots, K_\ell$ are all almost full of mutants at time $n^8$ (see Lemma~\ref{lem:fill-all-cliques}). The following lemma will be the linchpin of the proof.
\begin{lemma}\label{lem:try-fill-other-clique}
	\statelargeell Let $t_0 \ge 0$, and let $j \in [\ell-1]$. Let $f$ be a possible value of $\filt_{t_0}(X')$ which implies that for all $j' \in [j]$, $K_{j'}$ contains at most $\Delta$ non-mutants at time $t_0$. Then
	\[\Pr(\textnormal{for all }j' \in [j], |K_{j'} \setminus X'_{t_0+20c_r^2 k}| \le \Delta \mid \filt_{t_0}(X')=f) \ge 1-e^{-(\log n)^2}.\]
	Moreover,
	\[\Pr(|K_{j+1} \setminus X'_{t_0+20c_r^2 k}| \le \Delta \mid \filt_{t_0}(X')=f) \ge 
1/n^6.\] \end{lemma}
\begin{proof}
	The first part of the result follows immediately from Lemma~\ref{lem:iteration-cliques}(\ref{it:iteration-cliques-2}) and a union bound over all $j \in [j']$ by taking $I = (t_0, t_0+20c_r^2 k]$. 
	
	We now define some stopping times. Let
	\[T_1 = \min \{t \ge t_0 \mid K_1 \subseteq X'_t \mbox{ or } t=t_0 + 
30c_r^2 (\log n)^{6}\}.\]
	Let $T_2$ be the fourth time after $T_1$ at which a clock in $\moranclocks(\megastar[\ell])$ triggers. Let
	\[T_3 = \min  \{t \ge T_2 \mid K_{j+1} \subseteq X'_t \mbox{ or } t=T_2 + 10c_r^2k\}.\]
 	In addition, we define the following events.
	\begin{itemize}
		\item $\mathcal{E}_1$: $K_1 \subseteq X'_{T_1}$.
		\item $\mathcal{E}_2$: for some $v_1 \in K_1$, $v_2 \in R_{j+1}$ and $v_3 \in K_{j+1}$, the first four clocks in $\moranclocks(\megastar[\ell])$ to trigger in $(T_1, \infty)$ are $\vmc{v_1}{v^*}$, $\vmc{v^*}{v_2}$, $\vmc{v_2}{a_{j+1}}$ and $\vmc{a_{j+1}}{v_3}$, in that order.
		\item $\mathcal{E}_2'$: $T_2 \le T_1 + 1$. 
		\item $\mathcal{E}_3$: $K_{j+1} \subseteq X'_{T_3}$. 
		\item $\mathcal{E}_4$: $|K_{j+1} \setminus X'_{t_0+20c_r^2 k}| \le \Delta$.
	\end{itemize}
Our goal is to prove $\Pr(\mathcal{E}_4 \mid \filt_{t_0} = f) \geq 1/n^5$.

	By 
Corollary~\ref{cor:clique-absorb-fast-cts}(i), 
we have
	\begin{equation}\label{eqn:try-fill-other-clique-1}
	\Pr(\mathcal{E}_1 \mid \filt_{t_0}(X')=f) \ge 1-
4e^{-(\log n)^3}.	
	\end{equation}
	It is immediate that
	\begin{equation}\label{eqn:try-fill-other-clique-2}
	\Pr(\mathcal{E}_2 \mid \mathcal{E}_1, \filt_{t_0}(X')=f) = \frac{r}{n(1+r)} \cdot \frac{r}{\ell n(1+r)} \cdot \frac{r}{n(1+r)} \cdot \frac{r}{n(1+r)} \ge \frac{1}{n^5}.
	\end{equation}
For example, the first $r/(n(1+r))$ factor comes from the fact that 
the total rate of all clocks is $n(1+r)$ but the total rate of all mutant clocks with 
source in~$K_1$ and target~$v^*$
is $r$ since there are $k$ such clocks, each with rate $r/k$.

	Moreover, $\mathcal{E}_2'$ occurs if for all $h \in \{0,1,2,3\}$, at least one clock in $\moranclocks(\megastar[\ell])$ triggers in the interval $(T_2+h/4, T_2+(h+1)/4]$. Hence 
	\begin{equation}\label{eqn:try-fill-other-clique-3}
	\Pr(\mathcal{E}_2' \mid \mathcal{E}_1, \filt_{t_0}(X')=f) \ge 1 - 4e^{-(1+r)n/4}.
	\end{equation}
	It follows from \eqref{eqn:try-fill-other-clique-1}--\eqref{eqn:try-fill-other-clique-3} and a union bound that
	\begin{align}\label{eqn:fill-a-clique-1}
	\Pr(\mathcal{E}_1 \cap \mathcal{E}_2 \cap \mathcal{E}_2' \mid \filt_{t_0}(X')=f) &\ge \Pr(\mathcal{E}_1 \mid \filt_{t_0}(X')=f)\cdot \nonumber\\
	&\qquad \qquad  (\Pr(\mathcal{E}_2 \mid \mathcal{E}_1, \filt_{t_0}(X')=f) - \Pr(\overline{\mathcal{E}_2'} \mid \mathcal{E}_1, \filt_{t_0}(X')=f))\nonumber\\
	&\ge \big(1-
4 e^{-{(\log n)}^3}\big)
(1/n^5 
- 4e^{-(1+r)n/4}) \ge 1/(2n^5).
	\end{align}
	
	Now consider any $t_2 > t_0$ and any possible value $f_2$ of $\filt_{t_2}(X')$ which implies that $\filt_{t_0}(X')=f$, $T_2 = t_2$, and that $\mathcal{E}_1 \cap \mathcal{E}_2 \cap \mathcal{E}_2'$ occurs. 
Note that, if $\filt_{t_2}(X') = f_2$ then, since $\mathcal{E}_1 \cap \mathcal{E}_2 \cap \mathcal{E}_2'$ occurs, we must have $|K_{j+1} \cap X'_{t_2}| \ge 1$. It follows from Corollary~\ref{cor:clique-absorb-fast-cts} that
	\[\Pr(\mathcal{E}_3 \mid \mathcal{E}_1 \cap \mathcal{E}_2 \cap \mathcal{E}_2', \filt_{t_0}(X')=f) \ge \frac{r-1}{6r}.\]
	It therefore follows from \eqref{eqn:fill-a-clique-1} that
	\begin{equation}\label{eqn:fill-a-clique-2}
	\Pr(\mathcal{E}_1 \cap \mathcal{E}_2 \cap \mathcal{E}_2' \cap \mathcal{E}_3 \mid \filt_{t_0}(X') = f) 
	\ge \frac{r-1}{12rn^5}.
	\end{equation}
	
	Finally, consider any $t_3 \ge t_0$ and any possible value $f_3$ of $\filt_{t_3}(X')$ which implies that $\filt_{t_0}(X') = f$, that $T_3 = t_3$, and that $\mathcal{E}_1 \cap \mathcal{E}_2 \cap \mathcal{E}_2' \cap \mathcal{E}_3$ occurs. 
Since $\filt_{t_3}(X') = f_3$ implies that $\mathcal{E}_2'$ occurs, we have 
	\[t_3 \le t_0 +  
	30c_r^2(\log n)^{6} 
	+ 1 + 10c_r^2k\le t_0 + 20c_r^2k - (\log n)^{7}.\] 
If $\filt_{t_3}(X')=f_3$ then $\mathcal{E}_3$ occurs so $K_{j+1} \subseteq X'_{t_3}$, which 
obviously implies  $|K_{j+1} \setminus X'_{t_3}| \le \Delta$. 
It therefore follows from Lemma~\ref{lem:iteration-cliques}(\ref{it:iteration-cliques-2}) applied to $(t_3, t_0 + 20c_r^2 k]$ that
	\[\Pr(\mathcal{E}_4 \mid \mathcal{E}_1 \cap \mathcal{E}_2 \cap \mathcal{E}_2' \cap \mathcal{E}_3, \filt_{t_0}(X') = f) \ge 1 - e^{-\frac{1}{2}(\log n)^3},\]
	and therefore by \eqref{eqn:fill-a-clique-2},
	\[\Pr(\mathcal{E}_4 \mid \filt_{t_0}(X')=f) \ge \Pr(\mathcal{E}_1 \cap \mathcal{E}_2 \cap \mathcal{E}_2' \cap \mathcal{E}_3 \cap \mathcal{E}_4 \mid \filt_{t_0}(X')=f)\ge 1/n^6.\]
	The second part of the result therefore follows.
\end{proof}

We now apply Lemma~\ref{lem:try-fill-other-clique} repeatedly to prove the following.

\begin{lemma}\label{lem:fill-other-clique}
	\statelargeell Let $t_0 \ge 0$, and let $j\in [\ell-1]$. Let $f$ be a possible value of $\filt_{t_0}(X')$ which implies that for all $j' \in [j]$, $K_{j'}$ contains at most $\Delta$ non-mutants at time $t_0$. Then
	\[\Pr\big(\textnormal{for all }j' \in [j+1],\, |K_{j'} \setminus X'_{t_0+20c_r^2 n^7k}| \le \Delta \mid \filt_{t_0}(X')=f\big) \ge 1-n^8e^{-(\log n)^2}.\]
\end{lemma}
\begin{proof}
	For all $i\in\Zone$, let $t_i = t_0 + 20c_r^2 ki$. Let $\mathcal{E}_1(i)$ be the event that for all $j' \in [j]$, $|K_{j'} \setminus X'_{t_i}| \le \Delta$. Let $\mathcal{E}_2(i)$ be the event that $|K_{j+1} \setminus X'_{t_i}| \le \Delta$. For convenience, let $\mathcal{F}$ be the event that $\filt_{t_0}(X') = f$. By a union bound, we have
	\begin{align}\label{eqn:fill-other-clique-1}
	\Pr\big(\cup_{i=1}^{n^7}\, \mathcal{E}_2(i) \mid \mathcal{F} \big) 
		&\ge 1 - \Pr\left( \cap_{i=1}^{n^7} \big(\overline{\mathcal{E}_2(i)} \cap \mathcal{E}_1(i) \big)\,\big|\, \mathcal{F} \right)
		- \Pr\left( \cup_{i=1}^{n^7} \overline{\mathcal{E}_1(i)} \mid \mathcal{F} \right).
	\end{align}
	Moreover,
	\begin{align*}
	\Pr\Big( \cap_{i=1}^{n^7} \big(\overline{\mathcal{E}_2(i)} \cap \mathcal{E}_1(i) \big)  \,\big|\, \mathcal{F}\Big)
		&= \prod_{i=1}^{n^7}\Pr\left(\overline{\mathcal{E}_2(i)} \cap \mathcal{E}_1(i) \,\big|\, \cap_{s=1}^{i-1} \big(\overline{\mathcal{E}_2(s)} \cap \mathcal{E}_1(s)\big) \cap \mathcal{F} \right)\\
		&\le \prod_{i=1}^{n^7}\Pr\left(\overline{\mathcal{E}_2(i)} \,\big|\, \cap_{s=1}^{i-1} \big(\overline{\mathcal{E}_2(s)} \cap \mathcal{E}_1(s)\big) \cap \mathcal{F} \right).
	\end{align*}
	Since for all $i$, the event $\bigcap_{s=1}^{i-1} (\overline{\mathcal{E}_2(s)} \cap \mathcal{E}_1(s)) \cap \filt$ is determined by $\filt_{t_{i-1}}(X')$, by Lemma~\ref{lem:try-fill-other-clique} it follows that
	\begin{equation}\label{eqn:fill-other-clique-2}
	\Pr\Big( \cap_{i=1}^{n^7} \big(\overline{\mathcal{E}_2(i)} \cap \mathcal{E}_1(i) \big)  \,\big|\, \mathcal{F} \Big) \le \left(1 - \frac{1}{n^6}\right)^{n^7} \le e^{-n}.
	\end{equation}
	We also have
	\begin{equation*}
	\Pr\big( \cup_{i=1}^{n^7} \overline{\mathcal{E}_1(i)}  \,\big|\, \mathcal{F} \big) 
	\le \sum_{i=1}^{n^7}\Pr\big( \overline{\mathcal{E}_1(i)} \,\big|\, \cap_{s=1}^{i-1} \mathcal{E}_1(s) \cap \mathcal{F}\big).
	\end{equation*}
	Since for all $i$, the event $\cap_{s=1}^{i-1}\, \mathcal{E}_1(s) \cap \filt$ is determined by $\filt_{t_{i-1}}(X')$, it follows from Lemma~\ref{lem:try-fill-other-clique} that
	\begin{equation}\label{eqn:fill-other-clique-3}
	\Pr\big( \cup_{i=1}^{n^7} \overline{\mathcal{E}_1(i)}  \,\big|\, \mathcal{F} \big) \le n^7e^{-(\log n)^2}.
	\end{equation}
	Hence by \eqref{eqn:fill-other-clique-1}, \eqref{eqn:fill-other-clique-2} and \eqref{eqn:fill-other-clique-3}, we have
	\begin{equation}\label{eqn:fill-other-clique-4}
	\Pr\big(\cup_{i=1}^{n^7} \mathcal{E}_2(i)  \,\big|\, \mathcal{F} \big) \ge 1 - 2n^7e^{-(\log n)^2}.
	\end{equation}
	
	Now, let $i' \in [n^7-1]$. Since $\mathcal{E}_2(i')\cap\filt$ is determined by 
$\filt_{t_{i'}}(X')$, 
it follows by Lemma~\ref{lem:iteration-cliques}(ii) applied to the interval 
$(t_{i'}, t_{n^7}]$ 
that
	\[\Pr(\mathcal{E}_2(n^7) \mid \mathcal{E}_2(i') \cap \filt) \ge 1 - e^{-\frac{1}{2}(\log n)^3}.\]
	Hence by \eqref{eqn:fill-other-clique-4} we have $\Pr(\mathcal{E}_2(n^7) \mid \mathcal{F}) \ge 1 - 3n^7e^{-(\log n)^2}$. In addition, by \eqref{eqn:fill-other-clique-3}, we have $\Pr(\mathcal{E}_1(n^7)\mid \mathcal{F}) \ge 1 - n^7e^{-(\log n)^2}$. By a union bound, it follows that
	\[\Pr(\mathcal{E}_1(n^7) \cap \mathcal{E}_2(n^7)\mid \mathcal{F}) \ge 1 - n^8e^{-(\log n)^2},\]
	and so the result follows.
\end{proof}

The proof of Lemma~\ref{lem:fill-all-cliques}, the goal of Section~\ref{sec:fill-cliques}, now follows easily from Lemma~\ref{lem:fill-other-clique}.

\begin{lemma}\label{lem:fill-all-cliques}
	\statelargeell Suppose $X_0' \subseteq R_1$. Then with probability at least 
$1-20(\log n)^2/\ell$, 
for all $j \in [\ell]$, $K_j$ contains at most $\Delta$ non-mutants at time $n^8$.
\end{lemma}
\begin{proof}
	For each positive integer $i$, let $t_i = n + (i-1)20c_r^2 n^7k$. Let $\mathcal{E}_i$ be the event that at time $t_i$, for all $j \in [i]$ we have $|X'_{t_i} \setminus K_j| \le \Delta$. By Lemma~\ref{lem:fill-first-clique}, we have $\Pr(\mathcal{E}_1) \ge 1- 
	19(\log n)^2/\ell$. 
For all $i \in \{2, \dots, \ell\}$, by Lemma~\ref{lem:fill-other-clique}  
(applied with $j=i-1$ starting at $t_{i-1}$), 
we have
	\[\Pr(\mathcal{E}_i \mid \mathcal{E}_1 \cap \dots \cap \mathcal{E}_{i-1}) \ge 1 - n^8e^{-(\log n)^2}.\]
	It follows that
	\begin{align*}
	\Pr(\mathcal{E}_{\ell}) 
	&\ge 1 - \sum_{i=1}^{\ell} \Pr\big(\overline{\mathcal{E}_i} \mid \mathcal{E}_1 \cap \dots \cap \mathcal{E}_{i-1}\big)
	\ge 1 -  \frac{19(\log n)^2}{\ell} 
- \ell n^8e^{-(\log n)^2}. 
 	\end{align*}
	It therefore follows by Lemma~\ref{lem:iteration-cliques}(\ref{it:iteration-cliques-2}) applied to the interval $(t_{\ell}, n^8]$ combined with a union bound that with probability at least 
$$ 1 - \frac{19(\log n)^2}{\ell} 
- \ell n^8e^{-(\log n)^2}	 - \ell e^{-\frac{1}{2}(\log n)^3} \ge 1 - 20(\log n)^2/\ell,$$
 for all $j \in [\ell]$, $K_j$ contains at most $\Delta$ mutants at time $n^8$ as required.
\end{proof}
 
\subsection{Filling reservoirs from cliques}\label{sec:megastar-fixates}
\subsubsection{Setting up an iteration scheme --- Proof of Lemma~\ref{lem:megastar-fixates}}\label{sec:iteration}

In this section, we outline an iterative argument which, together with Lemma~\ref{lem:fill-all-cliques}, will allow us to prove our key lemma, Lemma~\ref{lem:megastar-fixates}.

\begin{definition}\label{def:iteration-scheme}
For all $i\in\Zzero$, let
\begin{align*}
I_i^-    &= n^8 + in(\log n)^3,\\
I_i^+    &= n^8 + (i+1)n(\log n)^3,\\ 
\alpha_i &= \floor{\max\{(2\log n)^{2}, m/(2\log n)^{2i}\}},\\
\beta_i  &= \floor{\max\{(2\log n)^{2}, \ell m/(2\log n)^{2i}\}}.
\end{align*}
Consider any $t\geq n^8$. 
Let $i$ be the integer such that $t\in [I_i^-,I_i^+)$.
Let $f$ be a possible value of $\filt_t(X')$. 
We say that $f$ is \emph{good}   if
the event $\filt_t(X')=f$ implies that the following events occur.
\begin{itemize}
\item  \Ptotalmut{i}: $|(R_1 \cup \dots \cup R_{\ell}) \setminus X'_{I_i^-}| \le \beta_i$.
\item  \Presmut{i}: For all $j \in [\ell]$, $|R_{j} \setminus X'_{I_i^-}| \le \alpha_i$.
\item \Pcliquefull{i}: For all $j \in [\ell]$, $|K_j \setminus X'_{I_i^-}| \le \Delta$.
\item \Pmostcliques{i}: For all but at most $\beta_i$ choices of $j\in[\ell]$, $R_{j} \cup \{a_j\} \cup K_j \subseteq X'_{I_i^-}$. \defend{}
\end{itemize} 
\end{definition}

Each interval $(I_i^-,I_i^+]$  
corresponds to a phase of our iterative argument, which we state in Lemma~\ref{lem:main-iteration}. At the end of each interval, the  number of non-mutants in each reservoir should drop by a factor of at least $(2\log n) ^{2}$ to a minimum of $\floor{(2\log n)^{2}}$, as should the   number of non-mutants in all reservoirs. In addition, every clique should remain almost full of mutants, and if there are fewer than $\ell$ non-mutants left in reservoirs then many branches should be completely full of mutants.  

\newcommand{\statemainiteration}{
\statelargeell Let $i$ be a non-negative integer, and let $f$ be a good possible value of $\filt_{I_i^-}(X')$.  Then
\[\Pr( \mbox{$\filt_{I_{i}^+}(X')$ is good}
 \mid \filt_{I_i^-}(X') = f) \ge 1 - 15e^{-(\log n)^2}.\]
Moreover, if $\beta_i = \floor{(2\log n)^{2}}$, then
\[\Pr(X_{I_i^+}' = V(\megastar[\ell]) \mid \filt_{I_i^-}(X') = f) \ge 1/2.\]
}
\begin{lemma}\label{lem:main-iteration}
\statemainiteration
\end{lemma}

Assuming Lemma~\ref{lem:main-iteration} for the moment, we give the proof of our key lemma, Lemma~\ref{lem:megastar-fixates}, which we restate here for convenience.
{\renewcommand{\thetheorem}{\ref{lem:megastar-fixates}}
\begin{lemma}
\statemegastarfixates
\end{lemma}
\addtocounter{theorem}{-1}
}
\begin{proof}
By the symmetry of the megastar, we may assume that $x_0\in R_1$. 

For a non-negative integer $i$, let $\mathcal{E}_i$ be the event that 
$\filt_{I_i^{-}}(X')$
 is good and let $\mathcal{E}_i'$ be the event that $X_{I_i^{+}}'=V(\megastar[\ell])$. To prove the lemma, it clearly suffices to show that 
\begin{equation}\label{eq:deikse24deikse}
\Pr(\cup^{2n}_{i=0}\, \mathcal{E}_i')\geq 1-\frac{42(\log n)^{2}}{\sqrt{n}}.
\end{equation}

By a union bound, we have
\begin{equation}\label{eq:simplemanip}
\Pr(\cup^{2n}_{i=0}\, \mathcal{E}_i')\geq 1-\Pr\Big(\cap^{2n}_{i=0} \big(\mathcal{E}_i\cap \overline{\mathcal{E}_i'}\big)\Big)-\Pr\big(\cup^{2n}_{i=0} \overline{\mathcal{E}_i}\big).
\end{equation}
For all $i \in [2n]$, $\filt_{I_i^-}(X')$ determines $\mathcal{E}_0 \cap \dots \cap \mathcal{E}_{i-1}$ and thus, by Lemma~\ref{lem:main-iteration}, we have
\begin{equation}\label{eq:timeiabv}
\Pr(\mathcal{E}_i \mid \mathcal{E}_0 \cap \dots \cap \mathcal{E}_{i-1}) \ge 1 - 15e^{-(\log n)^2}.
\end{equation}
Also, since $x_0\in R_1$, for $i=0$, we have by Lemma~\ref{lem:fill-all-cliques} that
\begin{equation}\label{eq:time0abv}
\Pr(\mathcal{E}_0)\geq 1-\frac{20(\log n)^2}{\ell}\geq 1-\frac{40(\log n)^2}{\sqrt{n}},
\end{equation}
since 
\Ptotalmut{0},
\Presmut{0} and \Pmostcliques{0} hold trivially  (from $\alpha_0= m$ and $\beta_0= \ell m$). Combining \eqref{eq:timeiabv} and \eqref{eq:time0abv} we obtain that 
\begin{equation}\label{eq:intermediate456}
\Pr\big(\cap^{2n}_{i=0} \mathcal{E}_i\big)\geq 1-\frac{41(\log n)^2}{\sqrt{n}}, \mbox{ so } \Pr\big(\cup^{2n}_{i=0} \overline{\mathcal{E}_i}\big)\leq \frac{41(\log n)^2}{\sqrt{n}}.
\end{equation}

For $i\geq n$, we have that $\beta_i=\floor{(2\log n)^{2}}$. Since $\filt_{I_i^-}(X')$ determines $\bigcap_{h=0}^{i-1}(\mathcal{E}_i \cup \overline {\mathcal{E}_i'})$ for all $i\in\{n+1,\hdots,2n\}$, it follows from Lemma~\ref{lem:main-iteration} that 
\[\Pr\Big(\big(\mathcal{E}_i\cap \overline{\mathcal{E}_i'}\big)\mid \big(\mathcal{E}_0\cap \overline{\mathcal{E}_0'}\big) \cap \dots \cap \big(\mathcal{E}_{i-1}\cap \overline{\mathcal{E}_{i-1}'}\big)\Big)\leq 1/2.\]
Hence, we obtain that  
\begin{equation}\label{eq:intermediate654}
\Pr\Big(\cap^{2n}_{i=0} \big(\mathcal{E}_i\cap \overline{\mathcal{E}_i'}\big)\Big)\leq 1/2^{n}.
\end{equation}
Plugging \eqref{eq:intermediate456} and \eqref{eq:intermediate654} in \eqref{eq:simplemanip} yields \eqref{eq:deikse24deikse}, as wanted.
\end{proof}

Recall that we already proved Theorem~\ref{thm:megastar} using Lemma \ref{lem:megastar-fixates} in  Section~\ref{sec:megastar-partial}. The remainder of Section~\ref{sec:megastar} will therefore focus on the proof of Lemma~\ref{lem:main-iteration}. The following stopping time will be important in what follows.

\begin{definition}\label{def:tdoom}
Consider $i\in \Zzero$. 
We define $\TTend{i}$ to be the first time $t \geq I_i^-$ such that one of the following holds.
\begin{enumerate}[\normalfont(D1)]
\item $t = I_i^+$.\label{D:end}
\item $v^*$ spawns $\beta_{i+1}$ non-mutants in the interval $(I_i^-,t]$.\label{D:beta}
\item For some $j \in [\ell]$, $v^*$ spawns $\alpha_{i+1}$ non-mutants onto vertices in $R_{j}$ in the interval $(I_i^-,t]$.\label{D:alpha}
\item For some $j \in [\ell]$, $|K_j \setminus X'_t| > 2\Delta$.\label{D:clique-empty}\defend{}
\end{enumerate}
\end{definition}
Note that 
if $\filt_{I_i^-}(X')$ is good then $|K_j \setminus X'_{I_i^-}| \leq 2\Delta$
so $\TTend{i} > I_i^-$. 

The crux of our argument will be a proof that if $\filt_{I_i^-}(X')$ is good, then $\TTend{i} = I_i^+$ with high probability (see the proof of Lemma~\ref{lem:branches-clear}).  

\subsubsection{Bounding the number of times cliques become active}\label{sec:clique-active}

In Section~\ref{sec:clique-active}, we   bound the number of times that cliques become active in the interval $(I^{-}_i, \TTend{i}]$ (see Lemma~\ref{lem:clique-active-bound}). We   require the following definitions.

\begin{definition}\label{defn:tmut}
Let $v \in V(\megastar[\ell])$ and 
$i\in \Zzero$. 
Let $\Tmut{v}{i}{-1} = I_i^-$.
We recursively define times $\Tnmut{v}{i}{h}$ and $\Tmut{v}{i}{h}$ for integers $h\ge 0$ as follows.
\begin{align*}
\Tnmut{v}{i}{h} &= 	 
		\min \{t \geq \Tmut{v}{i}{h-1} \mid v \notin X_t'\mbox{ or } t=\TTend{i}\}, \\
\Tmut{v}{i}{h} &= \min \{t \geq \Tnmut{v}{i}{h}\mid v \in X_t' \mbox{ or } t=\TTend{i}\}.	
\defenddisp{}
\end{align*} 
\end{definition}
The subscript ``$\mathsf m$'' stands for ``mutant'' and the subscript ``$\mathsf n$'' stands for ``non-mutant''.

\begin{definition}\label{def:wijh}
For each $i\in \Zzero$, let $I_i = (I_i^-,I_i^+]$.  Also, for
  $ h\in\Zzero$ and  $j \in [\ell]$, let 
  $\J{i,j}{h}$ be the interval $(\Tnmut{a_j}{i}{h}, \Tmut{a_j}{i}{h})$,
  let  
  $\W{i,j}{h}$ be the number of times at which $K_j$ becomes active in~$\J{i,j}{h}$,
  and let  $\Win{i,j}{h}$ be the number of times at which $K_j$ becomes inactive in~$\J{i,j}{h}$.
    \defend{}
\end{definition} 
Note that the intervals $\J{i,j}{h}$ are disjoint from each other and  they are all contained within
$[I_i^-,I_i^+]$.
Also, the interval $\J{i,j}{h}$ is empty if and only if $\Tnmut{a_j}{i}{h}=\TTend{i}$.
 The subscript ``$\mathsf a$''   stands for ``active'' and the subscript ``$\mathsf{in}$'' stands for
``inactive''.
The following lemmas,  Lemmas~\ref{lem:W-bound-1} and~\ref{lem:W-bound-2}, are stated only for 
$i\in \Zone$   --- we will deal with $i=0$ in the proof of Lemma~\ref{lem:clique-active-bound}.

\begin{lemma}\label{lem:W-bound-1}
\statelargeell  For $i\in\Zone$, let $f$ be a good possible value of $\filt_{I_i^-}(X')$. Let $j\in [\ell]$. Then,
\[\Pr\big(\W{i,j}{h}=0 \mbox{ for all } h \ge \floor{4\alpha_i\len{I_i}}+1\mid \filt_{I_i^-}(X') = f\big)\geq 1-e^{-n}.\]
\end{lemma}
\begin{proof}
We will show that with high probability,  the feeder vertex~$a_j$ becomes a non-mutant  at most $\floor{4\alpha_i\len{I_i}}$ times within $(I_i^-, \TTend{i}]$. Let $\mathcal{E}$ be the event that for all $v \in R_{j}$, $\vnc{v}{a_j}$ triggers at most $2\len{I_i}$ times in $I_i$. Recall that $\len{I_i} = n(\log n)^3$, so by Corollary~\ref{cor:pchernoff} combined with a union bound over all $v \in R_{j}$, we have
\[\Pr(\mathcal{E} \mid \filt_{I_i^-}(X')=f) \ge 1-e^{-n}.\]

Suppose that $\mathcal{E}$ occurs. Since $\filt_{I_i^-}(X')$ is good, by 
\Presmut{i}   we have $|R_{j} \setminus X'_{I_i^-}| \le \alpha_i$. Moreover, by \myref{D}{D:alpha}, at most $\alpha_{i+1}$ non-mutants are spawned into $R_{j}$ over the course of $(I_i^-, \TTend{i}]$. Hence all but at most $\alpha_i+\alpha_{i+1} \le 2\alpha_i$ vertices in $R_{j}$ are mutants throughout $(I_i^-, \TTend{i}]$, and therefore do not spawn any non-mutants onto $a_j$ within $(I_i^-, \TTend{i}]$. Since $\mathcal{E}$ occurs, the remaining vertices in $R_{j}$ each spawn at most $2\len{I_i}$ non-mutants onto $a_j$ within $(I_i^-, \TTend{i}]$, and so at most $\floor{4\alpha_i\len{I_i}}$ non-mutants are spawned onto $a_j$ in total over the course of $(I_i^-, \TTend{i}]$. 

Now, for all $h > 0$, $\Tnmut{a_j}{i}{h} = \TTend{i}$ or $a_j$ becomes a non-mutant at time $\Tnmut{a_j}{i}{h}$. It therefore follows that for all $h \ge \floor{4\alpha_i\len{I_i}}+1$, $\Tnmut{a_j}{i}{h} = \Tmut{a_j}{i}{h} = \TTend{i}$ 
so the interval $\J{i,j}{h}$   is empty 
and hence $\W{i,j}{h} = 0$ as required.
\end{proof}

In Lemma~\ref{lem:W-bound-2} we will show that, with high probability, the sum
$\sum_{h=0}^{\floor{4\alpha_i\len{I_i}}} \W{i,j}{h}$ is small. 
We will use this
in Lemma~\ref{lem:clique-active-bound},
 to show that, with high probability, $K_j$ doesn't become active too many times before 
$a_j$ has become a mutant more than $\floor{4\alpha_i\len{I_i}}$ times.

\begin{lemma}\label{lem:W-bound-2}
\statelargeell For $i\in \Zone$, let $f$ be a good possible value of $\filt_{I_i^-}(X')$. Let  $j \in [\ell]$. Then 
\[\Pr\left(\sum_{h=0}^{\floor{4\alpha_i\len{I_i}}} \W{i,j}{h} \le \alpha_i\sqrt{n}(\log n)^4-1 \,\Bigg|\,\filt_{I_i^-}(X') = f\right) \ge 1 - e^{-(\log n)^3}.\]
\end{lemma}

\begin{proof} 

We define the following random variables for $h\in \Zzero$.
These variables are local to this proof.   
First, if $K_j$ is inactive at $\Tnmut{a_j}{i}{h}$
(which is the left endpoint of $\J{i,j}{h}$)
 then $Q_h = \Win{i,j}{h}+1$.
Otherwise, $Q_h = \Win{i,j}{h}$.
It follows from the definition of
$\Win{i,j}{h}$ and $\W{i,j}{h}$ that $Q_h \geq \W{i,j}{h}$.
Next we define variables $T_h(0),\ldots,T_h(Q_h)$.  
First,  we define $T_h(Q_h) = \Tmut{a_j}{i}{h}$
(which is the right endpoint of $\J{i,j}{h}$).
If $Q_h>0$ 
then we define the remaining variables as follows.

\begin{description}
\item[Case 1.] If  
$K_j$ is inactive at  the left endpoint of $\J{i,j}{h}$
then  
$T_h(0)
= \Tnmut{a_j}{i}{h}$ and for
$q\in[\Win{i,j}{h}]$, 
$\Tinact{h}{q}$
is the $q$'th time that $K_j$ becomes inactive in $\J{i,j}{h}$.

\item [Case 2.]  If  
$K_j$ is active at  the left endpoint of $\J{i,j}{h}$
then  for $q\in [\Win{i,j}{h}]$, 
$T_h(q-1)$ is
the $q$'th time that $K_j$ becomes inactive in~$\J{i,j}{h}$.
 \end{description}

Now fix an integer $h \ge 0$.
Consider any time 
$\tnmut{a_j}{i}{h} \ge I_i^-$.
Consider any integers $w_{0}, \dots, w_{h-1}$ and $y \ge 0$, and any times 
$t_0,\ldots,t_y$
satisfying $\tnmut{a_j}{i}{h} \leq t_0 \leq \cdots \leq t_y$.
Suppose that $f'$ is a value of $\filt_{t_{y}}(X')$ which implies that
\begin{itemize}
\item $\filt_{I_i^-}(X') = f$,
\item $\Tnmut{a_j}{i}{h} = \tnmut{a_j}{i}{h}$,
\item $\W{i,j}{0} = w_{0}, \dots, \W{i,j}{h-1} = w_{h-1}$,
\item $\W{i,j}{h} \ge y$, and
\item $T_h(0) = t_{0}, \dots,  T_h(y) = t_{y}$.
\end{itemize}

 The event  $\filt_{t_y}(X') = f'$
determines whether or not 
$t_y = \Tmut{a_j}{i}{h}$. 
We split the analysis into two cases.
\begin{description}
\item[Case 1.]
If $\filt_{t_y}(X') = f'$ implies
 that $t_{y} = \Tmut{a_j}{i}{h}$, then 
since $y = Q_h \geq \W{i,j}{h} \geq y$, we have $\W{i,j}{h}=y$, so 
 \begin{equation}\label{eqn:W_h-1}
\Pr(\W{i,j}{h} = y \mid \filt_{t_{y}}(X') = f') = 1.
\end{equation}
\item[Case 2.]
Suppose that $\filt_{t_y}(X') = f'$ implies  $t_{y} < \Tmut{a_j}{i}{h}$. Let $\mathcal{E}$ be the event that in the interval $(t_{y},\infty)$, some mutant clock with source in $R_{j} \cap X'_{t_{y}}$ triggers before any non-mutant clock with source in $\{a_j,v^*\}$ triggers. 
If $\filt_{t_y}(X') = f'$  then $t_{y} < \Tmut{a_j}{i}{h} \le \TTend{i}$, so by \myref{D}{D:alpha} and the fact that $i \ge 1$, it follows that $R_{j}$ contains at least $m - \alpha_i - \alpha_{i+1} \ge m-2\alpha_1 \ge m/2$ mutants at time $t_{y}$. Hence
\begin{equation}\label{eqn:W_h-1.5}
\Pr(\mathcal{E} \mid \filt_{t_{y}}(X') = f') \ge 1 - \frac{2}{2+rm/2} \ge 1-\frac{5}{m}.
\end{equation}

We will now prove that if $\filt_{t_y}(X') = f'$
and  $\mathcal{E}$ occurs then $\W{i,j}{h} = y$. Let $T$ be the earliest time in the interval $(t_{y},\infty)$ at which some mutant clock with source in $R_{j} \cap X'_{t_{y}}$, say $\vmc{v}{a_j}$, triggers. 

If $t_{y} < \Tmut{a_j}{i}{h}$
then $t_y \leq \TTend{i}$ so
(D\ref{D:clique-empty}) implies that $K_j$ has at most $2\Delta$ non-mutants at $t_y$.
On the other hand, $T_h(y)=t_y$ implies that $K_j$ is
inactive at $t_y$. Hence, it must be full of mutants at $t_y$.
This means that     $K_j$ remains inactive after $t_{y}$ until $a_j$ spawns a non-mutant onto a vertex in $K_j$. If $\mathcal{E}$ occurs, then no non-mutant clock with source $a_j$ triggers in $(t_{y}, T]$, so $K_j$ is inactive throughout $(t_{y}, T]$.

Recall that $v \in X_{t_{y}}'$. If $\mathcal{E}$ occurs, then no non-mutant clock with source $v^*$ triggers in $(t_{y}, T]$, so $v \in X_T'$ also. Hence by definition, $v$ spawns a mutant onto $a_j$ at time $T$. Moreover, $K_j$ is inactive at time $T$ and so $a_j$ becomes a mutant at time $T$. Thus $\Tmut{a_j}{i}{h} \le T$, and so $K_j$ is inactive throughout $(t_{y}, \Tmut{a_j}{i}{h}]$. Thus $\W{i,j}{h} = y$ whenever $\mathcal{E}$ occurs, as claimed. By \eqref{eqn:W_h-1.5}, it follows that if $f'$ implies that $t_{y} < \Tmut{a_j}{i}{h}$,
\begin{equation}\label{dokeep}
\Pr\big(\W{i,j}{h} = y \mid \filt_{t_{y}}(X') = f'\big) \ge \Pr\big(\mathcal{E} \mid \filt_{t_{y}}(X')=f'\big) \ge 1-5/m.
\end{equation}
\end{description}

Combining Cases~1 and~2 
by considering all possible $f'$, $\tnmut{a_j}{i}{h}$, and
$t_0,\ldots,t_y$ and
combining Equations~\eqref{eqn:W_h-1} and \eqref{dokeep}, it follows that
\[\Pr\Bigg(\W{i,j}{h} = y \,\Bigg|\, \begin{array}{l}\W{i,j}{0} = w_{0}, \dots, \W{i,j}{h-1} = w_{h-1}, \\[0.2cm]
\W{i,j}{h} \ge y,\, \filt_{I_i^-}(X')=f\end{array}\Bigg) \ge 1 - \frac{5}{m}.\]
Let $W_0', W_1', \dots$ be i.i.d.\ geometric variables with parameter $1-5/m$. It is immediate that $\Pr(W_h' = y \mid W_h' \ge y) = 1-5/m$. Moreover, $\Pr(\W{i,j}{h} \ge 0) = \Pr(W_h' \ge 0) = 1$. It follows that conditioned on $\filt_{I_i^-}(X') = f$, the random variables $\W{i,j}{0}, \W{i,j}{1}, \dots$ are dominated above by $W_0', W_1', \dots$. 

Now, note that $(\log n)^3 \le 14 \cdot 5(\floor{4\alpha_i\len{I_i}}+1)/m \le \alpha_i\sqrt{n}(\log n)^4-1$ and that $1-5/m \ge 13/14$. It therefore follows by Lemma~\ref{lem:geo-chernoff} that
\[\Pr\left(\sum_{h=0}^{\floor{4\alpha_i\len{I_i}}} W_h' \ge \alpha_i\sqrt{n}(\log n)^4-1\right) \le e^{-(\log n)^3}.\]
The result therefore follows.
\end{proof}

The following definition is related to Definition~\ref{def:wijh}.

\begin{definition}\label{def:newwijh}
For each $i\in \Zzero$ and $j\in [\ell]$,
let $\Wnoh{i,j}$ be the number of times in $(I_i^-, \TTend{i}]$ at which $K_j$ becomes active.
\defend{}
\end{definition}

We now combine Lemma~\ref{lem:W-bound-1}  and Lemma~\ref{lem:W-bound-2}  to show that, with high probability, 
$K_j$ does not   become active too many times in $(I_i^-, \TTend{i}]$.

\begin{lemma}\label{lem:clique-active-bound}
\statelargeell Let $i$ be a non-negative integer, and let $f$ be a good possible value of $\filt_{I_i^-}(X')$. Let $j\in[\ell]$. Then 
\[\Pr\big(\Wnoh{i,j} \le \alpha_i\sqrt{n}(\log n)^4\mid \filt_{I_i^-}(X') = f\big) \ge 1-2e^{-(\log n)^3}.\]
\end{lemma}
\begin{proof}
First, note that $K_j$ can become active either when  $K_j$ has no mutants and a mutant is spawned into it or when $K_j$ is full of mutants and $a_j$ spawns a non-mutant into it. By \myref{D}{D:clique-empty}, $K_j$ has at most $2\Delta$ non-mutants throughout the interval $(I_i^-,\TTend{i}]$, so  that
\begin{equation}\label{eq:usingTendone}
\begin{aligned}
&\mbox{for $t\in(I_i^-,\TTend{i}]$, $K_j$ can only become active at time $t$ if $t\in [\Tnmut{a_j}{i}{h},\Tmut{a_j}{i}{h}]$}\\
&\mbox{(with $h\in \mathbb{Z}_{\geq 0}$) and a non-mutant clock with source $a_j$ triggers at time $t$.}
\end{aligned}
\end{equation}
We next consider cases on the value of $i$, namely, whether $i=0$ or not.

First suppose that $i=0$, so $\alpha_i = m$.  By Corollary~\ref{cor:pchernoff} (see also the proof of Lemma~\ref{lem:W-bound-1}), the probability that the non-mutant clocks with source $a_j$ trigger more than $2\len{I_i}$ times in $I_i$ is at most $e^{-n}$. It follows from \eqref{eq:usingTendone} that with probability at least $1-e^{-n}$, 
\[\Wnoh{i,j} \le 2\len{I_i} \le 4m\sqrt{n}(\log n)^3 = 4\alpha_i\sqrt{n}(\log n)^3\]
as required.

Suppose instead that $i \ge 1$.  We will show that with probability 1, it holds that 
\begin{equation}\label{eq:sumWsone}
\Wnoh{i,j} \leq 1+\sum_{h=0}^{\infty} \W{i,j}{h},
\end{equation}
so that the lemma follows from Lemmas~\ref{lem:W-bound-1} and \ref{lem:W-bound-2} (combined with a union bound).

It remains to justify that \eqref{eq:sumWsone} holds with probability 1. By \eqref{eq:usingTendone}, in the interval $(I_i^-,\TTend{i}]$, $K_j$ can only become active within the intervals $[\Tnmut{a_j}{i}{h},\Tmut{a_j}{i}{h}]$, $h\in \mathbb{Z}_{\geq 0}$. By the definition of the $\W{i,j}{h}$'s, \eqref{eq:sumWsone} will thus follow by showing that with probability 1, $K_j$ becomes active at most once at times $t$ with $t\in S$, where 
\[S:=\{\Tnmut{a_j}{i}{h}\mid h\in \mathbb{Z}_{\geq 0}\}\cup \{\Tmut{a_j}{i}{h}\mid h\in \mathbb{Z}_{\geq 0}\}.\]
Note that $a_j$ can become a non-mutant  either because $K_j$ was empty and a mutant was spawned into $K_j$ or because a non-mutant was spawned onto $a_j$. Now let $h$ be such that $\Tnmut{a_j}{i}{h} < \TTend{i}$.  By \myref{D}{D:clique-empty}, it must be the case that a non-mutant is spawned onto $a_j$  at time $\Tnmut{a_j}{i}{h}$. Thus, for $K_j$ to become active at time $\Tnmut{a_j}{i}{h}$, it must be the case from \eqref{eq:usingTendone} that two clocks triggered at the same time, which happens with probability 0. Similarly, for all $h$ such that $\Tmut{a_j}{i}{h} < \TTend{i}$, we have that a mutant is spawned onto $a_j$ at $\Tmut{a_j}{i}{h}$, and thus,  with probability 1, $K_j$ does not become active at $\Tmut{a_j}{i}{h}$. Thus, $K_j$ can only become active at a time $\Tnmut{a_j}{i}{h}$ (resp. $\Tmut{a_j}{i}{h}$) if $\Tnmut{a_j}{i}{h}=\TTend{i}$ (resp. $\Tmut{a_j}{i}{h}=\TTend{i}$).  It thus follows that the number of times that $K_j$ becomes active at times $t\in S$ is at most one (namely when $t=\TTend{i}$), as desired. This proves \eqref{eq:sumWsone} and concludes the proof of the lemma.
\end{proof}
\subsubsection{The behaviour of the centre vertex}\label{sec:vcent}

In Section~\ref{sec:vcent}, we will show that with high probability, $v^*$  does not spend
too much time as a non-mutant in the interval $(I_i^-, \TTend{i}]$ (see Lemma~\ref{lem:vcent-time}). We first apply Lemma~\ref{lem:clique-nm-spawns}  and Lemma~\ref{lem:clique-active-bound}, to give an upper bound for the total number of non-mutants each clique spawns onto $v^*$ within $(I_i^-, \TTend{i}]$.

\begin{lemma}\label{lem:clique-nm-spawns-2}
\statelargeell Let $i$ be a non-negative integer, and let $f$ be a good possible value of $\filt_{I_i^-}(X')$. Let $j \in [\ell]$. Then with probability at least $1-e^{-\frac{1}{2}(\log n)^3}$, conditioned on $\filt_{I_i^-}(X') = f$, vertices in $K_j$ spawn at most $10\alpha_i\sqrt{n}(\log n)^{17}/k$ non-mutants onto $v^*$ within $(I_i^-, \TTend{i}]$.
\end{lemma}
\begin{proof}
We start with the following mutually recursive definitions, which are the
same as the ones in Lemma~\ref{lem:iteration-cliques},  except that the endpoint is~$\TTend{i}$
rather than~$I^+$.
Let $\Tinact{}{-1}= I_i^-$.
Then, for $h\in \Zzero$,
\begin{align*}
\Tact{}{h} &=  \min \{ t\geq \Tinact{}{h-1}\mid
\mbox{$K_j$ is active at $t$  or $t=\TTend{i}$} \},  \\
\Tinact{}{h} &= \min \{t\geq \Tact{}{h}\mid 
\mbox{$K_j$ is inactive at $t$  or $t=\TTend{i}$} \}.
\end{align*}
For all $h\in\Zzero$, let $S_h$ be the number of non-mutants spawned onto $v^*$ in $(\Tact{}{h},\Tinact{}{h})$ by vertices in $K_j$. By (D\ref{D:clique-empty}), whenever $K_j$ is inactive in the interval $(I_i^-, \TTend{i}]$ it contains no non-mutants. Moreover, with probability 1 no non-mutants are spawned onto $v^*$ at any time $\Tact{}{h}$ or $\Tinact{}{h}$, except possibly $I_i^-$. (If $\Tact{}{h}$ or $\Tinact{}{h}$ is in $(I_i^-,\TTend{i})$ this is because a clock
with source~$a_j$ triggers at the relevant time and the probability that two clocks trigger simultaneously is~$0$. Also, it is clear from the definition of~$\TTend{i}$
that the probability that a clock with target~$v^*$ triggers at this time is~$0$.)
Thus precisely $\sum_{h=0}^\infty S_h$ non-mutants are spawned onto $v^*$ in $(I_i^-, \TTend{i}]$.

Let $\mathcal{E}_1$ be the event that $S_h = 0$ for all $h > \floor{\alpha_i\sqrt{n}(\log n)^4}$. Let $\mathcal{E}_2$ be the event that $S_h = 0$ for all but at most $10\alpha_i\sqrt{n}(\log n)^{14}/k$ values of $h \in \{0,\dots,\floor{\alpha_i\sqrt{n}(\log n)^4}\}$. Let $\mathcal{E}_3$ be the event that for all $h \in \{0,\dots,\floor{\alpha_i\sqrt{n}(\log n)^4}\}$, $S_h \le (\log n)^{3}$. Note that to prove the result, it suffices to show that
\begin{equation}\label{eqn:vcent-event}
\Pr\big(\mathcal{E}_1 \cap \mathcal{E}_2 \cap \mathcal{E}_3 \mid \filt_{I_i^-}(X') = f\big) \ge 1-e^{-\frac{1}{2}(\log n)^3}.
\end{equation}

By Lemma~\ref{lem:clique-active-bound}, we have
\[\Pr\big(\mathcal{E}_1 \mid \filt_{I_i^-}(X') = f\big) \ge 1-2e^{-(\log n)^3}.\]
Now, by \Pcliquefull{i}, every clique contains at most $\Delta$ non-mutants at time $I_i^-$. Moreover, by \myref{D}{D:clique-empty}, whenever $K_j$ becomes active in $(I_i^-, \TTend{i}]$ it contains only one non-mutant. It therefore follows that for all $h\in\Zzero$, either $\Tact{}{h} = \Tinact{}{h} = \TTend{i}$ or $|K_j \setminus X'_{\Tact{}{h}}| \in [\Delta]$. 

Now, consider any integer $h \ge 0$ and any time $t_h \ge I_i^-$. Suppose that $f_h$ is a possible value of $\filt_{t_h}(X')$ which implies that $t_h = \Tact{}{h}$ and $\filt_{I_i^-}(X')=f$. If $f_h$ implies that $|K_j \setminus X'_{t_h}| \in [\Delta]$, then by Lemma~\ref{lem:clique-nm-spawns}(\ref{it:clique-nm-spawns-1}) we have
\begin{equation}\label{eqn:nm-spawns}
\Pr(S_h > 0 \mid \filt_{t_h}(X') = f_h) \leq (\log n)^{10}/k.
\end{equation}
Otherwise, $f_h$ must imply that $t_h = \TTend{i}$ and hence $S_h = 0$, so \eqref{eqn:nm-spawns} is valid for all choices of $f_h$. Moreover, by Lemma~\ref{lem:clique-nm-spawns}(\ref{it:clique-nm-spawns-2}) we have
\begin{equation}\label{eqn:nm-spawns-2}
\Pr(S_h \le (\log n)^3 \mid \filt_{t_h}(X')=f_h) \ge 1 - 7e^{-(\log n)^3}.
\end{equation}

Now, recall that $\filt_{t_h}(X')$ determines $S_0, \dots, S_{h-1}$. It therefore follows from \eqref{eqn:nm-spawns} that the number of integers $h \in \{0,\dots,\floor{\alpha_i\sqrt{n}(\log n)^4}\}$ such that $S_h > 0$ is dominated above by a binomial variable with
 $\floor{\alpha_i\sqrt{n}(\log n)^4}+1$  trials, each with success probability $(\log n)^{10}/k$. It follows by Lemma~\ref{lem:bchernoff} that
\[\Pr(\mathcal{E}_2 \mid \filt_{I_i^-}(X') = f) \ge 1 - e^{-10\alpha_i\sqrt{n}(\log n)^{14}/k}\ge 1 - e^{-(\log n)^3}.\]
Finally, by \eqref{eqn:nm-spawns-2} combined with a union bound over all $h \in \{0, \dots, \floor{\alpha_i\sqrt{n}(\log n)^4}\}$, we have
\[\Pr(\mathcal{E}_3 \mid \filt_{I_i^-}(X') = f) \ge 1-7\big(\floor{\alpha_i\sqrt{n}(\log n)^4}+1\big)e^{-(\log n)^3} \ge 1-n^2e^{-(\log n)^3}.\]
Thus \eqref{eqn:vcent-event} follows by a union bound, which implies the result.
\end{proof}

We are now in a position to prove that in total, not too many non-mutants are spawned onto $v^*$ over the interval $(I_i^-, \TTend{i}]$.

\begin{corollary}\label{cor:total-nm-spawns}
\statelargeell Let $i$ be a non-negative integer, and let $f$ be a good possible value of $\filt_{I_i^-}(X')$. Then with probability at least $1-e^{-(\log n)^2}$, conditioned on $\filt_{I_i^-}(X') = f$, at most $80\beta_i\sqrt{n}(\log n)^{19}/k$ non-mutants are spawned onto $v^*$ within $(I_i^-, \TTend{i}]$.
\end{corollary}
\begin{proof}
Let $\mathcal{E}$ be the event that for all $j \in [\ell]$, vertices in $K_j$ spawn at most $10\alpha_i\sqrt{n}(\log n)^{17}/k$ non-mutants onto $v^*$ over the course of $(I_i^-, \TTend{i}]$. 

Suppose that $\mathcal{E}$ occurs and that $\alpha_i > \floor{(2\log n)^{2}}$, so   \[\beta_i \geq \floor{\ell m/(2\log n)^{2i}} \ge \ell\floor{m/(2\log n)^{2i}} = \ell \alpha_i.\] Then since $\mathcal{E}$ occurs, at most $10\ell \alpha_i \sqrt{n}(\log n)^{17}/k \le 80\beta_i\sqrt{n}(\log n)^{19}/k$ non-mutants are spawned onto $v^*$ over the course of $(I_i^-,\TTend{i}]$.

Now suppose that $\mathcal{E}$ occurs and that $\alpha_i = \floor{(2\log n)^{2}}$. By \myref{D}{D:beta} and 
\Pmostcliques{i},  for all but at most $\beta_i+\beta_{i+1} \le 2\beta_i$ values of $j \in [\ell]$, we have $K_j \subseteq X'_t$ for all $t \in (I_i^-,\TTend{i}]$. Certainly these cliques cannot spawn non-mutants onto $v^*$ in $(I_i^-,\TTend{i}]$, so since $\mathcal{E}$ occurs, it follows that at most 
\[20\beta_i\alpha_i\sqrt{n}(\log n)^{17} /k \le 80\beta_i\sqrt{n}(\log n)^{19}/k\]
non-mutants are spawned onto $v^*$ over the course of $(I_i^-, \TTend{i}]$.

We have therefore shown that whenever $\mathcal{E}$ occurs, at most $80\beta_i\sqrt{n}(\log n)^{19}/k$ non-mutants are spawned onto $v^*$ over the course of $(I_i^-, \TTend{i}]$. By Lemma~\ref{lem:clique-nm-spawns-2} combined with a union bound over all $j \in [\ell]$, we have
\[\Pr(\mathcal{E} \mid \filt_{I_i^-}(X') = f) \ge 1-e^{-(\log n)^2}.\]
The result therefore follows.
\end{proof}

We now use \myref{D}{D:clique-empty} to show that when $v^*$ is a non-mutant in $(I_i^-, \TTend{i}]$, the time until either $v^*$ becomes a mutant again or the interval ends is dominated above by an exponential variable with parameter $\ell/2$.
The proof of this lemma is similar to the proof of Lemma~\ref{lem:funnel-many-infections-needed}.

\begin{lemma}\label{lem:centre-nm-short-time}
\statelargeell Let $i \ge 0$ and $z \ge 0$ be integers, let $t_0 \in [I_i^-,I_i^+]$, and let $f$ be a good possible value of $\filt_{t_0}(X')$ which implies that $t_0 = \Tnmut{v^*}{i}{z}$. Let $S = \Tmut{v^*}{i}{z} - t_0$. Then for all $t \ge 0$,
\[\Pr(S \le t \mid \filt_{t_0}(X')=f) \ge 1 - e^{-\ell t/2}.\]
\end{lemma}

\begin{proof}
Note that $t_0 \le \TTend{i}$, and $f$ determines the event $t_0 = \TTend{i}$. Moreover, if $t_0 = \TTend{i}$, then $S=0$ with probability $1$ and the result follows. We may therefore assume that $f$ implies that $t_0 < \TTend{i}$, and in particular $v^* \notin X_{t_0}'$.

Let $\mathcal{A}$ be the set of all mutant clocks in $\moranclocks(\megastar[\ell])$ with target $v^*$, and let $\Phi$ encapsulate the behaviour of every clock in $\moranclocks(\megastar[\ell]) \setminus \mathcal{A}$ over the interval $I_i$. In particular, by (D\ref{D:end}), $\Phi$ determines the behaviour of these clocks in the interval $(I_i^-, \TTend{i}]$. Consider any possible value $\varphi$ of  $\Phi$ which is consistent with $\filt_{t_0}(X')=f$. 

We now define a time $\primetend{i}$ which depends only on the values $f$ and $\phi$. To do so, consider the situation in which $\filt_{t_0}(X')=f$, $\Phi = \varphi$ and no clock in $\mathcal{A}$ triggers in $(t_0, I_i^+]$, so that the evolution of $X'$ in this interval is entirely determined by $f$ and $\varphi$. Let $\primetend{i}$ be the time at which $\TTend{i}$ would occur in this situation. 

We claim that, if $\filt_{t_0}(X')=f$ and $\Phi = \varphi$, then $\Tmut{v^*}{i}{z} \le \primetend{i}$. To see this, suppose that $\filt_{t_0}(X')=f$ and $\Phi = \varphi$. If no mutant is spawned onto $v^*$ in $(t_0, \primetend{i}]$, then $X'$ evolves exactly as it would have done if no clocks in $\mathcal{A}$ triggered in $(t_0, \primetend{i}]$, and so $\primetend{i} = \TTend{i} \ge \Tmut{v^*}{i}{z}$. If a mutant is spawned onto $v^*$ at some time $t' \in (t_0, \primetend{i}]$, then $\Tmut{v^*}{i}{z} \le t' \le \primetend{i}$. So $\Tmut{v^*}{i}{z} \le \primetend{i}$ in all cases as claimed. Hence if $t \ge \primetend{i}-t_0$,
\[\Pr(S \le t \mid \filt_{t_0}(X')=f, \Phi=\varphi) = 1 \ge 1-e^{-\ell t/2},\]
as required in the lemma statement. We  now consider the case $t < \primetend{i}-t_0$.

Let $t_1 < \dots < t_y$ be the times in $(t_0, t_0+t]$ at which clocks in $\moranclocks(\megastar[\ell]) \setminus \mathcal{A}$ trigger, and let $t_{y+1} = t_0 + t$. Thus $t_0 < \dots < t_y \le t_{y+1} < \primetend{i}$.
For all $h \in \{0, \dots, y\}$, let $\chi(h)$ be the value that $X'_{t_h}$ would take in the situation where $\filt_{t_0}(X')=f$, $\Phi = \varphi$ and no clock in $\mathcal{A}$ triggers in $(t_0, t_h]$. Thus $\primetend{i}$, $y$, $t_0, \dots, t_{y+1}$, and $\chi(0), \dots, \chi(y)$ are all uniquely determined by $f$ and $\varphi$. 

For each $h \in [y+1]$, let $\mathcal{E}_h$ be the event that a mutant is spawned onto $v^*$ in the interval $(t_{h-1}, t_h)$. Note that with probability 1, no mutant is spawned onto $v^*$ at any time $t_h$. Thus
\begin{align}
\Pr(S \le t \mid \filt_{t_0}(X')=f, \Phi=\varphi) &= 1 - \Pr(\overline{\mathcal{E}_1} \cap \dots \cap \overline{\mathcal{E}_{y+1}} \mid \filt_{t_0}(X')=f, \Phi=\varphi)\nonumber\\
&= 1 - \prod_{h=1}^{y+1} \Pr(\overline{\mathcal{E}_h} \mid \filt_{t_0}(X')=f, \Phi=\varphi, \overline{\mathcal{E}_1} \cap \dots \cap \overline{\mathcal{E}_{h-1}}).\label{eqn:kersplunk}
\end{align}

Now fix $h \in [y+1]$, and consider any possible value $f_{h-1}$ of $\filt_{t_{h-1}}(X')$ which implies that $\filt_{t_0}(X')=f$ and that $\overline{\mathcal{E}_1} \cap \dots \cap \overline{\mathcal{E}_{h-1}}$ occurs, and is consistent with $\Phi = \varphi$. 
Consider the evolution of~$X'$ given
 $\filt_{t_{h-1}}(X') = f_{h-1}$ and $\Phi=\varphi$. Since $\overline{\mathcal{E}_1} \cap \dots \cap \overline{\mathcal{E}_{h-1}}$ occurs, no mutant  is spawned onto $v^*$ in the interval $(t_0, t_{h-1}]$ and so $X_{t_{h-1}}' = \chi(h-1)$. Moreover, $X'$ remains constant in $[t_{h-1}, t_h)$ unless a mutant is spawned onto $v^*$. Thus,
 given the condition that  $\filt_{t_{h-1}}(X') = f_{h-1}$ and $\Phi=\phi$,  $\mathcal{E}_h$ occurs if and only if a mutant clock with source in $\chi(h-1)$ and target $v^*$ triggers in the interval $(t_{h-1}, t_h)$. 

Now, since $\overline{\mathcal{E}_1} \cap \dots \cap \overline{\mathcal{E}_{h-1}}$ occurs and $\primetend{i} > t_{h-1}$, by \myref{D}{D:clique-empty} we have  
\begin{equation*}
|\chi(h-1) \cap (K_1 \cup \dots \cup K_\ell)| \ge k\ell/2\textnormal.
\end{equation*}
Hence
\[\Pr(\overline{\mathcal{E}_h} \mid \filt_{t_{h-1}}(X')=f_{h-1}, \Phi=\varphi) = e^{-r(t_h-t_{h-1})|\chi(h-1) \cap (K_1 \cup \dots \cup K_\ell)|/k} \le e^{-\ell(t_h-t_{h-1})/2}.\]
It therefore follows from \eqref{eqn:kersplunk} that
\[\Pr(S \le t \mid \filt_{t_0}(X')=f, \Phi=\varphi) \ge 1 - \prod_{h=1}^{y+1} e^{-\ell(t_h-t_{h-1})/2} = 1 - e^{-\ell t/2}.\]
The result therefore follows.
\end{proof}

\begin{definition}\label{def:iteration-sch}
For all $i\in\Zzero$, let
$\gamma_i = \beta_i (\log n)^{20}/k$.\defend{}
 \end{definition}

\begin{lemma}\label{lem:vcent-time}
\statelargeell Let $i$ be a non-negative integer, and let $f$ be a good possible value of $\filt_{I_i^-}(X')$. Then 
\[\Pr\left(\len{\{t \in (I_i^-, \TTend{i}]  \mid v^* \notin X_t'\}} \le \gamma_i \>\Big\vert\> \filt_{I_i^-}(X')=f\right) \ge 1-2e^{-(\log n)^2}.\]
\end{lemma}
\begin{proof}
For all $h\in\Zzero$, write $S_h = \Tmut{v^*}{i}{h} - \Tnmut{v^*}{i}{h}$. Consider any $s_0, \dots s_{h-1} \ge 0$, $t_h \ge I_i^-$, and any possible value $f_h$ of $\filt_{t_h}(X')$ which implies that $\filt_{I_i^-}(X')=f$, $S_0 = s_0, \dots, S_{h-1} = s_{h-1}$ and $\Tnmut{v^*}{i}{h} = t_h$. Then by Lemma~\ref{lem:centre-nm-short-time}, for all $t \ge 0$ we have
\[\Pr(S_h \le t \mid \filt_{t_h}(X') = f_h) \ge 1 - e^{-\ell t/2}.\]
It follows that $S_0, S_1, \dots $ are dominated above by i.i.d.\ exponential variables $S_0', S_1', \dots$ with parameter $\ell/2$. 

Now, by the definition of $\Tnmut{v^*}{i}{h}$ and $\Tmut{v^*}{i}{h}$, we have that 
\begin{equation}\label{eqn:vcent-time-1}
\len{\{t \in (I_i^-, \TTend{i}] \mid v^* \notin X_t'\}} = \sum_{h=0}^\infty S_h.
\end{equation}
By Corollary~\ref{cor:total-nm-spawns}, we have
\begin{equation}\label{eqn:vcent-time-2}
\Pr\big(S_h = 0 \textnormal{ for all } h > 80\beta_i\sqrt{n}(\log n)^{19}/k\mid \filt_{I_i^-}(X')=f\big) \ge 1-e^{-(\log n)^2}.
\end{equation}
Moreover, we have
\[\frac{3}{\ell}\bigg(\Floor{\frac{80\beta_i\sqrt{n}(\log n)^{19}}{k}}+1\bigg) \le \frac{241\beta_i\sqrt{n}(\log n)^{19}}{k\ell}\leq   \frac{\beta_i(\log n)^{20}}{k}= \gamma_i,\]
and so by Corollary~\ref{cor:expsum} we have
\begin{align*}
\Pr\left(\sum_{h=0}^{\Floor{\frac{80\beta_i(\log n)^{19}}{k}}} S_h < \gamma_i \,\Bigg|\, \filt_{I_i^-}(X')=f\right) &\ge \Pr\left(\sum_{h=0}^{\Floor{\frac{80\beta_i\sqrt{n}(\log n)^{19}}{k}}} S_h' < \gamma_i\right)\\
&\ge 1-e^{-\ell\gamma_i/32}  \ge 1-e^{-(\log n)^2}.
\end{align*}
The result therefore follows by \eqref{eqn:vcent-time-1}, \eqref{eqn:vcent-time-2} and a union bound.
\end{proof}

\subsubsection{Proving Lemma~\ref{lem:main-iteration}}\label{sec:megastar-final}

Recall that by Lemma~\ref{lem:vcent-time}, $v^*$ is very unlikely to spend more than $\gamma_i$ time as a non-mutant over the course of $(I_i^-, \TTend{i}]$. This motivates the following definition.

\begin{definition}\label{def:auxdefs}
Recall the definition of $\Psi(X')$ (Section~\ref{sec:starclock}). For all $i\in\Zzero$, let $t^-_i = \siout{v^*}{I_i^-}$ and let $t^+_i = \siin{v^*}{I_i^-}$. For all $j \in [\ell]$, let
\begin{align*}
\mathcal{T}^{i,j} &= \{t \in (t_i^-, t_i^- + \gamma_i] \mid \textnormal{for some } v \in R_{j},\, \vsnc{v^*}{v} \textnormal{ triggers at }t\},\\
\mathcal{T}^i &= \mathcal{T}^{i,1} \cup \dots \cup \mathcal{T}^{i,\ell},\\
\mathcal{U}_i &= \{v \in R_1 \cup \dots \cup R_{\ell} \mid \vsnc{v^*}{v}\textnormal{ triggers in }(t^-_i, t^-_i+\gamma_i]\}.\defenddisp{}
\end{align*}\end{definition}

In particular, if $v^*$ spends at most $\gamma_i$ time as a non-mutant in $(I_i^-,\TTend{i}]$, then it only spawns non-mutants onto vertices in $\mathcal{U}_i$. 

\begin{lemma}\label{lem:star-clock-bound}
\statelargeell Let $i$ be a non-negative integer, and let $f^*$ be a 
possible value of $\filt_{I_i^-}(\Psi(X'))$. Then with probability at least $1-3e^{-(\log n)^2}$ conditioned on $\filt_{I_i^-}(\Psi(X'))=f^*$, the following statements all hold.
\begin{enumerate}[(i)]
\item $|\mathcal{U}_i| \le |\mathcal{T}^i| < \beta_{i+1}$.\label{it:star-clock-bound-1}
\item For all $j \in [\ell]$, $|\mathcal{U}_i \cap R_{j}| \le |\mathcal{T}^{i,j}| < \alpha_{i+1}$.\label{it:star-clock-bound-2}
\item For all $v \in R_1 \cup \dots \cup R_{\ell}$, $\vsmc{v^*}{v}$ triggers in $(t^+_i, t^+_i + \len{I_i}/3]$.\label{it:star-clock-bound-3}
\end{enumerate}
\end{lemma}

\begin{proof}
Note that $t^-_i$ is determined by $f^*$, and that $f^*$ does not determine the behaviour of star-clocks $\vsnc{v^*}{v}$ in $(t_i^-, \infty)$. It follows that conditioned on $\filt_{I_i^-}(\Psi(X'))=f^*$, $|\mathcal{T}^i|$ follows a Poisson distribution with parameter $\gamma_i = \beta_i(\log n)^{20}/k \le \frac{1}{9}\frac{\beta_i}{(2\log n)^{2}}$.  

If $\beta_i \leq (2\log n)^{4}$, then this parameter is at most $\frac{1}{9}(2\log n)^2$. We also have that $\beta_{i+1}\ge \floor{(2\log n)^{2}}$. Using also Corollary~\ref{cor:pchernoff-3}, we thus obtain that 
\begin{align*}
\Pr(|\mathcal{T}^i| < \beta_{i+1} \mid \filt_{I_i^-}(\Psi(X'))=f^*) &\ge \Pr(|\mathcal{T}^i| < \floor{(2\log n)^{2}} \mid \filt_{I_i^-}(\Psi(X'))=f^*)\\
& \ge 1- e^{-\floor{(2\log n)^{2}}} \ge 1- e^{-(\log n)^2}.
\end{align*}
If instead $\beta_i> (2\log n)^{4}$, then $\beta_{i+1}\ge\floor{\frac{\beta_i}{(2\log n)^{2}}}\geq 8\gamma_i$ and so by Corollary~\ref{cor:pchernoff-3},
\begin{align*}
\Pr(|\mathcal{T}^i| < \beta_{i+1} \mid \filt_{I_i^-}(\Psi(X'))=f^*) &\ge 1-e^{-\beta_{i+1}} \ge 1-e^{-\floor{(2\log n)^2}} \ge 1-e^{-(\log n)^2}.
\end{align*}
Thus in all cases,
\begin{equation}\label{eqn:star-clock-bound-1}
\Pr(|\mathcal{T}^i| < \beta_{i+1} \mid \filt_{I_i^-}(\Psi(X')) = f^*)\ge 1-e^{-(\log n)^2}.
\end{equation}

By a similar argument to the above, we see that for all $j \in [\ell]$, $|\mathcal{T}^{i,j}|$ follows a Poisson distribution with parameter $\gamma_i/\ell$. Note that if $\beta_i = \floor{(2\log n)^2}$ then $\gamma_i/\ell \le 1/\ell$, and if $\beta_i = \floor{\ell m/(2\log n)^{2i}}$ then $\gamma_i/\ell \le \frac{1}{10}\frac{m}{(2\log n)^{2i+2}}$. Thus $\gamma_i/\ell \le \frac{1}{9}\frac{\alpha_i}{(2\log n)^{2}}$ in all cases. As in the above argument, it therefore follows from Corollary~\ref{cor:pchernoff-3} that
\begin{equation*}
\Pr(|\mathcal{T}^{i,j}| < \alpha_{i+1} \mid \filt_{I_i^-}(\Psi(X')) = f^*)\ge 1-e^{-\floor{(2\log n)^2}}.
\end{equation*}
A union bound over $j\in [\ell]$ thus gives  
\begin{equation}\label{eqn:star-clock-bound-2}
\Pr(|\mathcal{T}^{i,j}| < \alpha_{i+1} \textnormal{ for all }j\in[\ell]\mid \filt_{I_i^-}(\Psi(X'))=f^*) \ge 1-\ell e^{-\floor{(2\log n)^{2}}} \ge 1-e^{-(\log n)^2}.
\end{equation}

Finally, let $\mathcal{E}$ be the event that $\vsmc{v^*}{v}$ triggers in $(t^+_i, t^+_i+\len{I_i}/3]$ for all $v\in R_1 \cup \dots \cup R_{\ell}$. Note that $t_i^+$ is determined by $f^*$, and that $f^*$ does not determine the behaviour of star-clocks $\vsmc{v^*}{v}$ in $(t_i^+,\infty)$. It follows that conditioned on $\filt_{I_i^-}(\Psi(X'))=f^*$, for all $v \in R_1 \cup \dots \cup R_{\ell}$, the number of times $\vsmc{v^*}{v}$ triggers in $(t^+_i, t^+_i+\len{I_i}/3]$ is a Poisson variable with parameter $\len{I_i}/(3\ell m) \ge \frac{1}{3}(\log n)^3$. Hence
\[\Pr(\vsmc{v^*}{v}\textnormal{ triggers in }(t^+_i,t^+_i+\len{I_i}/3]\mid\filt_{I_i^-}(\Psi(X'))=f^*) \ge 1 - e^{-\frac{1}{3}(\log n)^3}.\]
Thus by a union bound over all $v \in R_1 \cup \dots \cup R_{\ell}$, we have
\begin{equation}\label{eqn:star-clock-bound-3}
\Pr(\mathcal{E}\mid \filt_{I_i^-}(\Psi(X'))=f^*) \ge 1-\ell m e^{-\frac{1}{3}(\log n)^3} \ge 1-e^{-(\log n)^2}.
\end{equation}
The result therefore follows from a union bound over \eqref{eqn:star-clock-bound-1}, \eqref{eqn:star-clock-bound-2} and \eqref{eqn:star-clock-bound-3}.
\end{proof}

The following lemma will be the heart of the proof of Lemma~\ref{lem:main-iteration}.

\begin{lemma}\label{lem:branches-clear}
\statelargeell Let $i$ be a non-negative integer, and let $f^*$ be a 
possible value of $\filt_{I_i^-}(\Psi(X'))$
such that the induced value~$f$ of $\filt_{I_i^-}(X')$ is good. 
Then with probability at least $1-10e^{-(\log n)^2}$ conditioned on $\filt_{I_i^-}(\Psi(X'))=f^*$, the following statements all hold.
\begin{enumerate}[(i)]
\item \label{it:branches-clear-1}$(R_1 \cup \dots \cup R_{\ell}) \setminus X'_{I_i^+} \subseteq \mathcal{U}_i$.
\item \label{it:branches-clear-2} If $i \ge 6$, then for all $j \in [\ell]$ such that $R_{j} \cap \mathcal{U}_i = \emptyset$ and all $t \in [I_i^+ - \len{I_i}/4, I_i^+]$, we have $R_{j} \cup \{a_j\} \cup K_j \subseteq X'_t$.
\end{enumerate}
\end{lemma}
\begin{proof}
We first define events as follows. For all $h \in \Zzero$, let $J_h^- = I_i^- + \len{I_i}/2 + h(\log n)^7$, let $J_h^+ = J_h^- + (\log n)^7$, and let $J_h = (J_h^-, J_h^+]$.
\begin{itemize}
\item $\mathcal{E}_1$: $\len{\{t \in (I_i^-, \TTend{i}] \mid v^* \notin X_t'\}} \le \gamma_i$.
\item $\mathcal{E}_2$: for all $t \in I_i$ and all $j \in [\ell]$, $|K_j \setminus X'_t| \le 2\Delta$.
\item $\mathcal{E}_3$: for all $j \in [\ell]$ and all $t \in [I_i^-, I_i^- + 2\len{I_i}/3]$, there exists $t' \in [t, t+(\log n)^7)$ such that $K_j \subseteq X_{t'}'$.
\item $\mathcal{E}_4$: $|\mathcal{T}^i| < \beta_{i+1}$.
\item $\mathcal{E}_5$: for all $j \in [\ell]$, $|\mathcal{T}^{i,j}| < \alpha_{i+1}$.
\item $\mathcal{E}_6$: for all $v \in R_1 \cup \dots \cup R_{\ell}$, $\vsmc{v^*}{v}$ triggers in $(t^+_i, t^+_i + \len{I_i}/3]$.
\item $\mathcal{E}_7$: for all $j \in [\ell]$, $\Wnoh{i,j} \le \alpha_i\sqrt{n}(\log n)^4$  
\item $\mathcal{E}_8$: for all $j \in [\ell]$ and all $h \in \{0, \dots, \floor{3\alpha_i\sqrt{n}(\log n)^4}\}$, some clock $\vmc{v}{a_j}$ triggers in the interval $J_h$.
\end{itemize}
Note that by Observation~\ref{obs:one}, $\filt_{I_i^-}(\Psi(X'))$ is uniquely determined by $\filt_{I_i^-}(X')$ and vice versa. Thus   
the value~$f$ from the statement of the lemma is the unique value of~$f$
such that $\filt_{I_i^-}(X') = f$ if and only if $\filt_{I_i^-}(\Psi(X'))=f^*$, allowing us to apply results like Lemma~\ref{lem:vcent-time} to $X'$ even though we are conditioning on $\filt_{I_i^-}(\Psi(X'))$ rather than $\filt_{I_i^-}(X')$.

Observe that by a union bound over all $j \in [\ell]$ and $h \in \{0, \dots, \floor{3\alpha_i\sqrt{n}(\log n)^4}\}$, 
\begin{equation}\label{eqn:E_7}
\Pr(\mathcal{E}_8 \mid \filt_{I_i^-}(\Psi(X')) = f^*) \ge 1 - \ell (\floor{3\alpha_i\sqrt{n}(\log n)^4}+1) e^{-rm(\log n)^7} \ge 1-e^{-(\log n)^2}.
\end{equation}
Moreover, by $\Pcliquefull{i}$, $f$ satisfies the conditions of Lemma~\ref{lem:iteration-cliques}, which together with a union bound over $j\in[\ell]$ implies that
\begin{equation}\label{eqn:E_2}
\Pr(\mathcal{E}_2 \cap \mathcal{E}_3 \mid \filt_{I_i^-}(\Psi(X')) = f^*) \ge 1 - \ell e^{-\frac{1}{2}(\log n)^3} \ge 1 - e^{-(\log n)^2}.
\end{equation}
Thus by Lemma~\ref{lem:vcent-time}, \eqref{eqn:E_2}, Lemma~\ref{lem:star-clock-bound}, Lemma~\ref{lem:clique-active-bound} and \eqref{eqn:E_7} (applied in order),
\begin{align}
\Pr\left(\bigcap_{s=1}^8 \mathcal{E}_s \,\Bigg|\, \filt_{I_i^-}(\Psi(X')) = f^* \right) 
&\ge  1 - \Pr(\overline{\mathcal{E}_1} \mid \filt_{I_i^-}(\Psi(X')) = f^*) 
        - \Pr(\overline{\mathcal{E}_2 \cap \mathcal{E}_3} \mid \filt_{I_i^-}(\Psi(X')) = f^*)\nonumber\\
& \qquad- \Pr(\overline{\mathcal{E}_4 \cap \mathcal{E}_5 \cap \mathcal{E}_6} \mid \filt_{I_i^-}(\Psi(X')) = f^*) 
        - \Pr(\overline{\mathcal{E}_7} \mid \filt_{I_i^-}(\Psi(X')) = f^*)\nonumber\\
& \qquad - \Pr(\overline{\mathcal{E}_8} \mid \filt_{I_i^-}(\Psi(X')) = f^*)\nonumber\\
&\ge 1 - 8e^{-(\log n)^2} - 2\ell e^{-(\log n)^3} \ge 1 - 10e^{-(\log n)^2}.\label{eqn:bigunionbound}
\end{align}
It therefore suffices to show that when $\mathcal{E}_1 \cap \dots \cap \mathcal{E}_8$ occurs, (\ref{it:branches-clear-1}) and~(\ref{it:branches-clear-2}) both hold.
 
Suppose $\mathcal{E}_1 \cap \dots \cap \mathcal{E}_8$ occurs. First note that since $\mathcal{E}_1$ occurs, by Observation~\ref{obs:one}, for all $t \in (I_i^-, \TTend{i}]$,
\[
\textnormal{$v^*$ spawns a non-mutant at time $t$ only if $\siout{v^*}{t} \in \mathcal{T}^i$.}
\]
Since $\mathcal{E}_4$ occurs, it follows that $v^*$ spawns  fewer than $\beta_{i+1}$ non-mutants in the interval $(I_i^-, \TTend{i}]$, and so \myref{D}{D:beta} does not hold at time $\TTend{i}$. Likewise, since $\mathcal{E}_1$ occurs, for all $t \in (I_i^-, \TTend{i}]$ and all $j \in [\ell]$,
\[
\textnormal{$v^*$ spawns a non-mutant onto a vertex in $R_{j}$ at time $t$ only if $\siout{v^*}{t} \in \mathcal{T}^{i,j}$.}
\]
and so since $\mathcal{E}_5$ occurs it follows that $v^*$ spawns fewer than $\alpha_{i+1}$ non-mutants onto vertices in $R_{j}$ in the interval $(I_i^-, \TTend{i}]$. Hence \myref{D}{D:alpha} does not hold at time $\TTend{i}$. Finally, since $\mathcal{E}_2$ occurs and $\TTend{i} \in I_i$, \myref{D}{D:clique-empty} does not hold at time $\TTend{i}$. Since none of \myref{D}{D:beta}--\myref{D}{D:clique-empty} hold at $\TTend{i}$, \myref{D}{D:end} must hold at $\TTend{i}$ and so $\TTend{i} = I_i^+$.

Since $\mathcal{E}_1$ occurs and $\TTend{i} = I_i^+$, it follows that 
\begin{equation}\label{eqn:mainprop}
\textnormal{Throughout $I_i$, $v^*$ only spawns non-mutants onto vertices in $\mathcal{U}_i$.}
\end{equation} 
Also since $\mathcal{E}_1$ occurs and $\TTend{i} = I_i^+$, we have 
\[\len{\{t \in (I_i^-, I_i^- + \len{I_i}/2] \mid v^* \in X'_t\}} \ge \frac{\len{I_i}}{2} - \gamma_i > \frac{\len{I_i}}{3}.\]
Thus since $\mathcal{E}_6$ occurs, by Observation~\ref{obs:one}, $v^*$ spawns a mutant onto every vertex in $(R_1 \cup \dots \cup R_{\ell}) \setminus \mathcal{U}_i$ in the interval $(I_i^-, I_i^- + \len{I_i}/2]$. Thus by \eqref{eqn:mainprop}, every vertex in $(R_1 \cup \dots \cup R_{\ell}) \setminus \mathcal{U}_i$ is a mutant throughout $[I_i^- + \len{I_i}/2, I_i^+]$, and in particular (\ref{it:branches-clear-1})~holds. 

To prove (\ref{it:branches-clear-2}), let $i \ge 6$ be an integer and let $j \in [\ell]$ be such that $R_{j} \cap \mathcal{U}_i = \emptyset$. By the above, we have $R_{j} \subseteq X'_t$ for all $t \in [I_i^- + \len{I_i}/2, I_i^+]$. 

Let
\begin{align*}
\mathcal{J} &= \{J_h \mid 0 \le h \le \floor{3\alpha_i\sqrt{n}(\log n)^4}, K_j \textnormal{ becomes active in $J_h$}\}\\
\mathcal{J}' &= \{J_h \mid 0 \le h \le \floor{3\alpha_i\sqrt{n}(\log n)^4}, \textnormal{there exists $t \in J_h$ such that $K_j$ is active at $t$.}\}
\end{align*}
From $i \ge 6$, we  have $\alpha_i \le \sqrt{n}/(\log n)^{12}$ and hence $J_h \subseteq (I_i^- + \len{I_i}/2, I_i^- + 2\len{I_i}/3]$ for all $h \le \floor{3\alpha_i\sqrt{n}(\log n)^4}$. Since $\mathcal{E}_3$ occurs, we have $|\mathcal{J}'| \le 2|\mathcal{J}|$. Since $\mathcal{E}_7$ occurs, it therefore follows that $|\mathcal{J}'| \le 2\alpha_i\sqrt{n}(\log n)^4$, and in particular there exists $h \in \{0, \dots \floor{3\alpha_i\sqrt{n}(\log n)^4}\}$ such that $K_j$ is inactive throughout $J_h$.  Moreover, since $\mathcal{E}_8$ occurs and $R_j \subseteq X_t$ for all $t \in J_h$, $a_j$ must become a mutant in $J_h$. Thus $R_{j} \cup \{a_j\} \subseteq X'_t$ for all $t \in [I_i^- + 2\len{I_i}/3, I_i^+]$. 

Finally, since $\mathcal{E}_3$ occurs, it follows that there exists $t \in [I_i^- + 2\len{I_i}/3, I_i^- - \len{I_i}/4]$ such that $K_j$ is full of mutants at time $t$. It follows that $R_{j} \cup \{a_j\} \cup K_j \subseteq X'_t$ for all $t \in [I_i^+ - \len{I_i}/4, I_i^+]$. Thus (\ref{it:branches-clear-2})~holds, and so the result follows from \eqref{eqn:bigunionbound}.
\end{proof}

We are now at last in a position to prove Lemma~\ref{lem:main-iteration}, which completes the proof of Theorem~\ref{thm:megastar}.

{\renewcommand{\thetheorem}{\ref{lem:main-iteration}}
\begin{lemma}
\statemainiteration
\end{lemma}
\addtocounter{theorem}{-1}
}
\begin{proof}
Consider the process $\Psi(X')$. We define the following events.
\begin{itemize}
\item $\mathcal{E}_1$: for all $j \in [\ell]$, $|K_j \setminus X'_{I_i^+}| \le \Delta$.
\item $\mathcal{E}_2$: $|\mathcal{U}_i| < \beta_{i+1}$.
\item $\mathcal{E}_3$: for all $j \in [\ell]$, $|\mathcal{U}_i \cap R_{j}| < \alpha_{i+1}$.
\item $\mathcal{E}_4$: $(R_1 \cup \dots \cup R_{\ell}) \setminus X'_{I_i^+} \subseteq \mathcal{U}_i$.
\item $\mathcal{E}_5$: either $i < 6$ or for all $j \in [\ell]$ such that $R_{j} \cap \mathcal{U}_i = \emptyset$ and all $t \in [I_i^+ - \len{I_i}/4, I_i^+]$, we have $R_{j} \cup \{a_j\} \cup K_j \subseteq X'_t$.
\item $\mathcal{E}_6$: $\mathcal{U}_i = \emptyset$.
\item $\mathcal{E}_7$: for some $v \in K_1 \cup \dots \cup K_\ell$, $\vmc{v}{v^*}$ triggers in $(I_i^+ - \len{I_i}/4, I_i^+]$.
\end{itemize}

Note that by Observation~\ref{obs:one}, there exists a unique value of $f^*$ such that $\filt_{I_i^-}(\Psi(X')) = f^*$ if and only if $\filt_{I_i^-}(X')=f$.  By $\Pcliquefull{i}$, $f$ satisfies the conditions of Lemma~\ref{lem:iteration-cliques}, and so by Lemma~\ref{lem:iteration-cliques}(\ref{it:iteration-cliques-2}) and a union bound over $j\in [\ell]$ we have 
\begin{equation}\label{eq:zxasqw12}
\Pr(\overline{\mathcal{E}_1} \mid \filt_{I_i^-}(X') = f) \geq 1-\ell e^{-\frac{1}{2}(\log n)^3}\geq 1-e^{-(\log n)^2}.
\end{equation}

Thus, by \eqref{eq:zxasqw12}, Lemma~\ref{lem:star-clock-bound}(\ref{it:star-clock-bound-1}) and (\ref{it:star-clock-bound-2}), and Lemma~\ref{lem:branches-clear} (applied in order) together with a union bound, we have
\begin{align}\label{eqn:main-events}
\Pr\left(\bigcap_{s=1}^5 \mathcal{E}_s \,\Bigg|\, \filt_{I_i^-}(X') = f \right) 
&\ge  1 - \Pr(\overline{\mathcal{E}_1} \mid \filt_{I_i^-}(X') = f) 
        - \Pr(\overline{\mathcal{E}_2 \cap \mathcal{E}_3} \mid \filt_{I_i^-}(X') = f)\nonumber\\
&\qquad\qquad\qquad - \Pr(\overline{\mathcal{E}_4 \cap \mathcal{E}_5} \mid \filt_{I_i^-}(X') = f)\nonumber\\
&\ge 1 - 15e^{-(\log n)^2}.
\end{align}

Suppose that $\mathcal{E}_1 \cap \dots \cap \mathcal{E}_5$ occurs. We claim that $\filt_{I_{i+1}^-}(X')$ is good. Indeed, since $\mathcal{E}_2 \cap \mathcal{E}_4$ occurs, 
\Ptotalmut{i+1}  is satisfied. Likewise, since $\mathcal{E}_3 \cap \mathcal{E}_4$ occurs, 
\Presmut{i+1} is satisfied. Since $\mathcal{E}_1$ occurs, 
\Pcliquefull{i+1} is satisfied. Finally, since $\mathcal{E}_2$ occurs, we have
\[|\{j \in [\ell]\mid R_{j} \cap \mathcal{U}_i \ne \emptyset\}| \le |\mathcal{U}_i| \le \beta_{i+1}.\]
If $i < 6$, then $\beta_i > \ell$ and so 
\Pmostcliques{i+1}  holds vacuously. If instead $i \ge 6$, then 
\Pmostcliques{i+1}   holds since $\mathcal{E}_5$ occurs. Thus $\filt_{I_{i+1}^-}(X')$ is good, and so the first part of the result follows from \eqref{eqn:main-events}.

Now suppose $\beta_i = \floor{(2\log n)^{2}}$, so   $\gamma_i \le 4(\log n)^{22}/k$ and $i \ge 6$. Note that $t_i^-$ is determined by $f$, and that $f$ does not determine the behaviour of clocks $\vsnc{v^*}{v}$ in $(t_i^-, \infty)$. It follows that conditioned on $\filt_{I_i^-}(X')=f$, $|\mathcal{T}^i|$ follows a Poisson distribution with parameter $\gamma_i$. It therefore follows that
\[\Pr(\mathcal{E}_6 \mid \filt_{I_i^-}(X')=f) = e^{-\gamma_i} \ge 1 - \gamma_i \ge 1 - \frac{4(\log n)^{22}}{k}  \ge \frac{5}{6}.\]
Moreover, it is immediate that
\[\Pr(\mathcal{E}_7 \mid \filt_{I_i^-}(X')=f) = 1 - e^{-r\ell \len{I_i}/4} \ge 5/6.\]
By Lemma~\ref{lem:branches-clear}(\ref{it:branches-clear-2}) and a union bound, it therefore follows that
\begin{equation}\label{eqn:fixation-events}
\Pr(\mathcal{E}_5 \cap \mathcal{E}_6 \cap \mathcal{E}_7 \mid \filt_{I_i^-}(X')=f) \ge 1/2.
\end{equation}

Suppose that $\mathcal{E}_5 \cap \mathcal{E}_6 \cap \mathcal{E}_7$ holds. Since $\mathcal{E}_5 \cap \mathcal{E}_6$ holds and $i \ge 6$, for all $j \in [\ell]$ and all $t \in [I_i^+ - \len{I_i}/4, I_i^+]$ we have $K_j \cup \{a_j\} \cup R_{j} \subseteq X_t'$. Since $\mathcal{E}_7$ holds, it therefore follows that $V(\megastar[\ell]) \subseteq X_{I_i^+}'$. Thus the second part of the result follows from \eqref{eqn:fixation-events}.
\end{proof}

\section{An upper bound on fixation probability of megastars}
\label{sec:megaUB}

The following lemma is similar to Lemma~\ref{lem:star-lower-die-instant} for superstars. 
 
\begin{lemma}\label{lem:megastar-lower-1}
Let $r>1$, and let $k$, $\ell$ and $m$ be arbitrary positive integers. Let $x_0 \in \megastar$ be chosen uniformly at 
random. Let $X$ be a  Moran  process with $G(X) = \megastar$ and $X_0 = \{x_0\}$. Then $X$ goes extinct with 
probability at least $k/(2r(m+k))$. 
\end{lemma}
\begin{proof}
We have
\[\Pr(x_0 \notin R_1 \cup \dots \cup R_\ell) = 1 - \frac{\ell m}{\ell(m+k+1)+1} \ge 1 - 
\frac{m}{m+k} 
= 
\frac{k}{m+k}.\]
Moreover, let $\mathcal{E}$ be the event that $x_0$ dies before spawning a mutant. Then we have
\begin{align*}
\Pr(\mathcal{E} \mid x_0 \in \{a_1, \dots, a_\ell\}) &= m/(m+r),\\
\Pr(\mathcal{E} \mid x_0 \in K_1 \cup \dots \cup K_\ell) &= 1/(1+r),\\
\Pr(\mathcal{E} \mid x_0 = v^*) &= \ell/(\ell + r).
\end{align*}
Thus we have
\[\Pr(\mathcal{E}) \ge \Pr(\mathcal{E} \mid x_0 \notin R_1 \cup \dots \cup R_\ell)\Pr(x_0 \notin R_1 \cup 
\dots \cup R_\ell) \ge \frac{1}{1+r} \cdot \frac{k}{m+k} \ge \frac{k}{2r(m+k)}.\]
Since $X$ goes extinct if $\mathcal{E}$ occurs, the result follows.
\end{proof}

The following lemma is similar to Lemma~\ref{lem:star-lower-die-before-path} for superstars.

\begin{lemma}\label{lem:megastar-lower-2}
Let $r>1$, and let $k$, $\ell$ and $m$ be arbitrary positive integers with  $m \ge 12r$. Let $x_0 \in R_1 \cup \dots 
\cup R_\ell$ be arbitrary. Let $X$ be  the Moran process with $G(X) = \megastar$ and $X_0 = \{x_0\}$. Then $X$ 
goes 
extinct with probability at least $1/(26r^2\ell)$.
\end{lemma}
\begin{proof}
Assume, without loss of generality, that $x_0\in R_1$. Let $\xi=\lfloor m/(2r) \rfloor$, $t^* = m/(4r^2)$ and  
 $J = [0,t^*]$.
For all $t\geq0$,
let $\calE^1$, $\calE^2$ and $\calE^3_t$  be events defined as follows. 
\begin{description}
\item $\mathcal{E}^1$:  $\vnc{v^*}{x_0}$   triggers in   $J$.
\item $\mathcal{E}^2$:  $\vmc{x_0}{a_1}$ triggers at most $\xi$ times in~$J$.
\item $\mathcal{E}^3_t$:   
$\min \{ t' > t \mid \mbox{for some $u \neq x_0$, $\vnc{u}{a_1}$  triggers at $t'$} 
\}
 <  \\ \mbox{\qquad\qquad\qquad} \min \{t' > t \mid \mbox{some clock $\vmc{a_1}{v}$  triggers at $t'$}\}$.
\end{description}
Finally, let 
 $\Tvone{i}$ be the $i$'th time at which the clock $\vmc{x_0}{a_1}$ triggers and define
$\mathcal{E}^3 =  \bigcap_{i=1}^{ \xi} \mathcal{E}^3_{\Tvone{i}}$.

As in the proof of Lemma~\ref{lem:star-lower-die-before-path} $X$ goes extinct if the events $\mathcal{E}^1$, 
$\mathcal{E}^2$ and $\mathcal{E}^3$ occur. Furthermore, these events have exactly the same probability as 
the corresponding events in the proof of Lemma~\ref{lem:star-lower-die-before-path}, so the result follows.
\end{proof}

Our upper bound on fixation probability now follows easily from Lemmas~\ref{lem:megastar-lower-1} and 
\ref{lem:megastar-lower-2}. 

\begin{theorem}\label{lem:megastar-lower-final}
Let $r>1$, and let $k$, $\ell$ and $m$ be arbitrary positive integers. Let $x_0 \in V(\megastar)$ be chosen uniformly 
at 
random. Let $X$ be  the Moran process with $G(X) = \megastar$ and $X_0 = \{x_0\}$. Then  $X$ fixates with 
probability at most $1-1/(52r^2\sqrt{n})$,
where $n= |V(\megastar)|$.
\end{theorem}
\begin{proof}
We prove the result by dividing into three cases.

\medskip\noindent\textbf{Case 1: $\boldsymbol{n \le 144r^2}$.} In this case, $x_0$ dies with rate at least $1/n$ and 
spawns a mutant with rate $r$, so $x_0$ dies before spawning a mutant with probability at least
\begin{equation}\label{eqn:megastar-lower-final-1}
\frac{\frac{1}{n}}{\frac{1}{n}+r} \ge \frac{1}{2rn} \ge \frac{1}{24r^2\sqrt{n}}.
\end{equation}

\medskip\noindent\textbf{Case 2: $\boldsymbol{n > 144r^2}$ and $\boldsymbol{m \le k\sqrt{n}}$.} In this case, by 
Lemma~\ref{lem:megastar-lower-1}, $X$ goes extinct with probability at least
\begin{equation}\label{eqn:megastar-lower-final-2}
\frac{k}{2r(m+k)} \ge \frac{1}{2r(\sqrt{n}+1)} \ge \frac{1}{3r\sqrt{n}}.
\end{equation}

\medskip\noindent\textbf{Case 3: $\boldsymbol{n > 144r^2}$, $\boldsymbol{m > k\sqrt{n}}$ and $\boldsymbol{\ell \le 
\sqrt{n}}$.} In this case, we have
\[\Pr(x_0 \in R_1 \cup \dots 
\cup R_\ell) = \frac{\ell m}{\ell(m+1+k)+1} \ge \frac{m}{m+3k} \ge \frac{\sqrt{n}}{\sqrt{n}+3} \ge \frac{1}{2}.\]
Moreover, we have $m \ge k\sqrt{n} \ge 12r$. Hence by Lemma~\ref{lem:megastar-lower-2}, $X$ goes extinct with 
probability at least
\begin{equation}\label{eqn:megastar-lower-final-3}
\frac{1}{2}\cdot \frac{1}{26r^2\ell} \ge \frac{1}{52r^2\sqrt{n}}.
\end{equation}

Since $n \ge \ell m$, we have either $m\le \sqrt{n} \le k\sqrt{n}$ or $\ell \le \sqrt{n}$, and so the above cases are 
exhaustive. The result therefore follows from \eqref{eqn:megastar-lower-final-1}--\eqref{eqn:megastar-lower-final-3}.
\end{proof}

\section{A clarification of the isothermal theorem}
\label{sec:isothermal}

Lieberman, Hauert and Nowak~\cite{LHN2005:EvoDyn} define the Moran
process in the more general setting of weighted graphs.  In this
section, we consider this generalisation.

In the ordinary Moran process, the offspring produced by reproduction
is equally likely to be placed at each neighbour of the reproducing
vertex.  Instead, given a graph~$G$, we may assign to each
edge~$(u,v)$ a weight~$w_{uv}$, taking $w_{uv}=0$ if there is no such
edge.  Without loss of generality, we   require that, for each
vertex~$u$, $\sum_{v\in V(G)} w_{uv} = 1$, so the weights are
probabilities.  When a vertex~$u$ is chosen to reproduce in a weighted
graph, its offspring goes to vertex~$v$ with probability~$w_{uv}$; the
process is, otherwise, identical to the process on unweighted graphs.
Note that the unweighted process is recovered by assigning
$w_{uv} = 1/d^+(u)$ for every edge~$(u,v)$.

A weighted graph~$G$ whose weights are probabilities as above is known
as an \emph{evolutionary graph} and is said to be
\emph{isothermal}~\cite{LHN2005:EvoDyn,SRJ2012:Review} if, for all
vertices~$u$, $\sum_{v\in V(G)} w_{vu}   = 1$.
This corresponds to the 
condition that the 
weighted adjacency matrix  of~$G$ is doubly-stochastic.  Broom and
Rycht\'ar~\cite{BR2008:FixProb} show that an undirected graph, when
considered as a weighted graph with edge weights $w_{uv} = 1/d^+(u)$,
is isothermal if and only if it is regular.

In the supplementary material to~\cite{LHN2005:EvoDyn}, Lieberman et
al.\@ state and prove the ``isothermal theorem'', which states that an
evolutionary graph is ``$\rho$-equivalent to the Moran process'' if
and only if it is isothermal.  Being $\rho$-equivalent to the Moran
process means that, for all sets $X\subseteq V(G)$, the probability of
reaching fixation from the state in which the set of mutants is~$X$ is
$(1-1/r^{|X|})/(1-1/r^n)$.  In particular, this condition implies that
the fixation probability given a single initial mutant placed
uniformly at random is $\rhoreg(r,n) = (1-1/r)/(1-1/r^n)$.

The isothermal theorem has been incorrectly described in the
literature.  This may stem from the ambiguity of the informal statement
of the theorem in the main text of~\cite[p.~313]{LHN2005:EvoDyn}.
Shakarian, Roos and Johnson~\cite{SRJ2012:Review} state the theorem in
the following form, which is very similar to the informal statement by
Lieberman et al.

\begin{proposition} (Theorem 1,\cite{SRJ2012:Review})
\label{prop:isothermal}
    An evolutionary graph is isothermal if and only if the fixation
    probability of a randomly placed mutant is $\rhoreg(r,n)$.
\end{proposition}

It is true that all $n$-vertex connected isothermal graphs do have fixation
probability~$\rhoreg(r,n)$. 
However, the converse direction of the proposition does not hold.
We prove the following.

{\renewcommand{\thetheorem}{\ref{prop:counterexample}}
\begin{proposition}
\stateexamplethm
\end{proposition}
\addtocounter{theorem}{-1}
}

\begin{figure}
\begin{center}
\begin{tikzpicture}[scale=0.8]
    \tikzstyle{vertex}=[fill=black, draw=black, circle, inner sep=2pt]

    \node[vertex] (v0) at  (0,0) [label=240:{$0$}] {};
    \node[vertex] (v1) at  (5,0) [label=300:{$1$}] {};
    \node[vertex] (v2) at (60:5) [label= 90:{$2$}] {};

    \begin{scope}[decoration={markings, mark=at position 0.52 with {\arrow[scale=2]{>}}}]
        \draw[postaction=decorate] (v0) arc (255.52:284.48:10);
        \draw[postaction=decorate] (v1) arc (15.52:44.48:10);
        \draw[postaction=decorate] (v2) arc (135.52:164.48:10);

        \draw[postaction=decorate] (v1) arc ( 75.52:104.48:10);
        \draw[postaction=decorate] (v2) arc (195.52:224.48:10);
        \draw[postaction=decorate] (v0) arc (315.52:344.48:10);
    \end{scope}

    \node at ($(v0)!0.5!(v1)+(270:0.9)$) {$\tfrac12$};
    \node at ($(v0)!0.5!(v2)+(330:0.75)$) {$\tfrac12$};

    \node at ($(v1)!0.5!(v0)+( 90:0.9)$) {$\tfrac34$};
    \node at ($(v1)!0.5!(v2)+( 30:0.75)$) {$\tfrac14$};

    \node at ($(v2)!0.5!(v0)+(150:0.75)$) {$\tfrac34$};
    \node at ($(v2)!0.5!(v1)+(210:0.75)$) {$\tfrac14$};
\end{tikzpicture}
\end{center}

\caption{A $3$-vertex graph which is not isothermal but has
  fixation probability~$\rhoreg(r,3)$.}
\label{fig:non-iso-graph}
\end{figure}
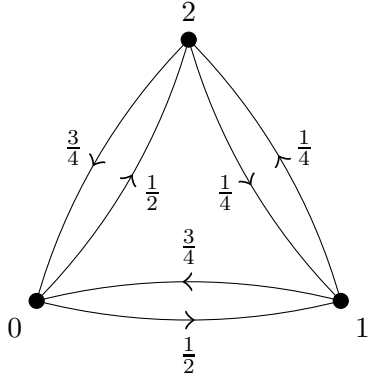

 A counterexample 
to Proposition~\ref{prop:isothermal},
  proving Proposition~\ref{prop:counterexample},
  is the graph shown in
Figure~\ref{fig:non-iso-graph}.  This is an evolutionary graph:
the total weight of outgoing edges from each vertex is~$1$.  It is
not isothermal, since, for example, $w_{10}+w_{20}\neq 1$.  However,
it can be shown that the fixation probability of a randomly placed
mutant with fitness~$r$ is, nonetheless, $\rhoreg(r,n)$.

Towards calculating the fixation probability of the graph shown in
Figure~\ref{fig:non-iso-graph}, let $p_i$ be the probability of
reaching fixation when the initial mutant is at vertex~$i$, for
$0\leq i\leq 2$.  For $0\leq i<j\leq 2$, let $p_{ij}$ be the
probability of reaching fixation from the configuration with mutants
at $i$ and~$j$ and a non-mutant at the remaining vertex.  Observe
that, by symmetry of the graph, $p_1 = p_2$ and $p_{01} = p_{02}$.
The fixation probability of the graph is given by
$\tfrac13(p_0 + p_1 + p_2) = \tfrac13(p_0 + 2p_1)$.

We obtain the following equations:
\begin{align*}
    p_0 &= \tfrac1{r+2}\big(
                  r\,\big(\tfrac12\,p_{01} + \tfrac12\,p_{02}\big)
                  + 2\cdot \tfrac14\, p_0
           \big)
         = \tfrac1{r+2}\big(rp_{01} + \tfrac12p_0\big)\,, \\[0.5ex]
    p_1 &= \tfrac1{r+2}\big(
               r\,\big(\tfrac34\,p_{01} + \tfrac14\,p_{12}\big)
               + \big(\tfrac12 + \tfrac34\big)\,p_1
           \big)\,, \\[0.5ex]
    p_{01} &= \tfrac1{2r+1}\big(
               r\,\big(\tfrac12 + \tfrac14
                       + \big(\tfrac12 + \tfrac34\big)\,p_{01}\big)
               + \tfrac14\,p_0 + \tfrac34\,p_1
           \big) \\
         &\qquad= \tfrac1{2r+1}\big(\tfrac34\,r + \tfrac54\,rp_{01}
               + \tfrac14\,p_0 + \tfrac34\,p_1\big)\,,\\[0.5ex]
    p_{12} &= \tfrac1{2r+1}\big(
               r\,\big(\tfrac34 + \tfrac34 + 2\cdot \tfrac14\, p_{12}\big)
               + \tfrac12 p_1 + \tfrac12 p_2
             \big) \\
         &\qquad= \tfrac1{2r+1}\big(r\,\big(\tfrac32 +
                                        \tfrac12\,p_{12}\big) + p_1\big)\,.
\end{align*}
Rearranging these gives
\begin{align*}
    (2r+3)p_0 &= 2rp_{01}          & (3r+4)p_{01} &= 3r + p_0 + 3p_1  \\
    (4r+3)p_1 &= 3rp_{01} + rp_{12} & (3r+2)p_{12} &= 3r + 2p_1\,.
\end{align*}
Routine solution of this linear system gives
\begin{equation*}
    p_0 = \frac{r^2(2r+1)}{2(r+1)(r^2+r+1)} \qquad
    p_1 = \frac{r^2(4r+5)}{4(r+1)(r^2+r+1)}\,,
\end{equation*}
which yields the fixation probability
\begin{align*}
    \frac13(p_0 + 2p_1) = \frac{r^2(2r+1+4r+5)}{6(r+1)(r^2+r+1)}
    = \frac{r^2}{r^2+r+1} = \frac{r^2(r-1)}{r^3-1} = \rhoreg(r,3)\,.
\end{align*}

\section{Heuristic analysis of superstars}
\label{sec:extra}
  
   As noted in Section~\ref{sec:JLH}, Jamieson-Lane and Hauert\cite{JLH2015:Superstars}  have
already provided a  heuristic analysis of the fixation probability of superstars.
The heuristic analysis contains good intuition.
The purpose of this final section is to explain some of the difficulties that arise
when converting such a heuristic argument to a rigorous proof.
This section is primarily for readers who are already familiar with the argument of~\cite{JLH2015:Superstars}.
This should not be regarded as a criticism of \cite{JLH2015:Superstars} --- that paper
provides an excellent heuristic analysis, so it does what it intends.
Rather, the purpose of this section is to   illustrate the complicating factors that arise in rigorous proofs
(these help to explain why our paper is so long!)
 
The evolution of the discrete-time Moran process on a superstar
is extremely complicated.
To avoid detailed analysis,  Jamieson-Lane and Hauert  
\cite[Appendix~E.2]{JLH2015:Superstars} use a simple random walk to
stochastically 
dominate  (from below)  the 
number of mutants in  reservoirs.  
They
claim that   there is always some forward bias in the actual process, in the sense that 
the number of mutants in reservoirs is always more likely to increase than 
to decrease. They say that the forward bias gets harder to 
quantify as the number of mutants increases, but that it is always positive.
Thus, the dominating walk~$Q$ that they consider   is forward-biased
when there are relatively few mutants, and unbiased when there are more.
More specifically, the dominating chain
$Q(h)$ is a simple random walk on $\{0,\ldots,m \ell\}$.
If $h$ is below some threshold~$\delta$
then $Q$  increases by~$1$ during each (discrete) step with probability 
$\gamma/(1+\gamma)$ (for some $\gamma>1$) and decreases  by~$1$
with probability $1/(1+\gamma)$.
If $h$ is above $\delta$ then it  increases or decreases (by~$1$) with probability~$1/2$.
 
There are several problems with this domination.
 First, the domination is invalid because
 there are actually configurations in which the number of mutants is more likely
to decrease than to increase.   
One such configuration is the configuration in which each reservoir contains 
$\tfrac{m}2$~mutants and
$\tfrac{m}2$~non-mutants
and the  centre
vertex~$v^*$ and all  path vertices are non-mutants.  
It is easy to see that the number of mutants is more likely to decrease than
to increase from this configuration, and even that
the number of mutants is likely to decrease  at least $k/(16 r)$ times before it ever increases.
Here is the  idea.
Before the
mutant population of the reservoirs can possibly increase,
$v^*$~must become a mutant.  
 This takes at least $k$~reproductions: a mutant  in a
reservoir must reproduce and a chain of~$k$ reproductions must  move the mutant down the corresponding path to the centre.  This is
very likely to take at least $nk/(2r)$ steps of the process.  
But during these
$nk/(2r)$ steps, $v^*$~is very likely to be chosen for reproduction
at least $k/(4r)$ times.  Since $v^*$~is a non-mutant throughout
this period, it must send a non-mutant into some reservoir each time
it reproduces.  Since half the reservoir vertices are mutants, it is
very likely that these $k/(4r)$ reproductions of the centre will cause the
mutant population of the reservoirs to decrease, not just once, but at
least $k/(16r)$ times before  it can  even go up at all.
Therefore, the
assumption that the mutant population of the reservoirs is as likely
to increase as to decrease does not hold for all configurations of
mutants. A rigorous proof needs to cover all such possibilities.
 
Even in the early evolution of the process, when there are few 
mutants in reservoirs, there are still problems with making the domination rigorous.
Jamieson-Lane and Hauert say \cite[Section 3.5]{JLH2015:Superstars}
(translating their variable names to ours and adding a little notation for future
reference)
\begin{quote}
At any given time step, the probability of losing the initial mutant in the reservoir   is 
$p_1:=1/(F_t \ell m)$. Based on the dynamics in the path, we derive the per time step probability
that a second mutant is generated in any reservoir  as the product of the
probability that a ``train'' is generated and the probability that the train succeeds
in producing a second mutant, which yields approximately $p_2:= r^4 T/(F_t m \ell)$.
\end{quote}
Here, $F_t$ is taken to be the overall fitness (sum of individual fitnesses)  in the configuration~$X_t$
 and $T$ is the expected  length of a ``train'' which is a chain of mutants at
the end of path.
The dominating Markov chain $Q$ is applied
with $\gamma \sim p_2/p_1 \sim r^4 T$.
Since $Q$ is a Markov chain, the domination is only valid if
it applies step-by-step to every configuration.
It applies, if, from every fixed configuration,
the probability that the number of mutants next goes down is
proportional to~$p_1$ and the probability that
it next goes up is proportional to~$p_2$.
But this is not proved. 
First, note that the event that the number of reservoir mutants
increases does not occur with probability~$p_2$ \emph{at any particular step} 
(conditioned on the configuration prior to the step).  
  Instead,
  the expression given in~$p_2$ is a heuristic  aggregated probability  which may, roughly, apply
 at some step or block of steps in the future.
 In order to rigorously dominate the number of reservoir mutants  using the Markov chain~$Q$
it is necessary to split the process into discrete pieces (whose length 
may be a random variable) so that
the number of reservoir mutants  decreases by at most one in each piece.
 It is important that, conditioned on  any  configuration at the start of any piece,
the probability that the number of reservoir mutants goes up must
be at least $p_2/p_1$ times the probability that it goes down.
The paper does not provide  such a domination.
Nevertheless, we do believe that there is an infinite family of
superstars  that is strongly amplifying.
Our  Theorem~\ref{thm:maintwo} in Section~\ref{sec:megastar}  
demonstrates strong amplification  for megastars.
A similar approach would presumably work for superstars, though of course it would not
guarantee  as strong amplification as Theorem~\ref{thm:maintwo}, 
since this is impossible by Theorem~\ref{thm:mainsuperstar}.

\bibliographystyle{plain}
\bibliography{\jobname}

\begin{thebibliography}{10}

\bibitem{ACN2015:amplifiers}
B.~Adlam, K.~Chatterjee, and M.~A. Nowak.
\newblock Amplifiers of selection.
\newblock {\em Proceedings of the Royal Society~A}, 471(2181), 2015.

\bibitem{ARLV2001:Influence}
Chalee Asavathiratham, Sandip Roy, Bernard Lesieutre, and George Verghese.
\newblock The influence model.
\newblock {\em IEEE Control Systems}, 21(6):52--64, 2001.

\bibitem{Ber2001:Monopolies}
Eli Berge.
\newblock Dynamic monopolies of constant size.
\newblock {\em Journal of Combinatorial Theory, Series B}, 83(2):191--200,
  2001.

\bibitem{Grimmett}
Carol Bezuidenhout and Geoffrey Grimmett.
\newblock The critical contact process dies out.
\newblock {\em Ann. Probab.}, 18(4):1462--1482, 10 1990.

\bibitem{BR2008:FixProb}
M.~Broom and J.~Rycht{\'a}r.
\newblock An analysis of the fixation probability of a mutant on special
  classes of non-directed graphs.
\newblock {\em Proceedings of the Royal Society~A}, 464:2609--2627, 2008.

\bibitem{DGMRSS2013:Superstars}
J.~D{\'\i}az, L.~A. Goldberg, G.~B. Mertzios, D.~Richerby, M.~J. Serna, and
  P.~G. Spirakis.
\newblock On the fixation probability of superstars.
\newblock {\em Proceedings of the Royal Society~A}, 469(2156):20130193, 2013.

\bibitem{DGMRSS2014:approx}
J.~D{\'\i}az, L.~A. Goldberg, G.~B. Mertzios, D.~Richerby, M.~J. Serna, and
  P.~G. Spirakis.
\newblock Approximating fixation probabilities in the generalised {M}oran
  process.
\newblock {\em Algorithmica}, 69(1):78--91, 2014.

\bibitem{DGRS}
Josep D{\'{\i}}az, Leslie~Ann Goldberg, David Richerby, and Maria~J. Serna.
\newblock Absorption time of the {M}oran process.
\newblock {\em Random Structures and Algorithms}, To Appear.

\bibitem{DS}
Richard Durrett and Jeffrey~E. Steif.
\newblock Fixation results for threshold voter systems.
\newblock {\em Ann. Probab.}, 21(1):232--247, 01 1993.

\bibitem{Durrett:NACS}
Rick Durrett.
\newblock Some features of the spread of epidemics and information on random
  graphs.
\newblock {\em Proceedings of the National Academy of Science},
  107(10):4491--4498, 2010.

\bibitem{Fel1968:Probability}
W.~Feller.
\newblock {\em An Introduction to Probability Theory and its Applications},
  volume~I.
\newblock Wiley, 3rd edition, 1968.

\bibitem{Gin2000:GT}
Herbert Gintis.
\newblock {\em Game Theory Evolving: A Problem-Centered Introduction to
  Modeling Strategic Interaction}.
\newblock Princeton University Press, 2000.

\bibitem{Hau2008:EvoDyn}
C.~Hauert.
\newblock Evolutionary dynamics.
\newblock In A.~T. Skjeltorp and A.~V. Belushkin, editors, {\em Proceedings of
  the NATO Advanced Study Institute on Evolution from Cellular to Social
  Scales}, pages 11--44. Springer, 2008.

\bibitem{JLH2015:Superstars}
A.~{Jamieson-Lane} and C.~Hauert.
\newblock Fixation probabilities on superstars, revisited and revised.
\newblock {\em Journal of Theoretical Biology}, 382:44--56, 2015.

\bibitem{JLR}
Svante Janson, Tomasz {\L}uczak, and Andrzej Rucinski.
\newblock {\em Random Graphs}.
\newblock Wiley, 2000.

\bibitem{KKT2003:Influence}
David Kempe, Jon Kleinberg, and {\'E}va Tardos.
\newblock Maximizing the spread of influence through a social network.
\newblock In {\em Proc. 9th ACM International Conference on Knowledge Discovery
  and Data Mining}, pages 137--146. ACM, 2003.

\bibitem{LHN2005:EvoDyn}
E.~Lieberman, C.~Hauert, and M.~A. Nowak.
\newblock Evolutionary dynamics on graphs.
\newblock {\em Nature}, 433(7023):312--316, 2005.
\newblock Supplementary material available at
  \url{http://www.nature.com/nature/journal/v433/n7023/full/nature03204.html}.

\bibitem{Lig1999:IntSys}
Thomas~M. Liggett.
\newblock {\em Stochastic Interacting Systems: Contact, Voter and Exclusion
  Processes}.
\newblock Springer, 1999.

\bibitem{Mertz1}
George~B. Mertzios, Sotiris~E. Nikoletseas, Christoforos Raptopoulos, and
  Paul~G. Spirakis.
\newblock Natural models for evolution on networks.
\newblock {\em Theor. Comput. Sci.}, 477:76--95, 2013.

\bibitem{Mertz2}
George~B. Mertzios and Paul~G. Spirakis.
\newblock Strong bounds for evolution in networks.
\newblock In {\em Automata, Languages, and Programming - 40th International
  Colloquium, {ICALP} 2013, Riga, Latvia, July 8-12, 2013, Proceedings, Part
  {II}}, pages 669--680, 2013.

\bibitem{MU}
Michael Mitzenmacher and Eli Upfal.
\newblock {\em Probability and Computing: Randomized Algorithms and
  Probabilistic Analysis}.
\newblock Cambridge University Press, 2005.

\bibitem{Mor1958:Moran}
P.~A.~P. Moran.
\newblock Random processes in genetics.
\newblock {\em Proceedings of the Cambridge Philosophical Society},
  54(1):60--71, 1958.

\bibitem{Shah}
Devavrat Shah.
\newblock Gossip algorithms.
\newblock {\em Found. Trends Netw.}, 3(1):1--125, January 2009.

\bibitem{SRJ2012:Review}
P.~Shakarian, P.~Roos, and A.~Johnson.
\newblock A review of evolutionary graph theory with applications to game
  theory.
\newblock {\em BioSystems}, 107(2):66--80, 2012.

\end{thebibliography}

\end{document}